\documentclass[graybox,envcountchap,sectrefs]{svmono}
\usepackage[utf8]{inputenc}

\usepackage{type1cm}

\usepackage{makeidx}         
\usepackage{graphicx}        
\usepackage{float}           
\usepackage{multicol}        
\usepackage[bottom]{footmisc}

\usepackage{newtxtext}       %

\usepackage[varvw]{newtxmath}       

\usepackage{amsmath,amssymb,amsthm}
\usepackage{graphicx}
\usepackage{booktabs,array,tabularx,rotating}
\usepackage{longtable}
\usepackage{ragged2e}
\usepackage{caption}
\usepackage{tikz}
\usetikzlibrary{arrows.meta,positioning}
\usepackage{listings,xcolor}
\usepackage[numbers,sort&compress]{natbib}
\usepackage[colorlinks=true,linkcolor=blue!45!black,
            citecolor=blue!45!black,urlcolor=blue!45!black]{hyperref}

\graphicspath{{figures/}}
\theoremstyle{plain}
\theoremstyle{definition}
\theoremstyle{remark}
\newtheorem{openproblem}{Open Problem}[chapter]

\newcommand{\dd}{\,\mathrm{d}}
\newcommand{\R}{\mathbb{R}}
\newcommand{\Prob}{\mathcal{P}}
\newcommand{\Ent}{\operatorname{Ent}}
\newcommand{\Hess}{\operatorname{Hess}}
\newcommand{\Ric}{\operatorname{Ric}}
\DeclareMathOperator{\divg}{div}
\newcommand{\ndivg}{\overline{\divg}}
\newcommand{\ngrad}{\overline{\nabla}}
\newcommand{\Lam}{\Lambda_{\log}}
\newcommand{\KL}{D_{\mathrm{KL}}}
\newcommand{\WW}{\mathcal{W}}
\newcommand{\eps}{\varepsilon}

\mainmatter
\let\cleardoublepage\clearpage

\newcolumntype{L}[1]{>{\RaggedRight\arraybackslash}p{#1}}

\makeindex             

\date{October 18, 2026 (Revised Version)}

\begin{document}

\author{Jinqiao Duan\\Ting Gao\\ Qiao Huang \\ Yuanfei Huang}
\title{Geometric Methods \\for Stochastic Dynamical Systems}
\subtitle{A Concise Introduction}
 
\maketitle

\frontmatter



\newcommand{\be}{\begin{equation}}
\newcommand{\ee}{\end{equation}}

\newcommand{\been}{\begin{enumerate}}
\newcommand{\enen}{\end{enumerate}}

\newcommand{\beeq}{\begin{eqnarray}}
\newcommand{\eneq}{\end{eqnarray}}

\newcommand{\beeqn}{\begin{eqnarray*}}
\newcommand{\eneqn}{\end{eqnarray*}}

\newcommand{\bd}{\begin{displaymath}}
\newcommand{\ed}{\end{displaymath}}


\newcommand{\e}{\epsilon}
\renewcommand{\eps}{\varepsilon}
\renewcommand{\a}{\alpha}
\renewcommand{\b}{\beta}
\newcommand{\om}{\omega}
\newcommand{\Om}{\Omega}
\newcommand{\De}{\Delta}
\newcommand{\de}{\delta}
\newcommand{\Dt}{\Delta t}

\newcommand{\s}{{\sigma}}
\newcommand{\la}{{\lambda}}

\renewcommand{\k}{\kappa}

\renewcommand{\phi}{\varphi}


\def\vX{\boldsymbol{X}}
\def\vx{\boldsymbol{x}}   
\def\RR{\mathbb{R}}

\def\PP{\mathbb{P}}
\def\ud{\mathrm{d}}
\def\vw{\mathbf{w}}

\newcommand{\0}{\mathbf{0}}



\newcommand{\p}{\partial}
\newcommand{\pp}{\frac{\partial}{\partial t}}
\newcommand{\ppn}{\frac{\partial}{\partial n}}

\newcommand{\ddt}{\frac{d}{dt}}

\newcommand{\cF}{\mathcal{F}}
\newcommand{\cG}{{\cal G}}
\newcommand{\cD}{{\cal D}}
\newcommand{\cO}{{\cal O}}

\newcommand{\F}{{\mathcal{F}}}
\newcommand{\B}{{\mathcal{B}}}
\newcommand{\cN}{{\mathcal{N}}}

\newcommand{\cL}{{\mathcal{L}}}

\newcommand{\EX}{{\mathbb{E}}}
\newcommand{\PX}{{\mathbb{P}}}

\newcommand{\sign}{\mbox{sign}}

\newcommand{\grad}{\nabla}
\newcommand{\n}{\nabla}
\newcommand{\curl}{\nabla \times}
\newcommand{\dive}{\nabla \cdot}



\def\cA{{\mathcal A}}
\def\cB{{\mathcal B}}
\def\cC{{\mathcal C}}

\newcommand{\dif}{\mathrm{d}}
\def\mR{{\mathbb R}}
\def\mE{{\mathbb E}}
\newcommand{\me}{\mathrm{e}}

\newcommand{\beQ}{\begin{eqnarray}}
\newcommand{\eeQ}{\end{eqnarray}}

\newcommand{\ceQ}{\begin{eqnarray*}}
\newcommand{\deQ}{\end{eqnarray*}}



\newcommand{\N}{\mathbb{N}}
\renewcommand{\R}{\mathbb{R}}
\newcommand{\Z}{\mathbb{Z}}

\newcommand{\Rn}{\mathbb{R}^n}
\newcommand{\Rone}{\mathbb{R}^1}

\newcommand{\ba}{{\bf a}}

\newcommand{\bu}{{\bf u}}
\newcommand{\bU}{{\bf U}}

\newcommand{\bv}{{\bf v}}
\newcommand{\bV}{{\bf V}}

\newcommand{\bk}{{\bf k}}

\newcommand{\bs}{{\bf s}}

\newcommand{\bz}{{\bf z}}
\newcommand{\bZ}{{\bf Z}}

\newcommand{\bn}{{\bf n}}

\newcommand{\bx}{{\bf x}}
\newcommand{\bX}{{\bf X}}

\newcommand{\bH}{{\bf H}}

\newcommand{\bL}{{\bf L}}

\newcommand{\bg}{{\bf g}}

\newcommand{\bj}{{\bf j}}

\newcommand{\br}{{\bf r}}


\renewcommand{\o}{\overline }


\newcommand{\Ae}{\emph{A}}
\newcommand{\Rb}{\mathbf{R}}
\newcommand{\Tb}{\mathbf{T}}
\newcommand{\Zb}{\mathbf{Z}}
\newcommand{\Nb}{\mathbf{N}}

\def\pvs{\vspace{11pt}}
\def\phs{\hspace{11pt}}


\newcommand{\uk}[1]{\ensuremath{u^{(#1)}(t,\omega)}}
\newcommand{\hse}{\ensuremath{h^s(\xi,\omega)}}
\newcommand{\hsk}[1]{\ensuremath{h^{(#1)}(\xi,\omega)}}

\newcommand{\sz}{\ensuremath{ {\int_s^t z(\theta_r (\omega))\,dr}}}
\newcommand{\sZ}{\ensuremath{ {\int_s^t Z(\theta_r (\omega))\,dr}}}
\newcommand{\zz}[1]{\ensuremath{{z(\theta_{#1} (\omega))}}}
\newcommand{\ZZ}[1]{\ensuremath{{Z(\theta_{#1} (\omega))}}}
\newcommand{\fu}[1]{\ensuremath{{F_u^{u_0 (#1)}}}}
\newcommand{\fus}[2]{\ensuremath{{\int^0_{#2} F_u^{u_0 (#1)}\,d{#1}}}}
\newcommand{\fuss}[2]{\ensuremath{{\int^{#2}_0 F_u^{u_0 (#1)}\,d{#1}}}}

\newcommand{\fuu}[1]{\ensuremath{{F_{uu}^{u_0(#1)}}}}
\newcommand{\rb}{\right)}
\newcommand{\lb}{\left(}
\newcommand{\rB}{\right]}
\newcommand{\lB}{\left[}

\newcommand{\nb}{\mathbf{n}}
\newcommand{\ub}{\mathbf{u}}
\newcommand{\xb}{\mathbf{x}}
\newcommand{\xnb}{\mathbf{x}_0}
\newcommand{\GaB}{\mathbf{\Gamma}}

\newcommand{\bo}{\mathcal {O}}
\newcommand{\so}{\mathcal {o}}

\newcommand{\BEN}{\begin{equation*}}
\newcommand{\EEN}{\end{equation*}}
\newcommand{\BAL}{\begin{align}}
\newcommand{\EAL}{\end{align}}
\newcommand{\BAN}{\begin{align*}}
\newcommand{\EAN}{\end{align*}}


\def\Var{\mbox{Var}}
\def\Cov{\mbox{Cov}}
\def\Corr{\mbox{Corr}}
\def\Tr{\mbox{Tr}}

\def\<{\left<}

\def\>{\right>}



\renewcommand{\E}{\mathbb E}
\renewcommand{\P}{\mathbb P}
\newcommand{\id}{\mathbf{Id}}
\renewcommand{\Ric}{\mathrm{Ric}}
\newcommand{\Pred}{\mathcal P}
\newcommand{\A}{\mathcal A}
\newcommand{\pt}{\partial}
\renewcommand{\D}{\mathbf{D}}
\newcommand{\Sym}{\mathrm{Sym}}
\newcommand{\vf}[1]{\frac{\partial}{\partial{#1}}}

\newtheorem{assumption}{Assumption}[chapter]

%
%

\preface
 
\textbf{Geometric Structures: Heteroclinic Orbits}

In deterministic dynamical systems, heteroclinic orbits are trajectories in the state space that connect distinct equilibrium states or invariant sets, over infinite time horizon. They are part of stable and unstable invariant manifolds. Serving as the geometric skeleton of a dynamical system, these connecting orbits do more than just define deterministic paths of state transitions; they act as a core mechanism driving global bifurcations, chaotic phenomena, and complex spatio-temporal behaviors. 
Consequently, geometric methods are indispensable for analyzing, predicting, and mitigating the complex behaviors inherent in nonlinear systems.

\textbf{The Most Probable Transition Path}

When noise is taken into account for modeling nonlinear phenomena, the geometric landscape shifts. In a stochastic dynamical system, a metastable state refers to a state or region in the state space where the system resides for a prolonged duration before undergoing a rare, fluctuation-induced transition to another regime, over finite time horizon. This phenomenon emerges from the intricate interplay between nonlinearity and uncertainty. For instance, a stable equilibrium of a underlying deterministic system manifests as a metastable state in its stochastic counterpart. When the deterministic system is gradient-like, these metastable states correspond precisely to the local minima of a potential energy landscape.
  
In this framework, the most probable transition path represents the optimal trajectory connecting one metastable state to another. In the case of Gaussian noise, this path minimizes the classical Onsager–Machlup action functional, effectively serving as the most likely route across the energy barrier.

\textbf{Stochastic Geometric Mechanics}

Stochastic geometric mechanics extends classical relations among action, dynamics, and geometry to random motion. In the mean-derivative approach, conditional mean velocities replace ordinary velocities, linking expected-action principles with stochastic Lagrangian and Hamiltonian equations and second-order Hamilton--Jacobi theory. Quadratic variation enters the geometric description explicitly, while second-order differential geometry provides a coordinate-independent framework on manifolds. Schr\"odinger's problem and Euclidean quantum mechanics have provided both the conceptual guidance and the impetus for developing this stochastic geometric framework.

\textbf{Schr\"odinger Bridges and Information Geodesics}

Lifting our perspective from individual sample paths in the state space to the infinite-dimensional space of probability densities allows for a more comprehensive geometric treatment of these transitions:

\textit{Schr\"odinger Bridges:} Traditionally framed as the minimizer of the Kullback–Leibler  divergence (relative entropy) subject to marginal constraints, a Schr\"odinger Bridge represents the optimal transition path between two boundary probability distributions. Notably, if the deterministic metastable states are idealized as Dirac delta distributions, the Onsager–Machlup most probable transition path can be mathematically recovered as a special case of the Schr\"odinger bridge.

\textit{Information Geodesics:} Beyond the Kullback–Leibler divergence, the space of probability densities can be endowed with richer dynamical information via $\alpha$-divergences. These functionals generalize relative entropy, quantifying the discrepancy between two probability densities, and are intimately related to   statistical mechanics (e.g., Tsallis and R\'enyi entropies). In nonequilibrium systems, they characterize the generalized thermodynamic cost of state transitions. Varying the parameter $\alpha$ yields distinct path-selection mechanisms and evolution modes. The minimizer of such a divergence functional defines an optimal distribution path, which we refer to as an ``information geodesic."

\medskip

 An earlier version of the foundational material for this book was published as ``Geometric Methods for Stochastic Dynamical Systems'' (by Jinqiao Duan and Hui Wang), appearing as Chapter 1 (pages 1--24) in \textit{Stochastic PDEs and Modelling of Multiscale Complex System} (Editors: Wei Wang, Xiaopeng Chen, and Yan Lv), World Scientific, 2019.

\bigskip 

\textbf{Who is this book for?}


 There is  growing interest in stochastic dynamics in the applied mathematics, interdisciplinary science, and artificial intelligence communities.
This book is written primarily for applied mathematicians and scientists who may not have  the   necessary background to go directly to advanced reference books or research literature in geometric methods for stochastic dynamics. 

Our goal   is to present a few selected geometric methods  for understanding solutions of stochastic differential equations. Some necessary background is reviewed, while more advanced materials are referred to relevant literature.  
It is our hope that   this book will provide a guide to    the reader in accessing  advanced monographs    and   research literature in  stochastic dynamics. We have tried to strike a balance between mathematical precision and accessibility for the readers of this book. For example,  some proofs are presented, whereas some are outlined and others are directed to references.  Some definitions are presented in   separate paragraphs starting with \emph{Definition}, but   others are introduced less formally as they occur in the body of the text.  



This book may be used as a textbook or  a reference for  advanced  doctoral     students, postdoctoral fellows,        researchers and   practitioners  in applied mathematics, machine learning, artificial intelligence, engineering  and  applied science.

\bigskip

\textbf{What does this book do?}

 After discussing motivation  and reviewing stochastic differential equations  (Chapter 1), we focus on three topics:

  $\bullet$  \textit{The Most Probable Dynamics via the Onsager-Machlup Action Functional (Chapter 2)}: This chapter examines  the most probable transition paths between metastable states, as the minimizer of the Onsager-Machlup action functional (which is like a Lagrangian action functional in classical mechanics).

  $\bullet$ \textit{Stochastic Variational and Stochastic Geometric Mechanics (Chapter 3):} This chapter develops stochastic Hamilton's and scaled Maupertuis's principles using forward mean derivatives. It relates Lagrangian and Hamiltonian dynamics to second-order Hamilton--Jacobi equations, discusses mechanical Lagrangians and the distinction between Schr\"odinger bridges and Onsager--Machlup most probable curves, and surveys stochastic geometric mechanics on manifolds.

   $\bullet$ \textit{Schr\"odinger Bridges and Information Geodesics in the Space  of Probability Densities (Chapter 4):} This chapter is an introduction to local and nonlocal Otto calculus, Schrodinger bridges, and information geodesics.


At the end of every chapter there is  a list of Problems, to inspire the readers.


\bigskip

\textbf{Why is this book useful for  artificial intelligence and machine learning? }

Stochastic dynamical systems stand at a vibrant and rapidly evolving intersection of mathematics and machine learning. They provide a rigorous foundational framework for extracting latent governing laws from noisy, high-dimensional data and for predicting complex, random evolutionary behaviors over time.

In particular, the mathematical paradigms of the  ``connecting orbits" --- the most probable transition paths, Schr\"odinger bridges, and information geodesics --- have recently emerged as a powerful cornerstone for deep generative modeling in artificial intelligence (such as diffusion models and flow matching). Beyond data synthesis, this framework equips machine learning with the prescriptive tools necessary for the early warning and active mitigation of critical transitions—or tipping phenomena—in complex engineering and natural systems.

These stochastic ``connecting orbits" delineate the optimal evolutionary trajectory of a complex system traversing between two observed macroscopic states (e.g., from a ``Healthy Brain" state to an ``Epileptic Seizure," or from a ``Stable Current" to a ``Collapsed Atlantic Circulation"). Crucially for modern AI architectures, these boundary states are typically available and modeled as empirical probability distributions or data manifolds, making the geometric and variational methods developed in this book  applicable to data-driven discovery and intelligent mitigation.

\bigskip

\textbf{What prerequisites are assumed?}

 For the reader, it is desirable to have   basic knowledge   of ordinary     differential equations, probability, and stochastic differential equations.   Realizing that some readers may not be familiar with stochastic differential equations, we   review this topic in Chapter 1.

\bigskip

\textbf{Notations} 


$\|\cdot\|$: Length or norm in Euclidean space $\Rn$

$A$ or $\mathcal{L}$: Generator for the solution process of a stochastic differential equation

$\E$:   Expectation 

$\E_x$: Conditional expectation, i.e., expectation conditioned on the initial state $x$.
        For example,  $\EX_x [f(X_t)] \triangleq  \EX [f(X_t) | X_0=x]$.

$\mathcal{N}(m, Q)$: Normal (or Gaussian) distribution with mean vector $m$ and covariance matrix $Q$

ODEs: Ordinary Differential Equations 

\noindent $(\Omega, \F, \P)$:   Probability space   equipped with a sample space $\Omega$,  a filtration $\{\F_t\}_{t \in \R}$, and probability $\P$ 

$\P$:  Probability or probability measure 

PDEs: Partial Differential Equations

$\R^n$:   Euclidean space of $n-$dimensions
$\mathcal{P}(\R^n)$:   Space of probability densities on $\R^n$  
$\mathcal{P}_2(\R^n)$:   Space probability densities on $\R^n$, with finite second moment

SDEs: Stochastic Differential Equations 

$W_2(\mu, \nu)$:  The  $2$-Wasserstein distance  between probability densities $\mu, \nu$ on $\R^n$

\bigskip

\textbf{Acknowledgments}

Part or all of the materials in this book were presented at the following events:   SIAM Conference on Applications of Dynamical Systems, May 2023, Portland, USA; 43rd Conference on Stochastic Processes and their Applications (SPA2023), July 2023, University of Lisbon, Portugal;  International Congress on Industrial and Applied Mathematics, August 2023, Tokyo, Japan; SIAM Conference on Applications of Dynamical Systems, May 2025, Denver, USA; the 8th International Conference on Random Dynamical Systems, July 2025, University of Konstanz, Germany; A Journey from Spin Glasses to Neural Networks, December 2025, Sapienza University of Roma, Italy; Machine Learning and Dynamical Systems (MLDS 5) \& Differential Equations for Data Science (DEDS 2026), February 2026, Kyoto University, Japan; the 15th AIMS Conference on Dynamical Systems, Differential Equations and Applications, July 2026, Athens, Greece; Second China-Spain Workshop on New Advances in Contact and Symplectic Hamiltonian Dynamics, July 2026, at the Institute of Mathematical Sciences (IcMat), Madrid, Spain; and the Tianyuan Mathematics Summer School on ``Dynamical Behaviors of Infinite Dimensional Dynamical Systems," July 2026, at the China University of Geosciences, Wuhan, China. We thank the event organizers for the opportunity to share our ideas and receive feedback from the participants.

\medskip

 An earlier version of the foundational material for this book was published as ``Geometric Methods for Stochastic Dynamical Systems'' (by Jinqiao Duan and Hui Wang), appearing as Chapter 1  in \textit{Stochastic PDEs and Modelling of Multiscale Complex System} (Editors: Wei Wang, Xiaopeng Chen, and Yan Lv), World Scientific, 2019.

\medskip

We would like to thank our many collaborators, friends and colleagues for inspiring discussions, especially Marc Arnaudon, Yuri Bakhtin, Peter Bates, Peter Baxendale, Nils Berglund, Alexandra Blessing, Dirk Blomker, Michal Branicki, Ying Chao, Linan Chen, Zhen-Qing Chen, Giovanni Conforti, Hans Crauel, Dan Crisan, Ana Bela Cruzeiro, Jianbo Cui, Manuel de Leon, Manfred Denker, Maximilian Engel, Chunrong Feng, Franco Flandoli, Fran\c{c}ois Gay-Balmaz, Hongjun Gao, Babara Gentz, Barbara Gentz, Martin Hairer, Darryl Holm, Peter Imkeller, Ioannis (Yannis) Kevrekidis, Peter E. Kloeden, Christian Kuehn, Rachel Kuske,   Jeroen Lamb, Christian L\'eonard, Qianxiao Li, Wuchen Li, Xiang-Dong Li, Xiaofan Li, Xue-Mei Li, Kening Lu, Valerio Lucarini, Yan Luo, Chenchen Mou, Qing Nie, Henri Orland, Juan-Pablo Ortega, Ilya Pavlyukevich, Paolo Piccione, Paolo Piccione, Paolo Piccione, Nicolas Privault, Huijie Qiao, Tudor Ratiu, Michael Scheutzow, Larissa Serdukova, Renming Song, Richard Sowers, Dmitry Treschev, Max von Renesse, Feng-Yu Wang, Hangxiao Wang, Guowei Wei, Pingyuan Wei, Steve Wiggins, Jiang-Lun Wu,  Jie Wu, Liming Wu, Yunnan Yang, Todd Young, Jean-Claude Zambrini, Huaizhong Zhao, Jiajie Zhu, and Johannes Zimmer.

We have benefited from comments by our current students and postdoctoral fellows, including Wenjie Chen,  Yuanke Chen, Libo Gao,  Shanqi Liu, You Wan, Zhihao Zhang, Feng Zhao, Zhihao Zhao, Miaolei Zheng, and Qiao Zheng.

\medskip

We gratefully acknowledge financial support from the National Natural Science Foundation of China (NSFC, Grant No. W2541005); the Guangdong Provincial Key Laboratory of Mathematical and Neural Dynamical Systems (Grant No. 2024B1212010004); the Cross-Disciplinary Research Team on Data Science and Intelligent Medicine (Grant No. 2023KCXTD054); the Guangdong–Dongguan Joint Research Fund (Grant No. 2023A1515140016); the Dongguan Key Laboratory for Data Science and Intelligent Medicine; and the Natural Science Foundation of Guangdong Province (Grant No. 2025A1515011188).

Jinqiao Duan would like to thank the Isaac Newton Institute for Mathematical Sciences, Cambridge, for support and hospitality during the programme `Metastability, critical transitions, and long transients: general mathematical framework and relevance for our changing planet' (August 2026), where part of the writing of this book was undertaken. His visit was supported by a Simons Foundation Fellowship and  EPSRC grant  EP/Z000580/1.

Ting Gao would like to thank China Scholarship Council for the grant sponsor when she visited the National University of Singapore where part of the writing of this book was undertaken. She is also supported by the National Natural Science Foundation of China (12401233), the Hubei Provincial Natural Science Foundation of China (2026AFA034), the National Key R$\&$D Program of China (2021ZD0201300),  Guangdong Major Project of Basic Research (2025B0303000003) and Hubei Key Laboratory of Engineering Modeling and Scientific Computing.

The work of Qiao Huang is supported by 
the National Natural Science Foundation of China under Grant No. 12501241, 
the Basic Research Program of Jiangsu Province under Grant No. BK20251280, 
the Fundamental Research Funds for the Central Universities under Grant No. 2242026RCB0008, 
the Jiangsu Provincial Scientific Research Center of Applied Mathematics under Grant No. BK20233002 
and 
the Start-up Research Fund of Southeast University under Grant No. RF1028624194.

\vspace{\baselineskip}
\begin{flushright}\noindent
Dongguan, China \hfill {\it Jinqiao Duan}\\
Wuhan, China \hfill {\it Ting Gao }\\
Nanjing, China \hfill {\it  Qiao Huang}\\
Pohang, South Korea \hfill {\it Yuanfei Huang}\\
October 2026 
\end{flushright}

\tableofcontents


\chapter*{Notations and Abbreviations}
\addcontentsline{toc}{chapter}{Notations and Abbreviations}
\markboth{Notations and Abbreviations}{Notations and Abbreviations}

\small

\begin{longtable}{@{}p{0.26\textwidth}p{0.69\textwidth}@{}}
\toprule
\textbf{Notation / Abbreviation} & \textbf{Meaning} \\
\midrule
\endfirsthead
\toprule
\textbf{Notation / Abbreviation} & \textbf{Meaning} \\
\midrule
\endhead
\midrule
\multicolumn{2}{r}{\small\itshape Continued on next page} \\
\endfoot
\bottomrule
\endlastfoot

\multicolumn{2}{@{}l}{\textbf{General notation}} \\[2pt]
\addlinespace[2pt]
$\mathbb{R}^n$ & $n$-dimensional Euclidean space. \textit{(Chs.~1--4)} \\
\addlinespace[2pt]
$|\cdot|,\ \|\cdot\|,\ \langle\cdot,\cdot\rangle$ & Euclidean absolute value/norm and inner product; on a Riemannian manifold,
the same notation is also used for the metric norm and inner product when the
meaning is clear from context. \textit{(Chs.~1--4)} \\
\addlinespace[2pt]
$(\Omega,\mathcal F,(\mathcal F_t)_{t\ge0},\mathbb P)$ & A filtered probability space: sample space, $\sigma$-algebra, filtration, and
probability measure. In Chapter~4, $\Omega$ is also used for a canonical path
space; the meaning is determined by context. \textit{(Chs.~1--4)} \\
\addlinespace[2pt]
$\mathbb E,\ \mathbb E_x$ & Expectation and conditional expectation given $X_0=x$, respectively:
$\mathbb E_x[f(X_t)]=\mathbb E[f(X_t)\mid X_0=x]$. \textit{(Chs.~1--4)} \\
\addlinespace[2pt]
$\operatorname{Law}(X)$ & Probability law (distribution) of the random variable or process $X$.
\textit{(Chs.~2--4)} \\
\addlinespace[2pt]
$\mathcal P(E)$ & Space of probability measures on a state space $E$; when a measure has a
density, the same notation is used for its density representation.
\textit{(Chs.~1, 3, 4)} \\
\addlinespace[2pt]
$\mathcal P_2(E)$ & Probability measures on $E$ with finite second moment.
\textit{(Chs.~1, 4)} \\
\addlinespace[2pt]
$\mathcal M_+(E)$ & Finite nonnegative measures on $E$. \textit{(Ch.~4)} \\
\addlinespace[2pt]
$\delta_x$ & Dirac probability measure concentrated at $x$. \textit{(Chs.~1, 4)} \\
\addlinespace[2pt]
$\nabla,\ \nabla^2,\ \Delta,\ \operatorname{div}$ & Gradient, Hessian, Laplacian, and divergence. On a Riemannian manifold these
refer to the corresponding metric operators. \textit{(Chs.~1--4)} \\
\addlinespace[2pt]
$H(f),\ \operatorname{Hess}f$ & Hessian matrix/tensor of a scalar function $f$. \textit{(Chs.~1, 2, 4)} \\
\addlinespace[2pt]
$\operatorname{Tr}$ & Trace of a matrix or linear operator. \textit{(Chs.~1--4)} \\
\addlinespace[2pt]
$\mu\ll\nu$ & $\mu$ is absolutely continuous with respect to $\nu$.
\textit{(Chs.~2--4)} \\
\addlinespace[2pt]
$d\mu/d\nu$ & Radon--Nikodym derivative of $\mu$ with respect to $\nu$.
\textit{(Chs.~2--4)} \\
\addlinespace[2pt]
\midrule

\multicolumn{2}{@{}l}{\textbf{Abbreviations}} \\[2pt]
\addlinespace[2pt]
AIG & accelerated information gradient (flow). \textit{(Ch.~4)} \\
\addlinespace[2pt]
CD & curvature--dimension condition. \textit{(Ch.~4)} \\
\addlinespace[2pt]
EL & Euler--Lagrange. \textit{(Chs.~2--3)} \\
\addlinespace[2pt]
FP & Fokker--Planck. \textit{(Chs.~1--4)} \\
\addlinespace[2pt]
HJ & Hamilton--Jacobi. \textit{(Chs.~2--4)} \\
\addlinespace[2pt]
HJB & Hamilton--Jacobi--Bellman. \textit{(Chs.~2--4)} \\
\addlinespace[2pt]
HWI & entropy--Wasserstein--Fisher-information inequality.
\textit{(Ch.~4)} \\
\addlinespace[2pt]
JKO & Jordan--Kinderlehrer--Otto minimizing-movement scheme.
\textit{(Ch.~4)} \\
\addlinespace[2pt]
KL & Kullback--Leibler divergence (relative entropy).
\textit{(Chs.~3--4)} \\
\addlinespace[2pt]
LSI & logarithmic Sobolev inequality. \textit{(Ch.~4)} \\
\addlinespace[2pt]
MFG & mean-field game. \textit{(Ch.~4)} \\
\addlinespace[2pt]
MLSI & modified logarithmic Sobolev inequality.
\textit{(Ch.~4.2)} \\
\addlinespace[2pt]
MPP & most probable path. \textit{(Ch.~2)} \\
\addlinespace[2pt]
MPTP & most probable transition path. \textit{(Ch.~2; also used in
the book-wide motivation)} \\
\addlinespace[2pt]
OM & Onsager--Machlup. \textit{(Chs.~2--3)} \\
\addlinespace[2pt]
OT & optimal transport. \textit{(Chs.~3--4)} \\
\addlinespace[2pt]
ODE & ordinary differential equation. \textit{(Chs.~1--4)} \\
\addlinespace[2pt]
PDE & partial differential equation. \textit{(Chs.~1--4)} \\
\addlinespace[2pt]
PIDE & partial integro-differential equation.
\textit{(Ch.~4.3)} \\
\addlinespace[2pt]
SB & Schr\"odinger bridge. \textit{(Chs.~3--4)} \\
\addlinespace[2pt]
SBP & Schr\"odinger bridge problem. \textit{(Ch.~4)} \\
\addlinespace[2pt]
SDE & stochastic differential equation. \textit{(Chs.~1--4)} \\
\addlinespace[2pt]
$S\alpha S$ & symmetric $\alpha$-stable. \textit{(Ch.~1)} \\
\addlinespace[2pt]
$T_2$ & Talagrand quadratic transport inequality.
\textit{(Ch.~4)} \\
\addlinespace[2pt]
TV & total variation. \textit{(Ch.~4.3)} \\
\addlinespace[2pt]
\midrule

\end{longtable}

\normalsize

\mainmatter
%
%
%
\chapter{Introduction  }
\label{chapter1intro} 

\abstract{  This chapter first discusses the motivation for applying geometric methods to stochastic dynamical systems, and the content of this book. It then provides a  review of  stochastic differential equations with  both Gaussian Brownian motion and non-Gaussian L\'evy noise. }

\section{Motivation: Geometric Methods for Stochastic Dynamics  }
\label{sec1.0} 

\textbf{Deterministic Dynamics vs. Stochastic Dynamics.} 
Dynamical systems provide   mathematical frameworks for modeling phenomena across science and engineering \cite{GuckenHolmes1983, Wiggins}. 

However, real-world systems are inevitably subject to random influences, including external perturbations, internal fluctuations, and parameter uncertainties. When constructing mathematical models, researchers often omit highly fluctuating, small-scale, or poorly understood processes due to observational limits or computational constraints. While these modeling simplifications are useful, the underlying randomness may have a profound and delicate impact on the system's dynamical evolution. Consequently, accounting for stochasticity is now recognized as critical when modeling complex phenomena in biological, chemical, and physical systems \cite{Arnold, Gentz}.


Stochastic differential equations are models to govern randomly influenced nonlinear systems \cite{kloeden2025stochastic}. While stochastic calculus provides a rigorous theoretical foundation for these equations, a deeper understanding of their dynamical behaviors remains essential—particularly when the systems are driven by non-Gaussian, heavy-tailed fluctuations.
Various mathematical frameworks have already been employed to study dynamics under uncertainty. These include  the analysis of stochastic flows \cite{Baudoin2004,Fang}, topological approaches \cite{li2005sternberg, li2008rotation, liu2008conley, chen2010sufficient}, random periodic structures  \cite{feng2023existence}, Lyapunov spectrum \cite{castro2025conditioned}, stochastic bifurcation \cite{crflan1998, doan2018hopf},  among others. See more advances in \cite{Arnold,  hairer2011asymptotic, huang2025dynamical, kuehn2013mathematical, pelayo2018poincare}. Furthermore,  examining deterministic quantities like mean exit times and escape probabilities offers crucial insights into the system's global behavior \cite[Ch. 5]{Duan2015}. Conceptually, these quantities serve a role analogous to eigenvalues or the Poincar\'e index in deterministic dynamics, or entropy in statistical physics.

\medskip
\textbf{Geometric Methods for Deterministic \& Stochastic Dynamics.} 
Geometric approaches have been  widely utilized to investigate   complex behaviors of deterministic dynamical systems \cite{VIArnold1988, Palis1982, Wiggins2025}.

Geometric  structures, such as random invariant manifolds \cite{Arnold}, have also been a significant topic for stochastic dynamical systems. Additionally, recent efforts have extended these concepts to construct the most probable phase portraits of stochastic dynamical systems \cite{Duan2015, duan2019geometric}.

 \medskip

\textbf{Connecting Orbits as a Geometric Skeleton for Stochastic Dynamics.} 
 Heteroclinic orbits, as connecting orbits,  for a  deterministic dynamical system are     trajectories in the state space that connect distinct equilibrium states or invariant sets. As a geometric skeleton of a dynamical system, heteroclinic orbits, also called connecting orbits, not only define deterministic paths of state transitions but also serve as a core mechanism for global bifurcations, chaotic phenomena, and complex spatio-temporal behaviors.  Under random fluctuations, this geometric skeleton encodes rare transitions  \cite{Bakhtinetal2026}.
  
In a stochastic dynamical system,  a metastable state informally means a state or region in the state space, where the system spends a very long time before making a rare transition to another state or region due to the interactions of nonlinearity and uncertainty. For example, a stable equilibrium state for the corresponding deterministic dynamical system (i.e.,  with noise absent) is a metastable state. In particular, when this deterministic dynamical system is a gradient system with a potential energy, a metastable state  is a local minimum of the   energy.     The most probable transition path is the connecting orbit from one metastable state to another, which minimizes the associated Onsager–Machlup action functional \cite{Durr1978, Chao2019}. 

In the space of probability densities, a Schr\"odinger bridge is a minimizer of the Kullback–Leibler divergence   (i.e., relative entropy)  and represents the optimal transition path connecting two probability distributions
\cite{leonard2014, chen2021stochastic, conforti2019, orland2025}. If metastable states are viewed as Dirac distributions, the Onsager–Machlup most probable transition path may be seen as a special case of the Schr\"odinger bridge. 
Moreover,  there are other divergence type of functionals, such as the $\alpha$-divergences \cite{amari2016, li2021bregman}, which carry dynamical information. These are generalizations of the Kullback–Leibler divergence or relative entropy, describing the discrepancy between two probability densities, and in certain cases can be interpreted as generalized entropies (e.g., Tsallis entropy and R\'enyi entropy). They characterize the cost of state transitions in nonequilibrium systems. Different values of the divergence parameter $\alpha$ imply different path-selection mechanisms or evolution modes of the system. Minimizing such divergence functionals yields the optimal transition path, which we   call   an ``information geodesic".

 \medskip

\textbf{Content of this book.} 
The most probable transition path,  Schr\"odinger bridges, and information geodesics may be regarded as the stochastic counterparts of heteroclinic orbits. 

This book  explores geometric methods for stochastic dynamical systems governed by stochastic differential equations (SDEs). As an introduction to this subject, we focus on       `connecting orbits' in stochastic dynamics.

(i) The Most Probable Transition Path via the Onsager-Machlup Action Formalism (Chapter 2): We examine the most probable dynamics of a stochastic system by analyzing the action functional, providing a variational pathway to determine the system's most likely trajectories.

(ii) Stochastic Geometric Mechanics  (Chapter 3): We provide a brief introduction to this framework  with stochastic variational principles. This  serves  as a conceptual link between the Onsager-Machlup most probable transition paths,   Schr\"odinger bridges, and information geodesics.

(iii)  Schr\"odinger bridges and information geodesics in the Space of Probability Densities  (Chapter 4):  We present a powerful modern paradigm by interpreting various random phenomena as transitions within the space of probability distributions, rather than merely as collections of individual particle trajectories. By adopting the Wasserstein space of probability densities as our primary geometric setting, we naturally connect the study of these transitions with Schr\"odinger bridges, information geodesics, optimal transport, and the Riemainnian-like structures of Otto calculus.

Viewed through this geometric lens, a solution to the Fokker-Planck equation associated with an SDE traces a continuous curve in the space of probability densities. A Schrödinger bridge or information geodesic represents a specific curve segment within this space. The intrinsic geometry of these curves directly encodes the vital dynamical information   and path-selection mechanisms of the underlying stochastic system.

This book synthesizes recent advances in geometric methods for stochastic dynamics, forging novel conceptual linkages and establishing a  geometric perspective on the subject.
These topics underlie many of the most vibrant current research  in artificial intelligence and machine learning—-particularly deep generative learning, such as diffusion models and flow matching \cite{wei2022optimal, li2021machine, song2021scorebased,debortoli2021diffusion, huang2024levy,li2022transport}. Furthermore, these tools provide prescriptive, data-driven frameworks for the early warning and intelligent mitigation of critical transitions or tipping phenomena across a vast spectrum of complex systems, including biomedicine, neuroscience, brain science, chemical physics, geophysics, ecology, and climate science \cite{DuanGao2026,   lucarini2022levy, serdukova2016stochastic, wang2026, xu2026early, zhang2025action}.



In the rest of this chapter, we review stochastic differential equations with (Gaussian) Brownian motion and (non-Gaussian) L\'evy motion; see \cite{Duan2015, Oksendal, Applebaum}.

\section{Brownian Motion and Stochastic Differential Equations  }
\label{sec1.1} 
 
Brownian motion, discovered in 1827 by the Scottish botanist Robert Brown, is caused by continuous impacts on pollen particles from the much smaller molecules of a surrounding liquid. These independent collisions occur at an extremely high rate within any small time interval, and the effect of any single hit is tiny compared to the total cumulative effect.  Albert Einstein formulated its physical theory in 1905—suggesting the following mathematical definition.


 

\index{Brownian motion}

\subsection{Brownian Motion }

  Brownian motion $B_t$, taking values in Euclidean space $\Rn$, is a
Gaussian stochastic process on an underlying probability space
$(\Om, \cF, \PX)$. Being a Gaussian process, $B_t$ is characterized by its
mean vector (taken to be the zero vector) and its covariance
    matrix (taken to be the identity matrix). More specifically, $B_t$ satisfies the
following conditions \cite{KS, Nelson}:

\noindent (i)\ \ \ $B_0=0, \;$  a.s.;  \\
(ii)\ \ \ $B_t$ has continuous paths, \; a.s.;\\
(iii)\ \ \ $B_t$ has independent increments;  \\
(iv)\ \ \ $B_t$ has stationary increments, and $B_t-B_s \sim \mathcal{N}(0, (t-s)I)$, for $  t
> s \geq 0$, where $I$ is the $n\times n$ identity matrix.


By this definition, we conclude that: 

$\diamond$ The covariance matrix for the Brownian motion $B_t $  in $\Rn$   is $ t \; I$, with $I$ the $n\times n$ identity matrix, and its trace is $ \Tr(t I)=n t$. For convenience,   we just call $I$ the covariance matrix for $B_t$.


$\diamond$  Because the  covariance matrix is $I$, the components of $B_t$ are pair-wise uncorrelated. The Gaussianity further implies that they are pair-wise independent.

$\diamond$  $B_t  \sim \mathcal{N}(0, t I)$, i.e., $B_t$ has probability density
function $p_t(x) = \frac1{(2\pi t)^{\frac{n}{2}}}
e^{-\frac{x_1^2+...+x_n^2}{2t}}$. This joint probability density function is the product of the probability density functions for the scalar components of $B_t$. Thus, the components of $B_t$ are independent scalar Brownian motions (not just pair-wise independent).



 The Brownian motion so defined is   called the standard Brownian motion, as the covariance matrix is the identity matrix.
We may   revise the preceding definition to allow   the covariance matrix to be a  general positive definite, symmetric matrix $Q$.
 \begin{definition} \label{QBM}  (Brownian motion with covariance matrix $Q$)\\
An $n$-dimensional Brownian motion with covariance matrix $Q$  is defined by
$B_t^Q = \sigma B_t,
 $
 where $\sigma$ is an $n \times m$ real non-zero matrix and $B_t$ is an $m-$dimensional standard Brownion motion,  such that
 $Q=\sigma \sigma^T$.
 \end{definition}


\index{Standard Brownian motion}
\index{$B_t^Q$}
\index{Brownian motion with covariance matrix $Q$}


\subsection{Stochastic Differential Equations}

Recall that a deterministic ordinary differential equation may be interpreted as an integral equation, while the integral is in Riemann-Stieltjes  sense. To consider stochastic differential equations, we also need a concept of stochastic integration,   that is,   integration with respect to Brownian motion.

Consider an   It\^{o}  SDE system in $\Rn$
\begin{eqnarray}
 dX_t = f(X_t)dt + \s(X_t) dB_t,   \label{system123}
\end{eqnarray}
where  $f$ is an $n$-dimensional   vector function, $\s$ is an $n\times n$ matrix function, and   $B_t(\omega)$ is an $n$-dimensional Brownian motion.
 The corresponding Stratonovich SDE is 
\begin{eqnarray} \label{SDEabcd}
dX_t = f(t,X_t) dt + \s(t,X_t) \circ dB_t, \; X_0=x,
\end{eqnarray}
 In these SDEs, $f$ is called the vector field or the drift term, and $\s$ is the diffusion coefficient or noise intensity. When $\s$ does not depend on the system state $X_t$, the SDE is said to have an additive noise, otherwise it is said to have a multiplicative noise.

Note that the It\^{o} stochastic differential $ \s(X_t) dB_t$ in SDE \eqref{system123} and  the Stratonovich stochastic differential $\s(X_t) \circ
dB_t$ in SDE  \eqref{SDEabcd} are
interpreted through their corresponding definitions of stochastic
integrals $\int_0^T \s(t,X_t)  dB_t$ and $\int_0^T \s(t, X_t) \circ dB_t$, respectively.


\index{Additive noise}
\index{Multiplicative noise}
\index{Vector field}
\index{Drift term}
\index{Noise intensity}
\index{Diffusion coefficient}

\medskip

In order to analyze SDEs, we need
a stochastic chain rule, i.e.,  the It\^{o}'s formula. This is  introduced in the next subsection.

\subsection{Generators and It\^{o}'s Formula }
\index{It\^{o}'s formula}



\index{It\^{o}'s formula}

 Let $g(t, x)$ be a given (deterministic) scalar smooth function in both $x\in \R^n$ and $t\in \R$. Then
 \begin{eqnarray}\label{itoRn}
 dg(t, X_t) & = & \frac{\p g}{\p t} dt + \grad g \cdot dX_t
  +   \frac12 (dX_t)^T H(g) dX_t            \nonumber  \\
  &=& \frac{\p g}{\p t} dt +  \sum_{i=1}^n \frac{\p g}{\p x_i}  dX_t^i
   +   \frac12  \sum_{i, j=1}^n \frac{\p^2 g}{\p x_i \p x_j}  dX_t^i dX_t^j,
\end{eqnarray}
 where $\frac{\p g}{\p t}$, $\frac{\p g}{\p x_i}$ and $ \frac{\p^2 g}{\p x_i \p x_j}$  are evaluated at $(t, X_t)$.

Symbolically, we may also use the following rules in manipulating the preceding It\^{o} differential:
\begin{eqnarray}
 dt dt =0, \;\; dt dB_t = 0, \;\;  dB_t \cdot dB_t =\Tr (Q)dt =ndt.
\end{eqnarray}
Note that the  covariance matrix $Q=I$, for $n-$dimensional Brownian motion $B_t$.

With these symbolic operations, the
It\^{o}'s formula in differential form becomes:
\begin{eqnarray}
 d g(t,X_t) & = & \{\frac{\p g}{\p t}(t, X_t)+ f \cdot \grad g(t, X_t)
 +\frac12 \Tr[\s \s^T H(g)](t, X_t)\} \; dt \nonumber \\
 & + & (\grad g(t, X_t))^T \; \s(X_t) \; dB_t,   \label{ito12345}
\end{eqnarray}
where $\cdot $ denotes scalar product in $\R^n$,  $ ^T$ denotes transpose of a   matrix,   $H(g)= (g_{x_i x_j})$ is
the $n \times n$ symmetric Hessian matrix which is also denoted as $D^2(g)$, and $\Tr$ denotes the trace of a matrix.

\begin{remark}
Note that $H(g) (x, y) = x^T H(g) y$, $H(g)$ is bilinear, and $B_t = \sum_{i=1}^n B_i(t) e_i$. \\
We have the following interpretation:
\begin{eqnarray*}
 H(g) (\s dB_t, \s dB_t)    
 &=& H(g) (\s \sum_{i=1}^n dB_i(t) e_i, \s \sum_{j=1}^n dB_j(t) e_j)  \\
 &=& \sum_{i, j=1}^n dB_i(t) dB_j(t) H(g) (\s e_i, \s e_j)   \\
 &=&  [\sum_{i=1}^n   e_i^T \s^T H(g) \s e_i] \; dt   
 = \Tr [\s^T H(g)\s ] dt    \\
 &=& \Tr[ H(g) \s \s^T ] dt 
 =   \Tr[\s \s^T H(g)  ] dt.
\end{eqnarray*}

\end{remark}

The generator $A$  for this SDE system \eqref{system123}, or for its solution process $X_t$, is
\begin{eqnarray} \label{gtor}
A h = f \cdot  \grad h   +\frac12 \Tr[\s \s^T H(h)],
\end{eqnarray}
for $h$  in Sobolev space $H_0^2(\Rn)$. 
For example, Brownian motion $B_t$ in $\Rn$ is the solution for $dX_t = 0dt + dB_t, \; X_0=x$. Thus, by \eqref{gtor},   the generator for  Brownian motion starting at $x$,  $X_t= x+B_t$,  is $\frac12 \Delta$ (Laplacian operator).

\index{Generator for Brownian motion in $\Rn$ }
\index{Brownian motion starting at $x$ }

With the generator $A$,  the It\^{o}'s formula \eqref{ito12345} can be rewritten as
\begin{eqnarray} \label{ito67890}
 d g(t,X_t)  =  \{ \frac{\p g}{\p t}(t, X_t)+ A g(t, X_t)\} dt
  +  (\grad g(t, X_t))^T \s(X_t) dB_t.
\end{eqnarray}

It\^{o}'s formula in integral form is:
\begin{eqnarray}
  g(t,X_t) & = & g(0, X_0)
  + \int_0^t \{\frac{\p g}{\p t}(s, X_s)+ b^T \grad g(s, X_s)+\frac12 \Tr[\s \s^T H(g)](s, X_s)\}  ds    \nonumber \\
& + & \int_0^t  (\grad g(s, X_s))^T \s(X_s)   dB_s.
\end{eqnarray}


\medskip

\textbf{Stochastic product rule and integration by parts}
\index{Stochastic product rule}

Let $X_t, Y_t$ be solutions of two scalar SDEs, respectively.
Then, by applying the two dimensional It\^{o}'s formula to $g(x, y) =xy$, we get the stochastic product rule
\begin{eqnarray}\label{product8989000}
 d (X_tY_t) &=& \frac{\p g}{\p x} dX_t + \frac{\p g}{\p y}dY_t + \frac12 (dX_t dY_t + dY_t dX_t) \nonumber\\
 &=& X_t dY_t + Y_t dX_t  + dX_t dY_t.
\end{eqnarray}
The corresponding integral form is the   stochastic integration by parts
\begin{eqnarray} \label{parts8989}
  \int_0^T X_t dY_t = X_T Y_T -X_0 Y_0 -\int_0^T Y_t dX_t - \int_0^T dX_t dY_t.
\end{eqnarray}

\index{Stochastic integration by parts}


\subsection{Kolmogorov and Fokker-Planck Equations }
\label{sec1.2.5}

We consider first  the Kolmogorov background equation  and then Kolmogorov forward  equation (or Fokker-Planck equation) for  the following  SDE system in $\R^n$
\begin{eqnarray} \label{sde906}
dX_{t}=f(X_{t})dt+\sigma(X_{t})dB_{t},  
\end{eqnarray}
where  $f$ is an $n$-dimensional   vector function, $\s$ is an $n\times n$ matrix function, and   $B_t(\omega)$ is an $n$-dimensional Brownian motion.

Recall that the   generator for this SDE system is
 \begin{eqnarray}  \label{A888}
Ah(x) &=& \sum_{i}f_{i} \frac{\partial h}{\partial x_{i}}
+ \frac12 \sum_{i,j}(\sigma\sigma^{T})_{ij} \frac{\partial^{2} h}{\partial
x_{i}\partial x_{j}}    \nonumber  \\
& =& f \cdot \nabla h + \frac12 \, \Tr (\sigma \sigma^T H(h)), \;\;  h \in H_0^2(\Rn),
\end{eqnarray}
where $H(h)$ is the Hessian matrix for $h$.


The Kolmogorov backward equation for \eqref{sde906}  is \cite[Ch.8]{Oksendal} 
\begin{eqnarray} \label{backwardeqn}
\p_t u (t, x) = A u(t, x), \;\; u(0, x)= u_0(x),
\end{eqnarray}
where $u(t, x)=\E u_0(X_t)$, for each given   $u_0$ in the domain of  the generator $A$.


The Kolmogorov forward equation is also called the Fokker-Planck equation.

The adjoint operator for the  generator $A$ in $L^2(\Rn)$ is
\begin{eqnarray}
A^{*}p=-\sum_{i}{\partial\over
\partial x_{i}}(f_{i}p)+{1\over2}\sum_{i,j}{\partial^{2}\over\partial
x_{i}\partial x_{j}}((\sigma\sigma^{T})_{ij} \; p), \;\; p \in H_0^2(\Rn).
 \label{FPK-operator}
\end{eqnarray}
Note that $A^{*}$ is  often called the Fokker-Planck operator.
\index{Fokker-Planck operator}
It may be rewritten as
\begin{eqnarray} \label{A*}
 A^{*} p  &=& \frac12 \sum_{i,j} \p_{x_i x_j} ( (\sigma \sigma^T)_{ij} p)
 -\sum_i \p_{x_i} (f_i(x) p ) \nonumber  \\
 &=& \frac12 \Tr (\nabla \nabla^T (\sigma \sigma^T p)) -\nabla \cdot (f p) \nonumber \\
&=& \frac12 \Tr (H (\sigma \sigma^T p)) -\nabla \cdot  (f p),
\end{eqnarray}
where we have used the fact for  the (symmetric) Hessian matrix:
$$
H = \nabla \nabla^T = (\p_{x_i x_j}).
$$
Note that $p$ is a scalar function, and  we interpret $H (\sigma \sigma^T p)$ as matrix multiplications of $H, \sigma $ and $\sigma^T p$.
The adjoint operator (\ref{A*}) may be further rewritten as
\begin{eqnarray} \label{AA*}
 A^{*} p  &=&   \frac12  \nabla \cdot (\nabla \cdot (\sigma \sigma^T p) ) -\nabla \cdot (f p)   \nonumber \\
 &= & \frac12  \mbox{div}  ( \mbox{div} (\sigma \sigma^T p) ) -\mbox{div}  (f p),
\end{eqnarray}
where $\nabla \cdot (\nabla \cdot (\sigma \sigma^T p ) )$ is explained as follows as a two-step operation: First,
take  divergence for each of the $n$ row vectors of an $n \times n$ symmetric matrix $\sigma \sigma^T p$. This gives us     an $n$-vector, namely, $V= \nabla \cdot (\sigma \sigma^T p)$. Then, take the divergence of $V$.

 The   Fokker-Planck equation for the probability density function $p$ of the solution process $X_t$, or for the SDE system  \eqref{sde906}, is 
\begin{eqnarray}\label{multi604}
{\partial\over\partial t}p(x, t)=A^{*} p(x, t).
\end{eqnarray}
When the SDE system  \eqref{sde906} is given the initial condition $X_{0}=x_{0}$,  the  Fokker-Planck equation is supplemented with the  initial condition
$ p(x, 0)=\delta(x-x_0)$.  
Note that $p(x, 0)$ may also be  a given initial probability density function.






\section{L\'evy Motions and Stochastic Differential Equations  }
\label{sec1.3}

In this section, we consider stochastic differential equations driven by non-Gaussian processes—specifically, $\alpha$-stable L'evy motions and their associated dynamical systems. Unlike Brownian motion, which relies on normal distributions, $\alpha$-stable L'evy motions are non-Gaussian processes defined via stable random variables. While both possess independent and stationary increments, $\alpha$-stable L'evy motions better capture complex fluctuations observed across physics, geophysics, biophysics, chemical physics, and engineering.

 \subsection{L\'evy Motions}





  L\'evy motions \cite[p.43]{Applebaum} are  defined similarly as for Brownian motion $B_t$.

\begin{definition}
A L\'evy motion (L\'evy process)  $L_t$, in $\Rn$, is a stochastic process satisfying the following conditions:\\
(i) $L_0=0$, a.s.;\\
(ii) Independent increments: For $ t_1 < t_2 < \cdots <t_{n-1}<t_n$,   the
random variables $L_{t_2}-L_{t_1}, \cdots, L_{t_n}-L_{t_{n-1}}$ are independent;\\
(iii) Stationary increments: $L_t -L_s $ and $L_{t-s}$ have the same distribution;\\
(iv)  Stochastically continuous sample paths (i.e., sample paths are continuous in probability): For every $\delta >0$ and every $s\geq 0$,
   $ \PX (\|L_t -L_s\| > \delta) \to 0$
   as $t \to s$.
\end{definition}

\index{Stochastically continuous sample paths}
\index{Stochastic continuity}
\index{Continuity in probability}
\index{Non-Gaussian stochastic process}

\begin{remark}
   For a stochastically continuous process, there exists a modification (i.e., a version)    whose   paths are continuous
   from the right and have   left limits (``c\`adl\`ag") at every time \cite[Ch.2]{Applebaum}. We are going to take this modification for the L\'evy motion.
    Therefore, the paths of a L\'evy motion are c\`adl\`ag.
   Note that a c\`adl\`ag function can only have (at most) a countable number of jumps (see \cite{Duan2015}), and the jumps are their only possible
   discontinuities in time.  Brownian motion $B_t$, as a Gaussian stochastic process,  is a special L\'evy motion wtih no jumps. 
\end{remark}

\index{C\`adl\`ag}

\subsection{L\'evy-It\^{o} Decomposition and L\'evy-Khintchine Formula  }

For    a L\'{e}vy motion  $L_t$  in $\R^n$, define the jump process $\Delta L(t)$ by
\begin{eqnarray} \label{jump1}
  \Delta L(t) \triangleq L_t-L_{t-}, \; t \geq 0,
\end{eqnarray}
where $L_{t-}$ is the left limit of $L_t$ at time $t$.

\index{Left limit ($L(t-)$)}



Let us count the jumps of specified size. For a Borel set $S \in \B(\R^n \setminus \{0\})$ and for $t > 0$,
define
\begin{eqnarray} \label{jump2}
  N(t,S)(\omega)  \triangleq \#\{0\leq s <t: \Delta L(s)(\omega) \in S \},
\end{eqnarray}
whenever $\omega \in \Omega_0$, and $N(t, S)(\omega)  \triangleq 0$, whenever
$\omega \in \Omega_0^c$. Recall that $L_t$ has a cadlag modification on $\Omega_0$ and $\PX (\Omega_0)=1$ (\cite[p.88]{Applebaum}).

Define a Borel measure $\nu$ on $\B(\R^n \setminus \{0\})$ by
\begin{eqnarray} \label{jump3}
 \nu(S)  \triangleq \EX N(1,S)(\omega).
\end{eqnarray}
Furthermore, define the compensated Poisson random measure by
\begin{eqnarray} \label{jump3.5}
\tilde{N} (t, S) \triangleq   N(t, S)-t\, \nu(S) .
\end{eqnarray}

We recall the L\'evy-It\^{o} decomposition theorem \cite[p.126]{Applebaum}.

\begin{theorem}(L\'evy-It\^{o} decomposition)\\
If $L_t$ is a L\'{e}vy motion  in $\R^n$, then there exist a vector $b \in \R^n$, a covariance matrix $Q$, and an independent Poisson random measure $N$ on $\R^+ \times (\R^n \setminus \{0\})$ such that for each $t\geq 0$,
\begin{align}\label{levyito8}
L_t=bt+B^Q_t +\int_{\|y\|< 1}y\tilde{N}(t, dy)+\int_{\|y\|\geq 1}y{N}(t, dy),
\end{align}
where $N(dt, dx)$ is the Poisson random measure (quantifying the number of jumps of $L_t$),
$\tilde{N}(dt, dx)\triangleq N(dt, dx)-\nu(dx)dt$
  is the compensated Poisson random
measure, $\nu(S)\triangleq \EX N(1,S)$ is the jump measure, and $B_t^Q$ is an independent $n$-dimensional Brownian motion with covariance matrix $Q$.
\end{theorem}

The triplet $(b, Q, \nu)$ is   called the    generating triplet for the  L\'{e}vy motion $L_t$. Brownian motion with covariance matrix $Q$  means that $B_t^Q = \sigma B_t$,
 where $\sigma$ is an $n \times m$ real non-zero matrix and $B_t$ is an $m-$dimensional standard Brownian motion,  such that $Q=\sigma \sigma^T$. For this reason, we occasionally write $\sigma = Q^{\frac12}$. Therefore,
 \begin{equation}
 B^Q_t = Q^{\frac12} B_t.
 \end{equation}

\index{Brownian motion with covariance matrix $Q$}
\index{$B_t^Q$}


The number $1$ in  $\|y\|< 1$ and $\|y\| \geq 1$ allows us to specify relatively `small' and `large' jumps, respectively \cite[p.364]{Applebaum}. It may be replaced by an arbitrary positive number $c$. The standard     Brownian motion $B_t$ in $\R^n$ has the identity matrix $I$ as covariance matrix (i.e., $Q=I$), with   the generating triplet $(0, I, 0)$.

 \index{Jumps (relatively `small' or `large')}

\index{Generating triplet for L\'{e}vy motion }
\index{Triplet for L\'{e}vy motion }

\medskip


The L\'evy-Khintchine formula  specifies the expression for the characteristic function of a L\'evy motion (\cite[p.45]{Applebaum}). Recall the indicator function, $I_s$, for a set $S$, is defined as
$$
I_S(y) =
    \begin{cases}
        1,   &\text{if $y \in S $,}\\
        0,    &\text{if $y \notin S$.}
    \end{cases}
$$
\index{Indicator function}

\begin{theorem}[L\'evy-Khintchine formula]
The characteristic function for    a L\'evy motion  $L_t$ in $\Rn$    is
\begin{equation*}
\Phi_{t}(u) \triangleq \EX e^{i u \cdot L_t} =e^{t\eta(u)}\qquad for\ each\
t\geq0,u\in\mathbb{R}^{n},
\end{equation*}
where
\begin{equation}\label{eq4}
\eta(u)=ib\cdot u-\frac{1}{2}u\cdot
Qu+\int_{\mathbb{R}^{n} \setminus \{0\}}[e^{iu\cdot y}-1-i I_{\{\|y\|<1\}} \; u\cdot
y] \; \nu(dy)
\end{equation}
for a vector $b\in\mathbb{R}^{n}$, a non-negative definite symmetric
$ n \times n$ matrix $Q$, and a Borel measure $\nu$ on
$\mathbb{R}^{n} \setminus \{0\}$ for which
$\int_{\mathbb{R}^{n} \setminus \{0\}}(\|y\|^{2}\wedge1)\nu(dy)<\infty$, or equivalently
\begin{align}\label{s1_2}
\int_{\mathbb{R}^n\setminus \{0\}}\frac{\|y\|^2}{1+\|y\|^2}\nu(dy)<\infty.
\end{align}
Here $\|\cdot \|$ is the usual Euclidean norm  in $\mathbb{R}^n$.\\
Conversely, given a mapping of the form \eqref{eq4}, there exists a L\'evy motion with characteristic function  $\Phi_{t}(u)=e^{t\eta(u)}$.
It is called the L\'evy motion with   triplet $(b, Q,\nu)$.
\end{theorem}

\index{L\'evy-Khintchine formula for L\'evy motions}

Both   dot and $<\cdot, \cdot>$ in this theorem denote the scalar product in $\R^n$.  We use both notations in this book.


 In this theorem, $(b, Q, \nu)$ is the  triplet, or generating triplet, for the L\'evy motion $L_t$.
 \index{Triplet (for a L\'evy motion)}
 The vector $b$ is usually called the drift vector, $Q$ is called the covariance matrix or diffusion matrix, and the   Borel measure $\nu$  is called the \emph{jump measure}, for $L_t$.

\index{Drift vector}
\index{Covariance matrix}
\index{Diffusion matrix}
 \index{Jump measure (for a L\'evy motion)}


Every L\'evy motion is obtained as a sum of independent processes with
three types of triplets $(b, 0, 0)$, $(0, Q, 0)$ and
$(0, 0, \nu)$.




The generator $A$ of a  L\'evy motion with triplet $(b, Q, \nu)$ is \cite{Applebaum}:
\begin{eqnarray} \label{big-generator}
 A  \phi &=& b \cdot \nabla \phi + \frac12 \Tr(Q H(\phi))   \nonumber \\
& +&  \int_{\R^n \setminus\{0\}} [\phi(x+  y)-\phi(x) -I_{\{\|y\|<1\}} \; y \cdot \nabla \phi(x) ] \; \nu(dy),
\end{eqnarray}
for $\phi$ in the domain of definition for  $A$.


\subsection{The $\alpha$-stable L\'evy Motions }
\label{examplelevymotion}

 We now consider a special but important class of L\'evy motions, the $\alpha$-stable L\'evy motions (\cite[p.30]{JW}).  
The stable   random variables   are used to define   (non-Gaussian) $\alpha$-stable L\'evy motion $L_t^\alpha$, just like the normal random variables are used to define    (Gaussian) Brownian motion $B_t$.

Recall that an  $n$-dimensional  normal random variable (i.e., Gaussian random variable) has distribution $ \cN(m, Q)$, with mean vector $m$ in $\Rn$ and $n\times n$ covariance matrix $Q$.   In particular,  $ \cN(0, I)$
is the standard normal distribution. 




\medskip 

\subsection*{Stable random variables in $\mathbb{R}^1$  }

We now first consider scalar and then   vector stable random variables, as generalizations of normal random variables.


\index{Stable random variable}

Unfortunately, the probability density functions for stable random variables  are generally not representable via elementary functions. So we examine them via their characteristic functions.


Let $\Phi_X(u) \triangleq \EX e^{i u   X}$ be the characteristic function for a scalar random variable $X$. 

\index{Characteristic function (for a random variable)}

The following definition for a stable random variable is in terms of characteristic functions. This is a ``local" characterization for a stable random variable, as the  characteristic function is in terms of point-wisely defined elementary functions.

\begin{definition}   \label{stable1}
A scalar random variable $X$ is stable if there exist four real parameters, i.e., a stability parameter $\alpha \in (0, 2]$, a scaling parameter $\sigma >0$, a symmetry parameter $\beta \in [-1, 1]$ and a shift parameter $\gamma \in \R^1$, such that its characteristic function $\Phi_X(u)$ has the following representation\\
(i) $0<\alpha<1$: \\
$$\Phi_X(u)=\exp\{i\gamma u-\sigma |u|^\alpha [1-i\beta \;\sign(u) \; \tan\frac{\pi \alpha}{2}] \}; $$
(ii) $\alpha=1$: \\
$$\Phi_X(u)=\exp\{i\gamma u-\sigma |u| [1+i\beta\frac{2}{\pi}\; \sign(u) \; \ln |u|] \};$$
(iii) $1<\alpha<2$: \\
$$\Phi_X(u)=\exp\{i\gamma u-\sigma |u|^\alpha [1-i\beta \; \sign(u) \; \tan\frac{\pi \alpha}{2}] \}; $$
(iv)  $\alpha=2$: \\
 $$\Phi_X(u)=\exp\{i\gamma u-\frac12 \sigma^2 u^2\}, $$
\end{definition}
where
$$
\sign (u) = \begin{cases}
1, \;\; u >0, \\
0, \;\; u=0,\\
-1, \;\; u<0.
\end{cases}
$$

\index{Definition for a stable random variable}
\index{Characteristic function for a sable random variable}

Note that $\Phi_X(u)=\exp\{i\gamma u-\frac12 \sigma^2 u^2\}$ is the characteristic function for a Gaussian random variable. So when $\alpha=2$, the stable random variable is just the Gaussian random variable.




\bigskip

The distribution for a scalar stable random variable is denoted as $S_{\alpha}(\sigma, \beta, \gamma )$. Usually, $\alpha$ is called the index of stability (or non-Gaussianity index), $\sigma$ the scale parameter,
$\beta$ the skewness parameter and $\gamma$ the shift parameter.
The symbol $S_{\alpha}(\sigma, \beta, \gamma )$ refers to either the distribution function or the probability density function for a stable random variable.
To indicate the importance of the index of stability, $\alpha$,
we often call such a   random variable the $\alpha$-stable random variable.


\index{$\alpha$-stable random variable}
\index{Stable distribution $S_{\alpha}(\sigma, \beta, \gamma )$}

\index{$S_\alpha(\sigma, \beta, \gamma)$}

\begin{remark}
As we know in \cite[p.37]{Applebaum} and \cite[p.16]{taqqu}, the following estimates hold: \\
(i) For $\alpha=2$, i.e., the normal random variable $X$ has the following tail estimate
\begin{eqnarray}
\PX(X>y) \sim \frac{e^{-\frac{y^2}2}}{\sqrt{2\pi} \, y} \; \mbox{ as }\; y \to \infty.
\end{eqnarray}
We say that the normal (or Gaussian) random variable $X$ has `light tail', as the tail estimate decays exponentially. \\
(ii) For $0<\alpha<2$,   the stable random variable $X$ has the following tail estimate
\begin{eqnarray}
\lim_{y\to \infty} y^\alpha \PX(X>y) =
 C_\alpha \frac{1+\beta}2 \sigma^\alpha,  \\
\lim_{y\to \infty} y^\alpha \PX(X<-y) =
 C_\alpha \frac{1-\beta}2 \sigma^\alpha,
\end{eqnarray}
where $C_\alpha$ is a positive constant.
A stable random variable $X$ (with $0<\alpha<2$) has `heavy tail' (or fat tail), as the tail estimate decays polynomially.
\end{remark}

\index{Light tail}
\index{Heavy tail}




We will focus on symmetric  $\alpha$-stable random variables. 
\begin{definition}
$X \sim S_\alpha(\sigma, 0, 0)$ 
 is called a symmetric $\alpha$-stable random variable (i.e., skew parameter $\beta=0$ and shift parameter  $\gamma=0$). This distribution is often denoted by $S \alpha S$. When $\sigma=1$, it is called a standard symmetric
$\alpha$-stable random variable, and we denote it by $X \sim S_\alpha(1, 0, 0)$.
\end{definition}

  

Figure~\ref{fig:gaussian-stable} shows the probability density
functions of (a) the standard Gaussian random variable
$X\sim\mathcal{N}(0,1)$ and (b) the standard symmetric
$\alpha$-stable random variable $X\sim S_\alpha(1,0,0)$
for $\alpha=0.5$, $1.0$, and $1.5$.

\begin{figure}[htbp]
  \centering
  \begin{minipage}[t]{0.48\linewidth}
    \vspace{0pt}
    \centering
    \includegraphics[width=\linewidth]{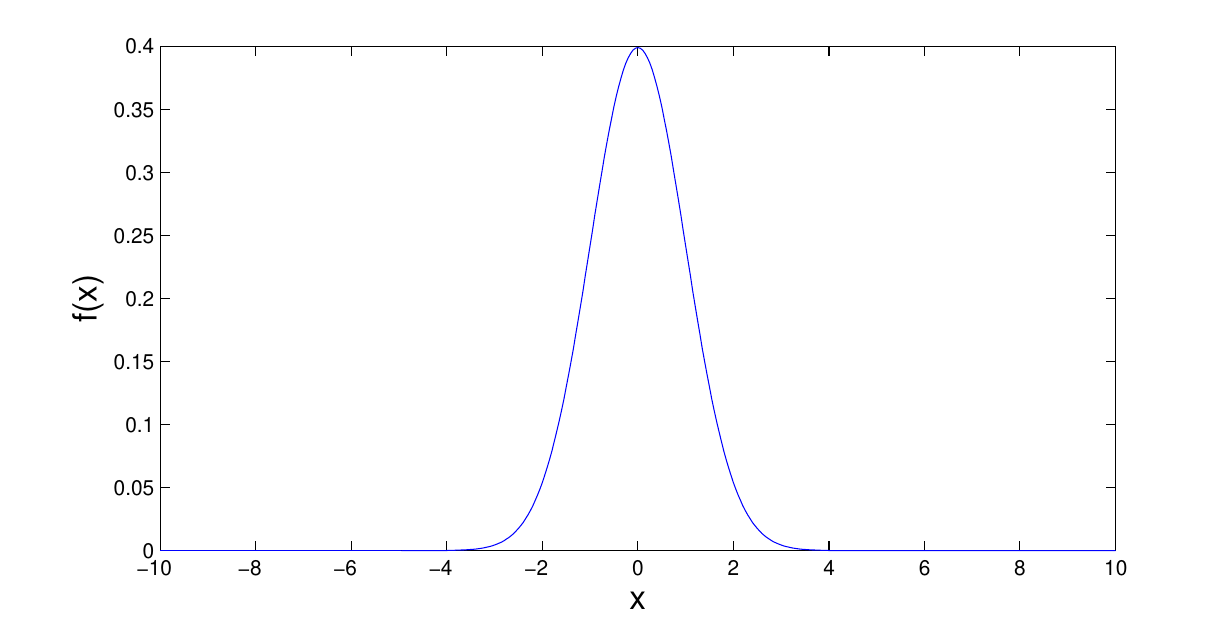}
    \par\smallskip
    {\small (a) Standard Gaussian density}
  \end{minipage}\hfill
  \begin{minipage}[t]{0.48\linewidth}
    \vspace{0pt}
    \centering
    \includegraphics[width=\linewidth]{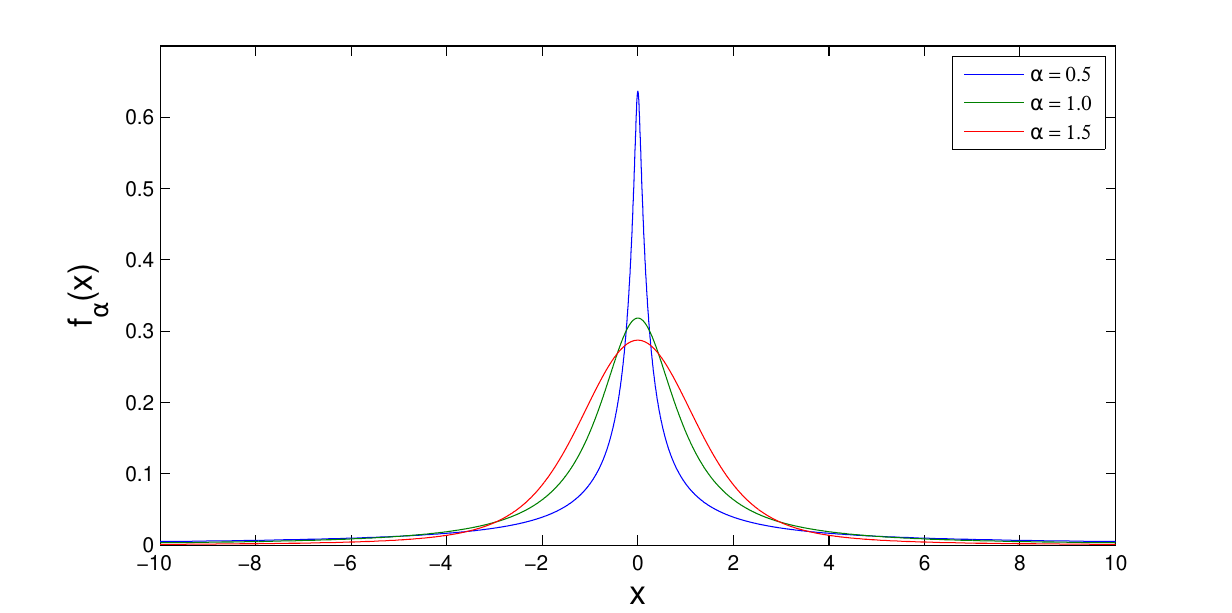}
    \par\smallskip
    {\small (b) Symmetric $\alpha$-stable densities}
  \end{minipage}

  \caption{Probability density functions.
  (a) The standard Gaussian random variable
  $X\sim\mathcal{N}(0,1)$.
  (b) The standard symmetric $\alpha$-stable random variable
  $X\sim S_\alpha(1,0,0)$ for
  $\alpha=0.5$ (blue, highest peak),
  $\alpha=1.0$ (green, intermediate peak), and
  $\alpha=1.5$ (red, lowest peak).}
  \label{fig:gaussian-stable}
  \label{bell}
  \label{bell3}
\end{figure}


A symmetric scalar $\alpha$-stable random variable $X$ with  distribution $S_{\alpha}(\sigma, 0, 0)$ has the characteristic function
\begin{eqnarray}
\Phi_X(u) = \begin{cases}
     e^{-\sigma^\alpha |u|^\alpha}, \;\; 0<\alpha<2,\\
     e^{-\frac12 \sigma^2 |u|^2}, \;\; \alpha =2.
     \end{cases}
\end{eqnarray}




Note that a  symmetric   $\alpha$-stable random variable
$X \sim S_\alpha(\sigma, 0, 0)$
has the following moment properties \cite[p.24]{JW}:  $\EX |X|<\infty$ if and only if $\alpha \in (1, 2]$, and
$\EX |X|^2 <\infty$ if and only if $\alpha=2$ (i.e., $X$ is a Gaussian random variable).


\subsection*{Stable random  variables in $\Rn$ }


This is a ``nonlocal” characterization for a stable random variable, as the
characteristic function involves with an integration with respect to a jump measure.

\begin{definition}  \label{stable2vector}
A   random vector $X$ is stable if   its characteristic function is
as follows\\
(i) $0<\alpha<1$: \\
$$\Phi_X(u)=\exp\{\int_{\Rn\setminus \{0\}}(e^{i<x, u>}-1) \nu(dx)+i<\gamma_0, u> \}; $$
(ii) $\alpha=1$: \\
$$\Phi_X(u)=\exp\{\int_{\Rn\setminus \{0\}}(e^{i<x,u>}-1-i <x, u> I_{\{\|x\| < 1\}}(x)) \nu(dx)+i<\gamma_*, u> \}; $$
(iii) $1<\alpha<2$: \\
$$\Phi_X(u)=\exp\{\int_{\Rn\setminus \{0\}}(e^{i<x, u>}-1-i <x, u>) \nu(dx)+i<\gamma_1, u> \};$$
(iv) $\alpha=2$: \\
$$\Phi_X(u)=\exp\{i<\gamma, u>-\frac12 \sigma^2 <u, u>\}, $$
where $\gamma_0, \gamma_*, \gamma_1$ are real vectors in $\Rn$, and $\nu(dx)$ is a Borel measure
on  $\Rn\setminus \{0\}$  (called   jump measure). 

\end{definition} 
It is known that  a scalar stable random variable in Definition \ref{stable1}  (local characterization) is a special case for this definition  (nonlocal characterization)  when $n=1$; see 
\cite[Problem 7.1]{Duan2015}.

A  stable random  vector in $\Rn$ is   rotationally symmetric if its characteristic function has the following special form (Applebaum \cite[Ch. 1]{Applebaum})
\begin{eqnarray}
\Phi_X(u) &=& e^{-\sigma^{\alpha}\|u\|^{\alpha}}, \mbox{ if} \ \alpha \neq 2,\\
\Phi_X(u) &=& e^{-\frac{\sigma^2}{2}\|u\|^2}, \mbox{ if}  \ \alpha=2.
\end{eqnarray}



 We now consider scalar symmetric and then high dimensional rotationally symmetric  $\alpha$-stable L\'evy motions.

\index{$\alpha$-stable  L\'evy motion}


\subsubsection*{The $\alpha$-stable L\'evy motions in $\mathbb{R}^1$}

A symmetric $\alpha$-stable scalar  L\'evy motion $L_t^{\alpha}$, with $0<\alpha <2$, is a stochastic process with the following properties

(i) $L_0^{\alpha} =0$, a.s.;

(ii) $L_t^{\alpha}$ has independent increments;

(iii) $L_t^{\alpha} -  L_s^{\alpha} \sim S_{\alpha}((t-s)^{\frac1{\alpha}}, 0, 0)$; and

(iv) $L_t^{\alpha}$ has  stochastically continuous sample paths, i.e., for every $s>0$, $L_t^{\alpha} \to  L_s^{\alpha}$ in probability, as $t \to s$.

\medskip

From this definition, we see that $L_t^{\alpha}  \sim  S_{\alpha}(t^{\frac1{\alpha}}, 0, 0)$.
It is known that if $X \sim S_\alpha(1, 0, 0)$, then
$ t^{\frac1{\alpha}} X \sim S_{\alpha}(t^{\frac1{\alpha}}, 0, 0)$, for $t >0$.
Moreover, for every $c>0$,
$L^\alpha_{ct}$ and $c^{\frac1{\alpha}} L^\alpha_{t}$ have the same distribution. See \cite[p.113]{taqqu}.

Using the  facts that
$\PX(t^{\frac1{\alpha}} X \leq x) = \PX(X \leq t^{-\frac1{\alpha}} x)$ and
$\frac{d}{dx}\PX(X \leq t^{-\frac1{\alpha}} x)= t^{-\frac1{\alpha}} \frac{d}{d\tilde{x}}\PX(X \leq \tilde{x})|_{\tilde{x}=t^{-\frac1{\alpha}}x}$, we conclude that
  the probability density function for $L_t^{\alpha}$ is
\begin{eqnarray}
 t^{-\frac1{\alpha}} \, f_\alpha(t^{-\frac1{\alpha}} \, x),
\end{eqnarray}
where $f_\alpha$ is the probability density function for the standard symmetric $\alpha$-stable random variable $X \sim S_\alpha(1, 0, 0)$. The generalized time derivative $\frac{dL_t^\alpha}{dt}$ as a model for non-Gaussian white noise is discussed in \cite{Nun, Lee, Shih}.


\index{Symmetric $\alpha$-stable  L\'evy motion}

\bigskip

A symmetric $\alpha$-stable  L\'evy motion $L_t^{\alpha}$, for $\alpha \in (0, 2)$,  has the generating triplet $(0, 0, \nu_{\alpha})$ where the jump measure $  \nu_{\alpha}(du)=  c_\alpha \frac{du}{|u|^{1+\alpha}}$,
with
\begin{eqnarray} \label{CCC}
\displaystyle{c_{\alpha} =
\frac{\alpha}{2^{1-\alpha}\sqrt{\pi}}
\frac{\Gamma(\frac{1+\alpha}{2})}{\Gamma(1-\frac{\alpha}{2})}}.
\end{eqnarray}
Here $\Gamma$ is the Gamma function.


When $\alpha=2$, this family reduces to the well-known Brownian motion $B_t$.


\medskip


The generator of a scalar symmetric $\alpha$-stable L\'evy motion  $L_t^{\alpha}$,  with triplet $(0, 0,  \nu_{\alpha})$,  is
\begin{eqnarray}
 A_\alpha  \phi &=&   \int_{\R^1 \setminus\{0\}} [\phi(x+  y)-\phi(x)  ] \; \nu_\alpha(dy),
\end{eqnarray}
where the right hand side is understood as  a Cauchy principal value. The domain of $A$ is the collection of  function  $\phi$ such that this Cauchy principal value integral exists. 
 Note that the integrand $  I_{\{|y|<1\}} \; y  $ is an odd function in $y$ and the corresponding integral in (\ref{big-generator}) is zero.


\subsection*{The $\alpha$-stable L\'evy motions in $\Rn$}
\label{rotationsymmetric}

In this subsection, we discuss rotationally symmetric
$\alpha$-stable L\'evy motions  $L_t^\alpha$ in $\Rn$.

\index{Rotationally symmetric $\alpha$-stable L\'evy motion}



\begin{definition}\label{rid}
For $\alpha\in(0,2)$, an $n$-dimensional  rotationally symmetric $\alpha$-stable L\'evy motion
  $L_t^\alpha$ is a L\'evy motion with characteristic function
\begin{equation}
\EX e^{i <u, L_t^\alpha>} = e^{-C t \|u\|^\alpha}, \quad u \in  \mathbb{R}^n, 
\end{equation}
where
$$
C=\pi^{-1/2}\frac{\Gamma((1+\a)/2)\Gamma(n/2)}{\Gamma((n+\a)/2)}.
$$
\end{definition}
The value of $C$ is $1$ when the dimension $n=1$.



\medskip

We recall the following result (\cite{Chen}, \cite[Ch. 3]{Sato}).  
A rotationally symmetric $\alpha$-stable  L\'evy motion $L_t^{\alpha}$ in $\R^n$ has the generating triplet $(0, 0, \nu_{\alpha})$,  with the jump measure
\begin{equation}\label{levyalpha888}
  \nu_{\alpha}(du)= c(n, \alpha) \frac{du}{\|u\|^{n+\alpha}},
\end{equation}
and the intensity constant
\begin{equation}\label{levyc}
     c(n, \alpha)= \frac{ \a \Gamma((n+\a)/2)}{2^{1-\a} \pi^{n/2} \Gamma(1-\a/2)  },
\end{equation}
where $\Gamma$ is the Gamma function.

When the spatial dimension $n$ is clear in the context, we often denote   $c(n, \alpha)$ as $c_{\alpha}$ or just $c$.
For example, in the case of $\a=1$, $c=\frac1{\pi}$ (for $n=1$) and   $c=\frac1{2 \pi}$ (for $n=2$). See \cite{Chen}.

\begin{remark}
Thus, for an $n$-dimensional rotationally symmetric $\alpha$-stable L\'evy motion  $L_t^\alpha$, its diffusion matrix $Q=0$ and the drift vector $\gamma=0$.   It is characterized by the jump measure $\nu_\alpha$, for $\alpha \in (0, 2)$.
\end{remark}






The generator for this rotationally symmetric $\alpha$-stable L\'evy motion $L_t^\alpha$ in $\Rn$ is (\cite[Theorem 3.3.3]{Applebaum} or \cite{Chen})
\begin{eqnarray}
 A_\alpha  \phi (x) &=&   \int_{\Rn \setminus\{0\}} [\phi(x+  y)-\phi(x)  ] \; \nu_\alpha(dy),
\end{eqnarray}
where the right hand side is understood as  a Cauchy principal value. The domain of $A$ is the collection of  function  $\phi$ such that this Cauchy principal value integral exists.
This can also be shown directly,   from the definition of generator $A_\alpha  \phi (x) = \frac{d}{dt}|_{t=0} \EX \phi(X_t)$  for   $X_t= x+L_t^\alpha$ (a $\alpha$-stable L\'evy motion starting at $x$). 


This   generator  $A_{\alpha} $ has a unique extension to   a self-adjoint operator (\cite{wu}) in the domain of definition
$W^{\alpha, 2}(\Rn) \triangleq \{g \in L^2(\Rn): \;  \|k\|^{\alpha} \mathbb{F} (g) \in L^2(\Rn)  \}$. Here the Fourier transform for $g$ is defined by
\begin{eqnarray} \label{Fourier}
 \mathbb{F} (g)(k) = \frac1{(2\pi)^{\frac{n}2}} \int_{\Rn} e^{-i\; <k, x>} g(x) dx.
\end{eqnarray} 
Sometimes, $\mathbb{F} (g)(k)$ is also denoted by  $ \hat{g}(k)$.
 Recall that $W^{\alpha, 2}(\Rn) = W^{\alpha, 2}_0(\Rn)$, i.e., all functions in $W^{\alpha, 2}(\Rn)$ have compact support in $\Rn$.
Further note that this integral operator is related to the fractional Laplacian operator \cite{wu}. Indeed, for $\alpha \in (0, 2)$, by Fourier inverse transform,
\begin{equation}\label{flaplace}
    A_{\alpha} u \triangleq   \int_{\R^n \setminus\{0\}} [u(x+  y)-u(x) ] \; \nu_{\a}(dy)=\theta_{\alpha, n} \mathbb{F}^{-1}(\|k\|^\alpha \mathbb{F}(u)(k)) = \theta_{\alpha, n} \; (-\Delta)^{\frac{\alpha}2} u(x),
\end{equation}
where
\begin{equation}
\theta_{\alpha, n} \triangleq  \int_{\R^n \setminus\{0\}} (\cos (e \cdot y) -1 ) \; \nu_{\a}(dy) <0,
\end{equation}
with $e$ being any unit vector in $\Rn$.
Here we have used the notation for the  fractional Laplacian operator:
\begin{eqnarray} \label{fractional-defn}
\mathbb{F}^{-1}(\|k\|^\alpha  \mathbb{F}(u) (k)) \triangleq (-\Delta)^{\frac{\alpha}2} u(x),
\end{eqnarray}
\begin{eqnarray} \label{fractional-defn9090}
\mathbb{F} ((-\Delta)^{\frac{\alpha}2} u(x) )  =  \|k\|^\alpha  \mathbb{F}(u) (k).
\end{eqnarray}
Clearly, this notation is inspired by the fact that
$$
\mathbb{F} (-\Delta u(x) ) =  \|k\|^2  \mathbb{F}(u) (k).
$$

\index{Fourier   transform of  fractional Laplacian operator }
\index{Fractional Laplacian operator}

Thus, the generator for the rotationally symmetric $\alpha$-stable L\'evy motion $L_t^\alpha$ in $\Rn$  is also written as
\begin{eqnarray}
 A_{\alpha}  \phi &=&   \theta_{\alpha, n} \; (-\Delta)^{\frac{\alpha}2} \phi,  \;\; \alpha \in (0, 2),
\end{eqnarray}
for $\phi$ in the domain of definition of $A_\alpha$, i.e., Sobolev space $W^{\alpha, 2}(\Rn)$. This is especially true for the scalar symmetric $\alpha$-stable L\'evy motion discussed earlier in this section.





\subsection{Stochastic Differential Equations with L\'evy Motions}
\label{SDELevynoise}




\index{Jump measure}


By the L\'evy-It\^{o} decomposition, a L\'{e}vy motion  with the generating
triplet $(b, Q, \nu)$ has the following representation
$$
L_t=b\; t+ Q^{\frac12} B_t+\int_{\|y\|< 1}y \tilde{N}(t, dy)+\int_{\|y\|\geq 1}y {N}(t, dy),
$$
where $N(dt, dx)$ is the Poisson random measure (quantifying the number of jumps of $L_t$), $\tilde{N}(dt,
dx)=N(dt, dx)-\nu(dx)dt$  is the compensated Poisson random
measure, $\nu(S) = \EX N(1, S)$ is the jump measure, $b$ and $y$ are in $\R^n$, $Q$ is a non-negative definite symmetric $n\times n$ covariance matrix, and $B_t$ is a standard $n$-dimensional Brownian motion (i.e., Wiener process). The small jumps ($\|y\|< 1$) are controlled by
  $\tilde{N}(t, dy)$, while   large jumps ($\|y\| \geq 1$) are governed by
   $N(t, dy)$.

\index{L\'evy-Ito decomposition}
\index{Small or large jumps}

As the covariance matrix $Q$ is non-negative definite and symmetric, it has non-negative eigenvalues $\lambda_1, \cdots, \lambda_n$, and an orthonormal basis $\{e_1, \cdots, e_n\}$  formed by the corresponding eigenvectors. The definition of  $Q^{\frac12}$ is via  $Q^{\frac12}x= \sum_{i=1}^n \lambda_i^{\frac12} <x, e_i> e_i$. Moreover,   $Q^{\frac12} B_t$ has covariance matrix $Q^{\frac12} (Q^{\frac12})^T = Q$.

In the L\'evy-It\^{o} decomposition, we  usually include (or absorb) the drift term $ bdt$ and Gaussian noise term $ Q^{\frac12} {\rm d} B_t$ in the corresponding terms in a stochastic differential equation. Thus the L\'evy motion $L_t$ appears to have its generating triplet $(0, 0, \nu)$ and becomes 
\begin{align}\label{s1_3aaaa}
L_t= \int_{\|y\|< 1}y \;\tilde{N}(t, dy)+\int_{\|y\|\geq 1}y\; {N}(t, dy).
\end{align}
In differential form, this is the `standard additive L\'evy noise'
\begin{align}\label{s1_4b}
{\rm d} L_t =   \int_{\|y\|< 1}y\; \tilde{N}({\rm d} t, {\rm d} y)+\int_{\|y\|\geq 1}y\;{N}({\rm d} t, {\rm d} y).
\end{align}
The `multiplicative L\'evy noise' is then
\begin{align}\label{s1_4bb}
G d L_t  \triangleq  \int_{\|y\|< 1} G(X_{t-}, y) \; \tilde{N}({\rm d} t, {\rm d} y)+\int_{\|y\|\geq 1} G(X_{t-}, y) \; {N}({\rm d} t, {\rm d} y),
\end{align} 
with an appropriate  ($n-$dimensional vector) L\'evy noise intensity $G(x,y)$.
For stochastic integrals with respect to L\'evy motion, see \cite[\S 4.3]{Applebaum}.

We consider an $n$-dimensional SDE with multiplicative L\'evy noise:
\begin{eqnarray}\label{sdelevy}
dX_t &=& f(X_{t})dt+\sigma(X_{t})  dB_t  \nonumber \\
&+& \int_{\|y\|< 1} G(X_{t-}, y) \; \tilde{N}({\rm d} t, {\rm d} y)+\int_{\|y\|\geq 1} G(X_{t-}, y) \; {N}({\rm d} t, {\rm d} y),
\end{eqnarray} 
where the vector field (or drift) $f$ is an $n-$ dimensional vector function, Gaussian noise intensity $\sigma$ is an $n\times n$   matrix, and $B_t$ is a standard $n$-dimensional Brownian motion, independent of $L_t$.  When   L\'evy noise intensity $G=G(y)$,    not depending explicitly on $X_t$, this SDE contains the `additive L\'evy noise'. In particular, when $G\equiv y$, this SDE has the `standard additive L\'evy noise' $dL_t$ in \eqref{s1_4b} and it becomes 
\begin{eqnarray}\label{sdelevy2}
dX_t &=& f(X_{t})dt+\sigma(X_{t})  dB_t + d L_t.
\end{eqnarray}






\subsection{Generators and It\^o Formula}
\label{Generator777}


The generator   for  $L_{t}$  in  \eqref{s1_3aaaa} is
(see   \S \ref{examplelevymotion} or  \cite[Theorem 3.3.3]{Applebaum} )
\begin{equation} \label{A}
A_0 h(x) =    \int_{\Rn
\setminus\{0\}} [h(x+ y)-h(x) -  I_{\{\|y\|<1\}} \; y \cdot \nabla h(x)] \; \nu(dy),
\end{equation}
where $I_S$ is the indicator function of the set $S$, i.e., it
takes value $1$ on this set and takes zero value otherwise.




The generator $A $ for the SDE  \eqref{sdelevy} or for the solution process $X_t$, is (\cite[Theorem 6.7.4]{Applebaum})
\begin{eqnarray} \label{generatorbig}
&& A h(x)  \triangleq  \sum_{i=1}^n f_i(x) \p_i h(x)+ \frac12  \sum_{i, j=1}^n (\sigma\sigma^T)_{ij}\p_i\p_j h(x) \nonumber \\
&+& \int_{\R^n \setminus \{0\}} [h(x+ G(x,y))-h(x)-I_{\{\|y\|<1\}} \sum_{i=1}^n G_i(x, y) \p_i h(x)] \; \nu(dy),
\end{eqnarray}
for $h$ in the domain of the generator. 
In vector form, this generator $A$ becomes
\begin{eqnarray}  \label{generatorbig2}
&& A h = f \cdot \grad h +\frac12 \Tr[\s \s^T H(h)]  \nonumber \\
&+& \int_{\R^n \setminus \{0\}} [h(x+G(x,y))-h(x)-I_{\{\|y\|<1\}}\; G(x, y) \cdot \grad h(x)] \; \nu(dy).
\end{eqnarray}
  Here $\; \text{Tr}\; $ is the trace of a matrix,   $^T$ denotes the transpose of a matrix, and $H$ is the Hessian matrix of a scalar function. 




 For each $h \in C^2 (\Rn)$, the It\^o formula  for   SDE  \eqref{sdelevy} is  \cite[Theorem 4.47]{Applebaum} 
 \begin{eqnarray}  \label{generatorbig8882}
 && dh(X_t) = \{ f(X_t) \cdot \grad h(X_t) +\frac12 \Tr[\s \s^T H(h(X_t))] \} dt  
 + (\nabla h(X_t)^T \sigma(X_t) dB_t     \nonumber \\
&+& \int_{\|y\|< 1} [h(X_{t-}+ G(X_{t-},y))- h(X_{t-})]\tilde{N}({\rm d} t, {\rm d} y)   \nonumber \\
&+&\int_{\|y\|\geq 1} [h(X_{t-}+G(X_{t-},y))- h(X_{t-})] {N}({\rm d} t, dy)    \nonumber \\
&+& \int_{\|y\|< 1} [h(X_{t-}+G(X_{t-},y))- h(X_{t-})- G(X_{t-},y) \cdot \nabla h(X_{t-})] \nu(dy) dt.
\end{eqnarray}
 Note that the generator and  It\^o formula  for   SDE  \eqref{sdelevy2}, with standard additive L\'evy noise, are   simpler (i.e., replacing $G(X_{t-},y)$  in \eqref{generatorbig2} and \eqref{generatorbig8882}, with $y$).


\subsection{Kolmogorov  and Fokker-Planck Equations  }
\label{sec1.3.8}

We consider first the Kolmogorov backward equation and then the Kolmogorov forward equation (i.e., the Fokker-Planck equation)  for SDE  systems in $\R^n$.


The Kolmogorov backward    equation for    SDE     \eqref{sdelevy}, with initial condition $X_0=x$,   is     \cite[\S 3.5.3]{Applebaum} 
\begin{eqnarray} \label{backwardlevy}
\p_t u (t, x) = A u(t, x), \;\; u(0, x)= u_0(x),
\end{eqnarray}
 where $u(t, x)= \EX_x [u_0(X_t)] \triangleq  \EX [u_0(X_t) | X_0=x]$, for each   observable  $u_0$ in the domain of the generator $A$.  Thus $ u(t,x)$ is the ensemble-averaged value of the observable $u_0$   at time $t$, conditioned on the initial state $X_0=x$.


The  Fokker-Planck equation    for   SDE with the generator $A$ is \cite[\S 3.5.3]{Applebaum} 
\begin{eqnarray}  \label{FPEgeneral}
\partial _t p = A^* p.
\end{eqnarray}
It is in terms of the adjoint operator $A^*$ for the generator $A$. This adjoint operator $A^*$ is easily available in the case of   additive, symmetric L\'evy noise. For example, we consider the following  SDE  
\begin{equation}\label{levysde8989}
dX_t=f(X_{t-})dt + \sigma(X_{t-}) dB_t +  dL_t^\alpha,  
\end{equation}
 where $f$ is a vector field, $\sigma$ is an $n \times n$ matrix, $B_t$ is a Brownian motion in $\Rn$,  and $L_t^\alpha$ is a symmetric $\alpha$-stable L\'{e}vy motion   in $\R^n$,  with the generating triplet     $(0, 0, \nu_\alpha)$.  The jump measure is 
$$
\nu_\a({\rm d}y)=c(n,\alpha) \|y\|^{-(n+\alpha)}\, {\rm d}y,
$$
with
$\displaystyle{c(n,\alpha)=\frac{\a\Gamma((n+\a)/2)}{2^{1-\a}\pi^{n/2}\Gamma(1-\a/2)}}$. The processes $B_t$ and $L_t^\alpha$ are independent.


The  Fokker-Planck equation    for   SDE \eqref{levysde8989}  is then
\begin{eqnarray}  \label{FPEvector}
 \partial_t p    
  = - \grad \cdot (f p) + \frac12 \Tr[H(\s \s^T p)]    \nonumber \\
+\int_{\R^n \setminus \{0\}} [p(t,x+y)-p(t, x) ] \nu_\alpha(dy),
\end{eqnarray}
where $H(\s \s^T p)$ is interpreted as matrix multiplication of the Hessian $H$ and $ \s \s^T p$ (note that $p$ is a scalar function).  Here the integral in the right hand side is understood as a Cauchy principal value.
If SDE \eqref{levysde8989}  is given initial condition $X_0=x_0$, then the initial condition for the Fokker-Planck equation above is $p(0,x)=\delta(x-x_0)$.








To obtain the Fokker-Planck equation for SDE  \eqref{sdelevy}, with multiplicative L\'evy noise,  it is considerably more complicated,   as we need to find   the adjoint operator $A^*$ for the generator $A$. See \cite{XuSun2012, zlotchevski2024} for a way to achieve this goal.








\section*{Problems}
\addcontentsline{toc}{section}{Problems}
%


\begin{prob}
\label{BandL}
\textbf{Brownian motion vs. L\'evy motion }\\
Compare the basic properties of scalar (Gaussian) Brownian motion $B_t$ and (non-Gaussian)  $\alpha$-stable L\'evy motion $L_t^\alpha$.     \\
\textit{ Hint}:   Compare sample-path regularity (almost surely continuous paths for $B_t$ versus almost surely c\`adl\`ag paths with jumps for $L_t^\alpha$), self-similarity exponents ($1/2$ versus $1/\alpha$), moment and tail behavior (all moments finite versus only moments of order less than $\alpha$ finite), and generators (the Laplacian versus the fractional Laplacian). See  also \cite{Applebaum} for a systematic discussion.
 
\end{prob}

\begin{prob}
\label{LMgenerator}
\textbf{Generator for L\'evy motion }\\
Find the generator $A_\alpha$ for a scalar symmetric $\alpha$-stable L\'evy motion $L_t^\alpha$,     using the definition   $A_\alpha  \phi (x) = \frac{d}{dt}|_{t=0} \EX \phi(X_t)$  for   $X_t= x+L_t^\alpha$ (a scalar $\alpha$-stable L\'evy motion starting at $x$).  \\
\textit{ Hint}: Write $\mathbb{E}\varphi(x+L_t^\alpha)$ using the characteristic function of $L_t^\alpha$ via Fourier inversion, differentiate in $t$ at $t=0$, and identify the resulting operator as the fractional Laplacian $-(-\Delta)^{\alpha/2}$. See Theorem~3.3.3 in \cite{Applebaum} for the full computation. 
\end{prob}


\begin{prob}
\label{multi888999} 
\textbf{Kolmogorov   equations and Fokker-Planck equations }\\
Consider a stochastic dynamical system in $\Rn$ with multiplicative   Gaussian and non-Gaussian noise
\begin{eqnarray*}\label{sdelevy888}
dX_t &=& f(X_{t})dt+\sigma(X_{t})  dB_t  \nonumber \\
&+& \int_{\|y\|< 1} G(X_{t-}, y) \; \tilde{N}({\rm d} t, {\rm d} y)+\int_{\|y\|\geq 1} G(X_{t-}, y) \; {N}({\rm d} t, {\rm d} y),
\end{eqnarray*} 
 where $f$ is a vector field, $B_t$ is a Brownian motion in $\Rn$, $\sigma$ is an  $n \times n$ matrix of functions,    L\'evy  noise intensity   $G$  is an $n-$dimensional vector of functions, and  $L_{t}$ is a   L\'evy motion in $\Rn$ with   generating triplet $(0, 0, \nu)$. The processes $B_t$ and $L_t$ are independent. What is the generator for the solution process $X_t$? What is the Kolmogorov backward equation and the Fokker-Planck equation? \\
\textit{ Hint}: Decompose the generator into the second-order differential part coming from the Brownian term and the integro-differential part coming from the compensated small-jump and uncompensated large-jump Poisson integrals; apply It\^o's formula for jump-diffusions to obtain the Kolmogorov backward equation directly, then pass to the (formal) adjoint operator to obtain the Fokker--Planck equation.  See  \cite{zlotchevski2024} and \cite{XuSun2012} for the general derivation, and Sections~1.3.6. above for the additive L\'evy noise case.

\end{prob}

\begin{prob}
\label{FPE2026}
\textbf{Geometry of a Fokker-Planck equation in the space of probability densities  }\\
Consider the       Fokker-Planck equation associated with  the scalar stochastic differential equation $dX_t = (X_t-X_t^3)dt + dB_t$. A solution of this equation is a curve  in  $\PP_2(\RR)$, the space of probability densities in $\RR$.  Try to devise a way to visualize such a curve.
How to calculate or analyze the slope and curvature of such a curve?\\
\textit{Hint:} Represent the evolving density $p_t$ as a curve in $\mathcal{P}_2(\mathbb{R})$ and use the continuity-equation description $\partial_t p + \partial_x(pv) = 0$ to identify the tangent vector (velocity field) at each time; the ``speed'' of the curve is then the  norm of this tangent vector, and curvature can be probed by comparing successive velocity fields along the flow. See Otto's original geometric formulation [171] and the local theory developed in Section~4.1 of this book.
\end{prob}

\begin{prob}
\label{Bridge}
\textbf{Brownian bridges, L\'evy  bridges and Markovian bridges  }\\
 Discuss the definition, existence, uniqueness and properties of    Brownian bridges, L\'evy  bridges and Markovian bridges.   With the help of AI, generate these bridges in the Euclidean plane.\\
 \textit{ Hint}:  Recall that a Brownian bridge conditions Brownian motion on its terminal value and can be constructed explicitly via a time-change; Markovian bridges generalize this through an $h$-transform of the transition density; L\'evy bridges require extra care since the transition density need not be explicit, and existence/uniqueness typically relies on absolute continuity of the bridge law with respect to the unconditioned process law. For the AI-assisted simulation, generate unconditioned sample paths and apply $h$-transform-based or rejection sampling to enforce the endpoint constraint. See \cite[Ch. 5]{Oksendal} for Brownian bridges, and \cite{privault2004, Chaumont2011, orland2024} for Markovian and L\'evy bridges. 
\end{prob}


%
%
%
\chapter{The Most Probable Transition Path via the Onsager-Machlup Action  }\label{ch:2}

\abstract*{Each chapter should be preceded by an abstract (no more than 200 words) that summarizes the content. The abstract will appear \textit{online} at \url{www.SpringerLink.com} and be available with unrestricted access. This allows unregistered users to read the abstract as a teaser for the complete chapter.
Please use the 'starred' version of the new \texttt{abstract} command for typesetting the text of the online abstracts (cf. source file of this chapter template \texttt{abstract}) and include them with the source files of your manuscript. Use the plain \texttt{abstract} command if the abstract is also to appear in the printed version of the book.}

\abstract{ This chapter establishes a geometric and analytical framework to quantify the most probable dynamics of stochastic systems using the Onsager-Machlup action functional.  We analyze transition paths for diffusion processes via Lagrangian and Hamiltonian formulations, and extend this framework to systems driven by multiplicative noise. Finally, we generalize the theory beyond traditional Gaussian assumptions to non-Gaussian jump-diffusion processes.}
 
\section{The Most Probable Transition Path for Diffusion Processes}\label{Mathematical introduction}
Stochastic differential equations (SDEs) are widely used to model complex dynamics in physical, biological, and engineering systems. In nonlinear systems, random perturbations may induce transitions between distinct dynamical states. A central problem is therefore to understand the mechanisms underlying such transitions and to determine the corresponding \emph{most probable transition paths (MPTPs)}. The historical development of the Onsager--Machlup (OM) theory may be
divided into two broad stages. During the early period, from 1953 to
1977, Onsager and Machlup introduced the action functional for
thermodynamic fluctuations \cite{OM53,MO53}; Tisza and Manning extended
the theory to systems with nonlinear coefficients
\cite{tisza1957fluctuations}; and path-integral formulations for
diffusion processes were further developed, notably by Graham
\cite{Graham1977}. Beginning in 1978, the theory acquired a rigorous probabilistic
interpretation in terms of small-tube probabilities
\cite{Durr1978,Ikeda1980}. Rather than assigning a positive probability
to an individual diffusion path, this approach compares the
probabilities that the process remains in equal-radius tubes around
smooth reference curves. Subsequent developments include the works
\cite{Zeitouni1989} and the more recent
applications discussed below. The contributions of It\^{o}'s school to this development are
presented in the monograph of Ikeda and Watanabe \cite{Ikeda1980}. A direct variational interpretation of the OM functional leads to an apparent difficulty. Its Euler--Lagrange equations produce twice-differentiable minimizers, whereas sample paths of diffusion processes are almost surely nowhere differentiable, and every individual continuous path has probability zero. To resolve this issue, the OM functional was given a rigorous interpretation through small-tube probabilities \cite{Durr1978,Ikeda1980}: rather than assigning probabilities to individual paths, one compares the probabilities that the diffusion remains inside thin tubes of equal radius around smooth reference curves. The minimizers of the OM functional therefore describe the most probable infinitesimal tubes rather than actual sample paths. Nevertheless, they are conventionally called most probable transition paths. This viewpoint has motivated extensive developments in path-integral methods \cite{Wiegel1986,Wio2013,Khandekar2000,Graham1977}, trajectory sampling \cite{KA20,Lu2017}, transition-path theory \cite{EVE10,huang2026transition}, transition-rate analysis \cite{Zuckerman2000,Gobbo2012}, noise-induced transitions under non-Gaussian perturbations \cite{Chao2019,Huang2019}, information projection \cite{selk2021information,bierkens2014explicit}, and Bayesian inverse problems \cite{dashti2013map,AKLS21}. Further discussions of OM functionals from different perspectives can be found in \cite{gasteratos2026onsager,selk2024smallnoise,selk2025reweighting}.

We begin with the following SDE in $\mathbb{R}^n$:
\begin{equation}\label{firstequation}
    dX_t = -\nabla U(X_t)dt+\sigma dB_t,\
    X_0 =x_0,
\end{equation}
where $U:\mathbb{R}^n\rightarrow\mathbb{R}$ is a scalar potential, $B=\{B_t\}_{t\ge0}$ is an $n$-dimensional standard Brownian motion, and $\sigma$ is a strictly positive diffusion intensity. Additional assumptions on the coefficients will be specified when needed.

\begin{definition}
A state $x_0\in\mathbb R^n$ is called a \emph{metastable state} of \eqref{firstequation} if it is a stable equilibrium of the deterministic system $\dot x_t=-\nabla U(x_t)$ \cite{Huang2019}. Equivalently, $x_0$ is a local minimizer of the potential $U$.
\end{definition}
For the one-dimensional Ginzburg--Landau double-well potential $U=\frac{1}{4}x^4-\frac{1}{2}x^2$,
the metastable states are $\pm1$. Throughout the following discussion, we assume that $x_0\in\mathbb R^n$ is a metastable state of \eqref{firstequation}.

\begin{remark}
We first consider the gradient drift $-\nabla U(x)$ for both pedagogical and physical reasons, postponing the general drift $f(x)$ to Section \ref{chapter2.sec2.2}. Equation \eqref{firstequation} describes overdamped Langevin dynamics in the potential-energy landscape $U(x)$. Since metastable states correspond to local minima of $U$, a transition may be viewed as a particle escaping from one potential well, crossing an energy barrier---typically near a saddle point---and entering another well. This geometric interpretation is generally unavailable for non-gradient systems.
\end{remark}

The classical Onsager--Machlup problem can now be stated as follows: for a prescribed time horizon $T$, among all smooth curves connecting two metastable states $x_0$ and $x_T$, which one is the most probable transition route of \eqref{firstequation}? In rigorous probabilistic terms, this question is formulated through the probability that the diffusion remains in a thin tubular neighborhood of a smooth reference curve \cite[Chapter 6, Section 9]{Ikeda1980}. This interpretation underlies the analyses in \cite{Durr1978,Ikeda1980,Zeitouni1989}. Under suitable regularity assumptions, the probability that the solution of \eqref{firstequation} remains inside a $\delta$-tube around a reference curve $\psi$ satisfies
\begin{equation}\label{OMapproximate}
\begin{split}
\mathbb{P}(\|X -\psi \|_T<\delta) \sim \exp(-S_X^{OM}(\psi))\mathbb{P}(\|\sigma B \|_T<\delta), \quad \delta\downarrow0.
\end{split}
\end{equation}
Here, $\|\cdot\|_T$ is the supremum norm on $C([0,T],\mathbb R^n)$, and $S_X^{OM}$ is the OM functional associated with \eqref{firstequation}:
\begin{equation}\label{OM}
\begin{split}
S_{X}^{OM}(\psi)=\frac{1}{2}\int_0^T \left[ \frac{|\dot{\psi}(s) + \nabla U(\psi(s))|^2}{\sigma^2} - \triangle U(\psi(s)) \right]ds.
\end{split}
\end{equation}
Thus, determining the MPTP amounts to minimizing $S_X^{OM}$ over an appropriate path space. The resulting minimizer may be regarded as an infinite-dimensional analogue of a statistical mode and provides a descriptor of the path distribution \cite{AKLS21}. Various qualitative properties of such optimal paths have been studied in \cite{DLL21,Li2021}. Applying the calculus of variations to the OM functional yields the associated Euler--Lagrange equation. This equation provides a basis for numerical methods for computing MPTPs \cite{E2002,E2004} and is also useful for analyzing their geometric structure \cite{Soskin2006}. Figure~\ref{fig:transition-path-tube} summarizes the geometric and probabilistic picture underlying the above discussion. The jagged curve represents a sample path, whereas the smooth curve is the MPTP of a tubular neighborhood whose probability is compared in the OM formulation.
\begin{figure}[htbp]
    \centering
    \includegraphics[width=0.8\textwidth]
    {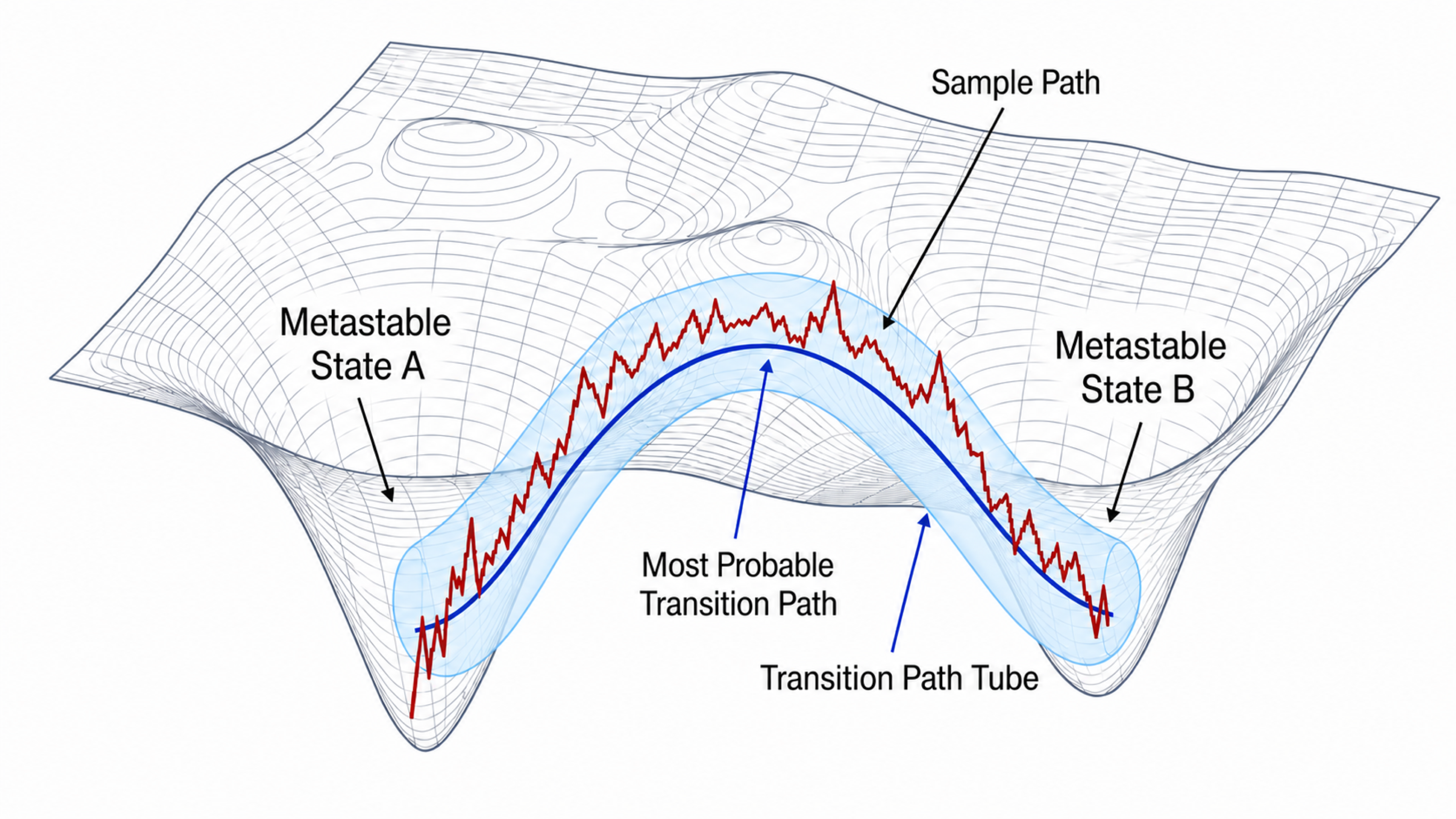}
    \caption{A stochastic sample path (red), the most probable transition path
(blue), and its surrounding transition-path tube (light blue) connecting two
metastable states.}
    \label{fig:transition-path-tube}
\end{figure}

\subsection{Preliminaries on Path Measures Induced by Diffusion Processes}

Let $C[0,T]=C([0,T],\mathbb R^n)$ be equipped with the supremum norm $\|\phi\|_T=\sup_{t\in[0,T]}|\phi(t)|,\quad \phi\in C[0,T]$.
Let $\mathcal B$ be its Borel $\sigma$-field and $\{\mathcal B_t\}_{t\in[0,T]}$ the canonical filtration $\mathcal B_t=\sigma\{\omega(s)\mid \omega\in C[0,T],\,0\le s\le t\}$.
Recall the infinitesimal generator
$A=-\nabla U\cdot\nabla+\frac12\sigma^2\triangle$. We impose the following assumptions.

\begin{assumption}\label{wellposedfirst}
(1) $U\in C^3(\mathbb R^n,\mathbb R)$.\\
(2) The local martingale problem associated with $A$ is well posed in $C[0,T]$ (see \cite[Chapter 5, Definition 4.5]{Karatzas1991}). For any $(s,x)\in[0,T)\times\mathbb R^n$, there exists a probability measure $P^{s,x}$ on $(C[0,T],\mathcal B)$ such that $P^{s,x}(\omega(r)=x,r\le s)=1$. Moreover, for every $f\in C^2(\mathbb R^n)$, $M_t^f(\omega):=f(\omega(t))-f(\omega(s))-\int_s^tAf(\omega(r))dr$ is a local martingale on $[s,T]$ with respect to the right-continuous augmentation of $\{\mathcal B_t\}_{t\in[s,T]}$ under $P^{s,x}$.
\end{assumption}

\begin{lemma}[Well-posedness of SDEs \cite{huang2024most}]\label{strong-X}
Under Assumption \ref{wellposedfirst}, the SDE \eqref{firstequation} admits a pathwisely unique strong solution $X=\{X_t\}_{t\in[0,T]}$ on $(\Omega,\mathcal F,\{\mathcal F_t\}_{t\ge0},\mathbb P^{x_0})$. The process is strongly Markovian and satisfies $\mathbb P^{x_0}(X_0=x_0)=1$.
\end{lemma}

The transition semigroup $(T_{s,t})_{0\le s<t}$ of \eqref{firstequation} is defined by $(T_{s,t}f)(x)=\mathbb E(f(X_t)\mid X_s=x)$, 
for $x\in\mathbb R^n$ and $f\in\mathfrak B_b(\mathbb R^n)$, where $\mathfrak B_b(\mathbb R^n)$ denotes the bounded Borel measurable functions on $\mathbb R^n$. Assume that $T_{s,t}$ admits a transition density $p(\cdot,t|x,s)$ with respect to a $\sigma$-finite measure $\nu$: $T_{s,t}f(x)=\int_{\mathbb R^n}f(y)p(y,t|x,s)\nu(dy)$. 
Choosing $\nu$ as the Lebesgue measure and using the time-independence of $-\nabla U$ and $\sigma$, the density is time homogeneous \cite{Ikeda1980}: $p(y,t+s|z,t)=p(y,s|z,0)$. 

Under Assumption \ref{wellposedfirst}, the density satisfies \cite[Chapter 6]{Bogachev2015}:\\
(\lowercase \expandafter {\romannumeral 1}) $(s,y,x)\mapsto p(y,s|x,0)$ is jointly continuous;\\
(\lowercase \expandafter {\romannumeral 2}) $p$ satisfies the Kolmogorov forward equation
\begin{equation}\label{forward}
\begin{split}
\frac{\partial p(x,t|x_0,0)}{\partial t}
=\nabla(\nabla U(x)p(x,t|x_0,0))
+\frac12\sigma^2\triangle p(x,t|x_0,0),
\end{split}
\end{equation}
and the Kolmogorov backward equation
\begin{equation}\label{backward}
\begin{split}
\frac{\partial p(x_T,T|x,t)}{\partial t}
=\nabla U(x)\cdot\nabla p(x_T,T|x,t)
-\frac12\sigma^2\triangle p(x_T,T|x,t).
\end{split}
\end{equation}

Since $X$ is strongly Markovian, the Chapman--Kolmogorov equation $p(y,T|x_0,0)=\int_{\mathbb R^n}p(z,T-s|x_0,0)p(y,s|z,0)dz$, holds for $0<s<T$ and $y\in\{x\mid p(x,T|x_0,0)>0\}$.
Let $x_0$ be the initial state of \eqref{firstequation} and $x_T$ a target state. Define $C_{x_0}[0,T]
=\{\phi:[0,T]\to\mathbb R^n~\mbox{continuous},~\phi(0)=x_0\}$. 
Except when $x_0=0$, this is not a vector space. It inherits the subspace topology from $C[0,T]$, with Borel $\sigma$-field $\mathcal B^{x_0}_{[0,T]}$. Equivalently, $\mathcal B^{x_0}_{[0,T]}$ is generated by cylinder sets $I=\{\phi\in C_{x_0}[0,T]\mid \phi(t_1)\in E_1,\cdots,\phi(t_n)\in E_n\}$, 
where $0\le t_1<\cdots<t_n\le T$ and each $E_i$ is Borel in $\mathbb R^n$ \cite{Karatzas1991}. For $\psi\in C_{x_0}[0,T]$ and $\delta>0$, define the open and closed tubes $K_T(\psi,\delta)
=\{\phi\in C_{x_0}[0,T]\mid \|\psi-\phi\|_T<\delta\}$, 
and $
\bar K_T(\psi,\delta)
=\{\phi\in C_{x_0}[0,T]\mid \|\psi-\phi\|_T\le\delta\}$.
Denote by $B_\rho(x)$ and $\bar B_\rho(x)$ the open and closed balls in $\mathbb R^n$. The law of $X$ on $(C_{x_0}[0,T],\mathcal B^{x_0}_{[0,T]})$ is $\mu_X^{x_0}(B)
=\mathbb P^{x_0}(\{\omega\in\Omega\mid X(\omega)\in B\})$, 
for $B\in\mathcal B^{x_0}_{[0,T]}$. The measure associated with $\sigma W$ is the Wiener measure $\mu_{\sigma W}^{x_0}$. Comparing tube probabilities $\mu_X^{x_0}(\bar K_T(\psi,\delta))$ for different $\psi$ leads naturally to the notion of most probable transition paths.

\begin{definition}[Most probable transition path]
For \eqref{firstequation}, the most probable transition path connecting $x_0$ and $x_T$ is a curve $\psi^*$ minimizing the OM functional \eqref{OM} over $C^2_{x_0,x_T}[0,T]
:=\{\psi:[0,T]\to\mathbb R^n\mid
\dot\psi,\ddot\psi~\mbox{continuous},~
\psi(0)=x_0,~
\psi(T)=x_T\}$, 
that is, $S_X^{OM}(\psi^*)
=\inf_{\psi\in C^2_{x_0,x_T}[0,T]}S_X^{OM}(\psi)$.
\end{definition}

By \eqref{OMapproximate}, this is equivalent to
$\lim_{\delta\downarrow0}
\frac{\mu_X^{x_0}(K_T(\psi^*,\delta))}
{\mu_X^{x_0}(K_T(\psi,\delta))}
\ge1$, 
for all $\psi\in C^2_{x_0,x_T}[0,T]$. Open tubes may also be replaced by closed tubes; see \cite[Theorem 9.1]{Ikeda1980}, \cite[Section 4]{Durr1978}, and \cite[Theorem 1]{Zeitouni1988}.

\begin{lemma}[Approximation for probabilities of closed tubes \cite{huang2021estimating,huang2024most}]\label{cylinderset}
Let $\mu$ be a probability measure on $(C_{x_0}[0,T],\mathcal B^{x_0}_{[0,T]})$. For any $\psi\in C_{x_0}[0,T]$ and $\delta>0$, there exists a sequence of closed cylinder sets $\{\bar I_n\}_{n=1}^\infty$ such that $\mu(\bar K_T(\psi,\delta))
=\lim_{n\to\infty}\mu(\bar I_n)$.
\end{lemma}

Consider the SDE in $\mathbb R^n$
\begin{equation}\label{SDEt}
dZ_t=f(t,Z_t)dt+\sigma dB_t,\quad Z_0=x_0,
\end{equation}
and its deterministic counterpart
\begin{equation}\label{MPPequ}
d\psi(t)=f(t,\psi(t))dt,\quad \psi(0)=x_0,
\end{equation}
where $f:[0,T]\times\mathbb R^n\to\mathbb R^n$. We define most probable paths (MPPs). Unlike MPTPs, MPPs are conditioned only on the initial state.

\begin{definition}[Most probable paths]
A curve $\psi^*\in C^2_{x_0}[0,T]$ is called a most probable path of \eqref{SDEt} if
\begin{equation}\label{eqn-1}
\begin{split}
\lim_{\delta\downarrow0}
\frac{\mu_Z^{x_0}(\bar K_T(\psi^*,\delta))}
{\mu_Z^{x_0}(\bar K_T(\psi,\delta))}
\ge1,
\quad \forall\,\psi\in C^2_{x_0}[0,T].
\end{split}
\end{equation}
\end{definition}

\begin{lemma}[Most probable paths \& 1st order ordinary differential equations \cite{huang2024most}]\label{MPP}
Assume weak existence and uniqueness hold for \eqref{SDEt} on $[0,T]$. Then $\psi^*\in C^2_{x_0}[0,T]$ is a most probable path of \eqref{SDEt} if and only if it solves \eqref{MPPequ}, either when the drift is linear or in the small-noise limit $\sigma\to0$.
\end{lemma}

\subsection{Markovian Bridge and the Most Probable Transition Path}\label{Markovianbridge}
Within the current mathematical configuration, the conditional probability law $\mathbb{P}^{x_0}(X\in\cdot \mid X_T)$ is known to admit a regular version; mathematically, it defines a regular conditional distribution for the process $X$ given the terminal state $x_T$ under the probability measure $\mathbb{P}^{x_0}$ \cite{Chaumont2011,Fitzsimmons1993}. We designate $\mu_{X}^{x_0,\cdot}$ as the associated probability kernel mapping from $\mathbb R^n$ into $C_{x_0}[0,T]$, which will be referred to as the \emph{bridge measure}. Explicitly, for $(\mathbb P^{x_0}\circ X_T^{-1})$-almost every $x_T\in\mathbb R^n$ and any Borel subset $B\in \mathcal{B}^{x_0}_{[0,T]}$, this relation implies $\mu_{X}^{x_0, x_T}(B) = \mathbb{P}^{x_0}(X\in B \mid X_T = x_T)$. 

Evaluated under the conditional measure $\mathbb{P}^{x_0}(\cdot\mid X_T=x_T)$, the stochastic trajectory $\{X_t\}_{0\leq t<T}$ functions as the specific $(x_0,T,x_T)$-bridge deduced from the original process $X$. Notably, this derived bridge strictly preserves the strong Markov property \cite{Chaumont2011,Fitzsimmons1993}, possessing the transition densities:
\begin{equation}\label{conditionalp}
\begin{split}
p^{x_0,x_T}(y,t|x,s)& =\frac{p(y,t|x,s)p(x_T,T|y,t)}{p(x_T,T|x,s)},~~0\leq s<t<T.
\end{split}
\end{equation}
Furthermore, it holds trivially that $\mu_{X}^{x_0,x_T}(\{\psi\in C_{x_0}[0,T]\mid \psi(T)=x_T\})=1$.

Considering an arbitrary cylinder set $I=\{\psi\in C_{x_0}[0,T]\mid\psi(t_1)\in E_1,\cdots,\psi(t_n)\in E_n\}$ structured with temporal nodes $0< t_1<t_2<\cdots<t_n<T$ and spatial Borel constraints $E_i\subset\mathbb{R}^n$, its probabilistic mass is evaluated as:
\begin{equation}\label{finitedistribution}
\begin{split}
\mu_X^{x_0,x_T}(I)
=&\int_{\{x_i\in E_i\}}\frac{p(x_1,t_1|x_0,0) \cdots p(x_T,T|x_{n},t_{n})}{p(x_T,T|x_0,0)}dx_1\cdots dx_{n}.
\end{split}
\end{equation}

By synthesizing the fundamental relations provided in $(\ref{forward})$, $(\ref{backward})$, and $(\ref{conditionalp})$, one can deduce that the bridge transition density function $p^{x_0,x_T}(x,t|x_0,0)(\equiv p^{x_0,x_T})$ intrinsically governs the following partial differential equation \cite{Cetin2016}:
\begin{equation*}
\begin{split}
\frac{\partial p^{x_0,x_T}}{\partial t}= \nabla\left[ \left(\nabla U - \sigma^2 \nabla\log   p(x_T,T|x,t)  \right)p^{x_0,x_T} \right]+\frac{1}{2}\sigma^2\triangle p^{x_0,x_T}.
\end{split}
\end{equation*}
Throughout the remainder of this exposition, any spatial derivative operator acting on $\log  p$ (or $p$) is explicitly executed with respect to the third argument of the function, e.g., $\nabla_x \log  p(\cdot, \cdot| x,\cdot)$. Structurally, this PDE mirrors the classical Fokker-Planck equation. This analytical parallel allows us to systematically pair the transition density $p^{x_0,x_T}(x,t|x_0,0)$ with a novel $n$-dimensional SDE. This derived SDE is constructed on a sufficiently expansive probability space $(\tilde{\Omega},\mathcal{\tilde F}, \{\tilde{\mathcal{F}}_t\}_{t\ge0}, \mathbb{\tilde P})$ capable of supporting an independent $n$-dimensional Brownian motion $\tilde{B}=(\tilde{B}^1,\cdots,\tilde{B}^n)$:
\begin{equation}\label{newsde}
   dY_t=\left[ - \nabla U(Y_t)+\sigma^2\nabla\log  p(x_T,T|Y_t,t) \right]dt+\sigma d\tilde{B}_t,\quad t\in[0,T).
\end{equation}
Equation \eqref{newsde} is hereafter designated as the \emph{bridge SDE} appended to the core dynamics \eqref{firstequation}. From a historical and probabilistic standpoint, this elegant formulation was pioneered by Doob \cite{Doob1957} and is predominantly recognized as the \emph{Doob h-transform} of the SDE \eqref{firstequation}. To maintain notational elegance, the specific drift vector driving this bridge SDE \eqref{newsde} is encapsulated by:
\begin{equation}\label{m-drift}
f(t,x) := - \nabla U(x)+\sigma^2\nabla\log  p(x_T,T|x,t),
\end{equation}
which we shall conventionally label the modified drift. Rigorous proofs establishing both the weak and strong well-posedness of \eqref{newsde} were presented in \cite{Cetin2016}, contingent upon relatively modest analytical constraints. In particular, provided that Assumption \ref{wellposedfirst} is satisfied (alongside the supplementary Assumption \ref{assumptiononEU} when the spatial dimension $k\geq2$), the strict existence and uniqueness of a strong solution to \eqref{newsde} are mathematically assured (cf. \cite[Theorem 4.1]{Cetin2016}). Following our previous notational conventions, let $\mathbb{\tilde P}^{x_0}$ denote the conditionally restricted probability $\mathbb{\tilde P}(\cdot\mid Y_0=x_0)$. Consequently, one arrives at the definitive identity for any arbitrarily chosen Borel subset $E\subset\mathbb{R}^n$:
\begin{equation}\label{densityY}
\begin{split}
\mathbb{\tilde P}^{x_0}(Y_t\in E)=\int_E\frac{p(y,t|x_0,0)p(x_T,T|y,t)}{p(x_T,T|x_0,0)}dy, \quad 0<t<T,
\end{split}
\end{equation}
which inherently guarantees the terminal constraint $\mathbb{\tilde P}^{x_0}(Y_T=x_T)=1$.

\begin{assumption}\label{assumptiononEU}
  For multi-dimensional cases where $n\geq2$, we impose the additional criteria that the strict positivity condition $p(x_T,T|x_0,0)>0$ holds, and that the potential scalar curvature is bounded from below as $-\triangle U \ge \xi$ for some constant $\xi\in\mathbb R$. Furthermore, the mapped function $h(t,x)=p(x_T,T|x,t)$ must belong to the continuity class $C^{1,2}([0,T)\times\mathbb R^n)$.
\end{assumption}

\begin{remark}\label{remark-1}
It is crucial to note that the foundational stochastic architecture $(\tilde{\Omega},\mathcal{\tilde F}, \{\tilde{\mathcal{F}}_t\}_{t\ge0}, \mathbb{\tilde P})$ and the driving noise $\tilde W$ utilized for the bridge SDE \eqref{newsde} are \emph{not} rigidly required to be identical to the underlying configuration $(\Omega,\mathcal{F}, \{\mathcal{F}_t\}_{t\ge0}, \mathbb{P})$, $W$ supporting the unconditioned SDE \eqref{firstequation}. This flexibility arises because the strong well-posedness arguments articulated in Lemmas \ref{strong-X} critically rely on the classical Yamada--Watanabe framework. Such a framework dictates that as long as the supporting probability space is sufficiently rich to accommodate a Brownian motion, the corresponding SDE can reliably admit a uniquely defined strong solution (compare with \cite[Section 5.3]{Karatzas1991}).
\end{remark}

Within the scope of this specific subsection, we establish a pivotal theoretical link that serves as the architectural foundation for our primary theorem presented subsequently. This critical relation is encapsulated in the ensuing lemma:
\begin{lemma}[Onsager--Machlup functional \& bridge measures]\label{OMderived}
There exists an absolute strictly positive constant $C>0$ guaranteeing that, for any smooth reference path $\psi\in C^2_{x_0,x_T}[0,T]$, the following asymptotic proportionality holds valid:
\begin{equation*}
\mu_X^{x_0,x_T}(\bar{K}_T(\psi,\delta))\sim C\exp\left( -S^{OM}_X(\psi) \right)\ \mu_{\sigma W}^{0,0}(\bar{K}_T(0,\delta)) \quad \text{as } \delta\downarrow0.
\end{equation*}
\end{lemma}
\begin{remark}
In the landmark treatments found in \cite{Durr1978,Ikeda1980}, the elegant formulation of the OM action functional was extracted directly from the unconditioned measure $\mu_X^{x_0}$ via a dual application of the classical Girsanov theorem. Such a pathway is logically sound primarily because $\mu_X^{x_0}$ maintains absolute continuity with respect to the scaled Wiener measure $\mu_{\sigma W}^{x_0}$, and both underlying measures possess the property of quasi-translation invariance (an extensive elaboration is provided in \cite{Durr1978}). Conversely, the strictly conditioned bridge measures $\mu_X^{x_0,x_T}$ and $\mu_{\sigma W}^{x_0,x_T}$ completely lose this quasi-translation invariance within the path space $C_{x_0}[0,T]$. This profound structural discrepancy marks the primary divergence between the analytical approach presented here and the historical methodologies utilized in \cite{Durr1978,Ikeda1980} for deriving the OM functional.
\end{remark}

Conditioning the solution of \eqref{firstequation} to reach $x_T$ at time $T$ yields the regular conditional probability measure (bridge measure $\mu_X^{x_0,x_T}$). It is heuristically represented by the following stochastic boundary value problem, called a \emph{conditioned SDE} \cite{Pinski2012}:
\begin{equation}\label{sdetwobridge} 
dX_t = - \nabla U(X_t)dt+\sigma dB_t,\ 
X_0 =x_0,\quad X_T=x_T. 
\end{equation}
This formulation represents the original dynamics \eqref{firstequation} under the conditional law $\mathbb{P}^{x_0}(\cdot\mid X_T=x_T)$. It differs from the bridge SDE \eqref{newsde}: the former imposes two-point boundary conditions, whereas the latter is an initial value problem with a modified drift.

The modified drift in \eqref{newsde} becomes singular as $t\to T$. This singular attraction forces trajectories of $Y$ to converge to $x_T$ at time $T$ \cite{Cetin2016}. Hence, the terminal constraint is already encoded in the dynamics of $Y$ and need not be imposed separately. The optimization problem therefore reduces to finding, among all smooth paths starting from $x_0$, the most probable trajectory of \eqref{newsde}. The continuous spatial progression of $Y$, governed by $(\ref{newsde})$, seamlessly induces an empirical measure $\mu_Y^{x_0}$ mapping over the Borel $\sigma$-field $\mathcal{B}^{x_0}_{[0,T]}$, formalized via its distributional law $\mu_Y^{x_0}(B)=\mathbb{\tilde P}^{x_0}\left( \{w\in\tilde{\Omega}\mid Y (\omega)\in B\} \right),\ B\in\mathcal{B}^{x_0}_{[0,T]}$.

To meticulously unravel the mathematical equivalency connecting the most probable transition paths of the base system $X$ to the unconstrained most probable paths of the auxiliary system $Y$, one must leverage the analytical properties of the bridge measure $\mu_{X}^{x_0,x_T}$. Even though the dual representations $(\ref{newsde})$ and \eqref{sdetwobridge} might technically inhabit disparate background probability spaces, their respective path-space induced measures, $\mu_Y^{x_0}$ and $\mu_X^{x_0,x_T}$, share identical residence within the common measurable domain $(C_{x_0}[0,T],\mathcal{B}^{x_0}_{[0,T]})$. The ensuing lemma establishes the absolute functional equivalence bridging the measures $\mu_Y^{x_0}$ and $\mu_{X}^{x_0,x_T}$.

\begin{lemma}[Bridge measures $=$ laws of bridge SDEs]\label{equivalence}
The two measures $\mu_{X}^{x_0,x_T}$ and $\mu_Y^{x_0}$ coincide, i.e., $\forall B\in\mathcal{B}^{x_0}_{[0,T]}$, $\mu_{X}^{x_0,x_T}(B)=\mu_Y^{x_0}(B)$.
\end{lemma}



Operating under the unified theoretical canopy provided by Lemmas \ref{OMderived} and \ref{equivalence}, for any competitive pairing of smooth curves $\psi_1,\psi_2\in C^2_{x_0,x_T}[0,T]$, the following precise asymptotic equality sequence is mandated:
\begin{equation*}
\begin{split}
\lim_{\delta\downarrow0} \frac{\mu_Y^{x_0}(\bar K_T(\psi_1,\delta))}{\mu_Y^{x_0}(\bar K_T(\psi_2,\delta))} =&\ \lim_{\delta\downarrow0} \frac{\mu_X^{x_0,x_T}(\bar K_T(\psi_1,\delta))}{\mu_X^{x_0,x_T}(\bar K_T(\psi_2,\delta))}
 = \lim_{\delta\downarrow0} \frac{\mu_X^{x_0}(\bar{K}_T(\psi_1,\delta))}{\mu_X^{x_0}(\bar{K}_T(\psi_2,\delta))}.
\end{split}
\end{equation*}
This chain of equalities confirms an incredibly powerful analytical shortcut: resolving the complex issue of isolating the most probable transition path belonging to the native dynamics \eqref{firstequation} can be seamlessly achieved by identifying the more accessible most probable path corresponding to the bridged formulation \eqref{newsde}. Adding to this mathematical convenience, courtesy of the mechanisms outlined in Lemma \ref{MPP}, the elusive most probable path attached to \eqref{newsde} can be analytically extracted via its corresponding deterministic ODE equivalent. These cascading theoretical deductions culminate to form the central thesis of the current investigation, officially formulated in the subsequent principal theorem.

\begin{theorem}[Sufficient \& necessary characterization for the most probable transition path \cite{huang2024most}]\label{maintheory}
Provided that the foundational hypotheses outlined in Assumptions \ref{wellposedfirst} and \ref{assumptiononEU} strictly hold, the following conclusions are valid: \\
(i) The targeted most probable transition path(s) corresponding to the original unconditioned stochastic environment \eqref{firstequation} map identically onto the most probable transition path(s) intrinsic to the algorithmically coupled bridge SDE (\ref{newsde}). \\
(ii) A deterministic optimal curve $\psi^*\in C^2_{x_0,x_T}[0,T]$ qualifies structurally as a most probable transition path for system \eqref{firstequation} if and only if it actively satisfies the governing first-order ordinary differential equation framework:
\begin{equation}\label{MPTP}
   d\psi^*(t)=\left[- \nabla U(\psi^*(t))+\sigma^2\nabla\log  p(x_T,T|\psi^*(t),t) \right]dt, \ t\in(0,T),\ \psi^*(0)=x_0,
\end{equation}
in which the kernel $p(\cdot,\cdot|\cdot,\cdot)$ strictly embodies the evolving transition density connected to the dynamic solution of \eqref{firstequation}, evaluated under the condition of possessing either a  linear drift  or undergoing an asymptotically vanishing noise intensity limit.
\end{theorem}

\subsection{The Most Probable Transition Path and Time Reversal}\label{timereversal}

The most probable transition path $\psi^*$ associated with the OM functional of \eqref{firstequation} coincides with the most probable path of the Markovian bridge SDE \eqref{newsde}. Hence, $\psi^*$ satisfies the following deterministic ODE; see \cite[Lemma 2.6]{huang2024most}:
\begin{equation}\label{2dMPTP}
\left\{
\begin{aligned}
\frac{d\psi^*(t)}{dt}
=&-\nabla U(\psi^*(t))
+\sigma^2\nabla\log p(x_T,T\mid x,t)\big|_{x=\psi^*(t)},
\quad t\in(0,T),\\
\psi^*(0)=&\,x_0.
\end{aligned}
\right.
\end{equation}

For a fixed time horizon $T$, let $\mathcal{L}_T$ denote the distribution of the solution $X_T$ to \eqref{firstequation}. Reversing time from $T$ to $0$ gives the reverse-time representation \cite{Nelson}
\begin{equation}\label{reversedSDE}
\left\{
\begin{aligned}
d_*X_t
=&\left[
-\nabla U(X_t)
-\sigma^2\nabla\log p(x,t\mid x_0,0)\big|_{x=X_t}
\right]d_*t
+\sigma d_*\widehat{B}_t,
\quad t\in(0,T],\\
X_T\sim&\,\mathcal{L}_T,
\end{aligned}
\right.
\end{equation}
where $\widehat{B}$ is a Brownian motion under the reversed time flow and $d_*$ denotes the backward-time differential. If the score term $\sigma^2\nabla\log p(x,t\mid x_0,0)\big|_{x=X_t}$ 
is available, trajectories of \eqref{reversedSDE} can be generated by sampling their terminal state from $\mathcal{L}_T$. Conditioning this terminal state on $X_T=x_T$ is equivalent to conditioning the forward process to reach $x_T$ at time $T$. The resulting reverse-time process has the same path-space distribution as the forward Markovian bridge \eqref{newsde}, up to time orientation. Therefore, the reverse-time Markovian bridge is
\begin{equation}\label{reversednews}
\left\{
\begin{aligned}
d_*X_t
=&\left[
-\nabla U(X_t)
-\sigma^2\nabla\log p(x,t\mid x_0,0)\big|_{x=X_t}
\right]d_*t
+\sigma d_*\widehat{B}_t,
\quad t\in(0,T],\\
X_T=&\,x_T.
\end{aligned}
\right.
\end{equation}
Thus, instead of computing the MPTP from the forward bridge, one may equivalently determine the most probable path of \eqref{reversednews}.

\begin{theorem}[Necessary and sufficient characterization of the most probable transition path]
\label{forwardreversetheory}
Assume that the MPTPs of \eqref{firstequation} coincide with those of the forward bridge \eqref{newsde} and the reverse-time bridge \eqref{reversednews}. Then a path
$\psi^*\in C^2_{x_0,x_T}[0,T]$ is an MPTP of \eqref{firstequation} if and only if it satisfies both the forward ODE
\begin{equation*}
\left\{
\begin{aligned}
d\psi^*(t)
=&\left[
-\nabla U(\psi^*(t))
+\sigma^2\nabla\log p(x_T,T\mid x,t)\big|_{x=\psi^*(t)}
\right]dt,
\quad t\in(0,T),\\
\psi^*(0)=&\,x_0,
\end{aligned}
\right.
\end{equation*}
and the reverse-time ODE
\begin{equation}\label{reverseMPTP}
\left\{
\begin{aligned}
d_*\psi^*(t)
=&\left[
-\nabla U(\psi^*(t))
-\sigma^2\nabla\log p(x,t\mid x_0,0)\big|_{x=\psi^*(t)}
\right]d_*t,
\quad t\in(0,T),\\
\psi^*(T)=&\,x_T.
\end{aligned}
\right.
\end{equation}
Here, $p(\cdot,\cdot\mid\cdot,\cdot)$ denotes the transition density of \eqref{firstequation}, considered for linear drift or in the vanishing-noise limit.
\end{theorem}

The forward and reverse-time solutions represent the same geometric MPTP, differing only in their time orientation. When a hat notation such as $\widehat{\psi}^*$ is used, it merely indicates that the path is obtained from the reverse-time formulation.

\subsection{Lagrangian and Hamiltonian Mechanics Viewpoints}\label{sec:SM-Lag-Ham}

{\bf Lagrangian mechanics viewpoint:} Let us examine the functional expansion $S_{X}^{OM}(\psi)
= \frac{S (\psi)}{\sigma^2} + \frac{U(x_T)-U(x_0)}{\sigma^2}$, 
where the classical action structural components are defined by:
\begin{equation}\label{hamiltonSL}
\begin{split}
S(\psi):=&\ \int_0^T L(\psi,\dot{\psi})ds,\ 
L(x,y):=\ \frac{1}{2} |y|^2 + V(x),\\
V(x):=&\  \frac{1}{2}|\nabla U(x)|^2 - \frac{\sigma^2}{2}\triangle U(x).
\end{split}
\end{equation}
For any stipulated constant noise intensity $\sigma$, the optimal trajectory $\psi^*$ simultaneously acts as the global minimizer for the following variational optimization problem $(\psi^*)=\inf_{\psi\in C^2_{x_0,x_T}[0,T]}S(\psi)$.
By invoking the standard calculus of variations, $\psi^*$ is required to satisfy the governing Euler-Lagrange equation tied to \eqref{firstequation}, expressed as:
\begin{equation}\label{2dEL} 
\ddot{\psi}^* +\nabla \left(\frac{1}{2}\sigma^2\Delta U(\psi^*)-\frac{1}{2}|\nabla U(\psi^*)|^2\right)=0,\ 
 \psi^*(0)=x_0,\ \psi^*(T)=x_T.
\end{equation}
This formulation can be elegantly mapped onto the following Hamiltonian dynamical system subject to boundary constraints $\dot{Q}(t)=\ P(t),\ \dot{P}(t)=\ \nabla V(Q(t)),\ 
    Q(0)=\ x_0,\ Q(T)=\ x_T$, 
by setting $Q=\psi^*$ and identifying the conjugate momentum as $P=\dot{\psi}^*$.

Classical mechanical principles dictate that the total energy (representing the Hamiltonian) of the system remains strictly conserved along the optimal path $\psi^*$. Mathematically, this translates to $\frac{1}{2}|\dot{\psi}^*|^2 - V(\psi^*)\equiv const,\ \forall t\in[0,T]$. 
Assuming the uniqueness of $\psi^*$, the invariant constant $E$ is completely predetermined by the chosen spatial boundaries $x_0$, $x_T$, combined with the transition duration $T$. Holding $x_0$ and $x_T$ fixed implies that $E$ effectively behaves as a function uniquely dependent on $T$. In such scenarios, the intrinsic relationship bridging $E$ and $T$ is governed by the integral equation $T=\int_{\gamma(\psi^*)}\frac{|d\psi^*|}{\sqrt{2K_{E(T)}(\psi^*)}}$, in which $K_{E(T)}(\psi^*)=E(T)+V(\psi^*)$ formulates the kinetic energy profile, and $\gamma(\psi^*)$ encapsulates the geometric trace of $\psi^*$.

{\bf Hamiltonian mechanics viewpoint:}
The forward and reverse-time characterizations
\eqref{2dMPTP} and \eqref{reverseMPTP} form a time-adjoint pair of
first-order ODEs for the MPTP. Our use of the reverse-time differential
$d_*$ follows the forward--backward kinematics of Markov diffusion
processes developed by Nelson \cite{Nelson}. In this formulation,
the additional score term in the backward or bridge drift carries the
conditioning information imposed at the opposite temporal endpoint;
in the forward-time interpretation, it accounts for conditioning on
the prescribed future state. Although these first-order equations provide concise characterizations
of the MPTP, their direct use presents a substantial practical
difficulty. They require the transition density
$p(x_T,T\mid\cdot,\cdot)$, or its forward counterpart
$p(\cdot,\cdot\mid x_0,0)$, to be evaluated throughout
$\mathbb{R}^n\times[0,T]$. These densities satisfy, respectively, the
Kolmogorov backward equation and the Fokker--Planck forward equation,
for which explicit solutions are rarely available. Short-time and
small-noise asymptotic approximations may be used, but their accuracy
can deteriorate for long transition times or non-negligible noise
intensities. This motivates the Lagrangian and Hamiltonian
reformulations developed below.

It should nevertheless be emphasized that the resulting Hamiltonian
equations, when supplemented with the prescribed initial and terminal
positions, constitute a two-point boundary-value problem rather than
an initial-value problem. Existence and uniqueness are therefore not
automatic: depending on the potential, the transition time, and the
boundary data, the Hamiltonian system may admit no solution, a unique
solution, or multiple solutions, possibly infinitely many. Moreover,
solutions of the Hamiltonian boundary-value problem are stationary
paths of the action functional and need not all be global minimizers;
when several solutions exist, their action values must be compared to
identify the MPTP.
The subsequent theorem introduces a robust alternative analytical methodology for determining the most probable transition path.

\begin{theorem}\label{MPTPinHamilton}
Under the assumptions of Theorem \ref{maintheory}, let $\psi^*$ denote the corresponding solutions to \eqref{2dMPTP} and \eqref{reverseMPTP}. By establishing the phase-space coordinates $(Q(t),P(t))=(\psi^*(t),\dot{\psi}^*(t))$ over the interval $t\in[0,T]$, it follows that the pair $(Q,P)$ strictly adheres to the ensuing Hamiltonian framework: 
\begin{equation}\label{HamiltonQP}
\left\{
\begin{aligned}
    \dot{Q}=&\ P,\quad
    \dot{P}=\ \nabla V(Q),\\
    Q(0)=&\ x_0,\ P(0)=\ -\nabla U(x_0)+\sigma^2\nabla\log  p(x_T,T|x,0)|_{x=x_0}
\end{aligned}\right.
\end{equation}
as well as its temporally reversed counterpart:
\begin{equation}\label{HamiltonQPreverse}
\left\{
\begin{aligned}
    \dot{Q}=&\ P,\quad 
    \dot{P}=\  \nabla V(Q),\\
    Q(T)=&\ x_T,\ P(T)=\ -\nabla U(x_T)-\sigma^2\nabla\log  p(x,T|x_0,0)|_{x=x_T}.
\end{aligned}\right.
\end{equation}
\end{theorem}
\begin{proof}
As previously established, the trajectory $\psi^*$ acts as the global minimizer for the action functional $S$ defined mathematically in \eqref{hamiltonSL}. In accordance with variational principles, it must satisfy the connected Euler-Lagrange constraints. Consequently, the optimal trajectory phase coordinates also seamlessly integrate into the corresponding Hamiltonian system $\dot{Q}=\ \frac{\partial H}{\partial P}=P,\ \dot{P}=\ -\frac{\partial H}{\partial Q}= \nabla V(Q)$, where the transformation $(Q(t),P(t))=(\psi^*(t),\dot{\psi}^*(t))$ is applied. Therefore, the boundary conditions governing the forward evaluation from $t=0$ are rigorously defined as $Q(0)=x_0,~P(0)=-\nabla U(x_0)+\sigma^2\nabla\log  p(x_T,T|x,0)|_{x=x_0},$ 
whereas the conditions anchoring the backward calculation from $t=T$ are given by $Q(T)=x_T,~P(T)=-\nabla U(x_T)-\sigma^2\nabla\log  p(x,T|x_0,0)|_{x=x_T}.$
\end{proof}

\begin{remark}
Theorem \ref{MPTPinHamilton} structurally implies that the most probable transition path(s) governing a stochastic system may potentially lack geometric uniqueness; nevertheless, in scenarios where multiple optimal trajectories coexist, they are strictly mathematically constrained to share identical initial and terminal velocity vectors.
\end{remark}

Having previously demonstrated the analytical efficacy of our theoretical framework on simpler setups, such as free Brownian motion and the Ornstein-Uhlenbeck process, we now pivot to evaluate a considerably more complex, nonlinear system that concurrently manages to retain an explicit transition density formulation.
\begin{example}[Hongler's model]
The fundamental dynamics of Hongler's model are captured by the following 1-dimensional nonlinear SDE:
\begin{equation}\label{hongler}
    \begin{aligned}
        dX_t=-\frac{dU}{dx}(X_t)dt+dB_t,\quad X_0=x_0,
    \end{aligned}
\end{equation}
equipped with the structured potential $U(x)= \frac{\sqrt{A}}{2} x^2-\log  \left[ {}_1F_1\left(\frac{B}{4\sqrt{A}}+\frac{1}{4};\frac{1}{2};\sqrt{A}x^2 \right) \right]$, 
where ${}_1F_1$ symbolically represents the standard confluent hypergeometric function, and the defining constants $A$ and $B$ satisfy the parameter constraints $A\geq0,\quad B\geq-\sqrt{A}$. Figure \ref{Honglerpotential} visualizes the landscape of the potential $U$ configured with the exact parameters $A=1$ and $B=-0.6$. The rigorous transition density function corresponding to the solution process of \eqref{hongler} can be analytically expressed as: 
\begin{equation*}
\begin{split}
p(x,t|y,s)=&\ \exp\left \{ \frac{\sqrt{A}}{e^{-2\sqrt{A}(t-s)}-1} \left(x-ye^{-\sqrt{A}(t-s)}\right)^2 -\frac{B}{2}(t-s) \right\} {}_1F_1\left(\frac{B}{4\sqrt{A}}+\frac{1}{4};\frac{1}{2};\sqrt{A}x^2 \right)\\
&\ \times \left[ \sqrt{ \frac{2\pi}{\sqrt{A}} \sinh\left(\sqrt{A}(t-s)\right)        }    {}_1F_1\left(\frac{B}{4\sqrt{A}}+\frac{1}{4};\frac{1}{2};\sqrt{A}y^2 \right)\right]^{-1}.
\end{split}
\end{equation*}

Consequently, the core system \eqref{hongler} can be explicitly expanded into:
\begin{equation*} 
    \begin{aligned}
        dX_t=\left(-\sqrt{A}X_t + \left(B+\sqrt{A}\right)\frac{ {}_1F_1\left(\frac{B}{4\sqrt{A}}+\frac{5}{4};\frac{3}{2};\sqrt{A}x^2 \right)}{{}_1F_1\left(\frac{B}{4\sqrt{A}}+\frac{1}{4};\frac{1}{2};\sqrt{A}x^2 \right)}X_t \right)dt+dB_t,\quad X_0=x_0.
    \end{aligned}
\end{equation*}
Upon executing algebraic simplification, the Hamiltonian system governing the target MPTP structurally reduces to $\frac{dQ}{dt}=\ P,\ 
    \frac{d P}{dt}=\ AQ$,
driven by the precise initial boundaries $Q(0)=x_0,~P(0)= -x_0 + 2e^{-\sqrt{A}T}[\sqrt{A}/(1-e^{-2\sqrt{A}T})  ](x_T-x_0e^{-\sqrt{A}T})$. Figure \ref{fig:hongler} encapsulates the empirical numerical outcomes tested under time constraints $T=5, 10,$ and $20$. These computational results conclusively demonstrate that enhancing the algorithmic precision of the initial conditions yields a correspondingly more accurate rendering of the final MPTP. Furthermore, the extreme inherent sensitivity of the resulting MPTP to the specific initial momentum configuration becomes glaringly apparent. Thus, securing a highly exact initial momentum parameter proves to be an indispensable prerequisite within our computational framework.
\begin{figure}
    \centering {\includegraphics[width=0.4\textwidth]{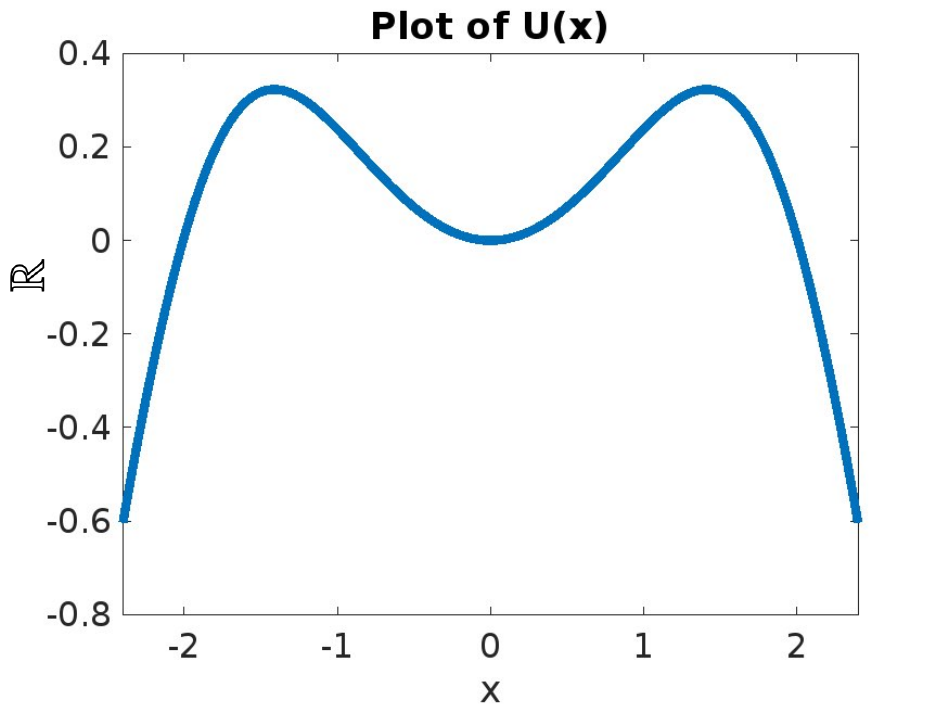} }
    \caption{Hongler's potential evaluated with  parameters $A= 1,
B=-0.6$}
    \label{Honglerpotential}%
\end{figure}

\begin{figure}
    \centering {{\includegraphics[width=0.4\textwidth]{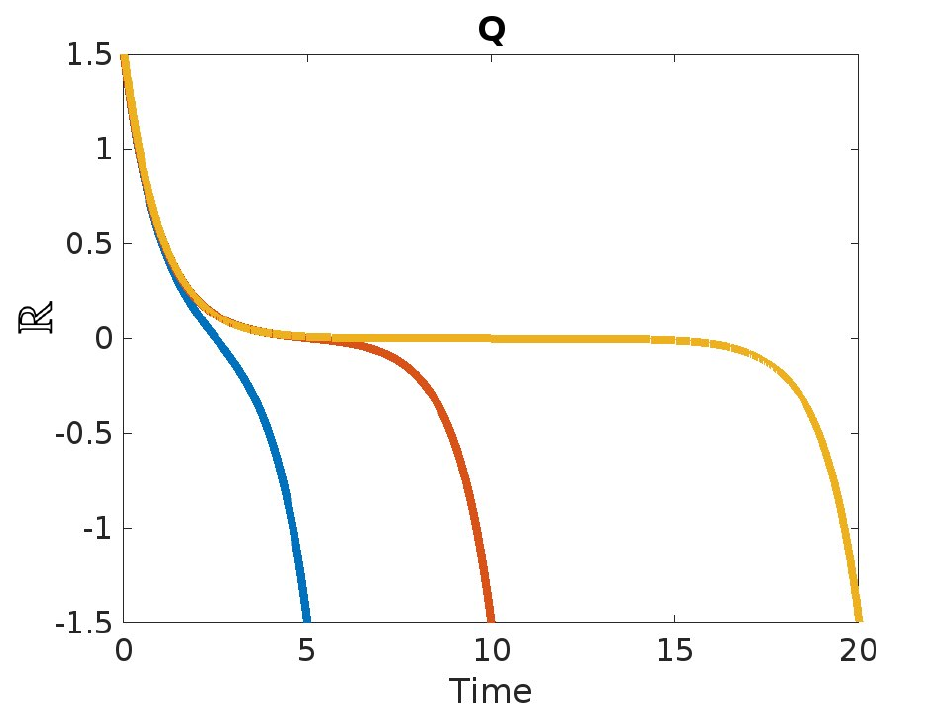} }}
    \qquad {{\includegraphics[width=0.4\textwidth]{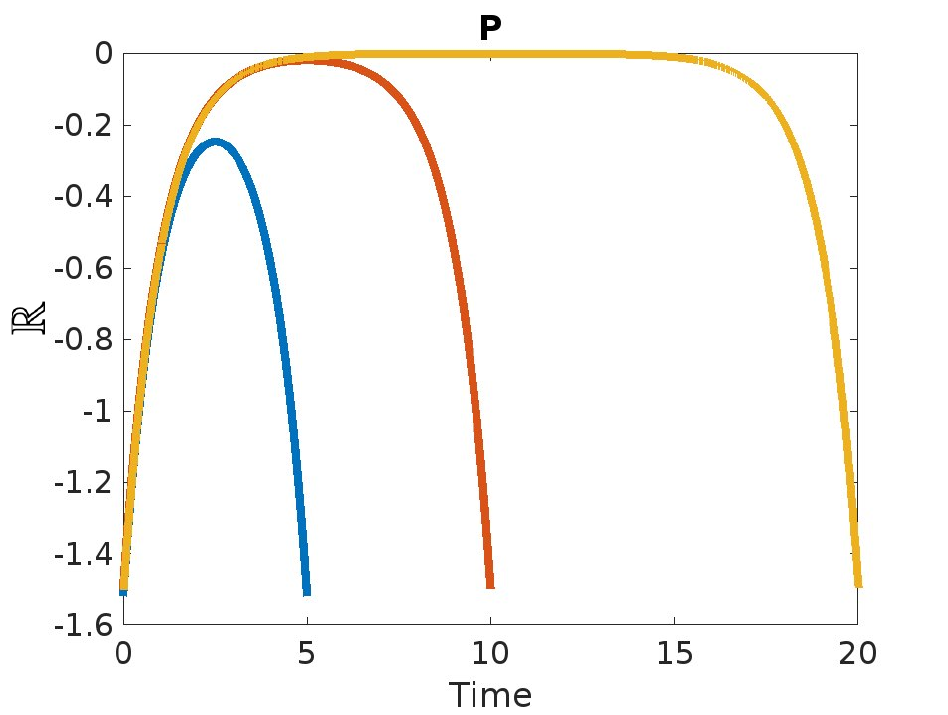} }}
    \caption{Comparative analysis of the most probable transition paths simulated for Hongler's model across varying time horizons.}
    \label{fig:hongler}
\end{figure}

\end{example}

\subsection{The Most Probable Transition Path as a Characteristics Line of PDEs}

The method of characteristics serves as a fundamental and pervasive technique within the realm of partial differential equations (PDEs) for constructing analytical solutions. Although classically deployed for first-order equations, its mathematical validity inherently encompasses broader hyperbolic PDE frameworks. Fundamentally, this approach operates by reducing a PDE into an equivalent system of ordinary differential equations (ODEs) evaluated along designated trajectories, universally recognized as characteristic curves. Subsequent integration along these paths allows one to recover the target solution, provided appropriate initial data is specified on a compatible hypersurface.

We begin by introducing an alternative representation of the Hamiltonian system governing the MPTP.
\begin{theorem}\label{MPTPinnewHamilton}
Subject to the hypotheses established in Theorem \ref{maintheory}, the spatial curves $\psi^*$ and $\widehat{\psi}^*$ qualify as the most probable transition paths for the system \eqref{firstequation} if and only if they strictly obey the ensuing systems of ordinary differential equations:
\begin{equation}\label{MPTPnew}
\left\{
  \begin{aligned}
    \frac{d\psi^*(t)}{dt}=&\ - \nabla U(\psi^*(t)) +  f_p(t),\\
    \frac{d f_p(t)}{dt}=&\ -\nabla \left( \frac{1}{2}\sigma^2 \Delta U(\psi^*(t)) \right) + \nabla^2 U(\psi^*(t))\cdot f_p(t),\\
    \psi^*(0)=&\ x_0,\quad f_p(0)=\ \sigma^2\nabla\log  p(x_T,T|x,0)|_{x=x_0},
\end{aligned}\right.
\end{equation}
and
\begin{equation*} 
\left\{
  \begin{aligned}
    \frac{d_*\psi^*(t)}{d_*t}=&\  -\nabla U(\psi^*(t))  -  \widehat{f}_p(t),\\
    \frac{d_* \widehat{f}_p(t)}{d_*t}=&\  \nabla \left( \frac{1}{2}\sigma^2 \Delta U(\psi^*(t)) \right) + \nabla^2 U(\psi^*(t))\cdot \widehat{f}_p(t),\\
    \psi^*(T)=&\ x_T,\quad \widehat{f}_p(T)=\ \sigma^2\nabla\log  p(x,T|x_0,0)|_{x=x_T},
\end{aligned}\right.
\end{equation*}
respectively, wherein $p(\cdot,\cdot|\cdot,\cdot)$ designates the transition density corresponding to the solution process of \eqref{firstequation}, and the operator $\nabla^2$ represents the Hessian matrix. 
\end{theorem}

We now proceed to establish that these most probable transition paths functionally serve as the characteristic curves for a specific class of second-order PDEs.
\begin{theorem}\label{relationHamiltonNS}
Presume that the conditions of Theorem \ref{maintheory} hold, and let
$F_p,\widehat{F}_p\in C^{1,2}(\R\times\R^k)$.
Then the following statements are equivalent.

\medskip
\noindent
{\rm (i)}
The field $F_p$ satisfies
\begin{equation*}
\begin{split}
\frac{\partial F_p}{\partial t}
-\nabla^2 U\cdot F_p
-(\nabla U\cdot\nabla)F_p
+(F_p\cdot\nabla)F_p
+\frac{1}{2}\sigma^2\triangle F_p
=0,
\ t\in[0,T).
\end{split}
\end{equation*}
Meanwhile, $\widehat{F}_p$ satisfies
\begin{equation}\label{NS}
\begin{split}
\frac{\partial \widehat{F}_p}{\partial t}
-\nabla(\triangle U)
-\nabla^2 U\cdot\widehat{F}_p
-(\nabla U\cdot\nabla)\widehat{F}_p
-(\widehat{F}_p\cdot\nabla)\widehat{F}_p
-\frac{1}{2}\sigma^2\triangle\widehat{F}_p
=0,
\ t\in(0,T].
\end{split}
\end{equation}

\medskip
\noindent
{\rm (ii)}
Let $\psi^*$ be any MPTP of the SDE \eqref{firstequation}, and define
the associated momentum curves by $f_p(t)=F_p(\psi^*(t),t),
\quad
\widehat{f}_p(t)=\widehat{F}_p(\psi^*(t),t)$. If $\psi^*$ satisfies the corresponding spatial equations, then
$f_p$ and $\widehat{f}_p$ satisfy the associated momentum equations.
More precisely, the forward-time pair $(\psi^*,f_p)$ satisfies
\begin{subequations}\label{Hamilton}
\begin{align}
\frac{d\psi^*(t)}{dt}
={}&-\nabla U(\psi^*(t))+f_p(t),
\label{Hamilton-spatial}\\
\frac{d f_p(t)}{dt}
={}&-\nabla\left(
\frac{1}{2}\sigma^2\Delta U(\psi^*(t))
\right)
+\nabla^2 U(\psi^*(t))\cdot f_p(t),
\label{Hamilton-momentum}
\end{align}
\end{subequations}
whereas the reverse-time pair
$(\psi^*,\widehat{f}_p)$ satisfies the adjoint system
\begin{subequations}\label{backwardHamilton}
\begin{align}
\frac{d_*\psi^*(t)}{d_*t}
={}&-\nabla U(\psi^*(t))-\widehat{f}_p(t),
\label{reverseHamilton-spatial}\\
\frac{d_*\widehat{f}_p(t)}{d_*t}
={}&\nabla\left(
\frac{1}{2}\sigma^2\Delta U(\psi^*(t))
\right)
+\nabla^2 U(\psi^*(t))\cdot\widehat{f}_p(t).
\label{reverseHamilton-momentum}
\end{align}
\end{subequations}
\end{theorem}
\begin{proof}
Recalling the transition density function $p(\cdot,\cdot|\cdot,\cdot)$ tied to the solution process $X$ of \eqref{firstequation}, the associated Kolmogorov backward equation is formulated as $\frac{\partial p(x_T,T|x,t)}{\partial t}= \nabla U(x)\cdot\nabla p(x_T,T|x,t)-\frac{1}{2}\sigma^2\triangle p(x_T,T|x,t)$, 
subject to the singular terminal condition $p(x_T,T|x,T)=\delta_{x_T}(x)$. By introducing the logarithmic transformation $S_p(x,t)=\log p(x_T,T|x,t)$, one deduces that $S_p$ governs the following Hamilton-Jacobi-Bellman (HJB) equation $\frac{\partial S_p}{\partial t}-\nabla U\cdot \nabla S_p +  \frac{1}{2}\sigma^2| \nabla S_p|^2 +\frac{1}{2}\sigma^2\triangle S_p=0,\quad t\in[0,T)$,
which can be explicitly expanded as $\frac{\partial S_p}{\partial t}-\sum_{i=1}^k\frac{\partial U}{\partial x_i}\frac{\partial S_p}{\partial x_i} +  \frac{1}{2}\sigma^2 \sum_{i=1}^k\left(\frac{\partial S_p}{\partial x_i}\right)^2 +\frac{1}{2}\sigma^2\sum_{i=1}^k\frac{\partial^2 S_p}{\partial x^2_i}=0,\quad t\in[0,T)$.
Taking the partial derivative of both sides with respect to the spatial variables yields $\frac{\partial }{\partial t}\frac{\partial S_p}{\partial x_j}-\sum_{i=1}^k\frac{\partial^2 U}{\partial x_j\partial x_i}\frac{\partial S_p}{\partial x_i} -\sum_{i=1}^k\frac{\partial U}{\partial x_i}\frac{\partial^2 S_p}{\partial x_j\partial x_i} +  \sigma^2\sum_{i=1}^k \frac{\partial^2 S_p}{\partial x_j\partial x_i}\frac{\partial S_p}{\partial x_i} +\frac{1}{2}\sigma^2\sum_{i=1}^k\frac{\partial^2}{\partial x^2_i}\frac{\partial S_p}{\partial x_j}=0,\quad t\in[0,T)$. By defining the vector field $F_p(x,t)=\sigma^2\nabla\log p(x_T,T|x,t)$, it naturally follows that $F_p$ is a solution to the ensuing PDE $\frac{\partial F_p}{\partial t}-\nabla^2 U\cdot F_p - (\nabla U\cdot\nabla)F_p +  (F_p\cdot\nabla ) F_p +\frac{1}{2}\sigma^2\triangle F_p=0,\quad t\in[0,T)$. Turning our attention to the Kolmogorov forward equation (commonly referred to as the Fokker-Planck equation) $\frac{\partial p(x,t|x_0,0)}{\partial t}= \nabla(\nabla U(x)p(x,t|x_0,0))+\frac{1}{2}\sigma^2\triangle p(x,t|x_0,0)$. 
Applying a similar logarithmic mapping $\widehat{S}_p(x,t)=\log p(x,t|x_0,0)$ reveals that it complies with the adjoint Hamilton-Jacobi-Bellman equation $\frac{\partial \widehat{S}_p}{\partial t} 
-\triangle U  -\nabla U\cdot \nabla \widehat{S}_p -  \frac{1}{2}\sigma^2| \nabla \widehat{S}_p|^2 -\frac{1}{2}\sigma^2\triangle \widehat{S}_p=0,\quad t\in(0,T]$,
or equivalently written in expanded form $\frac{\partial \widehat{S}_p}{\partial t}-\sum_{i=1}^k\frac{\partial^2U}{\partial x_i^2} -\sum_{i=1}^k\frac{\partial U}{\partial x_i}\frac{\partial \widehat{S}_p}{\partial x_i} -  \frac{1}{2}\sigma^2 \sum_{i=1}^k\left(\frac{\partial \widehat{S}_p}{\partial x_i}\right)^2 -\frac{1}{2}\sigma^2\sum_{i=1}^k\frac{\partial^2 \widehat{S}_p}{\partial x^2_i}=0,\quad t\in(0,T]$. 
Applying spatial differentiation across this equation produces:
\begin{equation*} 
\begin{split}
&\frac{\partial }{\partial t}\frac{\partial \widehat{S}_p}{\partial x_j} - \sum_{i=1}^k\frac{\partial^3U}{\partial x_j\partial x_i^2} - \sum_{i=1}^k\frac{\partial^2 U}{\partial x_j\partial x_i}\frac{\partial S_p}{\partial x_i} -\sum_{i=1}^k\frac{\partial U}{\partial x_i}\frac{\partial^2  \widehat{S}_p}{\partial x_j\partial x_i} \\
&\quad \quad -  \sigma^2\sum_{i=1}^k \frac{\partial^2  \widehat{S}_p}{\partial x_j\partial x_i}\frac{\partial  \widehat{S}_p}{\partial x_i} -\frac{1}{2}\sigma^2\sum_{i=1}^k\frac{\partial^2}{\partial x^2_i}\frac{\partial  \widehat{S}_p}{\partial x_j}=0,\quad t\in(0,T].
\end{split}
\end{equation*}

Substituting the vector formulation $\widehat{F}_p(x,t)=\sigma^2\nabla\log p(x,t|x_0,0)$ demonstrates that $\widehat{F}_p$ acts as a valid solution to $\frac{\partial \widehat{F}_p}{\partial t}- \nabla(\triangle U) - \nabla^2 U\cdot \widehat{F}_p - (\nabla U\cdot\nabla)\widehat{F}_p -  (\widehat{F}_p\cdot\nabla ) \widehat{F}_p -\frac{1}{2}\sigma^2\triangle \widehat{F}_p=0,\quad t\in(0,T]$. 
Ultimately, the propositions asserted within this Theorem emerge as a direct logical consequence of synthesizing the preceding differential arguments with the structural properties established in Theorem \ref{MPTPinnewHamilton}.
\end{proof}

To elucidate the practical implications of this Theorem, let us analyze the free Brownian motion.

\begin{example}[Free Brownian motion]
Consider the one-dimensional diffusion with a constant potential
$U\equiv C$. Since only derivatives of $U$ enter the dynamics, we may
set $C=0$ without loss of generality. In the present
Onsager--Machlup setting, this choice reduces the unconditioned process
to free Brownian motion. This observation may be contrasted with Nelson's time-symmetric
stochastic mechanics \cite{Nelson}. In that framework, a vanishing
external potential does not necessarily reduce the associated
conservative diffusion to an ordinary Wiener process, because the
forward and backward drifts remain coupled through a density-dependent
osmotic term. The zero-potential case therefore remained a subtle
problem for many years; see, for example, \cite{Shucker1980}. The transition density of Brownian motion with noise intensity
$\sigma$ is $p(x,t\mid y,s)
=
\frac{1}{\sqrt{2\pi\sigma^2(t-s)}}
\exp\left(
-\frac{(x-y)^2}{2\sigma^2(t-s)}
\right)$, $s<t$. 
Consequently, the forward and reverse score fields in Theorem
\ref{relationHamiltonNS} are
\begin{equation}\label{Brownian-score-fields}
\begin{aligned}
F_p(x,t)
&=
\sigma^2\partial_x\log p(x_T,T\mid x,t)
=
\frac{x_T-x}{T-t},
&&t\in[0,T),\\
\widehat{F}_p(x,t)
&=
\sigma^2\partial_x\log p(x,t\mid x_0,0)
=
-\frac{x-x_0}{t},
&&t\in(0,T].
\end{aligned}
\end{equation}
Since $U'=U''=0$, the PDEs in Theorem
\ref{relationHamiltonNS} reduce to the forward and backward viscous
Burgers equations
\begin{equation*}
\begin{split}
\frac{\partial F_p}{\partial t}
+F_p\frac{\partial F_p}{\partial x}
+\frac{1}{2}\sigma^2\frac{\partial^2F_p}{\partial x^2}
&=0,
\qquad t\in[0,T),\\
\frac{\partial\widehat{F}_p}{\partial t}
-\widehat{F}_p\frac{\partial\widehat{F}_p}{\partial x}
-\frac{1}{2}\sigma^2
 \frac{\partial^2\widehat{F}_p}{\partial x^2}
&=0,
\qquad t\in(0,T].
\end{split}
\end{equation*}
The fields in \eqref{Brownian-score-fields} can be verified directly
to satisfy these equations.

Along an MPTP $\psi^*$, define $f_p(t)=F_p(\psi^*(t),t),
\ 
\widehat{f}_p(t)=\widehat{F}_p(\psi^*(t),t).
$
For $U=0$, systems \eqref{Hamilton} and
\eqref{backwardHamilton} become
\begin{equation*}
\left\{
\begin{aligned}
\frac{d\psi^*(t)}{dt}
    &=f_p(t),\\
\frac{d f_p(t)}{dt}
    &=0,
\end{aligned}
\right.
\qquad\text{and}\qquad
\left\{
\begin{aligned}
\frac{d_*\psi^*(t)}{d_*t}
    &=-\widehat{f}_p(t),\\
\frac{d_*\widehat{f}_p(t)}{d_*t}
    &=0.
\end{aligned}
\right.
\end{equation*}
Using $\psi^*(0)=x_0$ and $\psi^*(T)=x_T$, we obtain $f_p(t)
=
-\widehat{f}_p(t)
=
\frac{x_T-x_0}{T},
\ 
\psi^*(t)
=
x_0+\frac{x_T-x_0}{T}t$. 
Equivalently, $F_p(\psi^*(t),t)
=
\frac{x_T-\psi^*(t)}{T-t},
\ 
\widehat{F}_p(\psi^*(t),t)
=
-\frac{\psi^*(t)-x_0}{t}$. 
Thus, the forward and reverse-time formulations determine the same
linear MPTP.

Conditioning the process on $X_0=x_0$ and $X_T=x_T$ produces the
elementary Brownian bridge, which is the simplest Bernstein reciprocal
diffusion associated with the zero-potential heat kernel; see
Chapter~\ref{ch:4}. The fields $F_p$ and $\widehat{F}_p$ determine its
forward and reverse-time drifts, while $\psi^*$ is both its mean path
and its most probable transition path.

This explicit solvability extends to Gaussian Markov diffusions and
their bridges. In particular, Ornstein--Uhlenbeck processes can be
treated as explicitly computable Bernstein processes; see
\cite{cruzeiro1993ornstein}. Such Gaussian constructions also appear
in Euclidean quantum mechanics, where quadratic potentials generate
explicit heat kernels. Their real-time counterparts lead to familiar
Gaussian path-integral calculations, although the resulting
oscillatory expressions need not define genuine diffusion probability
measures.
\end{example}

\section{The Onsager--Machlup Functional in Multiplicative Noise Cases}\label{chapter2.sec2.2}

In the preceding section, our discussion was confined to gradient systems with additive noise, where the physical landscape provided a straightforward intuition. However, to accommodate more complex non-equilibrium systems, it is imperative to extend this framework to general drift fields and state-dependent fluctuations. A rigorous foundation for this generalization can be traced back to the seminal work of Dembo and Zeitouni \cite{Zeitouni1988}, which addressed the problem of finding the maximum a posteriori (MAP) estimator for the trajectory of a diffusion process with multiplicative noise. In this more general setting, the state-dependent nature of the noise fundamentally alters the geometry of the system: unlike the flat Euclidean metric in the additive noise case, the multiplicative noise induces a non-trivial Riemannian structure on the state space, where the local ``distance'' is intrinsically weighted by the inverse of the diffusion matrix.

\subsection{Model and Settings}

Consider a diffusion process $X_t$ governed by the following SDE in the It\^{o} sense:
\begin{equation}
dX_t = f(X_t) dt + \sigma(X_t) dB_t, \quad X_0 = x_0\in\mathbb{R}^n,
\end{equation}
where $B_t$ is an $m$-dimensional standard Brownian motion ($m \ge n$). We assume the following fundamental conditions hold (see \cite{Zeitouni1988} for more detailed assumptions):
\begin{enumerate}
    \item \textbf{Ellipticity:} The diffusion matrix $a(x) = \sigma(x)\sigma^T(x)$ is uniformly positive definite for all $x \in \mathbb{R}^n$.
    \item \textbf{Smoothness:} The coefficients $f^i(x)$ and $\sigma^{ij}(x)$ are of class $C_b^3(\mathbb{R}^n)$.
\end{enumerate}

The multiplicative noise defines a Riemannian metric $g$ on $\mathbb{R}^n$, whose components are given by the inverse of the diffusion matrix $g_{ij}(x) = [a^{-1}(x)]_{ij}$, 
which transforms the state space into a Riemannian manifold $(\mathbb{R}^n, g)$. All subsequent differential operators (divergence, gradient) are defined with respect to this metric.

\subsection{The Lagrangian and Geometric Corrections}

The Onsager-Machlup functional is derived by considering the limit of the probability that the process $X_t$ remains within an $\epsilon$-tube around a deterministic path $\phi \in C^2([0,T]; \mathbb{R}^n)$. Zeitouni and Dembo proved that $P\left( \sup_{0 \le t \le T} \|X_t - \phi_t\| < \epsilon \right) \sim C(\epsilon) \exp \left( -\int_0^T L(\phi_t, \dot{\phi}_t) dt \right)$. 
The Lagrangian $L(x, \dot{x})$ for the multiplicative noise case is given by $L(x, \dot{x}) = \frac{1}{2} g_{ij}(x) (\dot{x}^i - \tilde{f}^i(x))(\dot{x}^j - \tilde{f}^j(x)) + \frac{1}{2} \text{div}_g \tilde{f}(x) - \frac{1}{12} R(x)$ 
where:
\begin{itemize}
    \item $\tilde{f}^i(x) = f^i(x) + \frac{1}{2} \sum_{l,j}((\sigma(x)\sigma^\top)^{-1})^{lj}(x)\Gamma^i_{lj}$ is the drift adjusted for the manifold structure. In terms of the Stratonovich drift $\tilde{f}(x)$, it relates to the Ito drift $f(x)$ through the connection coefficients.
    \item $\text{div}_g \tilde{f} = \frac{1}{\sqrt{|g|}} \partial_i (\sqrt{|g|} \tilde{f}^i)$ is the Beltrami-Laplace divergence of the vector field $\tilde{f}$, where $|g| = \det(g_{ij})$.
    \item $R(x)$ is the \textbf{scalar curvature} of the Riemannian manifold $(\mathbb{R}^n, g)$.
\end{itemize}

\subsection{Impact of Multiplicative Noise}

The scalar-curvature correction in the Lagrangian originates from the
short-time asymptotic expansion of the heat kernel on a Riemannian
manifold. Curvature terms of this type appeared historically in
DeWitt's path-integral formulation of quantum dynamics in curved
spaces \cite{DeWitt1957} and were subsequently incorporated into
path-integral representations of general diffusion processes by
Graham \cite{Graham1977}. The precise form and numerical coefficient of the curvature correction
have been the subject of some discussion, since they depend on the
normalization of the generator, the reference volume measure, and the
discretization or operator-ordering convention used in the path
integral. Throughout this section, we adopt the convention leading to
the term $-\frac{1}{12}R(x)$ 
in the Onsager--Machlup Lagrangian, consistently with the formulation
used above. Under this convention, the geometry induced by the
state-dependent diffusion contributes directly to the local action
density. For the metric $g_{ij}$, the scalar curvature is $R=g^{ij}R_{ij}
=
g^{ij}\left(
\partial_k\Gamma^k_{ij}
-\partial_j\Gamma^k_{ik}
+\Gamma^k_{kl}\Gamma^l_{ij}
-\Gamma^k_{jl}\Gamma^l_{ik}
\right)$, 
where $\Gamma^i_{jk}
=
\frac{1}{2}g^{il}
\left(
\partial_jg_{lk}
+\partial_kg_{lj}
-\partial_lg_{jk}
\right)$ 
are the Christoffel symbols. In one dimension the scalar curvature
vanishes identically, so this correction does not appear in the
one-dimensional formula derived below.


Note that, the state space is equipped with a Riemannian metric defined by the inverse of the diffusion coefficient squared $V_\sigma(x) = (\sigma(x)\sigma^*(x))^{-1} = \frac{1}{\sigma^2(x)}$.
In one dimension, the corresponding Christoffel symbol is $\Gamma = \frac{1}{2} V^{-1}_\sigma(x) \frac{d V_\sigma(x)}{dx} = \frac{1}{2} \sigma^2(x) \left( \frac{-2\sigma'(x)}{\sigma^3(x)} \right) = -\frac{\sigma'(x)}{\sigma(x)}$, 
where $\sigma'(x) = \frac{d\sigma(x)}{dx}$. The scalar curvature $R(x)$ in 1D is trivially zero.

We transform  the It\^o drift into a modified drift $f(x)$ mapping to the Stratonovich representation $\tilde{f}(x) = f(x) + \frac{1}{2} V^{-1}_\sigma(x) \Gamma = f(x) - \frac{1}{2} \sigma(x)\sigma'(x)$.
The divergence operator acting on $f(x)$ under this metric is given by:
\begin{equation}
    \text{div} \tilde{f}(x) = \frac{1}{\sqrt{V_\sigma(x)}} \frac{d}{dx} \left( \tilde{f}(x) \sqrt{V_\sigma(x)} \right) = \sigma(x) \frac{d}{dx} \left( \frac{\tilde{f}(x)}{\sigma(x)} \right) = \tilde{f}'(x) - \tilde{f}(x)\frac{\sigma'(x)}{\sigma(x)}.
\end{equation}

The associated Lagrangian is $L(x,\dot{x})
    =\frac12\left[(\dot{x}-\tilde f(x))^2V_\sigma(x)+\mathrm{div}\,\tilde f(x)\right]$. 
For the one-dimensional case, $L(x,\dot{x})
    =\frac{(\dot{x}-\tilde f(x))^2}{2\sigma^2(x)}
    +\frac12\tilde f'(x)
    -\frac{\tilde f(x)\sigma'(x)}{2\sigma(x)}$.
Using $\frac{\partial L}{\partial\dot{x}}
=\frac{\dot{x}-\tilde f(x)}{\sigma^2(x)}$, and simplifying the resulting derivatives, we obtain
\begin{equation}
    \frac{\ddot{x}}{\sigma^2}
    -\frac{\sigma'}{\sigma^3}\dot{x}^2
    -\frac{\tilde f\tilde f'}{\sigma^2}
    +\frac{\tilde f^2\sigma'}{\sigma^3}
    -\frac12\tilde f''
    +\frac12\frac{\tilde f'\sigma'}{\sigma}
    +\frac12\frac{\tilde f\sigma''}{\sigma}
    -\frac12\frac{\tilde f(\sigma')^2}{\sigma^2}
    =0.
\end{equation}
Multiplying by $\sigma^2(x)$ gives the Euler--Lagrange equation for the one-dimensional multiplicative noise system:
\begin{equation}
    \begin{aligned}
    &\ddot{x}
    -\frac{\sigma'(x)}{\sigma(x)}\dot{x}^2
    -\tilde f(x)\tilde f'(x)
    +\tilde f^2(x)\frac{\sigma'(x)}{\sigma(x)}
    -\frac12\tilde f''(x)\sigma^2(x) \\
    &\quad
    +\frac12\tilde f'(x)\sigma'(x)\sigma(x)
    +\frac12\tilde f(x)\sigma''(x)\sigma(x)
    -\frac12\tilde f(x)(\sigma'(x))^2
    =0,
    \end{aligned}
\end{equation}
where $\tilde f(x)=f(x)-\frac12\sigma(x)\sigma'(x)$. 

In particular, when $f=0$,
\begin{equation}
    \ddot{x}
    -\frac{\sigma'(x)}{\sigma(x)}\dot{x}^2
    + \frac14\sigma^3(x)\sigma'''(x)
    =0.
\end{equation}

\section{The Onsager-Machlup Functional for Jump-Diffusions}\label{chapter2.sec2.3}

Many biological and physical systems are driven by non-Gaussian noise, yet the corresponding OM theory remains underdeveloped, particularly for jump diffusions. Motivated by D\"urr and Bach \cite{Durr1978}, we extend their classical diffusion results to dynamical systems driven jointly by Brownian and L\'evy noise. The pure-jump component introduces new analytical and probabilistic challenges. We retain the notation of \cite{Durr1978} to facilitate comparison with the classical diffusion case.

\subsection{Preliminaries on Path Measures Induced by Jump-Diffusion Processes}
Consider the following scalar stochastic differential equation (SDE) defined on the interval $\tau=[s,u]$ over the filtered probability space $(\Omega,\mathcal{F},(\mathcal{F}_t)_{t\geq0},\mathbb{P})$:
\begin{equation} \label{Equation-1}
  \begin{split}
  &dX_t=f(X_{t-})dt+g(X_{t-})dB_t+dL_t,\quad X_s=x_0\in \mathbb{R}.
  \end{split}
\end{equation}

\noindent Here, ${L_t}$ denotes a L\'{e}vy process with characteristics $(0,0,\nu)$ satisfying
$
\int_{|\xi|<1}\xi\nu(d\xi)<\infty,
$
$B_t$ represents a standard Brownian motion, and $X_{t-}$ stands for the left-hand limit of $X_t$, namely $
X_{t-}=\lim_{s\uparrow t}X_s.
$ Using the L\'{e}vy-It\^o decomposition, equation (\ref{Equation-1}) may equivalently be expressed as (see Chapter \ref{chapter1intro}, Section \ref{sec1.3})
\begin{equation} \label{Equation-2}
  \begin{split}
  &dX_t=f(X_{t-})dt+g(X_{t-})dB_t+\int_{|y|<1}y\tilde{N}(dt,dy)+\int_{|y|\geq1}yN(dt,dy).\notag
  \end{split}
\end{equation}
The final term  corresponds to the contribution of large jumps weighted by $y$. Such a term can be treated through the interlacing technique \cite[Page 365]{Applebaum}. Therefore, it is natural to first neglect this component and focus on stochastic systems driven by continuous fluctuations together with small jumps only. Accordingly, we study the following SDE:
\begin{equation} \label{Equation-main}
  \begin{split}
  &dX_t=f(X_{t-})dt+g(X_{t-})dB_t+\int_{|y|<1}y\tilde{N}(dt,dy),\quad X_s=x_0\in \mathbb{R}.
  \end{split}
\end{equation}

\noindent Assume additionally that the coefficients $f(x)$ and $g(x)$ are locally Lipschitz continuous and satisfy the following one-sided linear growth condition.

\noindent{\bf C1 }(Locally Lipschitz condition) For every $R>0$, there exists a constant $K_1>0$ such that for all $|y_1|,|y_2|\leq R$,
$
|f(y_1)-f(y_2)|^2+|g(y_1)-g(y_2)|^2\leq K_1|y_1-y_2|^2.
$

\noindent{\bf C2 }(One-sided linear growth condition) There exists a positive constant $K_2$ satisfying, for all $y\in\mathbb{R}$,
$
|g(y)|^2+2y\cdot f(y)\leq K_2(1+|y|^2).
$

\noindent Under assumptions {\bf C1} and {\bf C2}, equation (\ref{Equation-main}) admits a unique global solution whose sample paths are adapted and c\`adl\`ag; see Theorem 3.1 in \cite{albeverio2010existence}.
Let $C^1(\mathbb{R})$ denote the collection of all continuously differentiable real-valued functions, and let $C^2(\mathbb{R})$ denote the set of all twice continuously differentiable real-valued functions. Furthermore, let $H^1([s,u];\mathbb{R})$ be the Sobolev space consisting of square-integrable real-valued functions on $\tau=[s,u]$ whose weak derivatives are also square integrable.

Denote by $D_\tau^{x_0}$ the path space, namely the collection of c\`adl\`ag functions
$
D_\tau^{x_0}=\{x(t)\mid x:\tau\rightarrow \mathbb{R},\; x(t)\ \hbox{is right-continuous with left limits},\; x(s)=x_0\}.
$ For any finite subset $S\subset[s,u]$, let $\pi_S$ denote the projection map from $D_\tau^{x_0}$ to $\mathbb{R}^S$, assigning to each path $x$ the vector $\{x(t):t\in S\}$. These projection mappings generate the projection $\sigma$-field $\mathcal{G}$. Unless otherwise stated, all measurability considered in this section is understood with respect to $\mathcal{G}$. Since every c\`adl\`ag function on $\tau$ is bounded, the space $D_\tau^{x_0}$ becomes a Banach space when equipped with the uniform norm
$
\|x\|=\sup_{t\in[s,u]}|x(t)|,\ x(t)\in D_\tau^{x_0}.
$
No alternative norm on $D_\tau^{x_0}$ will be considered in this section. One may verify that the $\sigma$-field $\mathcal{B}_0$ generated by the closed balls induced from the uniform norm coincides with the projection $\sigma$-field $\mathcal{G}$, although it does not agree with the larger Borel $\sigma$-field $\mathcal{B}_\tau^{x_0}$. More details may be found in \cite[Page 87--90]{pollard2012convergence}. \par

\begin{lemma}[\cite{Chao2019}]\label{lemma 2.1}
$\mathcal{B}_0=\mathcal{G}$ and $\mathcal{G}\subsetneqq\mathcal{B}_\tau^{x_0}$.
\end{lemma}

The primary objective of this section is to investigate the most probable tube associated with the process $X_t$. For this reason, we consider the probability that the trajectories of the process remain inside the closed tube $K(z,\epsilon)=\{x\in D_\tau^{x_0}\mid z\in D_\tau^{x_0},\|x-z\|\leq\epsilon,\epsilon>0\}$. Such probabilities can be evaluated or approximated through the induced measure on the underlying function space. Define the measure $\mu_X$ on $\mathcal{G}$ induced by the process (\ref{Equation-main}) through
$
\mu_X(B)=\mathbb{P}(\{\omega\in\Omega\mid X_t(\omega)\in B\}),\ B\in\mathcal{G}.
$ Consequently, for any prescribed $\epsilon>0$, the probabilities of tubes having identical radius can be compared for different reference paths $z\in D_\tau^{x_0}$ via $\mu_X(K(z,\epsilon))
=
\mathbb{P}(\{\omega\in\Omega\mid X_t(\omega)\in K(z,\epsilon)\})$, 
where $K(z,\epsilon)\in\mathcal{G}$. 



Restricting attention to differentiable reference functions leads naturally to the following definition \cite{OM53,tisza1957fluctuations}. For $\epsilon>0$, as in the diffusion setting, consider a tube centered at a reference trajectory $z(t)$. If, as $\epsilon\to0$, its probability satisfies $\mathbb{P}(\{\|X-z\|\leq\epsilon\})
\propto
C(\epsilon)\exp\left\{-\frac{1}{2}\int_s^u OM(\dot{z},z)dt\right\}$. If $z_m(t)$ maximizes $\mu_X(K(z,\epsilon))$, then, for sufficiently small $\epsilon$, it can be determined variationally by minimizing the OM action functional $\int_s^uOM(\dot{z},z)dt$. This approach was introduced for diffusion processes in \cite{Durr1978} and extends naturally to the present jump-diffusion setting. Recall that $\mu_X\ll\mu_Y$ means that $\mu_X$ is absolutely continuous with respect to $\mu_Y$, namely, $\mu_Y(A)=0$ implies $\mu_X(A)=0$ for every $A\in\mathcal{G}$. The measures are equivalent, denoted by $\mu_X\sim\mu_Y$, if both $\mu_X\ll\mu_Y$ and $\mu_Y\ll\mu_X$ hold; see \cite[Page 161]{Oksendal}.

\begin{lemma}[Transform between measures \cite{Chao2019}]\label{lemma 3.1}
Let $X_t$ and $Y_t$ be two jump-diffusion processes governed by the following stochastic differential equations on the probability space $(\Omega,\mathcal{F},(\mathcal{F}_t)_{t\geq0},\mathbb{P})$:
\begin{eqnarray}
dX_t=&& f(X_{t-})dt+g(X_{t-})dB_t+\int_{|y|<1}y\tilde{N}(dt,dy),\label{Equation-x}\\
dY_t=&& k(Y_{t-})dt+g(Y_{t-})dB_t+\int_{|y|<1}y\tilde{N}(dt,dy),\label{Equation-y}
\end{eqnarray}
where the driving L\'{e}vy process possesses characteristic triplet $(0,0,\nu)$ and satisfies
$
\int_{|\xi|<1}\xi\nu(d\xi)<\infty.
$
Assume that $X_s=Y_s=x_0\in\mathbb{R}$ for $t\in\tau$, and that the coefficients $f,k,g\in C^1(\mathbb{R})$ satisfy condition ${\bf C2}$. Then the induced measures satisfy $\mu_X\sim\mu_Y$, and the corresponding Radon--Nikodym derivative of $\mu_X$ with respect to $\mu_Y$ is given by
\begin{equation}\label{RND1}
\frac{d\mu_X}{d\mu_Y}[Y_t(\omega)]
=
\exp\left\{
\int_s^u a(Y_{t-})dB_t
-\frac{1}{2}\int_s^u(a(Y_{t-}))^2dt
\right\},
\end{equation}
where $a(x)=\frac{f(x)-k(x)}{g(x)}$.
\end{lemma}

We next rewrite the Radon--Nikodym derivative in (\ref{RND1}) in the form of a path integral. The essential step consists in converting the stochastic integral appearing in (\ref{RND1}) through the It\^o formula \cite[Page 251--255]{Applebaum}. To simplify the presentation, we restrict ourselves to the case where $g(x)=c$. Introduce the potential function $V(x)=\frac{1}{c}\int^x a(y)dy$. Applying the It\^o formula to the stochastic integral yields the following representation $F[y(t)]$ for (\ref{RND1}), which reveals the functional structure of the Radon--Nikodym derivative on the path space $D_\tau^{x_0}$ with $y(t)\in D_\tau^{x_0}$:
\begin{align}\label{F}
&F[y(t)]
=\frac{d\mu_X}{d\mu_Y}[y(t)] 
= \exp\Bigg\{
V(y(u))-V(x_0)
-\frac{1}{2}\int_s^u f(y(t-))dt \notag\\
&+\int_s^u\int_{|\xi|<1}\frac{\xi}{c}a(y(t-))\nu(d\xi)dt
-\Sigma_{s\leq t\leq u}[V(y(t))-V(y(t-))]\chi_{|\xi|<1}(\Delta y(t))
\Bigg\},
\end{align}
where $f(y(t-))
=
\{a(y(t-))\}^2
+\frac{2}{c}a(y(t-))k(y(t-))
+c\frac{da(x)}{dx}(y(t-))$. 

Recall that a measure $L$ on $\mathbb{R}^n$ is called translation invariant if, for every $E\in\mathcal{B}^n$, one has $\mathcal{T}E\in\mathcal{B}^n$ together with
$
L(\mathcal{T}E)=L(E),
$
where $\mathcal{T}$ denotes a translation on $\mathbb{R}^n$ and $\mathcal{B}^n$ is the Borel $\sigma$-field on $\mathbb{R}^n$. As pointed out in \cite{Durr1978}, no genuinely translation invariant measure exists on the function space $D_\tau^{x_0}$. Nevertheless, for the present framework, a weaker notion known as quasi-translation invariance is sufficient. This concept is formulated as follows \cite{Durr1978}.\par

\begin{definition}\label{definition 3.1}
Let $\mathcal{T}:D_\tau^{x_0}\rightarrow D_\tau^{x_0}$ be a transformation of the form
\begin{eqnarray}\label{t1}
\mathcal{T}x\rightarrow x+z_0,
\end{eqnarray}
where $z_0\in D_\tau^0$ is differentiable. Assume that $\mathcal{T}$ is measurable on
$
\mathcal{M}=\{A\in\mathcal{G}:\mathcal{T}^{-1}(A)\in\mathcal{G}\}.
$ Consider the jump-diffusion $X_t$ together with the transformed process
$\mathcal{T}X_t=X_t+z_0(t)$. If the induced measures satisfy $\mu_X\sim\mu_{\mathcal{T}X}$ on $\mathcal{M}$, then $\mu_X$ and $\mu_{\mathcal{T}X}$ are said to be quasi-translation invariant.
\end{definition}

\begin{remark}\label{remark 3.2}
The translation operator $\mathcal{T}$ acting on $D_\tau^{x_0}$ is not necessarily measurable because the projection $\sigma$-field $\mathcal{G}$ does not coincide with the larger Borel $\sigma$-field. To ensure measurability, the translation is therefore restricted to the subset $\mathcal{M}\subset\mathcal{G}$. Under this restriction, the notion of quasi-translation invariance becomes well-defined. Moreover, the set $\mathcal{M}$ is nonempty. Indeed, if the translation $\mathcal{T}$ defined in (\ref{t1}) is generated by $z_0(t)$, then
$
\mathcal{T}^{-1}(K(z,\epsilon))=K(x_0,\epsilon),
$
which clearly belongs to $\mathcal{G}$.
\end{remark}

\begin{lemma}[Girsanov theorem between quasi-translation invariant measures \cite{Chao2019}]\label{lemma 3.2}
The stochastic differential equation
\begin{equation}\label{Equation-xm1}
dX_t=f(X_{t-})dt+c\;dB_t+\int_{|y|<1}y\tilde{N}(dt,dy)
\end{equation}
induces a quasi-translation invariant measure $\mu_X$, and
\begin{equation}\label{RND2}
\frac{d\mu_{\mathcal{T}X}}{d\mu_X}[X_t(\omega)]
=
\exp\left\{
\int_s^u a_X(X_{t-},z_0(t))dB_t
-\frac{1}{2}\int_s^u(a_X(X_{t-},z_0(t)))^2dt
\right\},
\end{equation}
where $a_X(x,z_0)=\frac{f(x-z_0)+\dot{z_0}-f(x)}{c}$. 
\end{lemma}

\begin{remark}\label{remark 3.3}
The subscript $X$ indicates that the corresponding quantity arises from the Radon--Nikodym derivative of the measure $\mu_{\mathcal{T}X}$ relative to $\mu_X$. Replacing $z_0$ with $-z_0$ in these expressions yields instead the Radon--Nikodym derivative of $\mu_{\mathcal{T}^{-1}X}$ with respect to $\mu_X$. It should also be emphasized that Lemma \ref{lemma 3.2} is no longer valid when $g(x)$ in (\ref{Equation-main}) is nonconstant. Equivalently, if $g(x)\neq c$, then the induced measure associated with (\ref{Equation-main}) fails to possess the quasi-translation invariance property. Consequently, in Section \ref{subsec:OM-jump} we restrict attention to the case $g(x)=c$ when deriving the OM function.
\end{remark}

The following result highlights the role played by the quasi-translation invariant measures introduced above.

\begin{lemma}[\cite{Chao2019}]\label{lemma 3.3}
Let $\mathcal{T}$ be the translation defined in (\ref{t1}), and let
$
\mathcal{M}=\{A\in\mathcal{G}:\mathcal{T}^{-1}A\in\mathcal{G}\}.
$
Then $\mathcal{T}$ is measurable on $\mathcal{M}$. Moreover, if $\Phi(x)$ is a measurable functional on $D_\tau^{x_0}$ and $\mu_X$ is quasi-translation invariant, then the identity
\begin{eqnarray}\label{y}
\int_{A}\Phi(y)d\mu_X(y)
=
\int_{\mathcal{T}^{-1}A}\Phi(x+z_0)J_X[x,-z_0]d\mu_X(x)
\end{eqnarray}
holds for every $A\in\mathcal{M}$.
\end{lemma}

\subsection{The Onsager--Machlup Functional for Jump Diffusion Processes}\label{subsec:OM-jump}

We proceed to derive the Onsager--Machlup function associated with SDE
\begin{equation}\label{Equation-xmm}
dX_t=f(X_{t-})dt+c\;dB_t+\int_{|y|<1}y\tilde{N}(dt,dy),\qquad X_s=x_0\in\mathbb{R},
\end{equation}
which corresponds to the special case of (\ref{Equation-main}) with constant diffusion coefficient $g(x)=c$. Throughout this section, we assume that $f\in C^1(\mathbb{R})$ and that condition {\bf C2} is satisfied. The principal result concerning the explicit form of the OM function for jump-diffusion processes is stated below.\par

\begin{theorem}[Approximation of the OM functional for jump-diffusion processes \cite{Chao2019}]
\label{theorem 4.1}
Consider the stochastic system (\ref{Equation-xmm}) with jump measure satisfying
$
\int_{|\xi|<1}\xi \nu(d\xi)<\infty.
$ Then, up to an additive constant, the Onsager--Machlup function admits the approximation
\begin{eqnarray}\label{LOM}
OM(\dot{z},z)
=
\left(\frac{\dot{z}-f(z)}{c}\right)^2
+
f'(z)
+
2\frac{\dot{z}-f(z)}{c^2}\int_{|\xi|<1}\xi \nu(d\xi),
\end{eqnarray}
where $z(t)\in D_\tau^{x_0}$ is differentiable. The third term in (\ref{LOM}) represents the contribution arising from the pure jump L\'evy component. In the absence of jump noise, the above expression reduces to the classical OM function for diffusion processes.
Moreover, in terms of the OM function, the probability of the tube
$
K(z,\epsilon)
=
\{x\in D_\tau^{x_0}\mid z\in D_\tau^{x_0},\|x-z\|\leq\epsilon,\epsilon>0\}
$
can be approximated by
 $\mu_X(K(z,\epsilon))
\propto
\mu_{B^c}(K(x_0,\epsilon))
\exp\left\{
-\frac{1}{2}\int_s^uOM(\dot{z},z)dt
\right\}$, 
where the notation $\propto$ denotes asymptotic equivalence for sufficiently small $\epsilon$. Here, the process $B_t^c$ is defined through $dB_t^c=cdB_t+\int_{|y|<1}y\tilde{N}(dt,dy)$. 
\end{theorem}

\begin{proof}
We start with 
$
\mu_X(K(z,\epsilon))
=
\int_{K(z,\epsilon)}d\mu_X(x).
$ For every differentiable path $z(t)\in D_\tau^{x_0}$, one can introduce a function $z_0(t)\in D_\tau^{0}$ satisfying (\ref{t1}) such that
\begin{equation}\label{Equation-xm}
z(t)=x_0+z_0(t),
\qquad
\dot{z}(t)=\dot{z_0}(t).
\end{equation}
Consequently, the translation $\mathcal{T}$ defined by (\ref{t1}) is determined by this choice of $z_0(t)$, and we have $\mathcal{T}^{-1}(K(z,\epsilon))
=
K(x_0,\epsilon)$. Observe that every jump-diffusion process $Y_t$ with diffusion coefficient $c$ and initial condition $Y_s=x_0$ induces a quasi-translation invariant measure $\mu_Y$ satisfying $\mu_X\sim\mu_Y$, in view of Lemma \ref{lemma 3.1} and Lemma \ref{lemma 3.2}. Combining (\ref{F}) and (\ref{y}) therefore yields
\begin{align}\label{hj}
\mu_X(K(z,\epsilon))
=
\int_{K(z,\epsilon)}F[y]d\mu_Y(y)
=
\int_{K(x_0,\epsilon)}F[x+z_0]J_Y[x,-z_0]d\mu_Y(x),
\end{align}
since $K(z,\epsilon)\in\mathcal{M}\subset\mathcal{G}$. In particular, because $\mu_X$ itself is quasi-translation invariant by Lemma \ref{lemma 3.2}, we may choose $\mu_Y=\mu_X$. In this case $F=1$, and (\ref{hj}) reduces to
$\mu_X(K(z,\epsilon))
=
\int_{K(x_0,\epsilon)}J_X[x,-z_0]d\mu_X(x)$.
To further shift the integration domain to $K(0,\epsilon)$, define $Y_t^0=Y_t-x_0$
and denote by $\mu_{Y^0}$ the induced measure on $D_\tau^{0}$. Applying relation (\ref{y}) once again with translation parameter $x_0$, we obtain
\begin{eqnarray}\label{hjf}
\mu_X(K(z,\epsilon))
=
\int_{K(0,\epsilon)}F[y+z]J_Y[y+x_0,-z_0]d\mu_{Y^0}(y).
\end{eqnarray}
Expanding the exponent in $F[y+z]J_Y[y+x_0,-z_0]$ around $y(t)=0$ and $y(t-)=0$, respectively, and isolating the zero-order contributions, the remaining terms become arbitrarily small whenever $\epsilon$ is sufficiently small, since for $y(t)\in K(0,\epsilon)$ and $y(t-)\in K(0,\epsilon)$,
$
\|y(t)\|\leq\epsilon,
\ 
\|y(t-)\|\leq\epsilon.
$ Since $z(t)$ is continuously differentiable, denote the higher-order remainder terms collectively by $\Delta[y,z]$. Then it can be shown that \cite{Chao2019} 
\begin{align}\label{FJ1}
&F[y+z]J_Y[y+x_0,-z_0] 
=\exp(\Delta[y,z])
\cdot
\exp\Bigg\{
V(z(u))-V(x_0)
-\frac{1}{2}\int_s^uf(z(t))dt
\notag\\
&\quad +\int_s^u\int_{|\xi|<1}\frac{\xi}{c}a(z(t))\nu(d\xi)dt
\Bigg\} \cdot
\exp\Bigg\{
V_Y(x_0,-z_0(u))-V_Y(x_0,-z_0(s))
\notag\\
&\quad -\frac{1}{2}\int_s^ud_Y(x_0,-z_0(t))dt
 +\int_s^u\int_{|\xi|<1}\frac{\xi}{c}a_Y(x_0,-z_0(t))\nu(d\xi)dt
\Bigg\}.
\end{align}

Substituting (\ref{FJ1}) into (\ref{hjf}) leads to
\begin{eqnarray}\label{approx}
\mu_X(K(z,\epsilon))
=
F[z]J_Y[x_0,-z_0]
\int_{K(0,\epsilon)}\exp(\Delta[y,z])d\mu_{Y^0}(y).
\end{eqnarray}

We now estimate the remainder term $\Delta[y,z]$ appearing in (\ref{approx}). Recall that for a functional $\Psi[y]$ defined on $D_\tau^{x_0}$ satisfying $\|\Psi[y]\|\leq\gamma$, one has
$
\int_B\Psi[y]d\mu(y)\leq\gamma\mu(B),
\ 
B\in\mathcal{G}.
$ Hence, if $\epsilon>0$ is chosen sufficiently small so that $\Delta[y,z]<\gamma$ with $\gamma\to0$, then expanding the exponential in (\ref{approx}) and neglecting terms of order smaller than $\gamma$ (with $\gamma\ll1$) yields the approximation $\mu_X(K(z,\epsilon))
\propto
F[z]J_Y[x_0,-z_0]\mu_{Y^0}(K(0,\epsilon))$. 

Using the identity
$
\mu_{Y^0}(K(0,\epsilon))
=
\mu_Y(K(x_0,\epsilon)),
$
we finally arrive at
\begin{eqnarray}\label{tubePf}
\mu_X(K(z,\epsilon))
\propto
F[z]J_Y[x_0,-z_0]\mu_Y(K(x_0,\epsilon)).
\end{eqnarray}
Here, the notation $\propto$ is understood in the asymptotic sense as $\epsilon$ becomes sufficiently small. Based on (\ref{tubePf}), maximizing $\mu_X(K(z,\epsilon))$ reduces to maximizing the functional $M[z]
=
F[z]J_Y[x_0,-z_0]$. We next simplify the functional $M[z]$ following the same philosophy used for diffusion processes in \cite{Durr1978}. Let $X_t$ satisfy (\ref{Equation-xmm}), and let $Y_t$ be defined by (\ref{Equation-y}) with $g(x)=c$ and $h(y,x)=x$. It can be shown that \cite{Chao2019}
\begin{eqnarray}
V(z(u))-V(x_0)
=&&
\frac{1}{c^2}
\int_s^u
[\dot{z}(t)\{f(z(t))-k(z(t))\}]dt,\label{Vf}\\
f(z(t))
=&&
\left\{
\frac{f(z(t))-k(z(t))}{c}
\right\}^2
+
\frac{df(x)}{dx}\mid_{x=z(t)}
\label{bf}\\
&&
-
\frac{dk(x)}{dx}\mid_{x=z(t)}
+
\frac{2}{c^2}
\{f(z(t))-k(z(t))\}k(z(t)),\notag\\
V_Y(x_0,-z_0(u))-V_Y(x_0,-z_0(s))
=&&
\frac{1}{c^2}
\int_s^u
[\dot{z}(t)k(z(t))-\ddot{z}(t)x_0]dt,\label{VYf}
\\
d_Y(x_0,-z_0(t))
=&&
-\frac{k^2(x_0)}{c^2}
-
\frac{d k(x)}{dx}\mid_{x=x_0}
+
\frac{{\dot{z}(t)}^2}{c^2}
\notag\\
&&
+
\frac{k^2(z(t))}{c^2}
-
\frac{2\ddot{z}(t)x_0}{c^2}
+
\frac{d k(x)}{dx}\mid_{x=z(t)}.\label{dYf}
\end{eqnarray}

Combining the previous four identities yields
\begin{align}\label{lnM}
\log (M[z])
=&
-\frac{1}{2}
\int_s^u
\left\{
\left(
\frac{f(z)-\dot{z}}{c}
\right)^2
+
f'(z)
-
\left(
\frac{k^2(x_0)}{c^2}
+
k'(x_0)
\right)
\right\}dt
\notag\\
&
+
\int_s^u
\int_{|\xi|<1}
\frac{\xi}{c^2}f(z)\nu(d\xi)dt
-
\int_s^u
\int_{|\xi|<1}
\frac{\xi}{c^2}(k(x_0)+\dot{z})\nu(d\xi)dt.
\end{align}
Notice that the maximizer $z_m(t)$ of (\ref{tubePf}) should not depend on the particular quasi-translation invariant measure $\mu_Y$, namely, it should be independent of the choice of $k(z)$. This requirement is indeed satisfied here. Furthermore, the term
$
\left(\int_{|\xi|<1}\xi\nu(d\xi)/c
\right)^2
-
\left( -k(x_0)+\int_{|\xi|<1}\xi\nu(d\xi) 
\right)^2/c^2
+
k'(x_0)
$
is merely a constant depending on the chosen measure. Hence, up to an additive constant, the OM function can be taken as
\begin{align}\label{OMP}
OM(\dot{z},z)
=
\left(
\frac{\dot{z}-f(z)}{c}
\right)^2
+
f'(z)
+
2\frac{\dot{z}-f(z)}{c^2}
\int_{|\xi|<1}\xi \nu(d\xi).
\end{align}
This completes the proof.
\end{proof}

Recall that, when treating the remainder term $\Delta[y,z]$ in (\ref{approx}), we neglected the Poisson integral contributions from trajectories inside the tube around the reference path. This is a nontrivial approximation, since the number of jumps is not uniformly bounded over all L\'evy sample paths within a given tube. Therefore, the Onsager--Machlup functional in (\ref{LOM}) should be regarded as an approximation rather than a complete path description for jump diffusions with infinite jump activity.

Other approaches have been developed under additional assumptions, including the \emph{probability flow equivalence method} for finite-dimensional jump-diffusion systems with finite jump activity \cite{huang2025probability} and a framework based on the Kramers--Moyal equation \cite{faria2025nonequilibrium}.

The main result of \cite{huang2025probability} concerns the jump-diffusion equation $d X_t=\ b(X_{t-}) d t +  \sigma d B_t +  J_t d N_t, \ 
         X_0=\ x_0\in\mathbb{R}^n,$. 
where $b$ is a measurable vector-valued function, $\sigma$ is a constant, and $X_t\in\mathbb{R}^n$ denotes the position of a random particle. Here, $B$ is a standard Brownian motion in $\mathbb{R}^n$, while $N$ is an independent Poisson process with state-dependent stochastic intensity $\lambda(X_{t-})$ and unit jump size. In addition, $J$ denotes a random jump amplitude with intensity $\nu_J$, independent of both $B$ and $N$. The authors proved that $\mathbb{P}\left(
    \sup_{t\in[0,T]}|X_t-\psi(t)|\leq\delta
    \right)
    \approx
    C(\delta)\exp(-S_X^{\mathrm{OM}}(\psi)),
    \ 
    \delta\downarrow0$, 
for some positive constant $C(\delta)$ depending on $\delta$, where $\psi\in C^1_{x_0}([0,T],\mathbb R^d)$, and
\begin{equation*}
    \begin{aligned}
S_\mathrm{finite}^{\mathrm{OM}}(\psi,\dot{\psi})
        =&\
        \frac{1}{2} \int_{0}^{T} \Bigg[ \frac{1}{\sigma^2}\left|\dot{\psi}_s -  b(\psi_s) - \ell_J(\psi_s) \right|^2   +  \nabla\cdot\left( b(\psi_s) +  \ell_J(\psi_s)\right) + \tilde{\ell}_J(\psi_s) \Bigg]d s, 
        \\
\ell_J(x):=&\ \int_{\mathbb{R}^d} z\int_0^1\frac{\lambda(x-\theta z)\nu_J( -\theta z)}{\nu_J(0)} d\theta\ \nu_J(d z),\\
\tilde{\ell}_J(x):=&\ \int_{\mathbb{R}^d}z\cdot\int_0^1\lambda(x-\theta z)\frac{\nabla\nu_J(-\theta z)\nu_J(0)-\nu_J(-\theta z)\nabla\nu_J(0)}{(\nu_J(0))^2} d\theta\ \nu_J(d z).
    \end{aligned}
\end{equation*}
Consider the Lagrangian $L(\psi,\dot{\psi})$ corresponding to the
Onsager--Machlup functional
$S_{\mathrm{finite}}^{\mathrm{OM}}(\psi)
=
\int_0^T L(\psi_s,\dot{\psi}_s)\,ds.$
Define the effective drift
$F(x):=b(x)+\ell_J(x)$,
where $\ell_J$ is the vector-valued jump correction defined above.
The Lagrangian then takes the form $L(\psi,\dot{\psi})
=
\frac{1}{2\sigma^2}
\left|\dot{\psi}-F(\psi)\right|^2
+
\frac12\nabla\cdot F(\psi)
+
\frac12\tilde{\ell}_J(\psi)$. For the following variational calculation, assume that $\sigma>0$,
$F\in C^2(\mathbb{R}^n;\mathbb{R}^n)$, and
$\tilde{\ell}_J\in C^1(\mathbb{R}^n)$.
We use the Jacobian convention
$(\nabla F)_{ij}=\partial_j F_i$. The derivatives of $L$ are $\frac{\partial L}{\partial\dot{\psi}}
=
\frac{1}{\sigma^2}
\left(\dot{\psi}-F(\psi)\right)$ 
and $\frac{\partial L}{\partial\psi}
=
-\frac{1}{\sigma^2}
\bigl(\nabla F(\psi)\bigr)^\top
\left(\dot{\psi}-F(\psi)\right)
+
\frac12\nabla(\nabla\cdot F)(\psi)
+
\frac12\nabla\tilde{\ell}_J(\psi)$. Moreover, along a twice continuously differentiable path,
$\frac{d}{ds}
\frac{\partial L}{\partial\dot{\psi}}
(\psi_s,\dot{\psi}_s)
=
\frac{1}{\sigma^2}
\left[
\ddot{\psi}_s
-
\nabla F(\psi_s)\dot{\psi}_s
\right]$. Consequently, the Euler--Lagrange equation becomes
\begin{equation*}
\begin{aligned}
\ddot{\psi}_s
={}&
\bigl(\nabla F(\psi_s)\bigr)^\top F(\psi_s)
+
\frac{\sigma^2}{2}
\nabla(\nabla\cdot F)(\psi_s)
+
\frac{\sigma^2}{2}
\nabla\tilde{\ell}_J(\psi_s)
\\
&+
\left[
\nabla F(\psi_s)
-
\bigl(\nabla F(\psi_s)\bigr)^\top
\right]\dot{\psi}_s.
\end{aligned}
\end{equation*}
For prescribed endpoints $\psi(0)=x_0$ and $\psi(T)=x_T$,
this equation is a necessary condition for a smooth minimizer
of $S_{\mathrm{finite}}^{\mathrm{OM}}$.
Solutions of the resulting boundary-value problem are stationary
paths of the action and need not all be global minimizers.

When the explicit drift vanishes, $b=0$, the effective drift reduces to
$F_J(x):=\ell_J(x)
=
\int_{\mathbb{R}^n}
z\int_0^1
\frac{
\lambda(x-\theta z)\nu_J(-\theta z)
}{
\nu_J(0)
}
\,d\theta\,\nu_J(dz)$.
The scalar correction $\tilde{\ell}_J$ remains present, and the
Euler--Lagrange equation reduces to
\begin{equation*}
\begin{aligned}
\ddot{\psi}_s
={}&
\bigl(\nabla F_J(\psi_s)\bigr)^\top F_J(\psi_s)
+
\frac{\sigma^2}{2}
\nabla(\nabla\cdot F_J)(\psi_s)
+
\frac{\sigma^2}{2}
\nabla\tilde{\ell}_J(\psi_s)
\\
&+
\left[
\nabla F_J(\psi_s)
-
\bigl(\nabla F_J(\psi_s)\bigr)^\top
\right]\dot{\psi}_s.
\end{aligned}
\end{equation*}
Thus, jump statistics influence the most probable transition path
through both the effective drift $F_J$ and the scalar correction
$\tilde{\ell}_J$.
In addition to the effective-drift and divergence contributions,
the scalar correction generates the force-like term
$\frac{\sigma^2}{2}\nabla\tilde{\ell}_J$ in the variational equation.
Spatial variations of the jump intensity $\lambda(x)$ can therefore
affect path selection through both corrections.
In this sense, the jump mechanism induces an effective
``virtual'' force in the Onsager--Machlup description;
this is a property of the action-minimizing curves, rather than
an additional deterministic force in the original SDE.
The Brownian contribution remains present through $\sigma$,
so setting $b=0$ does not reduce the model to a pure-jump process. Moreover, the probability flow equivalence method can also be
employed to study probability currents \cite{huang2024levy}
and entropy production \cite{huang2026entropy}
in non-Gaussian systems.

Ref.~ \cite{faria2025nonequilibrium} approaches the problem from the perspective of the Kramers--Moyal equation and shows that the OM functional can be generalized into an infinite-series representation. In particular, for the SDE (\ref{Equation-xmm}) with drift $f(x)=-U'(x)$ and small noise intensity $\epsilon$, the corresponding Kramers--Moyal OM functional is
\begin{equation}\label{KMOM}
    \begin{aligned}
        S_\mathrm{KM}^\mathrm{OM}(x,\dot{x})
        =
        \sum_{n=2}^\infty
        \frac{a_n}{n!}
        \left(
        \frac{\dot{x}-U'(x)}{\epsilon a_2}
        \right)^n,
    \end{aligned}
\end{equation}
where the constants $a_n$ are determined by the jump noise. According to \cite{faria2025nonequilibrium}, the term corresponding to $n=2$ recovers the classical OM functional for Gaussian diffusion processes, whereas all higher-order terms with $n\geq3$ reflect non-Gaussian effects.

The Euler--Lagrange equation associated with the Kramers--Moyal OM functional can be expressed compactly as
\begin{equation}\label{KM:EL}
    \frac{\phi''(y)}{(\epsilon a_2)^2}
    \bigl(
    \ddot{x}-U''(x)\dot{x}
    \bigr)
    +
    \frac{U''(x)}{\epsilon a_2}\phi'(y)
    =
    0,
\end{equation}
where $y=
\frac{\dot{x}-U'(x)}{\epsilon a_2}$. When $U(x)=0$, one has $U'(x)=0$ and $U''(x)=0$. Under this condition, (\ref{KM:EL}) simplifies to $\frac{d}{dt}\phi'(y)=0$. Equivalently, $\phi'(y)=\mathrm{const}$. Expressed in terms of $\dot{x}$, this becomes $\phi'
    \left(
    \frac{\dot{x}}{\epsilon a_2}
    \right)
    =
    \mathrm{const}$. If $\phi''(y)\neq0$, it follows further that $\dot{x}=\mathrm{const}$, 
which implies that the optimal trajectory is linear in time.

Beyond the theoretical developments presented in this chapter, the
Onsager--Machlup framework has recently been combined with machine-learning
methods for computing transition pathways and predicting stochastic tipping
events. Relevant approaches include neural-network and physics-informed
methods~\cite{LiDuanLiu2021,ChenHuWeiDuan2022}, deep reinforcement
learning~\cite{GuoGaoZhangHanDuan2024}, and latent stochastic models for
detecting early-warning signals from high-dimensional data
\cite{FengGaoXiaoDuan2024}. More recently, the Onsager--Machlup approach has
also been compared with the Schr\"odinger bridge framework
\cite{xu2026early}. These developments illustrate how the pathwise
theory developed here can support data-driven studies of rare transitions
and tipping phenomena.

\section*{Problems}
\addcontentsline{toc}{section}{Problems}

\begin{prob}
\textbf{The most probable transition path and Hamiltonian system}\\
Prove Theorem \ref{MPTPinnewHamilton} and derive the Hamiltonian
system satisfied by the most probable transition path.

\medskip
\noindent\textit{Hint.}
Set $F_p(x,t)=\sigma^2\nabla_x\log p(x_T,T\mid x,t),
\
f_p(t)=F_p(\psi^*(t),t)$. 
Apply the logarithmic transformation to the Kolmogorov backward
equation and differentiate the resulting equation with respect to
the spatial variables. Then use the chain rule along $\psi^*$ together
with the first-order characterization \eqref{MPTP}. Treat the
reverse-time system analogously using the Kolmogorov forward equation.
\end{prob}

\begin{prob}
\label{prob1}
\textbf{The most probable transition path and Markovian bridge}\\
Prove Lemma \ref{OMderived} to show that the most probable transition
path of the original system coincides with that of the associated
Markovian bridge system. Then prove Lemma \ref{equivalence} to show
that the bridge measure equals the law of the bridge SDE.

\medskip
\noindent\textit{Hint.}
Use Lemma \ref{cylinderset} to approximate a closed tube by closed
cylinder sets, and evaluate their probabilities using the
finite-dimensional bridge distribution \eqref{finitedistribution}.
The conditioning on the terminal point contributes only a
path-independent normalization factor to the small-tube asymptotics.
For Lemma \ref{equivalence}, compare the transition kernels and
finite-dimensional distributions of the two measures on cylinder
sets, and then extend the equality to
$\mathcal{B}^{x_0}_{[0,T]}$ by the monotone class theorem.
\end{prob}

\begin{prob}
\textbf{The most probable transition path of an OU process}\\
Consider the stochastic differential equation
\begin{equation*}
dX_t=\theta(\mu-X_t)\,dt+\sigma\,dB_t,
\qquad X_0=x_0\in\mathbb{R},
\end{equation*}
whose drift is associated with the potential $U(x)=\theta\left(\frac{x^2}{2}-\mu x\right)$.
Derive the PDE for which the most probable
transition paths of the above SDE are characteristic curves.

\medskip
\noindent\textit{Hint.}
Compute $U'(x)$, $U''(x)$, and $U'''(x)$, and substitute them into the
forward equation for $F_p$ in Theorem \ref{relationHamiltonNS}. Define $F_p(x,t)=\sigma^2\partial_x\log p(x_T,T\mid x,t)$. 
The spatial characteristic equation is obtained by evaluating
$-\theta(x-\mu)+F_p(x,t)$ along a curve $x=\psi(t)$; compare the
resulting characteristic system with \eqref{MPTPnew}.
\end{prob}

\begin{prob}
\textbf{The $\sigma$-field of the c\`{a}dl\`{a}g path space}\\
Prove Lemma \ref{lemma 2.1}, namely, $\mathcal{B}_0=\mathcal{G}
\ \text{and}\ 
\mathcal{G}\subsetneqq\mathcal{B}_\tau^{x_0}$.

\medskip
\noindent\textit{Hint.}
Use right continuity to express the supremum norm over $[s,u]$ through
the values at the countable dense set
$\mathbb{Q}\cap[s,u]$, together with the endpoint $u$. This relates
uniform closed balls to countable combinations of cylinder sets and
gives $\mathcal{B}_0=\mathcal{G}$. For strict inclusion, exploit the
nonseparability of $D_\tau^{x_0}$ under the uniform norm: construct an
uncountable uniformly separated family of one-jump paths and consider
unions of pairwise disjoint small open balls around them.
\end{prob}

\begin{prob}
\label{prob2}
\textbf{Euler--Lagrange equation for Kramers--Moyal version  of Onsager--Machlup function}\\
Derive the Euler--Lagrange equation \eqref{KM:EL} for the
Kramers--Moyal version of Onsager--Machlup functional.

\medskip
\noindent\textit{Hint.}
Introduce $\phi(y)=\sum_{n=2}^{\infty}\frac{a_n}{n!}y^n,
\ 
y=\frac{\dot{x}-U'(x)}{\epsilon a_2}$, 
so that the Kramers--Moyal Lagrangian is $L(x,\dot{x})=\phi(y)$.
Compute $\partial L/\partial\dot{x}$ and $\partial L/\partial x$,
and substitute them into $\frac{d}{dt}\frac{\partial L}{\partial\dot{x}}
-\frac{\partial L}{\partial x}=0$.
Finally, use
$\dot{y}=(\ddot{x}-U''(x)\dot{x})/(\epsilon a_2)$.
\end{prob}

%
%
%
\chapter{Stochastic Variational Principles and Stochastic Geometric Mechanics }
\label{ch:4} 

\abstract*{This chapter develops a forward mean-derivative formulation of stochastic Lagrangian and Hamiltonian mechanics. We establish Hamilton's and scaled Maupertuis's principles, relate configuration-space and phase-space dynamics through the Legendre transform, and derive second-order Hamilton--Jacobi equations. Mechanical Lagrangians yield stochastic Newton equations and a reduced Jacobi functional. We distinguish Schr\"odinger's entropy-minimization problem from the Onsager--Machlup description of most probable curves, and conclude with a survey of stochastic geometric mechanics on manifolds as developed in \cite{HZ23} using Schwartz--Meyer second-order differential geometry.}

\abstract{This chapter develops a mean-derivative formulation of stochastic Lagrangian and Hamiltonian mechanics. We establish Hamilton's and scaled Maupertuis's principles, relate configuration-space and phase-space dynamics through the Legendre transform, and derive second-order Hamilton--Jacobi equations. Mechanical Lagrangians yield stochastic Newton equations and a reduced Jacobi functional. We distinguish Schr\"odinger's entropy-minimization problem from the Onsager--Machlup description of most probable curves, and conclude with a survey of stochastic geometric mechanics on manifolds as developed in \cite{HZ23} using Schwartz--Meyer second-order differential geometry.}

\section{Stochastic Lagrangian Mechanics}\label{sec-7-2}

Hamilton's principle characterizes motion by stationary action \cite{VIArnold1989,Holm2011}. For nondifferentiable diffusion paths, velocity, the role of noise, and admissible variations must be specified.

Schr\"odinger's endpoint-distribution problem gave a variational description of path laws \cite{Schrodinger1931,Schrodinger1932}. Nelson's mean-derivative formulation of real-time Schr\"odinger dynamics \cite{Nelson} inspired the mid-1980s revival of stochastic variational methods and a distinct Euclidean quantum programme \cite{zambrini1986schrodinger,zambrini1986variational,Zam15}, with expected-action principles in \cite{CZ91,CZ03}.

Dynamic programming links stochastic control to second-order Hamilton--Jacobi--Bellman equations \cite{bismut1973conjugate,FS06}. Fleming and Soner \cite{FS06} relate suitable quadratic-control problems to imaginary-time Schr\"odinger equations by logarithmic transformations; endpoint constraints lead to stochastic optimal transport \cite{leonard2014,Mik21}. The actions and variations may differ between approaches.

Here forward mean velocity and adapted random Cameron--Martin variations yield Hamilton's principle at fixed covariance, while deterministic time scales give an autonomous Maupertuis principle at prescribed mean energy. This forward framework differs from Nelson's time-symmetric dynamics. Phase-space and mechanical formulations precede the manifold survey (Section \ref{sec:multip-SGM}) \cite{HZ23}. Noise intensity is normalized to one.

\subsection{General settings}

Fix $T>0$. We work on a filtered probability space $(\Omega,\F,\{\F_t\}_{t\in[0,T]},\P)$ satisfying the usual conditions and carrying an $m$-dimensional Brownian motion $B$ with respect to this filtration.

We first recall the notion of an It\^o process; see, e.g., \cite{Ikeda1980}. An $\R^n$-valued continuous adapted process $X$ is an \emph{It\^o process} if it admits a representation
\begin{equation}\label{Ito-process}
  X(t)=X(0)+\int_0^t b(s)\,ds+\int_0^t \sigma(s)\,dB(s),
\end{equation}
where $X(0)$ is $\F_0$-measurable and the coefficients $b$ and $\sigma$ are predictable processes. For the admissible processes $X$ entering the variational problem, we impose the following square-integrability condition:
\begin{equation}\label{Ito-integrability}
  \E|X(0)|^2+\E\int_0^T\left(|b(t)|^2+\|\sigma(t)\|_{\mathrm{HS}}^2\right)dt<\infty.
\end{equation}
Drift and diffusion coefficients that agree $dt\otimes\P$-a.e. are identified; continuous processes are taken in their continuous versions. We write $L^2_{\mathrm{pred}}([0,T]\times\Omega;\R^n)$ for the space of predictable processes that are square-integrable with respect to $dt\otimes\P$.

Here and below, a \emph{representative} refers to equality $dt\otimes\P$-a.e., or $dt$-a.e. for deterministic functions of time. It need not be a \emph{modification} of a given process, which requires almost-sure equality at each fixed time.

We also recall Nelson's forward mean derivatives \cite{Nelson}.
\begin{definition}[Mean derivatives]\label{def-mean-der}
For an $\R^n$-valued process $X$ and $0\le t<T$, its forward mean derivative $DX$ and forward quadratic mean derivative $QX$ are
\begin{gather}
  DX(t)=\lim_{h\downarrow0}\E\left[\frac{X(t+h)-X(t)}{h}\bigg|\F_t\right], \label{mean-der}\\
  QX(t)=\lim_{h\downarrow0}\E\left[\frac{(X(t+h)-X(t)) (X(t+h)-X(t))^\top}{h}\bigg|\F_t\right], \label{quad-mean-der}
\end{gather}
whenever the conditional expectations are well defined and the limits exist in $L^1(\Omega)$. For $f\in C^{1,2}([0,T]\times\R^n)$, its total mean derivative along $X$ is
\begin{equation*}
  (\D_t f)(t,X(t)):=D[f(\cdot,X(\cdot))](t),
\end{equation*}
whenever this mean derivative exists.
\end{definition}

Nelson's forward and backward mean derivatives \cite{Nelson} condition respectively on the past and future in the natural-filtration formulation. Write $\F_t^+:=\sigma\{X(s):t\le s\le T\}$, completed under $\P$. For $0<t\le T$, the past-oriented backward mean derivative is
\begin{equation*}
  \overleftarrow D X(t):=\lim_{h\downarrow0}\E\left[\frac{X(t-h)-X(t)}{h}\bigg|\F_t^+\right],
\end{equation*}
whenever the conditional expectations are well defined and the limit exists in $L^1(\Omega)$. This sign convention matches Figure \ref{fig:mean-derivatives}. For a differentiable deterministic path, $DX=\dot X$ and $\overleftarrow D X=-\dot X$, but generally $DX\ne-\overleftarrow D X$ for a diffusion. Time-reversal symmetry therefore requires additional structure, although reversible stochastic models exist. Throughout this chapter we use the forward formulation; its parallel backward counterpart reverses the time orientation, filtrations, and endpoint data, with corresponding signs and backward regularity assumptions.

\begin{figure}[htbp]
  \centering
  \includegraphics[width=0.6\textwidth]{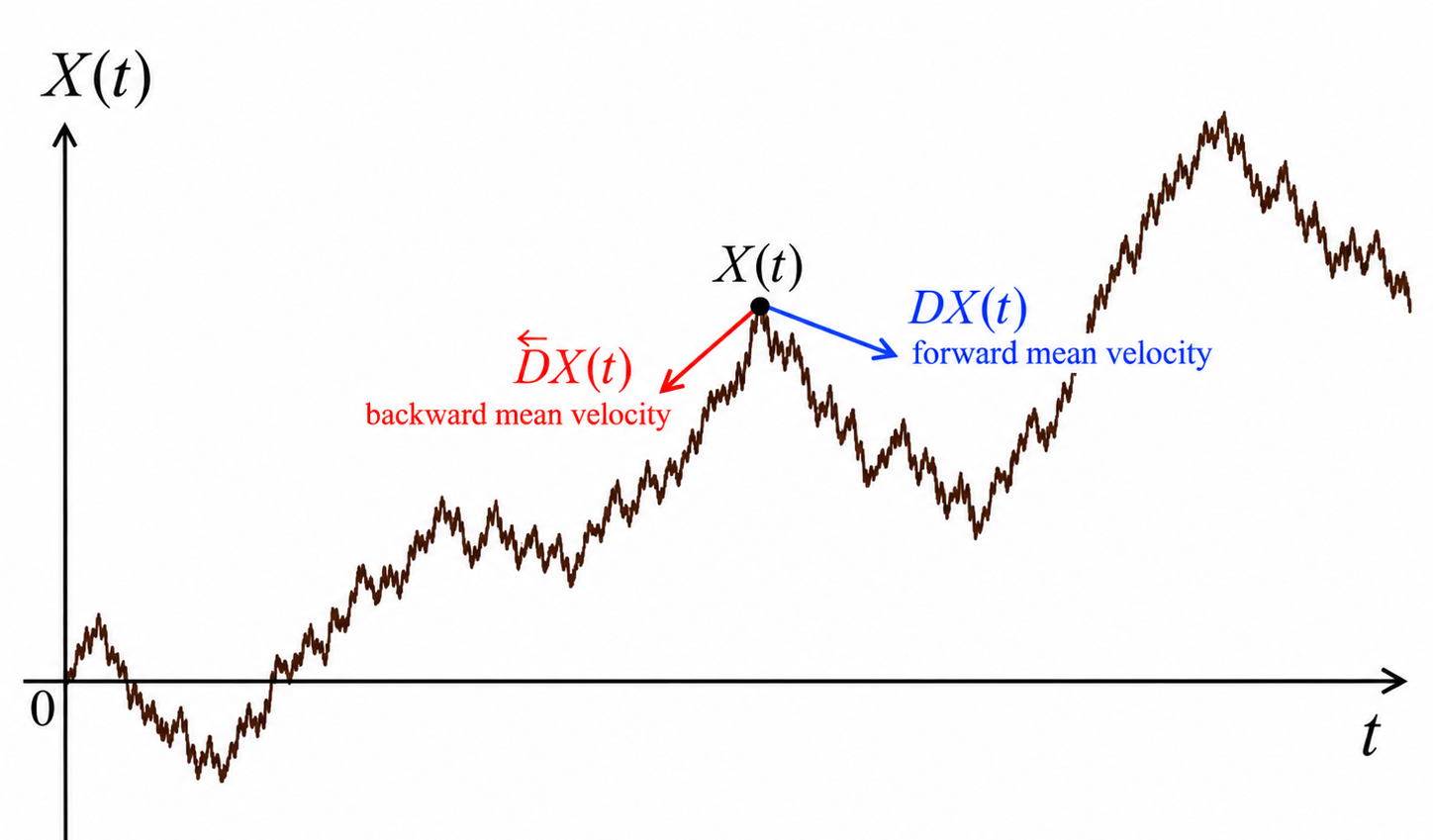}
  \caption{Forward and past-oriented backward mean velocities along a diffusion, shown schematically. The arrows represent conditional mean velocities, not pathwise tangents.}
  \label{fig:mean-derivatives}
\end{figure}

For the It\^o process \eqref{Ito-process} satisfying \eqref{Ito-integrability}, for Lebesgue-a.e. $t\in[0,T)$, the limits in \eqref{mean-der}--\eqref{quad-mean-der} exist in $L^1(\Omega)$ and
\begin{equation*}
  DX(t)=b(t),\qquad QX(t)=\sigma(t)\sigma(t)^\top.
\end{equation*}
Conditional martingale increments have mean zero and conditional second moments given by the integrated covariance. Lebesgue differentiation in $L^1(\Omega)$ identifies the coefficients; Cauchy--Schwarz and \eqref{Ito-integrability} make the drift contribution to the scaled second moment and its cross terms vanish.

The composite process $f(t,X(t))$ is not required to satisfy \eqref{Ito-integrability}. To compute its forward mean derivative, the following integrability conditions suffice, with all derivatives of $f$ in the display evaluated at $(t,X(t))$:
\begin{equation}\label{composite-mean-integrability}
\begin{gathered}
  \E\int_0^T\left(|\pt_t f|+|\nabla f|\,|b(t)|
  +\|\nabla^2 f\|_{\mathrm{op}}\,\|\sigma(t)\|_{\mathrm{HS}}^2\right)dt<\infty,\\
  \E\int_0^T|\sigma(t)^\top\nabla f|^2dt<\infty.
\end{gathered}
\end{equation}
Here $|\cdot|$ denotes the Euclidean norm of a vector. For a matrix $A$, $\|A\|_{\mathrm{op}}:=\sup_{|u|=1}|Au|$ is the operator norm induced by the Euclidean norm, and $\|A\|_{\mathrm{HS}}:=(\sum_{i,j}|A_{ij}|^2)^{1/2}$ is the Hilbert--Schmidt (or Frobenius) norm.
It\^o's formula then gives an integrable drift and a square-integrable martingale term. Taking conditional expectations and using Lebesgue differentiation yields, for Lebesgue-a.e. $t<T$,
\begin{equation}\label{total-mean-der}
\begin{split}
  (\D_t f)(t,X(t))
  =\left(\pt_t f+DX(t)\cdot\nabla f
  +\tfrac12\operatorname{Tr}(QX(t)\nabla^2 f)\right)(t,X(t)).
\end{split}
\end{equation}
Bounded $\pt_t f$, $\nabla f$ and $\nabla^2 f$ ensure \eqref{composite-mean-integrability}. Polynomial growth alone requires additional moment or integrability checks beyond the second-moment assumption \eqref{Ito-integrability}. These statements extend componentwise to vector-valued functions.

For stochastic Hamilton's principle in the next subsection, we assume that $L_0:[0,T]\times\R^n\times\R^n\to\R$ is continuous and continuously differentiable in $(x,\dot x)$, $\pt_xL_0$ and $\pt_{\dot x}L_0$ are jointly continuous, and there is a constant $C$ such that
\begin{equation}\label{Lagrangian-growth-assumption}
\begin{gathered}
  |L_0(t,0,0)|\le C,\\
  |\pt_x L_0(t,x,\dot x)|+|\pt_{\dot x}L_0(t,x,\dot x)|
  \le C(1+|x|+|\dot x|)
\end{gathered}
\end{equation}
for all $(t,x,\dot x)$. In particular, $|L_0(t,x,\dot x)|\le C'(1+|x|^2+|\dot x|^2)$ for some constant $C'$. 
The additional regularity needed to differentiate the momentum is specified in Theorem \ref{stoch-Hamilton-prin}.
The subscript $0$ indicates that $L_0$ is the classical, undeformed Lagrangian; it distinguishes $L_0$ from the deformed quantities introduced below.

\subsection{Stochastic Hamilton's Principle}

Let $\mu_0,\mu_T\in\Pred(\R^n)$ have finite second moments. For the prescribed covariance below we take $m\ge n$.

\begin{definition}[Admissible processes]\label{admissible-processes}
The class $\A([0,T];\mu_0,\mu_T)$ consists of It\^o processes of the form \eqref{Ito-process} satisfying \eqref{Ito-integrability} and
\begin{equation*}
\begin{gathered}
  QX(t)=I_n,
  \qquad dt\otimes\P\text{-a.e.},\\
  \operatorname{Law}(X(0))=\mu_0,\qquad
  \operatorname{Law}(X(T))=\mu_T.
\end{gathered}
\end{equation*}
\end{definition}

Figure \ref{fig:Schrodinger-bridge} depicts the prescribed endpoint laws through a Schr\"odinger-bridge picture, without asserting optimality. Section \ref{sec-7-3} discusses its entropy-minimization interpretation.

\begin{figure}[htbp]
  \centering
  \includegraphics[width=0.6\textwidth]{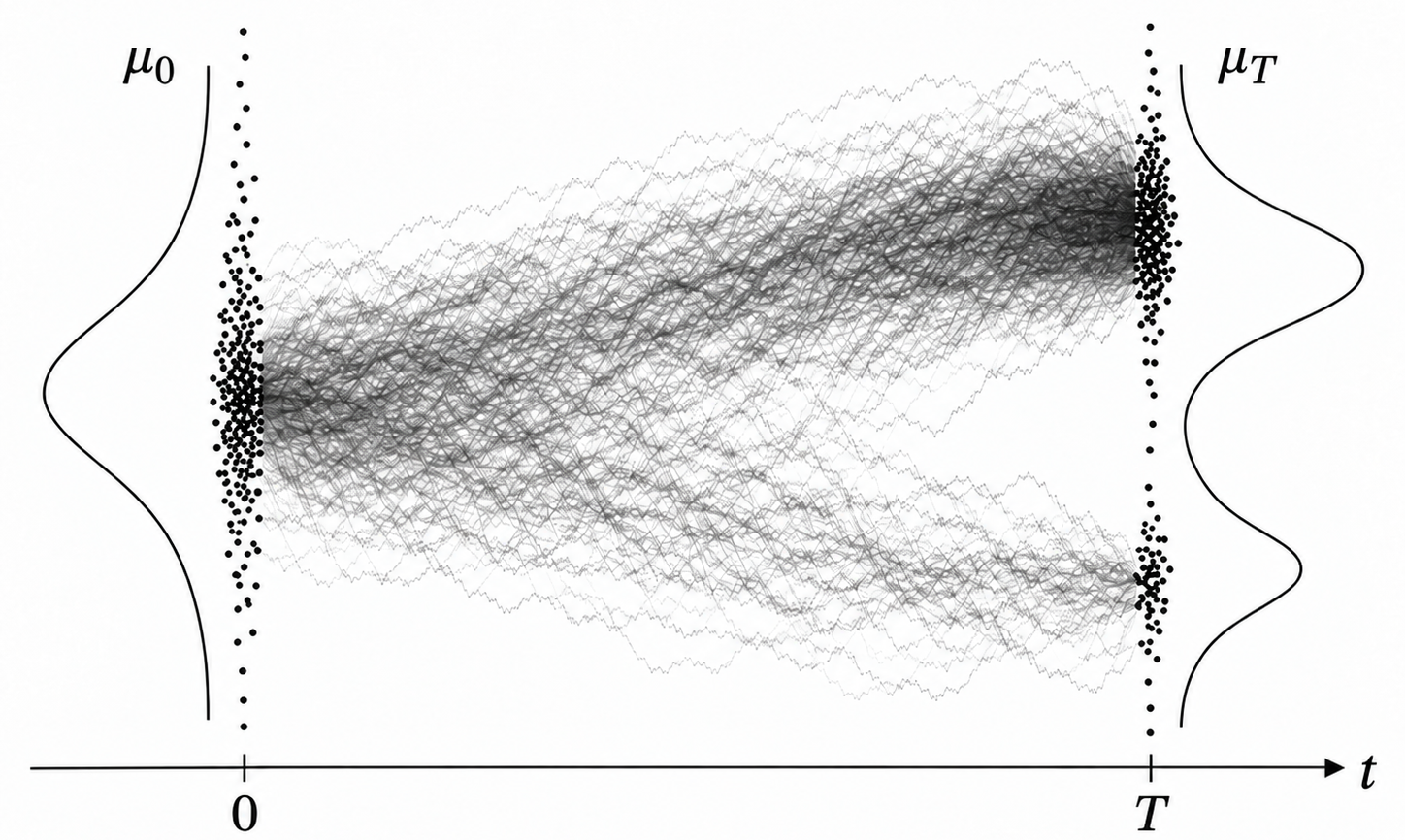}
  \caption{A schematic ensemble of diffusion paths connecting the prescribed endpoint laws $\mu_0$ and $\mu_T$. The endpoint distributions are fixed, while the endpoint values vary across realizations.}
  \label{fig:Schrodinger-bridge}
\end{figure}

The coefficient assumptions already give $\E\int_0^T|DX(t)|^2dt<\infty$. Moreover, Cauchy--Schwarz and Doob's inequalities imply
\begin{equation}\label{moment-assumptions}
\begin{split}
  \E\sup_{0\le t\le T}|X(t)|^2
  &\le 3\E|X(0)|^2+3T\E\int_0^T|DX(t)|^2dt
       +12nT<\infty.
\end{split}
\end{equation}

Consider the optimal-transport-type stochastic action
\begin{equation}\label{action}
  \mathcal S[X]:=\E\int_0^T L_0(t,X(t),DX(t))\,dt.
\end{equation}
The growth condition \eqref{Lagrangian-growth-assumption} and the preceding moment estimates ensure that \eqref{action} is finite.

We use endpoint-preserving directions of random Cameron--Martin type. Denote by $\mathcal H([0,T];\R^n)$ the space of adapted processes $v$ with absolutely continuous paths such that
\begin{equation*}
\begin{gathered}
  v(0)=v(T)=0\quad\P\text{-a.s.},\\
  \dot v\text{ is predictable},\qquad
  \E\int_0^T|\dot v(t)|^2dt<\infty.
\end{gathered}
\end{equation*}
Backward difference quotients give a predictable version of the path derivative of any adapted process with absolutely continuous paths. No separate integrability or boundedness assumption on $v$ is needed, since
\begin{equation}\label{variation-sup-estimate}
  \E\sup_{0\le t\le T}|v(t)|^2
  \le T\E\int_0^T|\dot v(t)|^2dt<\infty.
\end{equation}

\begin{definition}[Random Cameron--Martin variation]\label{variation}
For $X\in\A([0,T];\mu_0,\mu_T)$ and $v\in\mathcal H([0,T];\R^n)$, define
\begin{equation*}
  X^v_\e(t):=X(t)+\e v(t),\qquad \e\in\R.
\end{equation*}
\end{definition}

Then $X^v_\e$ is an It\^o process with the same diffusion coefficient as $X$ and drift $DX+\e\dot v$. In particular,
\begin{equation*}
  DX^v_\e=DX+\e\dot v,\qquad
  QX^v_\e=QX=I_n,\qquad dt\otimes\P\text{-a.e.}
\end{equation*}
Since $X^v_\e(0)=X(0)$ and $X^v_\e(T)=X(T)$ almost surely, $X^v_\e\in\A([0,T];\mu_0,\mu_T)$ for every $\e\in\R$: endpoint laws, prescribed quadratic variation and square-integrability are preserved. Also,
\begin{equation}\label{variation-deriv}
  \frac{\pt}{\pt\e}\bigg|_{\e=0}DX^v_\e=\dot v
  \quad\text{in }L^2(dt\otimes\P).
\end{equation}

An $X\in\A([0,T];\mu_0,\mu_T)$ is a \emph{critical point} of $\mathcal S$ if
\begin{equation}\label{stationary}
  \frac{d}{d\e}\bigg|_{\e=0}\mathcal S[X^v_\e]=0
\end{equation}
for every $v\in\mathcal H([0,T];\R^n)$. Problem \eqref{stationary} relates to stochastic control and transport; see \cite{CZ91,CZ95} for Cameron--Martin stochastic variations formulated using Malliavin calculus.

\begin{remark}[A separated-variable core]\label{remark-separated-variations}
The residual in the proof below can be tested against finite sums of variations
\begin{equation}\label{separated-variations}
  v(t,\omega)=\sum_{k=1}^N\xi_k(\omega)\phi_k(t),
\end{equation}
where $\xi_k\in L^\infty(\Omega,\F_{s_k};\R^n)$ and $\phi_k\in C_c^\infty((s_k,T))$, $0\le s_k<T$. These variations lie in $\mathcal H([0,T];\R^n)$. Bounded predictable step processes are dense in $L^2_{\mathrm{pred}}([0,T]\times\Omega;\R^n)$, and their indicators $\mathbf1_{(s,t]}$ admit $L^2(0,T)$ approximations by smooth functions supported in $(s,T)$. Hence the span of \eqref{separated-variations} is dense in $L^2_{\mathrm{pred}}$ under the $L^2(dt\otimes\P)$ norm \cite{Protter}.
\end{remark}

The following stochastic integration-by-parts identity makes the integrability assumptions explicit; cf. \cite[Lemma 4.4]{CZ91}. It requires no assumption on $Q\xi$ and extends componentwise using Euclidean inner products.
\begin{lemma}\label{integration-by-parts}
Let $\xi$ be a real-valued continuous semimartingale with decomposition
\begin{equation*}
  \xi(t)=\xi(0)+\int_0^t a(s)\,ds+M(t),
  \qquad \E|\xi(0)|^2+\E\int_0^T|a(t)|^2dt<\infty,
\end{equation*}
where $a$ is predictable and $M$ is a continuous square-integrable martingale with $M(0)=0$. Let $f$ be a real-valued adapted process with absolutely continuous paths, a predictable path derivative $\dot f$, and
\begin{equation*}
  \E|f(0)|^2+\E\int_0^T|\dot f(t)|^2dt<\infty.
\end{equation*}
Then $D\xi=a$ for Lebesgue-a.e. $t<T$, with the limit in $L^1(\Omega)$, and the following identity holds,
\begin{equation}\label{stochastic-ibp}
  \E\int_0^T\xi(t)\dot f(t)\,dt
  =\E[f(T)\xi(T)-f(0)\xi(0)]-\E\int_0^T f(t)D\xi(t)\,dt,
\end{equation}
where all terms in the identity are finite.
\end{lemma}

\begin{proof}
Martingale increments have conditional mean zero, so Lebesgue differentiation in $L^1(\Omega)$ gives $D\xi=a$ for a.e. $t<T$. Cauchy--Schwarz and Doob's inequality give $\E\sup_{t\le T}(|\xi(t)|^2+|f(t)|^2)<\infty$. Hence the boundary products and the ordinary integrals in \eqref{stochastic-ibp} are integrable.

Since $f$ has continuous finite-variation paths, $[\xi,f]=0$. Set $N(t):=\int_0^t f(s)\,dM(s)$, which is a continuous local martingale because $f$ is continuous and adapted. It\^o's product formula gives, for $0\le t\le T$,
\begin{equation*}
  f(t)\xi(t)-f(0)\xi(0)
  =\int_0^t\xi(s)\dot f(s)\,ds+\int_0^t f(s)a(s)\,ds+N(t).
\end{equation*}
Write $\xi^*:=\sup_{t\le T}|\xi(t)|$ and $f^*:=\sup_{t\le T}|f(t)|$. The product identity implies
\begin{equation*}
  \sup_{t\le T}|N(t)|
  \le 2\xi^*f^*+\int_0^T|\xi(s)\dot f(s)|\,ds
       +\int_0^T|f(s)a(s)|\,ds.
\end{equation*}
The right-hand side is integrable by the preceding estimates and Cauchy--Schwarz. For example,
\begin{equation*}
  \E\int_0^T|\xi(s)\dot f(s)|\,ds
  \le\sqrt T\,[\E(\xi^*)^2]^{1/2}
       \|\dot f\|_{L^2(dt\otimes\P)}<\infty,
\end{equation*}
and similarly for the last integral, using $a\in L^2(dt\otimes\P)$ and $f^*\in L^2(\Omega)$. Thus $\E\sup_{t\le T}|N(t)|<\infty$. For a localizing sequence $(\tau_k)$, the stopped martingales have mean zero because $N(0)=0$. Dominated convergence with bound $\sup_{t\le T}|N(t)|$ yields
\begin{equation*}
  \E N(T)=\lim_{k\to\infty}\E N(T\wedge\tau_k)=0.
\end{equation*}
Taking expectations in the product identity at $t=T$ and using $a=D\xi$ proves \eqref{stochastic-ibp}.
\end{proof}

\begin{theorem}[Stochastic Hamilton's principle]\label{stoch-Hamilton-prin}
Let $X\in\A([0,T];\mu_0,\mu_T)$. Suppose its momentum
\begin{equation*}
  P(t):=\pt_{\dot x}L_0(t,X(t),DX(t))
\end{equation*}
has a continuous adapted representative, agreeing with this expression $dt\otimes\P$-a.e., with decomposition
\begin{equation}\label{momentum-decomposition}
\left\{
\begin{aligned}
  P(t)=P(0)+\int_0^t a(s)\,ds+M(t),\\
  \E|P(0)|^2+\E\int_0^T|a(t)|^2dt<\infty,
\end{aligned}
\right.
\end{equation}
where $a$ is predictable and $M$ is a continuous square-integrable martingale with $M(0)=0$. Then $X$ is a critical point of $\mathcal S$ if and only if
\begin{equation}\label{stoch-EL}
  \D_t\left(\pt_{\dot x}L_0\right)=\pt_xL_0,
  \qquad dt\otimes\P\text{-a.e.},
\end{equation}
where the partial derivatives of $L_0$ are evaluated at $(t,X(t),DX(t))$.
\end{theorem}

The momentum assumption concerns only the candidate $X$, not its variations. Here $\D_t(\pt_{\dot x}L_0)$ means the forward mean derivative $DP$ of the whole momentum process, not the state-variable operator \eqref{total-mean-der} with velocity fixed.

\begin{proof}
Let $v\in\mathcal H([0,T];\R^n)$, and write $Z(t)=(X(t),DX(t))$ and $V(t)=(v(t),\dot v(t))$. Both belong to $L^2(dt\otimes\P)$. By the fundamental theorem of calculus in $(x,\dot x)$ and \eqref{Lagrangian-growth-assumption}, for $0<|\e|\le1$,
\begin{equation*}
  \left|\frac{L_0(t,Z(t)+\e V(t))-L_0(t,Z(t))}{\e}\right|
  \le C_1(1+|Z(t)|+|V(t)|)|V(t)|.
\end{equation*}
The right-hand side is independent of $\e$ and lies in $L^1(dt\otimes\P)$ by Cauchy--Schwarz. Thus dominated convergence on $[0,T]\times\Omega$ and \eqref{variation-deriv} give
\begin{equation}\label{eqn-18}
\begin{split}
  \frac{d}{d\e}\bigg|_{\e=0}\mathcal S[X^v_\e]
  &=\E\int_0^T\frac{\pt}{\pt\e}\bigg|_{\e=0}
    L_0(t,X^v_\e(t),DX^v_\e(t))\,dt\\
  &=\E\int_0^T\left[\pt_x L_0\cdot v(t)+\pt_{\dot x}L_0\cdot\dot v(t)\right]dt.
\end{split}
\end{equation}
All partial derivatives of $L_0$ are evaluated at $(t,X(t),DX(t))$. The growth and moment conditions place $\pt_xL_0$ and $\pt_{\dot x}L_0$ in $L^2(dt\otimes\P)$, ensuring integrability in \eqref{eqn-18}. By \eqref{momentum-decomposition}, Lemma \ref{integration-by-parts} applies componentwise to $\xi=P$ and $f=v$, giving $DP=a$. Since $v(0)=v(T)=0$,
\begin{equation}\label{eqn-19}
\begin{split}
  \frac{d}{d\e}\bigg|_{\e=0}\mathcal S[X^v_\e]
  &=\E\int_0^T\left[\pt_xL_0-DP(t)\right]\cdot v(t)\,dt\\
  &=\E\int_0^T\left[\pt_xL_0-\D_t(\pt_{\dot x}L_0)\right]\cdot v(t)\,dt.
\end{split}
\end{equation}
If \eqref{stoch-EL} holds, the right-hand side vanishes, proving sufficiency. Conversely, set
\begin{equation*}
  R(t):=\pt_xL_0(t,X(t),DX(t))-a(t).
\end{equation*}
Predictable versions of $DX$ and $a$, with \eqref{Lagrangian-growth-assumption}, \eqref{moment-assumptions} and \eqref{momentum-decomposition}, give $R\in L^2_{\mathrm{pred}}([0,T]\times\Omega;\R^n)$. For critical $X$, \eqref{eqn-19} vanishes on \eqref{separated-variations}; their density (Remark \ref{remark-separated-variations}) makes $R$ orthogonal to $L^2_{\mathrm{pred}}$. Approximating $R$ in $L^2$ by these tests gives $R=0$, $dt\otimes\P$-a.e., proving necessity.
\end{proof}

\begin{remark}
  (i) If one restricts the variations to deterministic $v\in H^1_0([0,T];\R^n)$, the same calculation yields only
  \begin{equation*}
    \E\left[\pt_xL_0-\D_t(\pt_{\dot x}L_0)\right]=0,
    \qquad dt\text{-a.e.},
  \end{equation*}
  rather than the pathwise equation \eqref{stoch-EL}. This corrects the result in \cite[Theorem 7.16]{HZ23}, where the displayed variational class consists of deterministic Cameron--Martin variations in $H_0^1([0,T];\R^n)$, while the stochastic Euler--Lagrange equation is stated in the stronger $dt\otimes \P$-a.e. sense. 


  (ii) These equations should be distinguished from Nelson's time-symmetric stochastic Newton law \cite{Nelson1966}. For example, with the free Lagrangian $L_0(x,\dot x)=\frac12|\dot x|^2$, Brownian motion has $DB=0$ and solves both \eqref{stoch-EL} and the corresponding Hamiltonian system \eqref{canonical-stoch-Hamilton-eqns}; it does not, in general, solve Nelson's free stochastic Newton equation. See also \cite{Shucker1980} for the zero-potential case.

  (iii) Lassalle and Zambrini formulated a weak stochastic Euler--Lagrange equation in \cite{LZ16}, where ``weak'' refers to a stochastic-integral formulation. Equation \eqref{stoch-EL}, by contrast, is written in terms of mean derivatives and therefore contains no martingale differential.
\end{remark}

\subsection{Stochastic Maupertuis's Principle}\label{app-3}

We introduce a deterministic time scale, motivated by the classical free-time Maupertuis principle \cite[Section 3.8]{AM78}. Related stochastic formulations appear in \cite[Appendix C]{HZ23}; here the admissible pairs and variations are specified independently, and the results follow directly from the preceding subsection.

Unless otherwise stated, $L_0$ is autonomous throughout this subsection: $L_0(t,x,\dot x)=L_0(x,\dot x)$. We assume the growth and differentiability conditions \eqref{Lagrangian-growth-assumption}, with the time argument omitted, and define the classical energy by
\begin{equation}\label{Maupertuis-classical-energy}
  E_0(x,\dot x)=\pt_{\dot x}L_0(x,\dot x)\cdot\dot x-L_0(x,\dot x).
\end{equation}
Fix $e\in\R$. No regular-value assumption on $e$ or nondegeneracy of the velocity Hessian is needed. Autonomy is essential for constant mean energy; time dependence is discussed after the proof.

Let $s\in[0,T]$ denote reference time. We admit deterministic, absolutely continuous clocks
\begin{equation}\label{Maupertuis-clocks}
  \tau(s)=\tau(0)+\int_0^s r(u)\,du,\qquad
  r,r^{-1}\in L^\infty(0,T),\quad r>0.
\end{equation}
Thus $r=\tau'$ almost everywhere and $\tau$ is bi-Lipschitz and strictly increasing. Both endpoint times are free. We use the physical-time filtration
\begin{equation*}
  \mathcal G_t=\F_{\tau^{-1}(t)},\qquad \tau(0)\le t\le\tau(T).
\end{equation*}
Mean derivatives of a process indexed by $t$ are taken with respect to $\mathcal G$; those of a reference process indexed by $s$ are taken with respect to $\F$.

\begin{definition}[Admissible pairs and scaled variations]\label{variation-2}
A pair $(X,\tau)$ is admissible if $\tau$ satisfies \eqref{Maupertuis-clocks}, $X$ is defined on $[\tau(0),\tau(T)]$, and its pullback
\begin{equation*}
  Y(s):=X(\tau(s))
\end{equation*}
belongs to $\A([0,T];\mu_0,\mu_T)$. Denote this class by $\widehat{\mathcal A}([0,T];\mu_0,\mu_T)$.
For $v\in\mathcal H([0,T];\R^n)$, a deterministic $\eta\in L^\infty(0,T)$, and $\alpha\in\R$, set
\begin{equation}\label{Maupertuis-pair-variation}
\begin{gathered}
  Y_\e=Y+\e v,\qquad r_\e=r+\e\eta,\\
  \tau_\e(s)=\tau(s)+\e\left(\alpha+\int_0^s\eta(u)\,du\right),\\
  X_\e(t)=Y_\e(\tau_\e^{-1}(t)),\qquad
  \tau_\e(0)\le t\le\tau_\e(T).
\end{gathered}
\end{equation}
For sufficiently small $|\e|$, $(X_\e,\tau_\e)$ are admissible pairs. We call them scaled variations of $(X,\tau)$.
\end{definition}

In particular, $X$ is an It\^o process for $\mathcal G$. The covariance condition implicit in this definition is
\begin{equation}\label{Maupertuis-physical-covariance}
  QX(t)=\frac{1}{r(\tau^{-1}(t))}I_n,
  \qquad dt\otimes\P\text{-a.e.}
\end{equation}
Unlike $\A([0,T];\mu_0,\mu_T)$, $\widehat{\mathcal A}([0,T];\mu_0,\mu_T)$ fixes covariance rate $I_n$ only for $Y=X\circ\tau$; the diffusion matrix of $X$ changes with $\tau$. We verify this below. The variations preserve the endpoint random variables, and hence their laws $\mu_0,\mu_T$:
\begin{equation*}
  X_\e(\tau_\e(0))=Y(0),\qquad
  X_\e(\tau_\e(T))=Y(T)\quad\P\text{-a.s.}
\end{equation*}
No energy constraint is imposed on the admissible pairs or their variations.

On this class, define the augmented free-time action in its physical-time form by
\begin{equation}\label{eq:02}
  \mathcal I[X,\tau]
  :=\E\int_{\tau(0)}^{\tau(T)}
  \left[L_0\left(X(t),DX(t)\right)+e\right]dt.
\end{equation}
A pair $(X,\tau)\in\widehat{\mathcal A}([0,T];\mu_0,\mu_T)$ is a critical point of $\mathcal I$ if
\begin{equation*}
  \frac{d}{d\e}\bigg|_{\e=0}\mathcal I[X_\e,\tau_\e]=0
\end{equation*}
for every variation \eqref{Maupertuis-pair-variation}. The spatial direction $v$ and rate direction $\eta$ are independent. Here $v$ varies the pullback $Y$ at fixed reference time, not $X$ at fixed physical time.

\begin{theorem}[Stochastic Maupertuis's principle]\label{stoch-Maupertuis-prin}
Let $(X,\tau)\in\widehat{\mathcal A}([0,T];\mu_0,\mu_T)$, and write $Y=X\circ\tau$ and $r=\tau'$. Assume that the reference momentum
\begin{equation*}
  \widehat{P}(s):=\pt_{\dot x}L_0\left(Y(s),\frac{DY(s)}{r(s)}\right)
\end{equation*}
has a continuous adapted representative satisfying \eqref{momentum-decomposition} in reference time, with $P$ replaced by $\widehat{P}$ and $t$ by $s$. Then $(X,\tau)$ is critical for $\mathcal I$ if and only if $X$ satisfies the stochastic Euler--Lagrange equation \eqref{stoch-EL} in physical time and has mean energy $e$, namely,
\begin{equation}\label{mean-energy}
  \E E_0(X(t),DX(t))=e,\qquad
  \text{for Lebesgue-a.e. }t\in[\tau(0),\tau(T)].
\end{equation}
\end{theorem}

At $\tau=\id_{[0,T]}$, this criterion is the stochastic Euler--Lagrange equation \eqref{stoch-EL} together with the additional mean-energy condition.

\begin{proof}
Write $dY=b\,ds+\sigma\,dB$, where $b=DY$ and $\sigma\sigma^\top=I_n$. The deterministic time change gives a Brownian motion for $\mathcal G$, starting from zero at time $\tau(0)$:
\begin{equation*}
  W(t)=\int_0^{\tau^{-1}(t)}\sqrt{r(s)}\,dB(s),
  \qquad \tau(0)\le t\le\tau(T).
\end{equation*}
Indeed, $W(\tau(0)+u)$ is a standard Brownian motion in the shifted time $u$. The change of variables for stochastic integrals then yields
\begin{equation*}
  dX(t)=\frac{b(\tau^{-1}(t))}{r(\tau^{-1}(t))}\,dt
  +\frac{\sigma(\tau^{-1}(t))}{\sqrt{r(\tau^{-1}(t))}}\,dW(t).
\end{equation*}
These coefficients satisfy the analogue of \eqref{Ito-integrability} on $[\tau(0),\tau(T)]$, since $r$ and $r^{-1}$ are bounded. Thus
\begin{equation}\label{Maupertuis-time-change-identities}
  DX(\tau(s))=\frac{b(s)}{r(s)},\qquad
  QX(\tau(s))=\frac{1}{r(s)}I_n,
  \qquad ds\otimes\P\text{-a.e.}
\end{equation}
This construction also applies to $(X_\e,\tau_\e)$ on its own physical-time filtration. By \eqref{Maupertuis-time-change-identities} and $t=\tau(s)$,
\begin{equation*}
  \mathcal I[X,\tau]
  =\E\int_0^T r(s)
  \left[L_0\left(Y(s),\frac{DY(s)}{r(s)}\right)+e\right]ds.
\end{equation*}
The same identity holds for every sufficiently small $\e$, with $(Y,r)$ replaced by $(Y_\e,r_\e)$.

The growth assumption and the bounds on $r,r^{-1}$ make this integral finite. For a fixed direction $(v,\eta)$, $r_\e$ and $r_\e^{-1}$ are uniformly bounded for small $|\e|$. Moreover,
\begin{equation*}
  \left|\frac{(b+\e\dot v)/(r+\e\eta)-b/r}{\e}\right|
  \le C_{r,\eta}(|b|+|\dot v|).
\end{equation*}
The growth bound \eqref{Lagrangian-growth-assumption}, together with the fundamental theorem of calculus in $(x,\dot x)$, therefore controls the integrand's difference quotient by
\begin{equation*}
  C_{r,\eta,e}\left(1+|Y|+|b|+|v|+|\dot v|\right)^2.
\end{equation*}
This is integrable on $[0,T]\times\Omega$ by \eqref{Ito-integrability}, \eqref{moment-assumptions} and \eqref{variation-sup-estimate}. We may therefore differentiate under the expectation and integral by dominated convergence.

In the next display, $L_0$ and its partial derivatives are evaluated at $(Y,b/r)$. The product and chain rules yield
\begin{equation*}
\begin{split}
  \frac{d}{d\e}\bigg|_{\e=0}\mathcal I[X_\e,\tau_\e]
  &=\E\int_0^T\bigg\{\eta(L_0+e)+r\,\pt_xL_0\cdot v +r\,\pt_{\dot x}L_0\cdot
        \left(\frac{\dot v}{r}-\frac{b\eta}{r^2}\right)\bigg\}ds\\
  &=\E\int_0^T\left\{r\,\pt_xL_0\cdot v+\widehat{P}\cdot\dot v
          +\left[e-E_0(Y,b/r)\right]\eta\right\}ds.
\end{split}
\end{equation*}
Here the last equality uses \eqref{Maupertuis-classical-energy}. By \eqref{momentum-decomposition}, using Lemma \ref{integration-by-parts},
\begin{equation*}
  \E\int_0^T \widehat{P}\cdot\dot v\,ds
  =\E[\widehat{P}(T)\cdot v(T)-\widehat{P}(0)\cdot v(0)]
     -\E\int_0^T D\widehat{P}\cdot v\,ds.
\end{equation*}
The boundary term vanishes because $v(0)=v(T)=0$. Only the unperturbed momentum requires the stated semimartingale regularity. Since $\eta$ is deterministic, we obtain
\begin{equation}\label{Maupertuis-first-variation-ibp}
\begin{split}
  \frac{d}{d\e}\bigg|_{\e=0}\mathcal I[X_\e,\tau_\e]
  &=\E\int_0^T\left[r\,\pt_xL_0-D\widehat{P}\right]\cdot v\,ds\\
  &\quad+\int_0^T\left[e-\E E_0(Y,b/r)\right]\eta\,ds.
\end{split}
\end{equation}

Criticality means that \eqref{Maupertuis-first-variation-ibp} vanishes for every independent choice of $(v,\eta)$. The residual $r\,\pt_xL_0-D\widehat{P}$ belongs to $L^2_{\mathrm{pred}}$, so the density in Remark \ref{remark-separated-variations} gives
\begin{equation}\label{Maupertuis-scaled-EL}
  D\widehat{P}(s)=r(s)\pt_xL_0\left(Y(s),\frac{DY(s)}{r(s)}\right),
  \qquad ds\otimes\P\text{-a.e.}
\end{equation}
Independently, taking $\eta=\operatorname{sgn}(e-\E E_0(Y,b/r))$ yields
\begin{equation}\label{mean-e}
  \E E_0\left(Y(s),\frac{DY(s)}{r(s)}\right)=e,
  \qquad ds\text{-a.e.}
\end{equation}
Conversely, \eqref{Maupertuis-scaled-EL} and \eqref{mean-e} make the first variation vanish for all $(v,\eta,\alpha)$.

Finally, \eqref{Maupertuis-time-change-identities} and the definition of $\widehat{P}$ give
\begin{equation*}
  \pt_{\dot x}L_0(X(t),DX(t))=\widehat{P}(\tau^{-1}(t)),
  \qquad dt\otimes\P\text{-a.e.}
\end{equation*}
We use the continuous process on the right as the representative of the physical momentum. 
Consequently, on $[\tau(0),\tau(T)]$,
\begin{equation*}
  D\big[\pt_{\dot x}L_0(X,DX)\big](t)
  =\frac{(D\widehat{P})(\tau^{-1}(t))}{r(\tau^{-1}(t))},
  \qquad dt\otimes\P\text{-a.e.}
\end{equation*}
The left derivative uses physical time and $\mathcal G$; $D\widehat{P}$ uses reference time and $\F$. With \eqref{Maupertuis-time-change-identities}, this identifies \eqref{Maupertuis-scaled-EL} with \eqref{stoch-EL}. The deterministic bi-Lipschitz clock also carries \eqref{mean-e} to \eqref{mean-energy}.
\end{proof}

\begin{remark}[Time-dependent Lagrangian]\label{remark-Maupertuis-interpretation}
Here we fix $\tau(0)$ and leave $\tau(T)$ free, so the variations \eqref{Maupertuis-pair-variation} have $\alpha=0$. For a time-dependent Lagrangian, use $L_0(t,X(t),DX(t))+e$ in \eqref{eq:02} and set $E_0(t,x,\dot x):=\pt_{\dot x}L_0(t,x,\dot x)\cdot\dot x-L_0(t,x,\dot x)$. The reference momentum is evaluated at $(\tau(s),Y(s),DY(s)/r(s))$. Assume $L_0\in C^1(I\times\R^n\times\R^n)$, where $I$ is an open interval containing $[\tau(0),\tau(T)]$, and assume \eqref{Lagrangian-growth-assumption} uniformly for $t\in I$, together with
\begin{equation*}
  |\pt_tL_0(t,x,\dot x)|\le C(1+|x|^2+|\dot x|^2).
\end{equation*}
For a critical pair, the mean energy $\overline E(t):=\E E_0(t,X(t),DX(t))$ has an absolutely continuous representative satisfying
\begin{equation}\label{Maupertuis-time-dependent-energy}
\begin{gathered}
  \frac{d}{dt}\overline E(t)=-\E\pt_tL_0(t,X(t),DX(t)),
  \qquad dt\text{-a.e.},\\
  \overline E(\tau(T))=e.
\end{gathered}
\end{equation}
The terminal value refers to this representative and follows from the free final time; no initial-energy condition is imposed. The proof is left to Problem \ref{prob3-2}.
\end{remark}

\begin{remark}[Random time scales]\label{remark-Maupertuis-random-clock}
In reference time, the autonomous problem allows positive predictable $r$ with $r,r^{-1}\in L^\infty(ds\otimes\P)$ and arbitrary bounded predictable directions $\eta$. The same differentiation and domination apply, assuming momentum regularity for $\pt_{\dot x}L_0(Y,DY/r)$ at the candidate pair. The scale term in \eqref{Maupertuis-first-variation-ibp} becomes
\begin{equation*}
  \E\int_0^T\left[e-E_0(Y,DY/r)\right]\eta\,ds.
\end{equation*}
Its residual is predictable and integrable. Choosing $\eta=\operatorname{sgn}(e-E_0(Y,DY/r))$ shows that stationarity imposes
\begin{equation*}
  E_0(Y,DY/r)=e,\qquad ds\otimes\P\text{-a.e.},
\end{equation*}
instead of \eqref{mean-e}. This stronger constraint need not hold for stochastic Euler--Lagrange solutions. The statement concerns the reference-time functional, not equivalence to a physical-time problem under arbitrary random clocks.
\end{remark}

The deterministic-scale reduction for mechanical Lagrangians is developed in Section \ref{sec:mechanical-Jacobi}.

\section{Stochastic Hamiltonian Mechanics}

At fixed covariance, independent position and momentum variations yield stochastic Hamilton's equations, equivalent to the Euler--Lagrange equation for hyperregular Lagrangians. Stochastic characteristics and momentum translations then connect them to the second-order Hamilton--Jacobi equation.

\subsection{Stochastic Variational Principles on the Phase Space}

As in classical mechanics, the Lagrangian formulation is related to a Hamiltonian one by the Legendre transform. For each fixed $t$, the Legendre map is
\begin{equation*}
  (x,\dot x)\longmapsto
  \bigl(x,p=\pt_{\dot x}L_0(t,x,\dot x)\bigr).
\end{equation*}
A $C^2$ Lagrangian is called hyperregular if these maps are diffeomorphisms. Its Hamiltonian is then
\begin{equation*}
  H_0(x,p,t)=p\cdot\dot x-L_0(t,x,\dot x),
\end{equation*}
where $\dot x$ is determined by the inverse Legendre map. We retain the argument order $(t,x,\dot x)$ for $L_0$ and $(x,p,t)$ for $H_0$.

Hyperregularity is unnecessary for a direct Hamiltonian formulation. Here $H_0:\R^{2n}\times[0,T]\to\R$ is continuous and continuously differentiable in $(x,p)$, with jointly continuous first derivatives satisfying
\begin{equation}\label{Hamiltonian-growth-assumption}
\begin{gathered}
  |H_0(0,0,t)|\le C,\\
  |\pt_xH_0(x,p,t)|+|\pt_pH_0(x,p,t)|
  \le C(1+|x|+|p|)
\end{gathered}
\end{equation}
for all $(x,p,t)$.

We consider pairs $(X,P)$ with $X\in\A([0,T];\mu_0,\mu_T)$ and $P\in L^2_{\mathrm{pred}}([0,T]\times\Omega;\R^n)$. No endpoint values or quadratic variation are prescribed for $P$. Define the phase-space action by
\begin{equation}\label{action-phase}
  \overline{\mathcal S}[X,P]
  :=\E\int_0^T\left[P(t)\cdot DX(t)-H_0(X(t),P(t),t)\right]dt.
\end{equation}
It is finite: Cauchy--Schwarz controls $P\cdot DX$, while \eqref{Hamiltonian-growth-assumption} gives $|H_0(x,p,t)|\le C'(1+|x|^2+|p|^2)$, and \eqref{moment-assumptions} controls $X$.

\begin{definition}[Phase-space variations]\label{phase-space-variations}
For independent directions $v\in\mathcal H([0,T];\R^n)$ and $w\in L^2_{\mathrm{pred}}([0,T]\times\Omega;\R^n)$, set
\begin{equation*}
  X_\e=X+\e v,\qquad P_\e=P+\e w.
\end{equation*}
These variations remain in the preceding class of pairs. A pair $(X,P)$ is critical for $\overline{\mathcal S}$ if
\begin{equation*}
  \frac{d}{d\e}\bigg|_{\e=0}
  \overline{\mathcal S}[X_\e,P_\e]=0
\end{equation*}
for every such $(v,w)$.
\end{definition}

\begin{samepage}
\begin{theorem}[Stochastic Hamilton's equations]\label{canonical-reduction}
Let $X\in\A([0,T];\mu_0,\mu_T)$ and $P\in L^2_{\mathrm{pred}}$. Suppose $P$ has a continuous adapted representative satisfying \eqref{momentum-decomposition}.
\setitemindent{ii)}
\begin{itemize}
\item[i)] The pair $(X,P)$ is critical for \eqref{action-phase} if and only if
\begin{equation}\label{canonical-stoch-Hamilton-eqns}
\left\{
\begin{aligned}
  DX(t)&=\pt_pH_0(X(t),P(t),t),\\
  DP(t)&=-\pt_xH_0(X(t),P(t),t),
\end{aligned}
\right.
\qquad dt\otimes\P\text{-a.e.}
\end{equation}
\item[ii)] If $H_0$ is the Legendre transform of a $C^2$ hyperregular $L_0$ satisfying \eqref{Lagrangian-growth-assumption}, these equations are equivalent to \eqref{stoch-EL} under the correspondence
\begin{equation*}
  P(t)=\pt_{\dot x}L_0(t,X(t),DX(t)),
  \qquad dt\otimes\P\text{-a.e.}
\end{equation*}
Along this correspondence, $\overline{\mathcal S}[X,P]=\mathcal S[X]$.
\end{itemize}
\end{theorem}
\end{samepage}

Here $DP$ is the forward mean derivative of the momentum process; no representation of $P$ as a function of $(t,X(t))$ is assumed. If $P(t)=p(t,X(t))$ for a $C^{1,2}$ field $p$ satisfying \eqref{composite-mean-integrability} componentwise, then \eqref{total-mean-der} gives
\begin{equation*}
  DP(t)=\big(\pt_t p+DX(t)\cdot\nabla p
  +\tfrac{1}{2}\Delta p\big)(t,X(t)).
\end{equation*}

\begin{proof}
i) Fix $(v,w)$ as in Definition \ref{phase-space-variations}, and put $U=(X,P)$ and $h=(v,w)$. Both belong to $L^2(dt\otimes\P)$. For $0<|\e|\le1$, \eqref{Hamiltonian-growth-assumption} and the fundamental theorem of calculus give
\begin{equation*}
  \left|\frac{H_0(U+\e h,t)-H_0(U,t)}{\e}\right|
  \le C_1(1+|U|+|h|)|h|.
\end{equation*}
The right-hand side is integrable by Cauchy--Schwarz. Moreover, $DX_\e=DX+\e\dot v$ gives
\begin{equation*}
  \frac{P_\e\cdot DX_\e-P\cdot DX}{\e}
  =w\cdot DX+P\cdot\dot v+\e w\cdot\dot v,
\end{equation*}
whose terms have an integrable bound independent of $\e$. Dominated convergence therefore justifies differentiation under the expectation and integral. Lemma \ref{integration-by-parts}, applied componentwise to $P$ and $v$, then gives
\begin{equation}\label{phase-space-first-variation}
\begin{split}
  \frac{d}{d\e}\bigg|_{\e=0}\overline{\mathcal S}[X_\e,P_\e]
  &=\E\int_0^T\left[
    w\cdot(DX-\pt_pH_0)+P\cdot\dot v-\pt_xH_0\cdot v
    \right]dt\\
  &=\E\int_0^T\left[
    (DX-\pt_pH_0)\cdot w-(DP+\pt_xH_0)\cdot v
    \right]dt.
\end{split}
\end{equation}
The derivatives of $H_0$ are evaluated at $(X(t),P(t),t)$; $v(0)=v(T)=0$ removes the boundary term. By \eqref{Hamiltonian-growth-assumption}, \eqref{moment-assumptions} and the decomposition of $P$, both residuals lie in $L^2_{\mathrm{pred}}$. Independent variations, with arbitrary $w$ and the dense class of Remark \ref{remark-separated-variations} for $v$, make criticality equivalent to their vanishing, hence to \eqref{canonical-stoch-Hamilton-eqns}.

ii) Under the Legendre relation $p=\pt_{\dot x}L_0(t,x,\dot x)$, differentiation gives
\begin{equation*}
  \pt_pH_0(x,p,t)=\dot x,\qquad
  \pt_xH_0(x,p,t)=-\pt_xL_0(t,x,\dot x).
\end{equation*}
These identities identify \eqref{canonical-stoch-Hamilton-eqns} with \eqref{stoch-EL}, with the same momentum regularity as in Theorem \ref{stoch-Hamilton-prin}. The defining identity for $H_0$ gives equality of the two actions. The detailed verification is left to Problem \ref{prob3-3}.
\end{proof}

\subsection{Second-Order Hamilton--Jacobi Equation}\label{sec:2HJ}

To pass from the stochastic Hamilton's equations to a PDE on the whole space, we consider characteristics starting from every space-time point, rather than fixing their terminal distribution.

The $C^{1,3}$ regularity used below includes the existence and joint continuity of the mixed derivative $\pt_t\nabla u$ for the function $u$ under consideration.

\begin{theorem}[Characteristics]\label{HJ-characteristics}
Assume that $H_0:\R^{2n}\times[0,T]\to\R$ is continuous and twice continuously differentiable in $(x,p)$, with jointly continuous derivatives, and satisfies \eqref{Hamiltonian-growth-assumption}. Let $S\in C^{1,3}([0,T]\times\R^n)$, and suppose that
\begin{equation}\label{HJ-regularity}
\begin{gathered}
  \|\nabla^2S(t,x)\|_{\mathrm{op}}\le C,\\
  |\pt_t\nabla S(t,x)|+|\nabla\Delta S(t,x)|\le C(1+|x|).
\end{gathered}
\end{equation}
For each $(s,x)\in[0,T)\times\R^n$, let $X$ be the solution on $[s,T]$ of
\begin{equation}\label{cond-1}
  dX(t)=\pt_pH_0\big(X(t),\nabla S(t,X(t)),t\big)\,dt
    +dB(t),\qquad X(s)=x,
\end{equation}
where $B$ denotes an $n$-dimensional Brownian motion. Then $(X,P)$, with $P(t)=\nabla S(t,X(t))$, satisfies \eqref{canonical-stoch-Hamilton-eqns} for every such initial point if and only if
\begin{equation}\label{HJB-4}
  \pt_tS(t,x)+H_0(x,\nabla S(t,x),t)
    +\tfrac{1}{2}\Delta S(t,x)=f(t)
\end{equation}
for some continuous function $f:[0,T]\to\R$.
\end{theorem}

\begin{proof}
The bound on $\nabla^2S$ implies that $\nabla S$ has at most linear growth, uniformly in time. The drift in \eqref{cond-1} is locally Lipschitz in $x$ and has at most linear growth. Thus each initial point determines a unique nonexplosive strong solution, with $\E\sup_{s\le t\le T}|X(t)|^2<\infty$.

It\^o's formula applied componentwise to $\nabla S$ gives
\begin{equation*}
  dP(t)
  =\left[\pt_t\nabla S+\nabla^2S\,\pt_pH_0(x,\nabla S,t)
    +\tfrac{1}{2}\nabla\Delta S\right]_{x=X(t)}dt+\nabla^2S(t,X(t))\,dB(t).
\end{equation*}
The drift is square integrable and the martingale is square integrable, by \eqref{HJ-regularity} and the preceding moment bound. Consequently, $P$ satisfies the counterpart of \eqref{momentum-decomposition} on $[s,T]$, and
\begin{equation}\label{HJ-characteristic-identity}
  DP(t)+\pt_xH_0(X(t),P(t),t)
  =\left[\nabla\big(\pt_tS+H_0(x,\nabla S,t)
       +\tfrac{1}{2}\Delta S\big)\right]_{x=X(t)}.
\end{equation}
Here $\pt_xH_0$ differentiates the first argument while keeping $p$ fixed; the outer gradient on the right differentiates the composite function.

Equation \eqref{HJB-4} makes the right-hand side vanish, proving the Hamilton equations. Conversely, if these hold for every initial point $(s,x)$, the right-hand side of \eqref{HJ-characteristic-identity} vanishes $dt\otimes\P$-a.e. Pathwise continuity extends this to every time, including $t=s$. Thus the spatial gradient vanishes everywhere, also at $T$ by continuity. Connectedness of $\R^n$ gives \eqref{HJB-4} with continuous $f$.
\end{proof}

Replacing $S(t,x)$ by $S(t,x)+\int_t^T f(u)\,du$ leaves $P$, the characteristics, and a prescribed terminal value $S(T,\cdot)=S_T$ unchanged, and reduces \eqref{HJB-4} to
\begin{equation}\label{HJ-normalized}
  \pt_tS+H_0(x,\nabla S,t)+\tfrac{1}{2}\Delta S=0.
\end{equation}
Prescribing $S(T,\cdot)=S_T$ gives a final-boundary problem for the second-order Hamilton--Jacobi equation \eqref{HJ-normalized}, a deformation of the classical equation. By Theorem \ref{canonical-reduction}, a characteristic with the prescribed endpoint laws gives a critical phase-space pair; arbitrary two-endpoint laws need not be attained.

In the following action and transformation formulas, we retain the assumptions on $H_0$ from Theorem \ref{HJ-characteristics}. For a smooth momentum field $p(t,x)$, define the \emph{second-order Hamiltonian} associated with $H_0$ by
\begin{equation}\label{HJ-deformed-Hamiltonian}
  H[p](t,x)
  :=H_0(x,p(t,x),t)+\tfrac{1}{2}\operatorname{div}_x p(t,x).
\end{equation}
The brackets emphasize that this expression depends on the spatial derivatives of the field $p$, not just on its value at a phase-space point. In particular, \eqref{HJ-normalized} reads $\pt_tS+H[\nabla S]=0$.

To relate \eqref{HJ-deformed-Hamiltonian} to the action, let $X\in\A([0,T];\mu_0,\mu_T)$ and $P(t)=p(t,X(t))$, where $p\in C^{1,2}$ and $\nabla p$ is bounded. Then $p$ has at most linear growth, and the It\^o--Stratonovich conversion gives
\begin{equation*}
\begin{split}
  \overline{\mathcal S}[X,P]
  &=\E\left[\int_0^T P(t)\circ dX(t)
      -\int_0^T H[p](t,X(t))\,dt\right].
\end{split}
\end{equation*}
Here $P\circ dX:=\sum_{i=1}^n P_i\circ dX^i$, with correction $\frac{1}{2}\int_0^T\operatorname{div}_x p(t,X(t))\,dt$. The It\^o martingale integral has zero expectation since $P\in L^2(dt\otimes\P)$. This additional field representation does not restrict the independent momentum variations of Definition \ref{phase-space-variations}.

General canonical transformations need not preserve $QX=I_n$ or our variation space. We consider momentum translations, which leave position unchanged. Let $F\in C^{1,3}([0,T]\times\R^n)$ satisfy \eqref{HJ-regularity} in place of $S$. Set
\begin{equation}\label{HJ-momentum-translation}
\begin{split}
  Q(t)&=P(t)-\nabla F(t,X(t)),\\
  K_0(x,q,t)&=H_0(x,q+\nabla F(t,x),t)
       +\pt_tF(t,x)+\tfrac{1}{2}\Delta F(t,x).
\end{split}
\end{equation}
This is the type-two transformation generated by $G_2(x,q,t)=x\cdot q+F(t,x)$ in the classical terminology \cite{AKN06}, with an It\^o correction in the transformed Hamiltonian.

The admissible class is preserved since $X$ is unchanged and $\nabla F(t,X(t))\in L^2(dt\otimes\P)$. The bounds give \eqref{Hamiltonian-growth-assumption} for $K_0$, while It\^o's formula for $\nabla F$ transfers \eqref{momentum-decomposition} from $P$ to $Q$. For the respective actions, It\^o's formula for $F$ yields
\begin{equation}\label{HJ-action-translation}
  \overline{\mathcal S}_{H_0}[X,P]
  =\overline{\mathcal S}_{K_0}[X,Q]
     +\E\big[F(T,X(T))-F(0,X(0))\big].
\end{equation}
All terms are integrable: $F$ and $\pt_tF$ have at most quadratic growth, $\nabla F$ at most linear growth, and $\nabla^2F$ is bounded.

The boundary expectation in \eqref{HJ-action-translation} is fixed by the endpoint laws. Moreover, a variation direction $(v,w)$ is transformed into
\begin{equation*}
  (v,w-\nabla^2F(t,X(t))v).
\end{equation*}
Boundedness of $\nabla^2F$ makes this a bijection of the direction space in Definition \ref{phase-space-variations}. Although the transformed variation need not be linear in $\e$, its difference quotient converges in $L^2$ by dominated convergence. The growth estimates used for \eqref{phase-space-first-variation}, applied to \eqref{HJ-action-translation}, then preserve criticality and, under the momentum assumption, the Hamilton equations.

When $P=p(t,X)$ as above, put $q=p-\nabla F$ and define $K[q]$ as in \eqref{HJ-deformed-Hamiltonian}. Then $K[q]=H[p]+\pt_tF$, so the Stratonovich chain rule gives the generating identity
\begin{equation*}
  P\circ dX-H[p](t,X)\,dt
  =Q\circ dX-K[q](t,X)\,dt+dF(t,X).
\end{equation*}
In particular, choosing $F=S$ with \eqref{HJ-normalized} gives
\begin{equation*}
  K_0(x,q,t)=H_0(x,q+\nabla S,t)-H_0(x,\nabla S,t),
  \qquad K_0(x,0,t)=0.
\end{equation*}
Thus the graph $P=\nabla S(t,X)$ becomes the zero-momentum graph $Q=0$. On that graph the transformed Hamilton equations give $DX=\pt_pH_0(X,\nabla S(t,X),t)$ and $DQ=0$; one does not require $K_0$ to vanish for all $q$.

For example, if $V\in C^2(\R^n)$ has a gradient of at most linear growth and
\begin{equation*}
  H_0(x,p,t)=\tfrac12|p|^2-p\cdot\nabla F(t,x)+V(x),
\end{equation*}
then \eqref{HJ-momentum-translation} gives
\begin{equation*}
  K_0(x,q,t)=\tfrac12|q|^2+\pt_tF-\tfrac12|\nabla F|^2
       +V+\tfrac{1}{2}\Delta F.
\end{equation*}
It reduces to the free Hamiltonian $\tfrac12|q|^2$ if
\begin{equation*}
  \pt_tF-\tfrac12|\nabla F|^2+\tfrac{1}{2}\Delta F+V=0.
\end{equation*}

\begin{samepage}
\begin{remark}[A martingale associated with time translation]\label{remark-HJ-martingale}
Suppose $H_0$ is independent of time, $S$ satisfies \eqref{HJ-normalized} under the assumptions of Theorem \ref{HJ-characteristics}, and additionally $\pt_tS\in C^{1,2}$. Then, along each characteristic, the process
\begin{equation*}
  H[\nabla S](t,X(t))=-\pt_tS(t,X(t))
\end{equation*}
is a martingale. Indeed, differentiating \eqref{HJ-normalized} in time shows that $\pt_tS$ has zero drift along \eqref{cond-1}. Its initial value is deterministic, and its stochastic integral is square integrable by \eqref{HJ-regularity} and the moment bound for $X$. This illustrates the relation between symmetries and martingales developed in \cite{HZ23}.
\end{remark}
\end{samepage}

\section{Mechanical Lagrangians}\label{sec:mechanical}

We specialize the preceding variational principles to the mechanical Lagrangian
\begin{equation}\label{Lagrangian}
  L_0(t,x,\dot x)=\tfrac12|\dot x|^2-V(t,x).
\end{equation}
Assume that $V$ is continuously differentiable in $x$, that $V$ and $\nabla V$ are jointly continuous, and that
\begin{equation}\label{mechanical-potential-growth}
  |\nabla V(t,x)|\le C(1+|x|),\qquad (t,x)\in[0,T]\times\R^n.
\end{equation}
Then $V$ has at most quadratic growth, uniformly in time, so the growth assumptions on $L_0$ in Section \ref{sec-7-2} hold. The Legendre transform gives
\begin{equation*}
  p=\dot x,\qquad H_0(x,p,t)=\tfrac12|p|^2+V(t,x).
\end{equation*}
We first describe the dynamics with fixed covariance rate and then, for autonomous potentials, reduce the action over deterministic clocks.

\subsection{Stochastic Newton and second-order Hamilton--Jacobi Equations}

Let $X\in\A([0,T];\mu_0,\mu_T)$, so $QX=I_n$, and assume that $P:=DX$ has a continuous adapted representative satisfying \eqref{momentum-decomposition}. Since $\pt_{\dot x}L_0=\dot x$ and $\pt_xL_0=-\nabla V$, Theorem \ref{stoch-Hamilton-prin} shows that $X$ is critical for $\mathcal S$ if and only if it satisfies the \emph{stochastic Newton equation}
\begin{equation}\label{mechanical-Newton}
  D(DX(t))=-\nabla V(t,X(t)),\qquad dt\otimes\P\text{-a.e.}
\end{equation}
Equivalently, the Hamilton equations are $DX=P$ and $DP=-\nabla V(t,X)$.

The passage to a scalar second-order Hamilton--Jacobi equation uses the gradient drift
\begin{equation}\label{SEL-HJB-cond}
  DX(t)=\nabla S(t,X(t)).
\end{equation}
The mechanical Hamiltonian $H_0(x,p,t)=\frac12|p|^2+V(t,x)$ satisfies \eqref{Hamiltonian-growth-assumption} by \eqref{mechanical-potential-growth}. Theorem \ref{HJ-characteristics} thus characterizes gradient-drift solutions of the stochastic Newton equation as follows.

\begin{corollary}[Gradient drifts]\label{SEL-HJB}
In addition to the standing assumptions on $V$, suppose that $V$ is twice continuously differentiable in $x$, with jointly continuous second derivatives. Let $S\in C^{1,3}([0,T]\times\R^n)$ satisfy \eqref{HJ-regularity}. For each $(s,x)\in[0,T)\times\R^n$, let $X$ be the solution on $[s,T]$ of
\begin{equation*}
  dX(t)=\nabla S(t,X(t))\,dt+dB(t),\qquad X(s)=x,
\end{equation*}
where $B$ is an $n$-dimensional Brownian motion. These processes satisfy \eqref{mechanical-Newton} for every initial point if and only if
\begin{equation}\label{HJB-2}
  \pt_tS+\tfrac12|\nabla S|^2+\tfrac{1}{2}\Delta S+V=f(t)
\end{equation}
for some continuous function $f:[0,T]\to\R$.
\end{corollary}

The same second-order Hamilton--Jacobi equation also yields a global minimality criterion for gradient-drift processes with prescribed endpoint laws.

\begin{proposition}[Global minimality]\label{prop-mechanical-verification}
Let $S\in C^{1,3}([0,T]\times\R^n)$ satisfy \eqref{HJ-regularity} and \eqref{HJB-2}. Any $X\in\A([0,T];\mu_0,\mu_T)$ satisfying \eqref{SEL-HJB-cond} is a global minimizer of $\mathcal S$ over this class, and hence a critical point.
\end{proposition}

\begin{proof}
Fix $X\in\A([0,T];\mu_0,\mu_T)$. Quadratic growth of $S,\pt_tS,V$, linear growth of $\nabla S$, boundedness of $\nabla^2S$, and \eqref{Ito-integrability} and \eqref{moment-assumptions} make the boundary and drift terms integrable and the martingale square integrable. It\^o's formula and \eqref{HJB-2} give
\begin{equation*}
\begin{aligned}
  dS(t,X(t))
  &=\left[\pt_tS+\nabla S\cdot DX(t)+\tfrac{1}{2}\Delta S\right]dt+dM(t)\\
  &=\left[f(t)-V-\tfrac12|\nabla S|^2+\nabla S\cdot DX(t)\right]dt+dM(t)\\
  &=\left[L_0(t,X(t),DX(t))+f(t)
      -\tfrac12|DX(t)-\nabla S|^2\right]dt+dM(t),
\end{aligned}
\end{equation*}
where $V$ and the derivatives of $S$ are evaluated at $(t,X(t))$. Integrating and taking expectations, with $\E M(T)=0$, yields
\begin{equation*}
\begin{split}
  \mathcal S[X]
  &\textstyle =\E[S(T,X(T))-S(0,X(0))]-\int_0^T f(t)\,dt\\
  &\textstyle \quad+\frac12\E\int_0^T|DX(t)-\nabla S(t,X(t))|^2\,dt.
\end{split}
\end{equation*}
The first two terms depend only on $S,f,\mu_0,\mu_T$; the last is nonnegative and vanishes under \eqref{SEL-HJB-cond}, proving global minimality. Both endpoint laws must hold: existence of $S$ alone does not ensure an arbitrary terminal law.
\end{proof}

\subsection{A Reduced Jacobi Functional}\label{sec:mechanical-Jacobi}

We now assume that the potential is independent of time and eliminate the deterministic scale in Theorem \ref{stoch-Maupertuis-prin}, using the scaled admissible class $\widehat{\mathcal A}$. Thus
\begin{equation*}
  L_0(x,\dot x)=\tfrac12|\dot x|^2-V(x),
\end{equation*}
where $V\in C^1(\R^n)$ and $|\nabla V(x)|\le C(1+|x|)$. In particular, $V$ has at most quadratic growth.
The associated classical energy follows from \eqref{Maupertuis-classical-energy}:
\begin{equation*}
  E_0(x,\dot x)=\tfrac12|\dot x|^2+V(x).
\end{equation*}

\begin{corollary}[A reduced Jacobi functional]\label{stoch-Maupertuis-Jacobi}
For $Y\in\A([0,T];\mu_0,\mu_T)$ with $e-\E V(Y(s))>0$ almost everywhere, define
\begin{equation}\label{Maupertuis-Jacobi-functional}
  \mathcal J_e[Y]
  :=\int_0^T\sqrt{2\,\E|DY(s)|^2\left[e-\E V(Y(s))\right]}\,ds.
\end{equation}
This is the reduction of \eqref{eq:02} over deterministic clocks:
\begin{equation}\label{Maupertuis-scale-reduction}
  \mathcal J_e[Y]
  =\inf_{\tau\ \mathrm{satisfies}\ \eqref{Maupertuis-clocks}}
      \mathcal I[Y\circ\tau^{-1},\tau].
\end{equation}
Suppose, in addition, that the following hold $ds$-a.e.:
\begin{equation}\label{Maupertuis-Jacobi-assumptions}
\begin{gathered}
  \E|DY(s)|^2>0,\\
  e-\E V(Y(s))\ge\delta>0.
\end{gathered}
\end{equation}
Then:
\setitemindent{ii)}
\begin{itemize}
\item[i)] For every $v\in\mathcal H([0,T];\R^n)$, $\mathcal J_e[Y+\e v]$ is finite for sufficiently small $|\e|$ and differentiable at $\e=0$.
\item[ii)] Define the deterministic rate
\begin{equation*}
  r_*(s):=\sqrt{\frac{\E|DY(s)|^2}{2[e-\E V(Y(s))]}}.
\end{equation*}
Assume $r_*,r_*^{-1}\in L^\infty(0,T)$ and set
\begin{equation*}
  \tau(s)=\tau(0)+\int_0^s r_*(u)\,du,\qquad
  X=Y\circ\tau^{-1}.
\end{equation*}
Then $(X,\tau)\in\widehat{\mathcal A}([0,T];\mu_0,\mu_T)$ and $X$ has mean energy $e$ as in \eqref{mean-energy}.
Moreover, if the reference momentum $\widehat{P}:=DY/r_*$ has a continuous adapted representative satisfying \eqref{momentum-decomposition} in reference time, then $Y$ is critical for $\mathcal J_e$ under all these variations if and only if
\begin{equation}\label{stoch-Newton}
  D(DX)=-\nabla V(X),\qquad dt\otimes\P\text{-a.e.},
\end{equation}
in physical time. Here $QX=(r_*\circ\tau^{-1})^{-1}I_n$, as in \eqref{Maupertuis-physical-covariance}; unlike in the preceding subsection, the physical covariance rate need not equal $I_n$.
\end{itemize}
\end{corollary}

\begin{proof}
First fix any $Y$ with $e-\E V(Y(s))>0$ almost everywhere, and introduce the mean kinetic energy and available mean energy
\begin{equation*}
  K_Y(s):=\tfrac12\E|DY(s)|^2,\qquad
  A_Y(s):=e-\E V(Y(s)).
\end{equation*}
The proof of Theorem \ref{stoch-Maupertuis-prin} gives
\begin{equation*}
  \mathcal I[Y\circ\tau^{-1},\tau]
  =\int_0^T\left(\frac{K_Y(s)}{r(s)}+r(s)A_Y(s)\right)ds.
\end{equation*}
For almost every $s$ and every $r(s)>0$,
\begin{equation*}
  \frac{K_Y(s)}{r(s)}+r(s)A_Y(s)
  \ge 2\sqrt{K_Y(s)A_Y(s)}.
\end{equation*}
When $K_Y(s)>0$, the unique pointwise minimizing rate is
\begin{equation*}
  r_*(s)=\sqrt{K_Y(s)/A_Y(s)}.
\end{equation*}
When $K_Y(s)=0$, the pointwise infimum is approached as $r(s)\downarrow0$. To respect \eqref{Maupertuis-clocks}, take
\begin{equation*}
  r_N(s)=\min\left\{N,\max\left\{N^{-1},
       \sqrt{K_Y(s)/A_Y(s)}\right\}\right\},
  \qquad N\ge1.
\end{equation*}
These admissible rates minimize the integrand over $[N^{-1},N]$, so their integrands are bounded by $K_Y(s)+A_Y(s)$ and converge almost everywhere to $2\sqrt{K_Y(s)A_Y(s)}$. Since $K_Y\in L^1(0,T)$ and the moment estimate and growth of $V$ bound $A_Y$, dominated convergence proves \eqref{Maupertuis-scale-reduction}.

i) Fix $v\in\mathcal H([0,T];\R^n)$ and set
\begin{equation*}
  k_\e(s):=\sqrt{K_{Y+\e v}(s)},\qquad
  a_\e(s):=\sqrt{A_{Y+\e v}(s)}.
\end{equation*}
Then
\begin{equation*}
  \mathcal J_e[Y+\e v]=2\int_0^T k_\e(s)a_\e(s)\,ds.
\end{equation*}
The gradient bound and the square-integrable suprema of $Y$ and $v$ give, for $|\e|\le1$,
\begin{equation*}
  \sup_{0\le s\le T}|A_{Y+\e v}(s)-A_Y(s)|\le C_v|\e|.
\end{equation*}
Thus $A_{Y+\e v}\ge\delta/2$ almost everywhere for sufficiently small $|\e|$. The functions $a_\e$ are uniformly bounded and $k_\e\in L^2(0,T)$, so $\mathcal J_e[Y+\e v]$ is finite. Moreover,
\begin{equation*}
  \frac{|a_\e-a_0|}{|\e|}\le C_{v,\delta},\qquad
  \frac{|k_\e-k_0|}{|\e|}
  \le\frac{1}{\sqrt2}\|\dot v(s)\|_{L^2(\Omega)}.
\end{equation*}
It follows that the difference quotient of $k_\e a_\e$ is bounded by
\begin{equation*}
  C\left(\|\dot v(s)\|_{L^2(\Omega)}+k_0(s)\right)\in L^1(0,T).
\end{equation*}
Under \eqref{Maupertuis-Jacobi-assumptions}, $k_0>0$ and $a_0\ge\sqrt\delta$ almost everywhere. 
Since $r_*=k_0/a_0$, differentiation of the $L^2(\Omega)$ norm and the potential expectation, followed by dominated convergence in $s$, yields
\begin{equation}\label{Maupertuis-Jacobi-variation}
\begin{split}
  \frac{d}{d\e}\bigg|_{\e=0}\mathcal J_e[Y+\e v]
  &=\int_0^T\left[
     \frac{a_0}{k_0}\E(DY\cdot\dot v)
     -\frac{k_0}{a_0}\E(\nabla V(Y)\cdot v)\right]ds\\
  &=\E\int_0^T\left[
     \frac{DY}{r_*}\cdot\dot v-r_*\nabla V(Y)\cdot v\right]ds.
\end{split}
\end{equation}
Cauchy--Schwarz bounds the expectations of the absolute values in the last line by $\sqrt{2A_Y}\|\dot v(s)\|_{L^2(\Omega)}$ and $C_{v,\delta}\sqrt{K_Y}$, respectively, both integrable in $s$.

ii) Under the bounds on $r_*,r_*^{-1}$, the clock $\tau$ is admissible and bi-Lipschitz, so $(X,\tau)\in\widehat{\mathcal A}([0,T];\mu_0,\mu_T)$. By \eqref{Maupertuis-time-change-identities} and the definition of $r_*$, for $ds$-a.e. $s$,
\begin{equation*}
  \E E_0(X(\tau(s)),DX(\tau(s)))
  =\frac{K_Y(s)}{r_*^2(s)}+\E V(Y(s))=A_Y(s)+\E V(Y(s))=e.
\end{equation*}
Since $\tau$ is deterministic and bi-Lipschitz, this proves \eqref{mean-energy}. The rate $r_*$ also attains the infimum in \eqref{Maupertuis-scale-reduction}: its bounds make the clock admissible, and $r_*^2=K_Y/A_Y$ gives
\begin{equation*}
  \frac{K_Y(s)}{r_*(s)}+r_*(s)A_Y(s)
  =2\sqrt{K_Y(s)A_Y(s)}.
\end{equation*}
Integrating this identity yields $\mathcal I[Y\circ\tau^{-1},\tau]=\mathcal J_e[Y]$.
Now assume the stated regularity of $\widehat{P}=DY/r_*$. Lemma \ref{integration-by-parts}, applied to $\widehat{P}$ and $v$, transforms \eqref{Maupertuis-Jacobi-variation} into
\begin{equation*}
  \frac{d}{d\e}\bigg|_{\e=0}\mathcal J_e[Y+\e v]
  =\E\int_0^T\left[-r_*\nabla V(Y)-D\widehat{P}\right]\cdot v\,ds.
\end{equation*}
The residual belongs to $L^2_{\mathrm{pred}}$, by the bound on $r_*$, the growth of $\nabla V$ and \eqref{momentum-decomposition}. The density in Remark \ref{remark-separated-variations} therefore shows that $Y$ is critical if and only if
\begin{equation}\label{Maupertuis-Jacobi-scaled-EL}
  D\widehat{P}=-r_*\nabla V(Y),\qquad ds\otimes\P\text{-a.e.}
\end{equation}
The time-change identities in the proof of Theorem \ref{stoch-Maupertuis-prin} give
\begin{equation*}
  DX(\tau(s))=\widehat{P}(s),\qquad
  D(DX)(\tau(s))=\frac{D\widehat{P}(s)}{r_*(s)},
  \qquad ds\otimes\P\text{-a.e.}
\end{equation*}
Here $\widehat{P}\circ\tau^{-1}$ represents $DX$ continuously. The derivatives of $X$ use physical time and $\mathcal G$; $D\widehat{P}$ uses reference time and $\F$. Thus \eqref{Maupertuis-Jacobi-scaled-EL} is equivalent to \eqref{stoch-Newton}.
\end{proof}

\begin{remark}[Normalization]\label{remark-Jacobi-attainment}
Under \eqref{Maupertuis-Jacobi-assumptions}, the additional normalization
\begin{equation*}
  \E\left[\tfrac12|DY(s)|^2+V(Y(s))\right]=e,
  \qquad ds\text{-a.e.},
\end{equation*}
is equivalent to $K_Y=A_Y$ and hence $r_*=1$. Taking $\tau(0)=0$ then gives $\tau=\id_{[0,T]}$ and $X=Y$. If $\widehat{P}=DY$ has the representative required in \eqref{momentum-decomposition}, the corollary reduces to criticality of $Y$ being equivalent to $D(DY)=-\nabla V(Y)$ in the original time parameter.
\end{remark}

\begin{remark}[The deterministic case]\label{remark-Maupertuis-deterministic}
For deterministic trajectories with no noise, the covariance distinction in \eqref{Maupertuis-physical-covariance} disappears, and \eqref{eq:02} becomes the usual autonomous free-time action; \eqref{Maupertuis-Jacobi-functional} reduces to the classical Jacobi functional (cf. e.g. \cite[Section 3.8]{AM78}).
\end{remark}

\section{Schr\"odinger's Problem v.s. Onsager--Machlup Functional}\label{sec-7-3}

The mechanical Lagrangian links two distinct problems: Schr\"odinger's varies path laws with fixed endpoint distributions; Onsager--Machlup varies curves for a fixed diffusion law. Here covariance rate is $I_n$, as in $\A([0,T];\mu_0,\mu_T)$, without rescaling.

\subsection{Schr\"odinger's Problem}

After Schr\"odinger's 1931--1932 formulation and Bernstein's reciprocal-process viewpoint \cite{Schrodinger1931,Schrodinger1932}, Fortet studied existence and Jamison the Markov construction \cite{fortet1940schrodinger,jamison1975schrodinger}. Nelson's mechanics inspired the mid-1980s revival and its distinct Euclidean quantum analogy \cite{Nelson,zambrini1986schrodinger,zambrini1986variational}. Mikami's 2004 zero-noise limit linked the problem to quadratic-cost Monge transport \cite{Mik04}; see \cite{leonard2014,Mik21} for surveys.

On $\mathcal C=C([0,T];\R^n)$, let $X$ be the coordinate process and $P_t=P\circ X(t)^{-1}$. For endpoint laws $\mu_0,\mu_T$ with finite second moments, let $R$ be Brownian motion's law with covariance rate $I_n$ and $R_0=\mu_0$. Schr\"odinger's problem minimizes relative entropy, or \emph{Kullback--Leibler divergence}:
\begin{equation}\label{entropy}
  \inf_{P_0=\mu_0,\;P_T=\mu_T}D_{\mathrm{KL}}(P|R),
  \qquad
  D_{\mathrm{KL}}(P|R):=
  \begin{cases}
    \int_{\mathcal C}\log\left(\frac{dP}{dR}\right)dP,&P\ll R,\\
    +\infty,&\text{otherwise}.
  \end{cases}
\end{equation}
We assume that a finite-entropy competitor exists. A minimizer is called a \emph{Schr\"odinger bridge}.

For $P_0=\mu_0$ and $D_{\mathrm{KL}}(P|R)<\infty$, Girsanov theory in the canonical filtration gives a predictable drift $b$, possibly depending on the whole past, such that
\begin{equation*}
  \textstyle dX(t)=b(t)\,dt+dB^P(t),
  \qquad
  D_{\mathrm{KL}}(P|R)=\frac12\E_P\int_0^T|b(t)|^2\,dt,
\end{equation*}
where $B^P$ is $P$-Brownian motion; $b$ need not be Markovian \cite{Leo12,leonard2014}.

To include the potential in \eqref{Lagrangian}, suppose additionally that $V$ is bounded and introduce the probability law
\begin{equation}\label{Schrodinger-potential-reference}
  \begin{aligned}
    \textstyle \frac{dR^V}{dR}
      &\textstyle =\frac{1}{Z_V}\exp\left(\int_0^T V(t,X(t))\,dt\right),\\
    Z_V&\textstyle =\E_R\exp\left(\int_0^T V(t,X(t))\,dt\right).
  \end{aligned}
\end{equation}
For finite-entropy competitors with the same endpoint constraints,
\begin{equation*}
  \textstyle D_{\mathrm{KL}}(P|R^V)
  =\E_P\int_0^T\left[\frac12|b(t)|^2-V(t,X(t))\right]dt
    +\log Z_V.
\end{equation*}
Thus entropy minimization relative to $R^V$ is a path-law formulation of the mechanical action. The exponential factor in \eqref{Schrodinger-potential-reference} incorporates $V$ via Feynman--Kac; $\frac12\Delta+V$ need not generate a conservative diffusion.

The corresponding dynamic formulation minimizes \cite{leonard2014}
\begin{equation*}
  \textstyle \int_0^T\int_{\R^n}
    \left[\frac12|v(t,x)|^2-V(t,x)\right]\rho_t(dx)\,dt
\end{equation*}
over narrowly continuous probability measures $\rho_t$ and measurable vector fields $v$ with finite kinetic energy, subject to the weak Fokker--Planck equation
\begin{equation}\label{FP-eqn}
  \pt_t\rho+\operatorname{div}(\rho v)-\tfrac{1}{2}\Delta\rho=0,
  \qquad \rho_0=\mu_0,\quad \rho_T=\mu_T.
\end{equation}
In the smooth regime, the optimality system couples \eqref{FP-eqn} with $v=\nabla S$ and \eqref{HJB-2} with $f=0$. The terminal value of $S$ must match the prescribed final law, not be prescribed independently; Proposition \ref{prop-mechanical-verification} gives a sufficient minimality criterion.

\emph{Euclidean quantum mechanics} is a probabilistic analogue of the quantum-mechanical $L^2$ framework, with positive forward/backward evolutions, Bernstein processes, and probabilistically interpreted observables under suitable domain and integrability assumptions \cite{albeverio1989,CZ91,CZ03}. It is neither merely Feynman--Kac nor Nelson's real-time theory, whose reconstructions involve a density-dependent Bohm potential distinct from our prescribed $V$ \cite{Nelson,CZ03,Zam15}. Wasserstein formulations offer another geometric description \cite{gentil2020}.

Reciprocity means conditional independence of interior and exterior paths given an interval's endpoint values. Reversal preserves reciprocity, not necessarily the path law; unequal endpoint laws preclude reversal invariance \cite{Leo14a}. Our forward formulation permits reciprocal structure; its backward counterpart reverses the filtration and endpoint data. Reciprocity requires more than the forward Euler--Lagrange equation; see the next chapter for bridges.

\subsection{Onsager--Machlup Functional}

Chapter \ref{ch:2} develops the Onsager--Machlup approach to most probable transition paths, with Lagrangian and Hamiltonian interpretations in Section \ref{sec:SM-Lag-Ham}. We now connect it to the preceding stochastic mechanics.

In contrast to \eqref{entropy}, the diffusion law is now fixed. Let $S\in C^{1,3}([0,T]\times\R^n)$ satisfy \eqref{HJ-regularity}, and consider
\begin{equation*}
  dX(t)=\nabla S(t,X(t))\,dt+dB(t),
  \qquad X(0)=0.
\end{equation*}
For deterministic curves $\gamma\in H^1([0,T];\R^n)$ with $\gamma(0)=0$, the \emph{Onsager--Machlup functional} is
\begin{equation}\label{OM-1}
  \textstyle \mathrm{OM}[\gamma]
  =\int_0^T\left[
    \frac12|\dot\gamma(t)-\nabla S(t,\gamma(t))|^2
    +\frac12\Delta S(t,\gamma(t))
  \right]dt.
\end{equation}
Its probabilistic meaning is the small-tube asymptotic: for a $C^2$ curve $\gamma$ starting at zero,
\begin{equation*}
  \lim_{\e\downarrow0}
  \frac{\P(\|X-\gamma\|_\infty<\e)}
       {\P(\|B\|_\infty<\e)}
  =\exp\left(-\mathrm{OM}[\gamma]\right),
\end{equation*}
where $\|\cdot\|_\infty$ is the uniform norm on $[0,T]$ \cite{OM53,MO53,Zeitouni1988}. This concerns neighborhoods of curves, not the probability of an individual path.

If $S$ also satisfies \eqref{HJB-2} with $f=0$, then
\begin{equation}\label{OM-id-drv}
  \textstyle \mathrm{OM}[\gamma]
  =\int_0^T L_0(t,\gamma(t),\dot\gamma(t))\,dt
    -S(T,\gamma(T))+S(0,0).
\end{equation}
Consequently, minimizing the OM functional is a deterministic Bolza problem, with mechanical running cost $L_0(t,\gamma,\dot\gamma)$ and terminal cost $-S(T,\gamma(T))$. A smooth minimizer with fixed initial point and free terminal point necessarily satisfies
\begin{equation}\label{eur-lag}
  \ddot\gamma(t)=-\nabla V(t,\gamma(t)),\quad 0<t<T,
  \qquad \gamma(0)=0,\quad
  \dot\gamma(T)=\nabla S(T,\gamma(T)).
\end{equation}
If the terminal point is fixed at $\gamma_T$, $\gamma(T)=\gamma_T$ replaces transversality, and the boundary term in \eqref{OM-id-drv} is constant. These necessary stationarity conditions assert neither existence nor sufficiency; see Problem \ref{prob3-4}.

A most probable curve in this small-tube sense is a deterministic variational object; a Schr\"odinger bridge is a probability law. Their shared mechanical Lagrangian does not identify the problems.

\section{Stochastic Geometric Mechanics on Manifolds}\label{sec:multip-SGM}

Following \cite{HZ23}, we survey smooth extensions to manifolds through Riemannian and second-order differential geometry, without proofs. Low regularity, nonexplosion on noncompact manifolds, and prescribed endpoint laws require additional analytic hypotheses.

\subsection{History and Motivation}

Classical geometric mechanics relates dynamics, variational principles, and symmetries on tangent and cotangent bundles \cite{VIArnold1989,Holm2011}. Constraints and symmetry reduction motivate manifolds; stochastic extensions must also encode quadratic variation.

Building on Bismut's diffusions, L\'azaro-Cam\'i and Ortega developed Stratonovich Hamiltonian and Hamilton--Jacobi theory on classical symplectic or Poisson structures \cite{LO08,LO09}. Here mean derivatives yield a deterministic second-order Hamilton--Jacobi equation; Schwartz--Meyer calculus encodes its It\^o correction in tangent and cotangent objects \cite{Eme07}.

Stochastic control motivates equations of the form
\begin{equation}\label{HJB}
  \pt_tS+H(x,\nabla S,\nabla^2S,t)=0,
\end{equation}
whose second derivatives reflect diffusion \cite{FS06}. Endpoint laws connect stochastic control to stochastic optimal transport and Schr\"odinger's problem \cite{leonard2014,Mik21}. Since the mid-1980s, stochastic deformation has sought probabilistic counterparts of classical relations among action, dynamics, and Hamilton--Jacobi theory \cite{Zam15}.

Using Schwartz--Meyer geometry, \cite{HZ23} constructs second-order symplectic and contact structures linking stochastic Hamilton equations, variational principles, and Hamilton--Jacobi theory. Its stochastic Noether theorem links second-order Hamilton--Jacobi symmetries to martingale conservation laws under suitable integrability conditions. Motivations include Schr\"odinger's problem and Euclidean quantum mechanics \cite{CZ03,leonard2014}; thermodynamic applications appear in \cite{HZ23a}.

\subsection{Stochastic Geometric Mechanics on Riemannian Manifolds}\label{sec:riemannian-SGM}

For simplicity, let $(M,g)$ be a compact, connected, smooth Riemannian manifold without boundary, with Levi--Civita connection $\nabla$. All fields in this subsection are smooth on $[0,T]\times M$. We use the convention $\Delta_g=\operatorname{div}_g\nabla$ and fix the covariance rate
\begin{equation}\label{Riemannian-covariance}
  QX(t)=g^{-1}(X(t)).
\end{equation}
A time-independent positive-definite diffusion tensor $a=\sigma\sigma^\top$ defines $g=a^{-1}$; see Section \ref{chapter2.sec2.2}. This metric turns the configuration space into a Riemannian manifold, with covariance $a$ in \eqref{Riemannian-covariance}. This inversion is unavailable for degenerate diffusion tensors.

The coordinate mean derivative $DX$ does not transform as a tangent vector. The connection corrects it to
\begin{equation}\label{mean-d}
  (D_\nabla X)^i=(DX)^i+\tfrac12\Gamma^i_{jk}(X)(QX)^{jk},
\end{equation}
Here $\Gamma^i_{jk}$ are the Christoffel symbols, with repeated indices summed. Unlike $DX$, $D_\nabla X$ is a vector along $X$; $QX$ is a symmetric contravariant two-tensor (Problem \ref{prob3-5}). If $D_\nabla X(t)=v(t,X(t))$ for a smooth vector field $v$, the generator is
\begin{equation*}
  \mathcal L_t f=v(t,\cdot)f+\tfrac{1}{2}\Delta_gf.
\end{equation*}
In particular, the local It\^o drift is $v^i-\frac{1}{2}\Gamma^i_{jk}g^{jk}$, not $v^i$.

For a smooth Lagrangian $L_0:[0,T]\times TM\to\R$, the geometric expected action is
\begin{equation*}
  \textstyle \mathcal S[X]=\E\int_0^T L_0(t,X(t),D_\nabla X(t))\,dt,
\end{equation*}
with prescribed endpoint laws. In the smooth Markov setting, put
$p(t,x)=d_{\dot x}L_0(t,x,v(t,x))$, a time-dependent one-form. The covariant stochastic Euler--Lagrange equation is
\begin{equation}\label{stoch-EL-mfld}
  \big(\pt_t+\nabla_v+\tfrac{1}{2}\Delta_{\mathrm{LD}}\big)p
    =d_xL_0(t,x,v(t,x)).
\end{equation}
Here $d_{\dot x}$ is the vertical differential, while $d_x$ is the horizontal differential defined by the connection. On one-forms the Laplace--de Rham operator has the sign convention
\begin{equation*}
  \Delta_{\mathrm{LD}}
    =-(dd^*+d^*d)
    =\operatorname{tr}_g\nabla^2-\operatorname{Ric}^*,
\end{equation*}
Here $d^*$ is the formal adjoint of the exterior derivative and $\operatorname{Ric}^*$ acts on covectors. Along $X$, the operator in \eqref{stoch-EL-mfld} is the damped mean covariant derivative of the smooth one-form field \cite[Section 7]{HZ23}, not a definition for arbitrary non-Markov momenta.

If $L_0$ is hyperregular, the Legendre transform gives the equivalent covariant equations
\begin{equation}\label{Riemannian-Hamilton}
  v=\pt_pH_0(x,p,t),\qquad
  \big(\pt_t+\nabla_v+\tfrac{1}{2}\Delta_{\mathrm{LD}}\big)p
    =-d_xH_0(x,p,t),
\end{equation}
where $d_xH_0$ is again the horizontal differential. For an exact momentum field $p=dS$, smooth solutions of
\begin{equation}\label{Riemannian-HJ}
  \pt_tS+H_0(x,dS,t)+\tfrac{1}{2}\Delta_gS=0
\end{equation}
generate stochastic Hamiltonian characteristics with geometric drift $v=\pt_pH_0(x,dS,t)$. In the mechanical case $L_0(t,x,\dot x)=\frac12|\dot x|_g^2-V(t,x)$, these are the Riemannian versions of the stochastic Newton and second-order Hamilton--Jacobi equations above.

A variational characterization of \eqref{stoch-EL-mfld}--\eqref{Riemannian-Hamilton} requires intrinsic random variations preserving \eqref{Riemannian-covariance} and covariant integration by parts; the additive variations of Definition \ref{variation} do not apply in arbitrary charts.

\subsection{Second-Order Differential Geometry}

Schwartz--Meyer geometry on a smooth manifold $M$ requires neither a metric nor a connection \cite{Eme07,HZ23}. A \emph{second-order tangent vector} at $q\in M$ is a differential operator of order at most two, without a zeroth-order term:
\begin{equation*}
  \textstyle A=b^i\frac{\pt}{\pt x^i}\big|_q
    +\frac12a^{jk}\frac{\pt^2}{\pt x^j\pt x^k}\big|_q,
  \qquad a^{jk}=a^{kj}.
\end{equation*}
Under coordinate changes, $a$ is a contravariant two-tensor, whereas $b$ acquires the It\^o correction. These operators form $\mathcal T^S M$, containing $TM$ as $a=0$. Diffusions correspond to its positive-semidefinite part, with $(b,a)=(DX,QX)$.

The dual bundle $\mathcal T^{S^*}M$ consists of second-order covectors, locally written as
\begin{equation*}
  \alpha=p_i\,d^2x^i+\tfrac12o_{jk}\,dx^j\cdot dx^k,
  \qquad o_{jk}=o_{kj},
  \qquad
  \langle\alpha,A\rangle=p_ib^i+\tfrac12o_{jk}a^{jk}.
\end{equation*}
The coordinates are $(x,p,o)$, and restriction to $TM$ gives $p_i\,dx^i$. Although $o_{jk}$ is not tensorial, $o_{jk}-\Gamma^i_{jk}p_i$ is tensorial after choosing a connection.

The basic second-order forms are characterized by
\begin{equation*}
  \langle d^2f,A\rangle=Af,\qquad
  \langle df\cdot dh,A\rangle=A(fh)-fAh-hAf.
\end{equation*}
In particular,
\begin{equation*}
  d^2f=\pt_if\,d^2x^i
        +\tfrac12\pt_{ij}f\,dx^i\cdot dx^j.
\end{equation*}
The symmetric product records quadratic covariation, incorporating first and second derivatives into the Hamiltonian description.

\subsection{Stochastic Hamiltonian Mechanics on General Manifolds}\label{sec:general-SGM}

On the extended second-order phase space $\mathcal T^{S^*}M\times[0,T]$, a smooth Hamiltonian $H(x,p,o,t)$ gives the Poincar\'e--Cartan form
\begin{equation*}
  p_i\,d^2x^i+\tfrac12o_{jk}\,dx^j\cdot dx^k-H\,dt.
\end{equation*}
Its second-order differential yields the Hamiltonian construction in \cite[Section 6]{HZ23}. Unlike classical mechanics, $H$ alone does not determine the full phase-space diffusion; momentum covariations and compatibility are also required. We consider a smooth compatible reduction.

Let $p(t,x)$ and $o(t,x)$ satisfy
\begin{equation*}
  o_{ij}=\pt_{x^j}p_i=\pt_{x^i}p_j.
\end{equation*}
This compatibility means that momentum is locally $d^2S$; full support does not imply it, and topology may prevent a global $S$.

Define the symmetric matrix $(\pt_{o_{jk}}H)_{j,k}$ by
$d_oH[U]=\sum_{j,k}\pt_{o_{jk}}H\,U_{jk}$ for every symmetric $U$.
Set
\begin{equation}\label{general-Hamilton-coefficients}
  \begin{aligned}
    b^i(t,x)&=\pt_{p_i}H(x,p(t,x),o(t,x),t),\\
    a^{jk}(t,x)&=2\pt_{o_{jk}}H(x,p(t,x),o(t,x),t).
  \end{aligned}
\end{equation}
The reduced stochastic Hamilton equations are
\begin{equation}\label{S-H}
  \begin{aligned}
    &(DX)^i=b^i(t,X),\qquad (QX)^{jk}=a^{jk}(t,X),\\
    &\big(\pt_t+b^j\pt_{x^j}
       +\tfrac12a^{jk}\pt_{x^jx^k}\big)p_i
       =-\pt_{x^i}H(x,p(t,x),o(t,x),t).
  \end{aligned}
\end{equation}
The last equation is a local PDE, with $\pt_{x^i}H$ taken at fixed $p,o$. Along $X$, the momenta $P_i(t)=p_i(t,X(t))$ and $O_{ij}(t)=o_{ij}(t,X(t))$ have induced covariations, and $DP_i=-\pt_{x^i}H$.

A diffusion requires positive-semidefinite $a$; a smooth square root gives a local SDE, and global existence requires nonexplosion. Equation \eqref{general-Hamilton-coefficients} determines $a$ from $H$, so covariance cannot be prescribed independently.

A smooth global solution of the second-order Hamilton--Jacobi equation
\begin{equation}\label{general-HJB}
  \pt_tS+H(x,\pt_xS,\pt_{xx}S,t)=0
\end{equation}
gives compatible fields $p_i=\pt_iS$, $o_{ij}=\pt_{ij}S$ satisfying \eqref{S-H}, establishing the smooth characteristic connection with \eqref{HJB}.

Finally, the Riemannian case is recovered by the globally defined Hamiltonian
\begin{equation*}
  H(x,p,o,t)
    =H_0(x,p,t)+\tfrac{1}{2}g^{jk}(x)
      \big(o_{jk}-\Gamma^i_{jk}(x)p_i\big).
\end{equation*}
Then $a=g^{-1}$, and \eqref{general-HJB} becomes \eqref{Riemannian-HJ}; for exact momenta, \eqref{S-H} reduces to \eqref{Riemannian-Hamilton}.

\section*{Problems}
\addcontentsline{toc}{section}{Problems}
%

\begin{prob}
\label{prob3-1}
\textbf{Lagrangian mechanics for a harmonic oscillator (or a spring-mass system).}\\
 For Newton's second law, $m\ddot x = -k x$, modeling a harmonic oscillator with mass $m$ and the Hooke spring constant $k>0$, verify that its   Lagrangian mechanics is equivalent to Newton's mechanics.\\
 \textit{Hint:} The Lagrangian is $\frac12 m \dot x^2 - \frac12 kx^2.$ The Newton's second law for this system is $m\ddot x= - kx.$
\end{prob}

\begin{prob}
\label{prob3-2}
\textbf{Mean-energy balance for a time-dependent Lagrangian.} \\
Under Remark \ref{remark-Maupertuis-interpretation}, prove \eqref{Maupertuis-time-dependent-energy}, including absolute continuity and the terminal condition, for a pair stationary under all deterministic clock variations with $\tau(0)$ fixed. \\
\textit{Hint:} Set $v=0$, $\alpha=0$ and $\zeta(s)=\int_0^s\eta(u)\,du$, so $\zeta(0)=0$. Differentiate the pulled-back action with time argument $\tau(s)$; first test against compactly supported $\zeta$, then use the arbitrary terminal value $\zeta(T)$.
\end{prob}

\begin{prob}
\label{prob3-3}
\textbf{Equivalence of stochastic Euler--Lagrange and Hamilton equations.} \\
Prove Theorem \ref{canonical-reduction}-(ii). \\
\textit{Hint:} Differentiate $H_0(x,p,t)=p\cdot\dot x-L_0(t,x,\dot x)$ using $p=\pt_{\dot x}L_0(t,x,\dot x)$, then substitute $x=X(t)$ and $p=P(t)$ into the resulting identities.
\end{prob}

\begin{prob}
\label{prob3-4}
\textbf{The Euler--Lagrange equation of most probable paths.} \\
Let $S$ solve \eqref{HJB-2} with $f=0$. Derive the Euler--Lagrange equation \eqref{eur-lag} from each expression \eqref{OM-1} and \eqref{OM-id-drv} of $\mathrm{OM}$. \\
\textit{Hint:} Use deterministic variations $h\in H^1([0,T];\R^n)$ with $h(0)=0$ and arbitrary $h(T)$. Integrate by parts and retain the terminal boundary term.
\end{prob}

\begin{prob}
\label{prob3-5}
\textbf{The vectorized mean derivative.} \\
Prove that the modified mean derivative \eqref{mean-d} transforms as a tangent vector along $X$ under changes of coordinates. \\
\textit{Hint:} Use It\^o's formula and the change-of-coordinate formula of Christoffel's symbols.
\end{prob}

\chapter[Schr\"odinger Bridges and Information Geodesics in the Space of Probability Densities]{Schr\"odinger Bridges and Information Geodesics in the Space of Probability Distributions}
\label{ch:4}

\abstract{
This chapter studies the Schr\"odinger bridge problem as relative-entropy
minimisation on path space under prescribed endpoint marginals.
It compares geometric formulations for Brownian references, finite-state
Markov chains, and continuous-state jump processes, including
L\'evy-driven models.
The Brownian theory connects stochastic and PDE formulations with
Wasserstein geometry, functional inequalities, and the Newton equation
for entropic interpolations.
For irreducible reversible finite-state chains, nonlocal transport
provides entropy gradient flows and a smooth Riemannian structure on
the interior of the probability simplex.
For continuous-state jump processes, we review nonlocal transport
and entropy gradient flows, and highlight open problems in developing
a general second-order differential calculus for the associated transport geometry.
The final section brings together Fisher--Rao geometry, dual connections,
entropy projections, and transport information geometry, with connections
to score-based diffusion, Schr\"odinger bridge generative modelling,
and flow matching.}

\section*{Introduction}

The preceding chapters study stochastic transitions through
individual trajectories and their geometric descriptions.
Chapter~2 develops the Onsager--Machlup approach to most probable
transition paths, while Chapter~3 examines stochastic geometric
mechanics on the state space. The present chapter turns to
probability laws and their evolution. Given a reference Markov
process and prescribed initial and final distributions, the
Schr\"odinger bridge problem selects a path law by minimising
relative entropy. Its time marginals form a curve
$t\mapsto\mu_t$ of probability distributions, called an
\emph{entropic interpolation}. This formulation describes how
probability is redistributed between the endpoints while retaining
the randomness of individual trajectories. Transport and information geometry provide complementary tools for studying these distributional transitions. Transport metrics quantify the cost of redistributing probability mass and support gradient-flow formulations of evolution equations. Information geometry describes statistical distinguishability through the Fisher--Rao metric, divergences, and dual connections.

\begin{remark}[Spatial order and geometric order]
The differential or fractional order of a Markov generator
describes its action on the state variable. The terms
\emph{first-order} and \emph{second-order} in the section titles
refer instead to the geometry of probability spaces.
First-order geometry concerns transport metrics, geodesics,
and gradient flows; second-order geometry concerns covariant
differentiation, Hessians, curvature, and Newton equations.
The scope of this geometric calculus depends on the state
space and the regularity of the chosen transport structure,
and cannot be inferred from the spatial order of the generator
alone.
\end{remark}

Section~\ref{sec:local} develops the Brownian theory, relating the
Schr\"odinger system to Wasserstein geometry, functional
inequalities, and the Newton equation for entropic interpolations.
Section~\ref{sec:discrete} treats irreducible reversible Markov
chains on finite graphs. Alongside the bridge problem, it develops
the nonlocal transport geometry on the interior of the probability
simplex, where gradients, Hessians, and curvature admit a smooth
finite-dimensional formulation.
Section~\ref{sec:levy} considers continuous-state jump processes,
including L\'evy-driven models. It reviews nonlocal transport and
entropy gradient flows, then discusses the analytical assumptions
and open problems involved in extending second-order differential
calculus to broader classes of jump kernels. Section~\ref{sec:information-geometry-generative} brings together
Fisher--Rao geometry, dual $\alpha$-connections, entropy projections,
and transport information geometry. It relates these ideas to
score-based generative modelling, Schr\"odinger bridge algorithms,
and flow matching, connecting the geometric study of probability
distributions with their computational construction.

\section[The Local Theory]{The Local Theory: Schr\"odinger Bridges under Brownian Motion}
\label{sec:local}

The Schr\"odinger bridge problem asks for the most likely stochastic
evolution connecting two given probability distributions: among all
random processes with prescribed initial and terminal laws, find the one
that minimises the relative entropy with respect to a reference measure.
The question was posed by Schr\"odinger in 1931--1932
{\ \citep{Schrodinger1931,Schrodinger1932}}. Erwin Schr\"odinger considered a Gedanken experiment involving a hot gas. He asked: given the observed probability distributions of a system of Brownian particles at two distinct times, what is the most likely evolution? Schr\"odinger formulated this as an entropy minimization problem over path measures subject to marginal constraints. This question has given rise to a rich mathematical theory and a variety of equivalent formulations, making it a cornerstone of modern data science, generative modeling, and control theory. Moreover, the SBP admits equivalent formulations in terms of controlled stochastic dynamics, forward--backward Kolmogorov or FP--HJB systems, and stochastic optimal control. The Wasserstein gradient-flow structure used below belongs to the reference Fokker--Planck dynamics and to related dissipative equations; the entropic interpolation of a Schr\"odinger bridge is not, in general, the gradient flow of a fixed functional. Its intrinsic Wasserstein dynamics is instead naturally expressed through a boundary-value Hamiltonian/Newton-type equation.

It is useful to separate two stages in the literature.
\emph{The foundational period (1931--1986)} begins with
Schr\"odinger's two-endpoint fluctuation problem
\citep{Schrodinger1931,Schrodinger1932}, Bernstein's 1932 reciprocal
construction, Fortet's analysis of the nonlinear
Schr\"odinger system \citep{fortet1940schrodinger}, Beurling's
product-measure existence--uniqueness theorem \citep{beurling1960},
and Jamison's reciprocal/Markov formulation
\citep{jamison1974,jamison1975schrodinger}.  Nelson's stochastic
mechanics \citep{Nelson1966} is a historically neighboring, but
mathematically distinct, reconstruction of quantum mechanics; the
relation between Schr\"odinger's programme, reciprocal processes, and
Nelson-type stochastic mechanics was explicitly revisited by Zambrini
in the mid-1980s
\citep{zambrini1986schrodinger,zambrini1986variational}.
\emph{The post-1986 period} develops these ideas through stochastic
control, entropy minimisation, optimal transport, contraction methods,
and computational regularisation
\citep{DaiPra1993,leonard2014,chen2016relation,Georgiou2015,cuturi2013}.
This division is only a historical guide: the modern optimal-transport
formulation does not replace the earlier reciprocal-process and
stochastic-mechanics strands, and the latter should not be retroactively
described as consequences of recent entropic-transport algorithms.

Two surveys anchor our treatment: L\'eonard's account of the entropy-minimisation theory \citep{leonard2014} and the control-theoretic synthesis in  \citep{chen2021stochastic, chen2021survey}. This section develops the local theory in five aspects: the stochastic-analysis, PDE,
geometric viewpoints, the completion of the first-order theory, and the second-order theory.

\subsection{The Stochastic-Analysis Viewpoint}
\label{subsec:local-stochastic}

We begin with some notations under the general framework in L\'eonard's survey \cite{leonard2014}, which is a foundational reference for Schr\"odinger Bridge Problem and its connections with optimal transport, large deviations, and stochastic control.

\textbf{Path measures and marginal notation.}
Let $\mathsf X$ denote the state space. We distinguish the continuous
and c\`adl\`ag canonical path spaces by
\[
\Omega_C:=C([0,1];\mathsf X),\qquad
\Omega_D:=D([0,1];\mathsf X).
\]
For diffusion references we take $\Omega=\Omega_C$, whereas for jump
references we take $\Omega=\Omega_D$. Let
$R\in\mathcal M_+(\Omega)$ be a positive reference measure and
$P\in\mathcal M_+(\Omega)$ another path measure. The canonical
evaluation map and time marginal are
\[
X_t:\Omega\to\mathsf X,\qquad X_t(\omega)=\omega(t),\qquad
P_t:=(X_t)_\#P\in\mathcal M_+(\mathsf X).
\]
Thus $\mathsf X$ is the state space, $X_t$ is the coordinate process,
and $\Omega$ is the path space.

\textbf{Relative Entropy.} For $P,R \in \mathcal{M}_+(\Omega)$, the relative entropy is
\begin{equation}
H(P \mid R) := 
\begin{cases}
\displaystyle \int_\Omega \log\frac{dP}{dR} \, dP, & \text{if } P \ll R, \\
+\infty, & \text{otherwise}.
\end{cases}
\end{equation}
If $P,R$ are probability measures, $H(P\mid R)=D_{\text{KL}}(P\|R)\ge 0$, with equality iff $P=R$.\\

Now we start with the dynamic Schr\"odinger under Brownian motion.

\begin{definition}[Dynamic Schr\"odinger's Problem]\label{def:dynSB}
The dynamic Schr\"odinger problem associated with the reference path measure is the following entropy minimization problem 
\begin{equation}
{\ 
\min \left\{ H(P \mid R) : P \in \mathcal{P}(\Omega), \;
P_0 = \mu_0, \; P_1 = \mu_1 \right\},
}
\label{eq:dynSB}
\end{equation}
where $\mu_0,\mu_1\in\mathcal P_2(\mathsf X)$ are prescribed initial
and terminal marginals and $H(\cdot\mid\cdot)$ is the relative entropy.
The subscript $2$ is reserved for measures on the state space having
finite second moment; no second-moment structure on the path space
$\Omega$ is required here.

Its unique minimiser (when finite) is the \emph{Schr\"odinger bridge}
$P^\star$, and the flow of marginals $\rho_t=P^\star_t$ is the
\emph{entropic interpolation} between $\mu_0$ and $\mu_1$.
\end{definition}

\textbf{Girsanov Formula and the energy of the drift.}
When $R$ is the law of a reversible diffusion (for definiteness, Wiener
measure with variance $\eps$) and $P\ll R$ has finite entropy, Girsanov's
theorem represents $P$ as the law of a controlled diffusion
$\dd X_t=v_t\dd t+\sqrt\eps\,\dd W_t$ and yields the decomposition
\begin{equation}
H(P\,|\,R)\;=\;H(P_0\,|\,R_0)\;+\;
\mathbb E_P\!\int_0^1 \frac{|v_t|^2}{2\eps}\dd t .
\label{eq:girsanov}
\end{equation}
Relative entropy on path space is thus, up to the initial term, the
expected kinetic energy of the control: entropy minimisation \emph{is}
a stochastic optimal control problem, namely the steering of the density
from $\mu_0$ to $\mu_1$ with minimal control effort
\citep{dai1996stochastic, chen2016relation}.\\

\textbf{Conditional Sanov Theorem and the Large Deviation.} Schrödinger's original question admits a precise interpretation through
Sanov's theorem and the Gibbs conditioning principle. Let
$ X^1,\ldots,X^N$ 
be independent path-valued random variables with common law
\(R\in\mathcal{P}(\Omega)\), where \(\Omega\) denotes the underlying path
space. Their empirical path measure is defined by
\begin{equation}
L_N
:=
\frac{1}{N}\sum_{i=1}^{N}\delta_{X^i}
\in\mathcal{P}(\Omega).
\end{equation}
Sanov's theorem states that the sequence \((L_N)_{N\geq1}\) satisfies a
large-deviation principle on \(\mathcal{P}(\Omega)\) with rate function
\[
P\longmapsto H(P\mid R),
\]
where
\begin{equation}
H(P\mid R)
:=
\begin{cases}
\displaystyle
\int_{\Omega}
\log\!\left(\frac{dP}{dR}\right)\,dP,
& P\ll R,\\[2mm]
+\infty,
& \text{otherwise},
\end{cases}
\end{equation}
is the relative entropy of \(P\) with respect to \(R\).

Let \(X_t:\Omega\to\mathsf{X}\) denote the canonical evaluation map at
time \(t\), and write
$P_t:=(X_t)_{\#}P$
for the time-\(t\) marginal of a path measure \(P\). Define the marginal constraint set
\begin{equation}
\mathcal{C}
:=
\left\{
P\in\mathcal{P}(\Omega):
P_0=\mu_0,\quad P_1=\mu_1
\right\}.
\end{equation}
Since the event of exact equality of empirical marginals may have
probability zero on a continuous state space, let
\(\mathcal{C}_{\delta}\) be a suitable \(\delta\)-neighborhood of
\(\mathcal{C}\). For example, with a metric \(d\) that metrizes weak
convergence on \(\mathcal{P}(\mathsf{X})\), one may take
\begin{equation}
\mathcal{C}_{\delta}
:=
\left\{
P\in\mathcal{P}(\Omega):
d(P_0,\mu_0)<\delta,\quad
d(P_1,\mu_1)<\delta
\right\}.
\end{equation}

Under the standard regularity assumptions required for the large-deviation
upper and lower bounds, the probability of observing these atypical
endpoint marginals satisfies
\begin{equation}
\lim_{\delta\downarrow0}
\lim_{N\to\infty}
-\frac{1}{N}
\log
R^{\otimes N}
\left(
L_N\in\mathcal{C}_{\delta}
\right)
=
\inf
\left\{
H(P\mid R):
P_0=\mu_0, P_1=\mu_1
\right\}.
\end{equation}

{ 
The large-deviation identity above determines the exponential cost of
the rare endpoint event, but it does not by itself identify the
conditional empirical law. To obtain that conclusion one also invokes
the \emph{Gibbs conditioning principle}: when the constrained entropy
minimiser $P^\star$ is unique, and the conditioning neighbourhoods
shrink in the standard admissible manner, the conditional law of the
empirical path measure concentrates at $P^\star$ as $N\to\infty$
\citep{leonard2014}. In this precise sense the Schr\"odinger bridge is
the conditional law-of-large-numbers limit of the rare marginal
fluctuation. Equivalently, among all path laws compatible with the
observed endpoint distributions, $P^\star$ is the one with minimal
relative entropy with respect to the reference dynamics $R$.
}\\

{ 
The solution to the SBP can be characterised by a system of integral
equations known as the Schr\"odinger system. Historically, Fortet gave
the first constructive existence analysis in 1940
\citep{fortet1940schrodinger,essid2018}; Beurling subsequently proved a
more general product-measure existence--uniqueness theorem
\citep{beurling1960}; and Jamison placed the construction in the modern
reciprocal/Markov-process framework
\citep{jamison1974,jamison1975schrodinger}.  The later
Hilbert-projective-metric interpretation explains the contraction
mechanism behind iterative scaling \citep{Georgiou2015}. Suppose for
simplicity that $R$ has identical marginals $R_0=R_1=m$.
}

\begin{definition} (\textbf{The Schr\"odinger System}.) Assume that $R$ has the same one-dimensional marginals: $R_0 = R_1 = m$ (e.g., Lebesgue measure). The Schr\"odinger system consists of two integral equations:
\begin{align}
f_0 g_0(x) &= f_0(x) \, \mathbb{E}_R\bigl( g_1(X_1) \mid X_0=x \bigr) = \frac{d\mu_0}{dm}(x), \label{eq:schsys1} \\
f_1 g_1(y) &= g_1(y) \, \mathbb{E}_R\bigl( f_0(X_0) \mid X_1=y \bigr) = \frac{d\mu_1}{dm}(y), \label{eq:schsys2}
\end{align}
where $f_0,g_1: X\to[0,\infty)$ are unknown functions. The solution $(\hat f_0,\hat g_1)$ gives the optimal path measure:
\begin{equation}
\hat P(d\omega) = \hat f_0(X_0(\omega)) \, \hat g_1(X_1(\omega)) \, R(d\omega).
\end{equation}
The optimal coupling $\hat\pi$ (joint distribution of $(X_0,X_1)$) is
\begin{equation}
\hat\pi(dx,dy) = \hat f_0(x)\,\hat g_1(y)\, R_{01}(dx,dy),
\label{eq:fgproduct}
\end{equation}
where $R_{01}$ is the joint distribution of $(X_0,X_1)$ under $R$.
\end{definition}

{ Under the relevant hypotheses the factors are unique up to the usual multiplicative rescaling. Fortet's constructive method is important both historically and computationally, while Beurling provides the more general product-measure existence--uniqueness theorem; Jamison gives a measure-theoretic Markov-process formulation. L\'eonard \citep{leonard2014} presents a modern entropy-minimisation treatment.}

\begin{remark}[Computation]\label{rem:sinkhorn}
Alternately enforcing \eqref{eq:schsys1} and \eqref{eq:schsys2} is the
iterative proportional fitting procedure, known in the discrete setting
as the Sinkhorn algorithm \citep{sinkhorn1967, cuturi2013}; under suitable positivity assumptions, convergence is geometric. We shall not pursue algorithmic questions in this
chapter and only flag, here and at analogous points below, that
efficient solvers exist.
\end{remark}

When the reference process is reversible (time‐symmetric), the solution of the dynamic SP admits a simple representation. For a reference Markov process with transition density $p(s,x;t,y)$ (with respect to a reference measure $m$), the solution $\mathbb{P}^*$ has transition densities
\[
p^*(s,x;t,y) = \frac{f(t,y)}{f(s,x)}\, p(s,x;t,y),
\]
where $f$ is a positive space–time function. This is an \emph{$(f,g)$-transform} or, in the time‐symmetric case, a \emph{Doob $h$-transform}. 

\textbf{The $(f,g)$-transform.}
The product form \eqref{eq:fgproduct} extends into the time interval.
Define
\begin{equation}
f_t(x)=\mathbb E_R\bigl[f_0(X_0)\,\big|\,X_t=x\bigr],\qquad
g_t(x)=\mathbb E_R\bigl[g_1(X_1)\,\big|\,X_t=x\bigr].
\label{eq:ftgt}
\end{equation}

\begin{theorem}[{ Marginals and Kolmogorov equations;
\citealp[Thm.~2.9 and Sect.~3]{leonard2014}}]\label{thm:fg}
Let $P=f_0(X_0)g_1(X_1)\,R$ with
$\mathbb E_R[f_0(X_0)g_1(X_1)]=1$. Then the marginal of $P$ at time $t$
is $\rho_t=f_t\,g_t\cdot m$, where $f_t$ solves the forward Kolmogorov
equation and $g_t$ the backward Kolmogorov equation associated with the
generator $L$ of $R$:
\begin{equation}
(-\partial_t+L^{*})f=0,\quad f|_{t=0}=f_0;
\qquad
(\partial_t+L)g=0,\quad g|_{t=1}=g_1.
\label{eq:fbkolmo}
\end{equation}
For time-symmetric $R$ this is Doob's $h$-transform performed
simultaneously from both time endpoints.
\end{theorem}

{ 
\begin{remark}[Reference generator versus quantum Hamiltonian]\label{rem:generator-hamiltonian}
The operator $L$ in Theorem~\ref{thm:fg} is the generator of the reference Markov process in the modern entropy-minimisation formulation of the Schr\"odinger problem. This should be distinguished from Schr\"odinger's original Euclidean analogy with quantum mechanics, whose starting operator is a quantum Hamiltonian, for example $\mathsf H=-\tfrac12\Delta+V$. Except in the free case $V=0$, $-\mathsf H$ is not the generator of a conservative diffusion; the potential term is instead represented on the Euclidean side through a Feynman--Kac weight (or, probabilistically, killing/branching after suitable interpretation). Thus the forward--backward Kolmogorov equations used here and the original quantum motivation are closely related historically and analytically, but they should not be identified operator by operator.
\end{remark}
}

\textbf{Entropy-regularized Optimal Transport.}
Projecting \eqref{eq:dynSB} onto the endpoint pair $(X_0,X_1)$ yields
the \emph{static} Schr\"odinger problem, which for the Wiener reference
becomes the entropically regularised Monge--Kantorovich problem
\begin{equation}
\min_{\pi\in\Pi(\mu_0,\mu_1)}
\left\{\int \frac{|x-y|^2}{2}\,\pi(\dd x,\dd y)
\;+\;\eps\,H(\pi\,|\,\mu_0\!\otimes\!\mu_1)\right\},
\label{eq:eot}
\end{equation}
with the noise intensity $\eps$ playing the role of a temperature
\citep{chen2021stochastic, cuturi2013}. The zero-temperature limit
$\eps\to0$ is the subject of Sect.~\ref{subsec:local-first-complete};
its dual, information-geometric reading is taken up in
Sect.~\ref{sec:ig-projection}.

Furthermore, the entropic interpolation $\mu_t$ (the marginal flow of the SBP) converges to the displacement interpolation (geodesic) in the $2$-Wasserstein space as $\epsilon\to 0$.

\subsection{The Partial Differential Equation Viewpoint}
\label{subsec:local-pde}

Let the reference dynamics be the It\^o diffusion
$$\dd X_t=b(t,X_t)\dd t+\sqrt{\eps}\,\dd W_t.$$  Writing the Schr\"odinger
factors of Theorem~\ref{thm:fg} as space--time functions
$\varphi(t,x)=g_t(x)$ and $\hat\varphi(t,x)=f_t(x)\,m(x)$, the bridge is
characterised by the \emph{forward--backward Kolmogorov system}
\begin{align}
\partial_t\hat\varphi+\nabla\!\cdot\!\bigl(b\,\hat\varphi\bigr)
-\tfrac{\eps}{2}\Delta\hat\varphi &= 0,
& \hat\varphi(0,\cdot)\,\varphi(0,\cdot)&=\rho_0, \label{eq:fk}\\
\partial_t\varphi+b\!\cdot\!\nabla\varphi
+\tfrac{\eps}{2}\Delta\varphi &= 0,
& \hat\varphi(1,\cdot)\,\varphi(1,\cdot)&=\rho_1, \label{eq:bk}
\end{align}
two \emph{linear} equations coupled only through their boundary
conditions, whose product $\rho_t=\varphi(t,\cdot)\hat\varphi(t,\cdot)$
is the entropic interpolation \citep{chen2021survey, Caluya2019}.

{ 
The same object admits a genuinely nonlinear description. Writing
$\rho_t$ for the controlled density and $\phi$ for the control potential,
normalised so that the optimal controlled drift is
$b+\nabla\phi$, the optimality conditions form the coupled
Fokker--Planck--Hamilton--Jacobi--Bellman system
\begin{equation}
\begin{aligned}
\partial_t\rho+\nabla\!\cdot\!\bigl(\rho\,(b+\nabla\phi)\bigr)
-\tfrac{\eps}{2}\Delta\rho &= 0,\\
\partial_t\phi+b\!\cdot\!\nabla\phi+\tfrac12|\nabla\phi|^2
+\tfrac{\eps}{2}\Delta\phi &= 0,
\end{aligned}
\qquad \rho(0,\cdot)=\rho_0,\quad \rho(1,\cdot)=\rho_1,
\label{eq:fphjb}
\end{equation}
Thus \eqref{eq:fphjb} describes the same reference diffusion
$dX_t=b(t,X_t)\,dt+\sqrt{\eps}\,dW_t$ as
\eqref{eq:fk}--\eqref{eq:bk}; the drift $b$ must be retained in both
representations. The passage between \eqref{eq:fphjb} and the linear pair
\eqref{eq:fk}--\eqref{eq:bk} is the \emph{Hopf--Cole transformation}
$\phi=\eps\log\varphi$ \citep{Hopf1950, Cole1951}. Indeed,
substituting this relation in the second equation of \eqref{eq:fphjb}
recovers the backward Kolmogorov equation \eqref{eq:bk}. The two-point
conditions are imposed on the density,
$\rho(0,\cdot)=\rho_0$ and $\rho(1,\cdot)=\rho_1$. L\'eger and Li
\citep{Léger2019} showed that generalized Hopf--Cole transformations are
symplectic submersions in Wasserstein symplectic geometry, so that the
Schr\"odinger bridge can be viewed as a boundary-value Hamiltonian system
on density space.
}\par

{ 
\textbf{Dual Formulation of Static SBP.} Let $R_{01}$ denote the
endpoint law of the reference process on $X\times X$. The primal problem is
\begin{equation}
\text{(P)}\qquad
\inf\Bigl\{H(\pi\mid R_{01}):\pi\in\Prob(X\times X),\
\pi_0=\mu_0,\ \pi_1=\mu_1\Bigr\}.
\label{eq:static-primal}
\end{equation}
For probability marginals, one convenient dual normalisation is the
log-partition form
\begin{equation}
\text{(D)}\qquad
\sup_{\phi,\psi\in C_b(X)}\left\{
\int_X\phi\,\dd\mu_0+\int_X\psi\,\dd\mu_1
-\log\!\int_{X\times X}e^{\phi(x)+\psi(y)}R_{01}(\dd x,\dd y)
\right\}.
\label{eq:static-dual}
\end{equation}
Under the usual finite-entropy and integrability assumptions, strong
duality holds \citep{leonard2014}. If $(\hat\phi,\hat\psi)$ is an optimal
dual pair, then
\begin{equation}
\hat\pi(\dd x,\dd y)
=\frac{e^{\hat\phi(x)+\hat\psi(y)}}{Z}\,
R_{01}(\dd x,\dd y),\qquad
Z:=\int_{X\times X}e^{\hat\phi(x)+\hat\psi(y)}R_{01}(\dd x,\dd y).
\label{eq:static-optcoupling}
\end{equation}
Because the dual potentials have an additive gauge freedom, they may be
renormalised so that $Z=1$; in that gauge
$e^{\hat\phi}=\hat f_0$ and $e^{\hat\psi}=\hat g_1$, and
\eqref{eq:static-optcoupling} reduces exactly to the Schr\"odinger
factorisation \eqref{eq:fgproduct}.
}\par

\begin{remark}[Computation]
The linearity of \eqref{eq:fk}--\eqref{eq:bk} underlies the dynamic
Sinkhorn recursion (alternate forward and backward solves) and
Wasserstein-proximal schemes \citep{Caluya2019}; we do not discuss these
further.
\end{remark}

\begin{remark}
    The Hopf--Cole transformation is of foundational importance for several reasons:
\begin{itemize}
    \item \textbf{Linearization}: It reduces the coupled Fokker--Planck--HJB system to linear PDEs, enabling analytical solutions in certain cases (e.g., Gaussian bridges) \cite{Chen2021}.
    \item \textbf{Dynamic Sinkhorn}: It underpins the dynamic Sinkhorn (iterative proportional fitting) algorithm, which converges to the optimal bridge by alternating forward and backward integration \cite{Georgiou2015, Caluya2019}.
    \item \textbf{Computational Efficiency}: For problems satisfying the matching condition, the transformation leads to linear PDEs amenable to efficient numerical methods such as finite elements or spectral methods \cite{Kalise2025}.
    \item \textbf{Generalizations}: Recent work extends the Hopf-Cole framework to control-affine systems \cite{teter2025}, mean-field interactions \cite{Eldesoukey2026}, quantum Schr\"odinger bridges \cite{Chen2021, Li2025}.
\end{itemize}
\end{remark}


\subsection{The Geometric Viewpoint I: First-Order}
\label{subsec:local-otto}

The space of probability measures is fundamental in statistics, physics, and machine learning. In the late 1990s, Felix Otto proposed to endow this infinite-dimensional space with a formal Riemannian metric \cite{otto2001}. This framework, known as \textbf{Otto calculus}, reveals a deep geometric structure underlying evolution equations such as the heat equation and the porous medium equation, interpreting them as gradient flows \cite{jko1998}. Otto acknowledged inspiration from earlier work by David Kinderlehrer and discussions with   McCann and   Villani \cite{villani2009}.

{ 
\textbf{Wasserstein space as a Riemannian manifold.}
For an absolutely continuous measure $\rho\,dx\in\mathcal P_2(\mathbb R^d)$,
the Otto tangent space is understood as the $L^2(\rho)$-closure
\[
T_\rho\mathcal P_2(\mathbb R^d)
=
\overline{\{\nabla\phi:\phi\in C_c^\infty(\mathbb R^d)\}}
^{\,L^2(\rho;\mathbb R^d)} .
\]
An absolutely continuous curve $t\mapsto\rho_t$ in
$(\mathcal P_2,W_2)$ admits a velocity field
$v_t\in L^2(\rho_t;\mathbb R^d)$ satisfying
\[
\partial_t\rho_t+\nabla\!\cdot(\rho_t v_t)=0
\]
in the weak sense. Among all such representatives, the tangent
velocity is identified with the minimal-$L^2(\rho_t)$ representative,
which belongs to the above closure of gradient fields
\citep{ambrosio2008}.
}

\textbf{The Riemannian Metric.} For two tangent vectors $v, w \in T_\rho \mathcal{P}_2$, their inner product is
\[
\langle v, w \rangle_\rho := \int_{\mathbb{R}^d} v(x) \cdot w(x) \, \rho(x) \, dx.
\]
This is the $L^2$ inner product with respect to the measure $\rho$ \cite{otto2001}. The geodesic distance induced by this metric is exactly the quadratic Wasserstein distance $W_2$ \cite{villani2009}:
\begin{equation}
W_2(\mu_0, \mu_1)^2 = \inf_{\rho_t, v_t} \left\{ \int_0^1 \int |v_t(x)|^2 \rho_t(x) \, dx \, dt \;\middle|\; \partial_t \rho_t + \nabla \cdot (v_t \rho_t)=0,\ \rho_0=\mu_0,\ \rho_1=\mu_1 \right\}.
\label{eq:BB}
\end{equation}
This is the Benamou–Brenier formula, which serves as the dynamic formulation of $W_2$.

Otto calculus allows us to study gradient flows, Newton equations, and functional Hessians in this infinite-dimensional space. 

Consider a functional $\mathcal{F}(\rho)$ on $\mathcal{P}_2$. Its gradient flow with respect to the $W_2$ metric is
\[
\partial_t \rho_t = - \operatorname{grad}_{W_2} \mathcal{F}(\rho_t),
\]
which is the PDE that steepest descends $\mathcal{F}$.

For a given $\mathcal{F}$, the Wasserstein gradient is given by
\[
\operatorname{grad}_{W_2} \mathcal{F}(\rho) = -\nabla \cdot (\rho \nabla \delta_\rho \mathcal{F}),
\]
where $\delta_\rho \mathcal{F}$ is the first variation.

The Jordan–Kinderlehrer–Otto (JKO) scheme is a time-discretization that forms a rigorous variational formulation \cite{jko1998}:
\[
\rho_{k+1} = \arg\min_{\rho \in \mathcal{P}_2} \left( \mathcal{F}(\rho) + \frac{1}{2\tau} W_2^2(\rho, \rho_k) \right).
\]
As $\tau \to 0$, the discrete solutions converge to the gradient flow.

For example, the Boltzmann entropy $\mathcal{F}(\rho) = \int \rho \log \rho \, dx$ yields $\operatorname{grad}_{W_2} \mathcal{F} = -\Delta \rho$, i.e., the heat equation. The porous medium equation arises from $\mathcal{F}(\rho) = \int \frac{\rho^m}{m-1} dx$ \cite{otto2001}.

Otto calculus has been extended in several important directions:

\textbf{Rigorous Framework and Synthetic Ricci Curvature.} The formal calculations of Otto have been made rigorous within the theory of \emph{metric measure spaces} \cite{ambrosio2008} and \emph{synthetic Ricci curvature} \cite{villani2009}. The Lott–Sturm–Villani theory defines Ricci curvature via displacement convexity of the entropy, a concept deeply connected to Otto calculus.

{ \textbf{Mean-field games and generative modelling.} Wasserstein PDEs also appear in mean-field control/games and in modern generative modelling. For the purposes of this chapter, the classical point needed from Jordan--Kinderlehrer--Otto is the minimizing-movement interpretation of Wasserstein gradient flows \citep{jko1998}. The 1998 JKO paper predates normalizing flows and should not be cited as a source for a normalizing-flow implementation.}

\textbf{Information Geometry and Other Cost Functions.} Recent work unifies Otto's Wasserstein geometry with classical information geometry, leading to the concept of the Kantorovich–Wasserstein–Otto metric using deformed logarithms. Extensions also replace $W_2$ with more general transport costs on manifolds.

Otto calculus provides a powerful and versatile geometric framework for probability distributions. Its impact spans pure mathematics, physics, and computational sciences. Key references include the original papers by Otto \cite{otto2001} and Jordan–Kinderlehrer–Otto \cite{jko1998}, the comprehensive books by Villani \cite{villani2009} and Ambrosio–Gigli–Savaré \cite{ambrosio2008}, as well as recent advances by Gentil–Leonard–Ripani \cite{gentil2020}, and Karatzas et al. \cite{karatzas2021}.

\subsection{The Geometric Viewpoint II: Second-Order}
\label{subsec:local-second}

{ 
Before using the Bochner identity, we fix the semigroup notation that
will also be used in Sect.~\ref{subsec:local-first-complete}. For the
weighted diffusion generator
\[
L=\Delta-\nabla U\cdot\nabla,
\]
define
\[
\Gamma(f,g):=\frac12\bigl(L(fg)-fLg-gLf\bigr),
\qquad \Gamma(f):=\Gamma(f,f),
\]
and
\begin{equation}
\Gamma_2(f):=\frac12\bigl(L\Gamma(f)-2\Gamma(f,Lf)\bigr).
\label{eq:gamma}
\end{equation}
For this diffusion generator, $\Gamma(f)=|\nabla f|^2$.
}

{ 
\textbf{Hessians in Otto calculus.}
The second-order structure of $\Prob_2$ begins with the Hessian of a
functional along Wasserstein geodesics,
\begin{equation}
\Hess_{W_2}\mathcal F(\rho)[\nabla\phi,\nabla\phi]
=\frac{\dd^2}{\dd s^2}\Big|_{s=0}\mathcal F(\rho_s),
\qquad \rho_s=(\mathrm{Id}+s\nabla\phi)_\#\rho ,
\label{eq:hess}
\end{equation}
made rigorous by Gigli's second-order analysis on
$(\Prob_2,W_2)$ \citep{gigli2012}; Lott \citep{lott2008} computed the
corresponding curvature tensor in the smooth setting. To match the
weighted diffusion used later in Sect.~\ref{subsec:local-first-complete},
let $(M,g)$ be a smooth Riemannian manifold and let
$\dd\mu=e^{-U}\dd\mathrm{vol}$ be the reference measure. For the relative
entropy
\[
\Ent_\mu(\rho)
:=\int_M \rho\log\!\left(\frac{\rho}{e^{-U}}\right)\dd\mathrm{vol},
\]
the fundamental second-variation identity is
\begin{equation}
\Hess_{W_2}\Ent_\mu(\rho)[\nabla\phi,\nabla\phi]
=\int_M\Bigl(\|\Hess\phi\|_{\mathrm{HS}}^2
+\Ric_U(\nabla\phi,\nabla\phi)\Bigr)\rho\,\dd\mathrm{vol},
\qquad
\Ric_U:=\Ric+\Hess U .
\label{eq:hessent}
\end{equation}
Equivalently, for the weighted generator
$L=\Delta-\nabla U\cdot\nabla$,
\[
\Gamma_2(\phi)=\|\Hess\phi\|_{\mathrm{HS}}^2
+\Ric_U(\nabla\phi,\nabla\phi).
\]
Thus, under the standard smoothness hypotheses, the lower bound
$\Ric_U\ge K g$ (equivalently $\Gamma_2\ge K\Gamma$) is the smooth
weighted-manifold form of geodesic $K$-convexity of the relative entropy.
The Lott--Sturm--Villani theory \citep{villani2009, ambrosio2008}
extends this principle synthetically to metric-measure spaces. In the
Euclidean unweighted case $U\equiv0$, one has $\Ric_U=0$.
}\par

{ 
\textbf{Newton's equation for entropic interpolations.}
Geodesics have zero covariant acceleration. For the Brownian reference
with generator $(\eps/2)\Delta$, Conforti
\citep{conforti2019} proved that the marginal flow of the
Schr\"odinger bridge satisfies
\begin{equation}
\frac{\mathbf D}{\dd t}v_t
=\frac{\eps^2}{8}\,\operatorname{grad}_{W_2}\mathcal I(\rho_t),
\qquad
\mathcal I(\rho)=\int|\nabla\log\rho|^2\rho\,\dd x .
\label{eq:newton}
\end{equation}
For a sufficiently smooth gradient velocity field
$v_t=\nabla\phi_t$, the covariant derivative is represented in Eulerian
coordinates by
$\frac{\mathbf D}{\dd t}v_t=\partial_t v_t+(v_t\cdot\nabla)v_t$.
Thus the coefficient in the Eulerian form is also $\eps^2/8$.
Equation \eqref{eq:newton} may be read as a Newton equation with
mechanical potential $-(\eps^2/8)\mathcal I$. In particular, the force is
$O(\eps^2)$ and vanishes in the zero-noise limit, consistently with the
convergence of entropic interpolations to $W_2$ geodesics. The factor
$1/8$ is the normalization corresponding to the generator
$(\eps/2)\Delta$ and is used consistently below.
}\par

Equation \eqref{eq:newton} is the reference point for the two nonlocal
sections: on graphs an exact analogue exists and is a finite-dimensional
Hamiltonian system (Sect.~\ref{subsec:discrete-second}); in the
continuous nonlocal setting, writing any analogue at all is an open
problem (Sect.~\ref{subsec:levy-second}).


\subsection{Functional Inequalities
and the Zero-Noise Limit}
\label{subsec:local-first-complete}

The gradient-flow reading of the heat equation
(Sect.~\ref{subsec:local-otto}) turns a family of classical functional
inequalities into statements about the behaviour of entropy along the
flow, and ties their sharp constants to a single geometric quantity: a
lower bound on the Ricci curvature. We first fix the three functionals
involved, then state the inequalities and the implications among them
(Fig.~\ref{fig:inequalities}), and finally record the curvature
criterion that generates the whole picture.

Throughout, the reference measure is
$\dd\mu=e^{-U}\dd x$, and for a probability density $\nu\ll\mu$ we write
the relative entropy, the Fisher information, and the quadratic
Wasserstein distance
\begin{equation}
H(\nu\,|\,\mu)=\int\frac{\dd\nu}{\dd\mu}\log\frac{\dd\nu}{\dd\mu}\dd\mu,
\qquad
I(\nu\,|\,\mu)=\int\Bigl|\nabla\log\frac{\dd\nu}{\dd\mu}\Bigr|^2\dd\nu,
\qquad
W_2(\nu,\mu).
\label{eq:HIW}
\end{equation}
Entropy measures information content, the Fisher information measures its
dissipation rate along the heat flow, and $W_2$ measures transport cost;
the inequalities below are precisely the statements that these three
quantities constrain one another.\\

We now move on to several important inequalities.\\

\textbf{Logarithmic Sobolev Inequality.} The \emph{logarithmic Sobolev inequality} $\mathrm{LSI}(K)$ bounds
entropy by Fisher information,
\begin{equation}
H(\nu\,|\,\mu)\;\le\;\frac{1}{2K}\,I(\nu\,|\,\mu);
\label{eq:LSI}
\end{equation}
along the heat semigroup it is equivalent to exponential decay of the
entropy at rate $2K$. \\

\textbf{Talagrand Transport Inequality.} The \emph{Talagrand transport inequality}
$\mathrm T_2(K)$ bounds transport by entropy,
\begin{equation}
\frac{K}{2}\,W_2^2(\nu,\mu)\;\le\;H(\nu\,|\,\mu),
\label{eq:T2}
\end{equation}
so that closeness in entropy forces closeness in transport distance. \\

\textbf{HWI inequality.} The \emph{HWI inequality} interpolates between the two,
\begin{equation}
H(\nu\,|\,\mu)\;\le\;W_2(\nu,\mu)\,\sqrt{I(\nu\,|\,\mu)}
-\frac{K}{2}\,W_2^2(\nu,\mu),
\label{eq:HWI}
\end{equation}
relating all three functionals of \eqref{eq:HIW} in a single line. Based on its name, this can also be called
the relative entropy–Wasserstein distance–Fisher information inequality. Finally, we linearize any of
these around $\mu$ produces the \emph{Poincar\'e inequality} (spectral
gap)
\begin{equation}
\operatorname{Var}_\mu(f)\;\le\;\frac{1}{\lambda}\int|\nabla f|^2\dd\mu,
\qquad
\lambda=\inf_{f\perp 1}\frac{\int|\nabla f|^2\dd\mu}
{\operatorname{Var}_\mu(f)},
\label{eq:poincare}
\end{equation}
the weakest of the four and the one governing the asymptotic rate of
convergence to equilibrium.\\


\begin{figure}[htbp]
\centering
\begin{tikzpicture}[
    >=Stealth,
    node distance=12mm and 22mm,
    every node/.style={
        font=\large,
        align=center
    },
    main/.style={
        inner sep=2pt
    },
    arr/.style={
        ->,
        line width=0.9pt
    },
    lab/.style={
        font=\small,
        fill=white,
        inner sep=1.5pt
    }
]

\node[main] (CD) {$CD(K,\infty)$};

\node[main, below left=16mm and 20mm of CD] (LSI) {$LSI(K)$};
\node[main, below right=16mm and 20mm of CD] (HWI) {$HWI(K)$};

\node[main, below left=16mm and 17mm of LSI] (T2) {$T_2(K)$};
\node[main, below right=16mm and 17mm of LSI] (PI) {Poincar\'e$(K)$};

\draw[arr] (CD) --
node[lab, above left=-1pt] {Bakry--\'Emery}
(LSI);

\draw[arr] (CD) --
node[lab, above right=-1pt]
{HWI theorem\\(Otto--Villani)}
(HWI);

\draw[arr] (HWI) --
node[lab, above] {Young inequality}
(LSI);

\draw[arr] (LSI) --
node[lab, left] {Otto--Villani}
(T2);

\draw[arr] (LSI) --
node[lab, right] {linearization}
(PI);

\end{tikzpicture}

\caption{Implication relations among functional inequalities.}
\label{fig:inequalities}
\end{figure}

\textbf{The curvature criterion.} All the downward arrows of Fig.~\ref{fig:inequalities} issue from a
single hypothesis, expressed through the carr\'e du champ calculus of
the generator $L=\Delta-\nabla U\cdot\nabla$ of the reference diffusion.
With $\Gamma$ and $\Gamma_2$ defined already in
Sect.~\ref{subsec:local-second}, the Bochner formula computes the
iterated form explicitly,
\begin{equation}
\Gamma_2(f)=\|\nabla^2 f\|_{\mathrm{HS}}^2
+\bigl(\operatorname{Ric}+\nabla^2 U\bigr)(\nabla f,\nabla f),
\label{eq:bochner}
\end{equation}
exhibiting the \emph{Bakry--\'Emery Ricci tensor}
$\operatorname{Ric}_U=\operatorname{Ric}+\nabla^2 U$, which folds the
convexity of the potential into an effective curvature. The
\emph{Bakry--\'Emery criterion}, or curvature-dimension condition
$CD(K,\infty)$, is the inequality
\begin{equation}
\Gamma_2(f)\;\ge\;K\,\Gamma(f)\qquad\text{for all }f,
\label{eq:BE}
\end{equation}
which, after discarding the nonnegative Hessian term in
\eqref{eq:bochner}, is exactly $\operatorname{Ric}_U\ge K$. The parameter
$K$ is the curvature lower bound; when $K>0$ it is simultaneously the
constant appearing in \eqref{eq:LSI}, \eqref{eq:T2}, \eqref{eq:HWI}, and
a lower bound for the spectral gap in \eqref{eq:poincare}. The symbol
$\infty$ is the dimensional parameter $N$ of the finer condition
$CD(K,N)$, whose defining inequality carries the extra term
$\tfrac1N(Lf)^2$ on the right of \eqref{eq:BE}; sending $N\to\infty$
removes it and yields the clean, dimension-free criterion used here,
while finite $N$ additionally produces Sobolev inequalities and
Bonnet--Myers diameter bounds. This is the precise content of the slogan
that the sharp constants of the first-order theory are curvatures.

\textbf{Comparison with the nonlocal setting.}
Formula \eqref{eq:bochner} is the analytic engine of the entire diagram:
it is what allows $\Gamma_2$ to be read as ``Hessian plus curvature'',
and it is the pointwise density of the entropy Hessian on Wasserstein
space computed in \eqref{eq:hessent} of Sect.~\ref{subsec:local-second}.
Its availability rests on the diffusion (chain-rule) property of $L$. For
the jump generators of Sect.~\ref{sec:levy} the identity
\eqref{eq:bochner} has no counterpart --- $\Gamma_2$ does not close, and
$\alpha$-stable semigroups satisfy no $CD(K,\infty)$ for any $K$ --- so
the top node of Fig.~\ref{fig:inequalities} is absent, and with it the
whole cascade of curvature-based inequalities. On finite graphs
(Sect.~\ref{sec:discrete}) the top node is recovered by a different
route: bypassing Bochner, one \emph{defines} the entropic Ricci bound as
the convexity constant of the entropy along $\WW$-geodesics
(Sect.~\ref{subsec:discrete-second}) and thereby restores discrete
analogues of $\mathrm{LSI}$, $\mathrm T_2$, and $\mathrm{HWI}$, with the
same implication pattern as in Fig.~\ref{fig:inequalities}.


\textbf{$\Gamma$-convergence to optimal transport.}
Let $\mathcal E_\eps(\mu_0,\mu_1)$ denote the entropic cost, the optimal
value of \eqref{eq:dynSB} for the Wiener reference of variance $\eps$
(suitably normalised). As $\eps\to0^+$,
\begin{equation}
\eps\,\mathcal E_\eps(\mu_0,\mu_1)\;
\xrightarrow{\ \Gamma\ }\;\tfrac12 W_2^2(\mu_0,\mu_1),
\label{eq:gammaOT}
\end{equation}
minimisers converge to optimal transport plans, and the entropic
interpolation $\rho_t^\eps$ converges to McCann's displacement
interpolation --- the $W_2$ geodesic \citep{leonard2014, chen2021stochastic}. The Schr\"odinger problem is thus a canonical
regularisation of optimal transport, and conversely optimal transport is
the zero-temperature shadow of the bridge; the stability of this picture
under perturbations of the reference measure \citep{ghosal2022stability} will be
essential when the reference becomes a L\'evy process in
Sect.~\ref{sec:levy}. The nonlocal analogue of \eqref{eq:gammaOT} ---
the localisation limit --- is Theorem~\ref{thm:localization} below.

\subsection{Extension: Changing the Base Geometry}
\label{subsec:local-extensions}

The constructions of this section used remarkably little about the
Wiener reference: an entropy, a Markov semigroup, and a Girsanov-type
energy identity. The variational principle \eqref{eq:dynSB} is therefore
\emph{stable under changes of the base geometry}, and the standard
extensions of the theory are best understood as three instances of this
one principle: change the underlying space, change the time axis, or
change the level of description of the particles.

\textbf{Changing the space: Riemannian bridges.}
Replacing $\R^d$ by a Riemannian manifold replaces the Euclidean Wiener reference by manifold Brownian motion. The Schr\"odinger system and the corresponding forward--backward characterisation then use the Laplace--Beltrami operator (together with the appropriate drift terms), and diffusion Schr\"odinger bridges can be formulated directly on Riemannian state spaces \citep{thornton2022}. We restrict the present discussion to this Riemannian setting.

\textbf{Changing the time axis: the Euclidean/quantum correspondence.}
For a Schr\"odinger-type operator with scalar potential $V$, the Euclidean forward--backward pair may be written schematically as
\begin{equation}
\partial_t\varphi=-\tfrac12\Delta\varphi+V\varphi,
\qquad
\partial_t\hat\varphi=\tfrac12\Delta\hat\varphi-V\hat\varphi,
\qquad
\rho=\varphi\hat\varphi .
\label{eq:euclideanpair}
\end{equation}
This is the operator-level Euclidean correspondence that motivated
Schr\"odinger's original problem
\cite{Schrodinger1931,Schrodinger1932}. A formal Wick rotation connects
parabolic and unitary Schr\"odinger evolutions, while the path-integral
formulation of quantum mechanics provides a distinct real-time
Lagrangian route \citep{feynmanhibbs1965}.  These constructions should
not be conflated with \emph{Nelson's stochastic mechanics}, introduced
as a separate probabilistic reconstruction of the Schr\"odinger
equation \citep{Nelson1966}.  In the mid-1980s Zambrini explicitly
revisited Schr\"odinger's programme and its relation to stochastic
mechanics
\citep{zambrini1986schrodinger,zambrini1986variational}; Carmona's
probabilistic construction of Nelson processes belongs to this latter
historical strand \citep{carmona1987nelson}.  Accordingly, references
to a ``quantum Schr\"odinger bridge'' must specify which construction is
intended.  In particular, Bohm-potential formulations associated with
Nelson/Madelung-type stochastic mechanics should not be presented as
part of Schr\"odinger's original 1931 hot-gas bridge problem.  Equation
\eqref{eq:euclideanpair} is therefore used here only as a historical and
analytic Euclidean correspondence, consistently with
Remark~\ref{rem:generator-hamiltonian}, not as an identification of
every diffusion bridge with a unitary quantum evolution.

\textbf{Changing the particle level: mean-field bridges.}
Replacing independent particles by a weakly interacting population
replaces the reference by a McKean--Vlasov diffusion
$\dd X_t=\bigl(b(X_t)+\nabla U\!*\!\rho_t(X_t)\bigr)\dd t
+\sqrt{2\eps}\,\dd W_t$, and the mean-field Schr\"odinger bridge seeks
the minimum-effort steering of the population between prescribed
marginals \citep{Eldesoukey2026}. The optimality system is
\eqref{eq:fphjb} augmented by nonlocal interaction terms,
\begin{equation}
\begin{aligned}
\partial_t\rho+\nabla\!\cdot\!\bigl(\rho\,
(\nabla\phi+\nabla U\!*\!\rho)\bigr)&=\eps\Delta\rho,\\
\partial_t\phi+\tfrac12|\nabla\phi|^2
+\nabla U\!*\!\rho\cdot\nabla\phi&=-\eps\Delta\phi,
\end{aligned}
\label{eq:mfsb}
\end{equation}
a nonconvex problem in general; a generalized Hopf--Cole transform
linearises the interaction under symmetry and Lipschitz assumptions on
$U$, and Sinkhorn-type recursions then apply \cite{Eldesoukey2026}. Closely related forward--backward systems also occur in mean-field control and mean-field games \citep{Carmona2018}, but the relation is not merely a change of boundary conditions. The variational objectives, the role of the population consistency relation, and the terminal data differ between these formulations. The useful point here is structural: the mean-field Schr\"odinger bridge retains a two-marginal steering objective while the dynamics themselves depend on the evolving population law.

\section[The Nonlocal Theory I: Schrödinger Bridges on Finite Graphs]{The Nonlocal Theory I: Schrödinger Bridges on Finite Graphs}
\label{sec:discrete}

We now leave the local world and consider a reversible Markov jump process on a finite graph. The crucial simplification is finite dimensionality: on the relative interior of the probability simplex the nonlocal transport metric is a smooth Riemannian metric, so its connection, Hessian, curvature, Jacobi equation, and Hamiltonian dynamics reduce to finite-dimensional differential geometry. The fact that the master equation contains finite differences is useful notation, but it is not what makes the second-order theory available. This section therefore provides the finite-dimensional nonlocal counterpart of Sect.~\ref{sec:local}.

Throughout, $\mathcal X$ is a finite set, $q_{xy}\ge0$ are jump rates of
an irreducible continuous-time Markov chain, reversible with respect to
a probability measure $\pi$ (detailed balance:
$\pi_xq_{xy}=\pi_yq_{yx}$). Densities are vectors
$\rho\in\Prob(\mathcal X)$, identified with points of the simplex.

\subsection{The Stochastic-Analysis Viewpoint}
\label{subsec:discrete-stoch}

The stochastic side of the theory transfers verbatim from
Sect.~\ref{subsec:local-stochastic}: with $R$ the law of the reversible
chain on path space $D([0,1];\mathcal X)$, the dynamic Schr\"odinger
problem \eqref{eq:dynSB} is well posed, the Schr\"odinger system
\eqref{eq:schsys1}--\eqref{eq:schsys2} becomes a pair of equations for
vectors $f_0,g_1\in\R_{\ge0}^{\mathcal X}$ coupled through the
transition matrix $R_{01}=e^{Q}$ (with $Q$ the rate matrix), and the
optimal bridge is the $(f,g)$-transform
$P^\star=f_0(X_0)g_1(X_1)\,R$, an $h$-transformed, time-inhomogeneous
jump chain with the \emph{same jump structure} as the reference. Fortet's
contraction argument applies directly, and convergence of the
alternating scheme of Remark~\ref{rem:sinkhorn} is geometric in
Hilbert's projective metric \citep{Georgiou2015}; in matrix language the
scheme is diagonal matrix scaling, and we shall say no more about
computation.

\textbf{The logarithmic-mean chain rule.}
The ODE-level structure is more interesting, because it is here that the
nonlocal world first genuinely differs from the local one. The local
theory rests on the chain rule $\nabla\rho=\rho\,\nabla\log\rho$, which
dictates the mobility $\rho$ in the continuity equation. On a graph there is no chain rule: a difference
of values is not a value times a difference of logarithms. The
resolution is an interpolation. For $a,b>0$ let
\begin{equation}
\Lam(a,b)=\frac{a-b}{\log a-\log b}\qquad(\Lam(a,a)=a)
\label{eq:logmean}
\end{equation}
denote the \emph{logarithmic mean}; then, identically,
\begin{equation}
\rho_y-\rho_x=\Lam(\rho_x,\rho_y)\,
\bigl(\log\rho_y-\log\rho_x\bigr).
\label{eq:chainrule}
\end{equation}
Identity \eqref{eq:chainrule} is the discrete chain rule, and the
logarithmic mean is the price of nonlocality.


\begin{theorem}[{Master equation as gradient flow;
\citealp[Thm.~1.2 and Cor.~4.8]{maas2011}}]
\label{thm:masterGF}
Define, for each edge, the mobility
$w_{xy}(\rho)=\Lam(\rho_x/\pi_x,\rho_y/\pi_y)\,q_{xy}\pi_x$ (in the
uniform case simply $\Lam(\rho_x,\rho_y)q_{xy}$). Then the master
equation
\begin{equation}
\dot\rho_x=\sum_{y}\bigl(q_{yx}\rho_y-q_{xy}\rho_x\bigr)
\label{eq:master}
\end{equation}
is the gradient flow of the relative entropy
$H(\rho\,|\,\pi)=\sum_x\rho_x\log\frac{\rho_x}{\pi_x}$ with respect to
the transport metric of Definition~\ref{def:maasmetric} below, built
from the mobility $w_{xy}$. The logarithmic mean is the \emph{unique}
mean for which this statement holds.
\end{theorem}

Proof (Sketch): 
Write
\[
\frac{\partial H}{\partial \rho_x}
=\log\frac{\rho_x}{\pi_x}+1.
\]
The additive constant disappears after taking an edge difference. With
the logarithmic-mean mobility and detailed balance,
\[
w_{xy}(\rho)
\left(
\log\frac{\rho_y}{\pi_y}
-
\log\frac{\rho_x}{\pi_x}
\right)
=
q_{xy}\pi_x
\left(
\frac{\rho_y}{\pi_y}
-
\frac{\rho_x}{\pi_x}
\right)
=
q_{yx}\rho_y-q_{xy}\rho_x .
\]
Hence the negative metric gradient of the entropy is exactly the
right-hand side of \eqref{eq:master}, which gives the master equation.
This is the finite-state counterpart of
$\rho\nabla\log\rho=\nabla\rho$. Conversely, if a symmetric two-point
mobility $\theta(a,b)$ is required to turn the entropy difference
$\log b-\log a$ into the linear difference $b-a$ edge by edge, then
\[
\theta(a,b)(\log b-\log a)=b-a,
\]
so necessarily
$\theta(a,b)=\Lambda(a,b)=(b-a)/(\log b-\log a)$. Thus the logarithmic
mean is forced by the chosen entropy and the quadratic transport
structure. This is the core calculation in
\citep[Thm.~1.2 and Cor.~4.8]{maas2011}; see also
\citep{mielke2011,chow2012}.

\subsection{The Geometric Viewpoint I: First-Order}
\label{subsec:discrete-geom}

\begin{definition}[Discrete transport metric
{\citep{maas2011, mielke2011}}]\label{def:maasmetric}
A discrete vector field is an antisymmetric edge function
$V_{xy}=-V_{yx}$; the discrete continuity equation is
$\dot\rho_x+\sum_y V_{xy}=0$. For
$\rho^0,\rho^1\in\Prob(\mathcal X)$ set
\begin{equation}
\WW(\rho^0,\rho^1)^2
=\inf_{(\rho_t,V_t)}
\left\{\int_0^1\sum_{x<y}
\frac{V_{xy}(t)^2}{w_{xy}(\rho_t)}\,\dd t
\;:\;\dot\rho_x+\sum_y V_{xy}=0,\
\rho_{t=0}=\rho^0,\ \rho_{t=1}=\rho^1\right\},
\label{eq:discBB}
\end{equation}
with $w_{xy}$ the logarithmic-mean mobility of
Theorem~\ref{thm:masterGF}.
\end{definition}

Formula \eqref{eq:discBB} is a verbatim discrete Benamou--Brenier
formula: kinetic action, mobility-weighted, minimised under a continuity
constraint. Its analytical content, however, is finite-dimensional: on
the interior of the simplex, $\WW$ is the geodesic distance of a genuine
Riemannian metric with smooth, explicit metric tensor. Geodesics solve
finite-dimensional Hamiltonian ODEs; existence and (interior) smoothness
are elementary. The metric degenerates at the boundary of the simplex,
but in sharp contrast with regime~(i) of the continuous nonlocal classification in Proposition~\ref{prop:expel} below, the boundary remains at \emph{finite} distance.

The construction extends beyond linear equations. Replacing the
logarithmic mean in \eqref{eq:discBB} by the mean
$\varsigma_\phi(a,b)$ adapted to a convex function $\phi$, Erbar and
Maas \citep{erbarmass2014} exhibited the discrete porous medium
equation
\[
\dot\rho_x=\sum_y q_{xy}\bigl(\phi(\rho_y)-\phi(\rho_x)\bigr)
\]
as the gradient flow of the corresponding internal energy: the entire
gradient-flow zoo of Sect.~\ref{subsec:local-otto} has a discrete
counterpart, mean by mean.

Finally, the discrete and local geometries are not merely analogous but
connected: on the lattice $\tfrac1N\mathbb Z^d\cap[0,1]^d$ with
nearest-neighbour rates, the metrics $\WW_N$ converge in the
Gromov--Hausdorff sense to $W_2$ on $\Prob([0,1]^d)$ as $N\to\infty$
\citep{gigli2013}. This is the discrete route back to the local world;
the continuous nonlocal route --- localisation of the jump kernel ---
is Theorem~\ref{thm:localization}.

\subsection{The Geometric Viewpoint II: Second-Order}
\label{subsec:discrete-second}

Because $(\mathring{\Prob}(\mathcal X),\WW)$ is a finite-dimensional
Riemannian manifold with an explicit metric tensor, the second-order
theory is a matter of calculus rather than of definition. Three layers
deserve record.

\textbf{Curvature.} Following Lott--Sturm--Villani in reverse, Erbar and Maas
\citep{erbar2012} \emph{define} the entropic Ricci curvature of the
chain as the best constant $\kappa$ such that the entropy is
$\kappa$-convex along $\WW$-geodesics,
\begin{equation}
\Hess_{\WW} H(\rho\,|\,\pi)\;\ge\;\kappa\,\WW\text{-metric},
\label{eq:entropicricci}
\end{equation}
and then, crucially, \emph{compute} it: the Hessian on the left is an
explicit finite-dimensional quadratic form (no Bochner formula is
needed --- this is the decisive advantage of finite dimension), and
lower bounds follow for concrete chains. For the discrete hypercube
$\{0,1\}^n$ with the Erbar--Maas simple-random-walk normalisation
$K_n(x,y)=1/n$ for Hamming neighbours (equivalently, continuous-time
jump rate $1/n$ along each of the $n$ coordinate edges, so the total
jump rate is $1$), the sharp bound is $\kappa=2/n$; discrete tori and
birth--death chains are likewise explicit \citep{erbar2012}. Through
Sect.~\ref{subsec:discrete-fi}, each such bound yields MLSI, Talagrand
and HWI inequalities with explicit constants.

\textbf{Connection, parallel transport, Jacobi fields.}
On the interior of the simplex the Christoffel symbols of the metric
\eqref{eq:discBB} are explicit rational expressions in $\rho$ (rational
in $\rho$ and in the values $\Lam(\rho_x,\rho_y)$ and their partial
derivatives), so parallel transport, geodesic deviation, Jacobi fields
and conjugate points are all governed by finite-dimensional linear ODEs
along geodesics. These constructions follow from finite-dimensional Riemannian calculus, and that is precisely the
point to be contrasted with Sect.~\ref{subsec:levy-second}, where none
of these objects is even defined.

\textbf{Hamiltonian flows and Newton's equation.}
The discrete analogue of the second-order dynamics of
Sect.~\ref{subsec:local-second} is furnished by Wasserstein Hamiltonian
flows on graphs \citep{li2019wassersteinhamiltonian}: for a Hamiltonian
$\mathcal H(\rho,\phi)=\tfrac12\sum_{x<y}w_{xy}(\rho)
(\phi_x-\phi_y)^2+\mathcal V(\rho)$ the canonical equations
\begin{equation}
\dot\rho=\partial_\phi\mathcal H,\qquad
\dot\phi=-\partial_\rho\mathcal H
\label{eq:graphhamilton}
\end{equation}
reproduce, for $\mathcal V=0$, the geodesic equations of
\eqref{eq:discBB}, and for the Fisher-information potential
$\mathcal V=-\tfrac{\eps^2}{4}\mathcal I$ they reproduce the discrete
Schr\"odinger bridge: the graph analogue of Newton's equation
\eqref{eq:newton} is not an open problem but a system of $2|\mathcal X|$
ODEs. Every term of \eqref{eq:newton} --- acceleration, Fisher force,
$\eps^2$ scaling --- has an exact finite-dimensional counterpart.

\subsection{Functional Inequalities}
\label{subsec:discrete-fi}

The triangle of inequalities of
Sect.~\ref{subsec:local-first-complete} transfers to graphs with the
metric $\WW$ in the role of $W_2$. The modified logarithmic Sobolev
inequality $\mathrm{MLSI}(\lambda)$ asserts
$H(\rho\,|\,\pi)\le\frac{1}{2\lambda}\mathcal I(\rho\,|\,\pi)$ with the
discrete Fisher information
$\mathcal I(\rho\,|\,\pi)=\tfrac12\sum_{x,y}
(\rho_x-\rho_y)\bigl(\log\tfrac{\rho_x}{\pi_x}
-\log\tfrac{\rho_y}{\pi_y}\bigr)q_{xy}\pi_x$
(here written for the uniform case), and is equivalent to exponential
entropy decay along the master equation. The discrete Talagrand
inequality $\mathrm T_{\WW}(\lambda)$:
$\tfrac{\lambda}{2}\WW(\rho,\pi)^2\le H(\rho\,|\,\pi)$, and a discrete
HWI inequality interpolate exactly as in the local case, and the
implications $\mathrm{Ric}\ge\kappa>0\Rightarrow\mathrm{MLSI}(\kappa)
\Rightarrow\mathrm T_{\WW}(\kappa)$ hold with the \emph{entropic Ricci
curvature} of the preceding subsection in the role of the Bakry--\'Emery
bound \citep{erbar2012}. The parallel with
Sect.~\ref{subsec:local-first-complete} is exact, including its closing
sentence: the optimal constants are curvatures, and on a finite graph
the curvature can actually be computed.

\section[The Nonlocal Theory II: Continuous State Spaces under L\'evy Noise]{The Nonlocal Theory II: Continuous State Spaces under L\'evy
Noise}
\label{sec:levy}

{  
The continuous nonlocal theory replaces the Brownian reference of
Sect.~\ref{sec:local} by a L\'evy process---a pure-jump process or a
jump-diffusion---and the finite graph of Sect.~\ref{sec:discrete} by
$\R^d$. Stable-type generators have fractional order
$\alpha\in(0,2)$; this spatial order should not be conflated with the
geometric ``second order'' of connections and curvature. The section
follows the arc of its predecessors: genuine jump Schr\"odinger bridges
are developed at the stochastic level
(Sect.~\ref{subsec:levy-stoch}); their forward--backward PIDE
description and extensions with killing appear in
Sect.~\ref{subsec:levy-pide}; and Sect.~\ref{subsec:levy-geom}
develops a related first-order nonlocal transport geometry. The latter
has its own topological classification and localization limit to $W_2$,
but it should not be identified with the bridge problem itself unless a
separate variational or limiting theorem establishes that connection.
Importantly, Proposition~\ref{prop:expel} does not give a trichotomy:
an intermediate kernel regime remains open. The obstruction to a
general second-order theory is tied instead to the infinite-dimensional,
density-dependent nonlocal operator and the absence of the
regularity/Bochner structure used in the local theory;
Sect.~\ref{subsec:levy-second} records these limitations and open
problems.

Throughout Sect.~\ref{sec:levy} we use the following taxonomy.
\begin{enumerate}
\item A \emph{L\'evy or jump-diffusion Schr\"odinger bridge} is a
path-space relative-entropy minimisation problem with a jump-process
reference law and prescribed endpoint marginals. The $h$-transform and
forward--backward PIDE results below belong to this class.
\item A \emph{nonlocal transport metric} such as
$\WW_{J,\Theta}$ is a variational geometry on curves of marginals
defined through a nonlocal continuity equation and kinetic action.
It is closely motivated by jump processes, but it is a different
variational problem from the path-space Schr\"odinger bridge.
\item Fractional mean-field control, nonlocal-in-time control, and other
jump-inspired transport models are \emph{neighboring models} unless an
explicit entropy-minimisation or bridge theorem has been proved.
They are included only for comparison and are not cited as evidence for
a L\'evy Schr\"odinger-bridge theorem.
\end{enumerate}
}

%

\subsection{The Stochastic-Analysis Viewpoint}
\label{subsec:levy-stoch}

The constructive theory of bridges for jump processes is due to Privault
and Zambrini~\citep{privault2004}, motivated by Euclidean quantum
mechanics in the momentum representation, where L\'evy generators arise
naturally. Their construction proceeds in two stages. Fix a reference
measure $\lambda$ for which the perturbed L\'evy generator $H=U+V(\nabla)$
and its $\lambda$-adjoint $H^\dagger$ are mutually adjoint, and let
$h(t,k,u,\dd l)$, $h^\dagger(s,\dd j,t,k)$ be the integral kernels of the
Feynman--Kac semigroups $e^{-(u-t)H}$, $e^{-(t-s)H^\dagger}$. The first
theorem builds a process from a pair of positive boundary functions.

\begin{theorem}[{  Markovian Bernstein bridge;
\citealp[Thm.~3.2]{privault2004}}]\label{thm:levy_bridge}
Let $\eta_r^\ast,\eta_v:\R^d\to\R_+$ be $\lambda$-a.e.\ strictly positive
and normalised so that, for some $t\in[r,v]$ (and then for all such $t$),
$\int_{\R^d}\eta_t^\ast\eta_t\dd\lambda=1$, where
$\eta_t^\ast=e^{-(t-r)H^\dagger}\eta_r^\ast$ and
$\eta_t=e^{-(v-t)H}\eta_v$. Then there exists an $\R^d$-valued process
$(z_t)_{t\in[r,v]}$ whose law at time $t$ has $\lambda$-density
$\rho_t=\eta_t^\ast\eta_t$, which is both forward and backward Markov,
with transition kernels
\begin{equation}
p(t,k,u,\dd l)=\frac{\eta_u(l)}{\eta_t(k)}\,h(t,k,u,\dd l),
\qquad
p^\ast(s,\dd j,t,k)=\frac{\eta_s^\ast(j)}{\eta_t^\ast(k)}\,
h^\dagger(s,\dd j,t,k),
\end{equation}
initial and terminal laws $\eta_r\eta_r^\ast\cdot\lambda$ and
$\eta_v\eta_v^\ast\cdot\lambda$, and boundary functions solving the
adjoint equations
$-\partial_t\eta_t^\ast=H^\dagger\eta_t^\ast$,
$\partial_t\eta_t=H\eta_t$. Moreover $(z_t)_{t\in[r,v]}$ is a Bernstein
process.
\end{theorem}

\begin{definition}[Bernstein (reciprocal) process;
{\cite{privault2004}}]\label{def:bernstein}
A process $(z_t)_{t\in[r,v]}$ is a \emph{Bernstein process} if, for all
$r\le s<t<u\le v$,
\begin{equation}
\mathbb P\bigl(z_t\in\dd k\mid\mathcal P_s\vee\mathcal F_u\bigr)
=\mathbb P\bigl(z_t\in\dd k\mid z_s,z_u\bigr),
\end{equation}
where $(\mathcal P_t)$ and $(\mathcal F_t)$ are the forward and backward
filtrations generated by the process. Every Markov process is a
Bernstein process; the converse fails.
\end{definition}

{  
\begin{remark}[Bernstein and reciprocal processes]\label{rem:bernstein-history}
The terminology goes back to S.~Bernstein's 1932 work; Jamison subsequently developed the modern
measure-theoretic theory of \emph{reciprocal processes}
\citep{jamison1974,jamison1975schrodinger}.  Reciprocity is the
two-sided conditional-independence property in
Definition~\ref{def:bernstein}; it is not synonymous with the Markov
property. Every Markov process is reciprocal, but a general reciprocal
process need not be Markov.  The bridges constructed in
Theorem~\ref{thm:levy_bridge} form a Markovian subclass because their
transition laws arise from the forward--backward $h$-transform
structure.  We therefore use ``Bernstein (reciprocal) process'' for the
general class and ``Markov bridge'' only when the additional Markov
property has actually been established.
\end{remark}
}

To run the construction from prescribed marginals rather than from
boundary functions, one inverts the product structure; this is the
existence theorem for the Schr\"odinger system, going back to Fortet,
Beurling and Jamison.

\begin{theorem}[{  Boundary factors for prescribed marginals;
\citealp[Thm.~3.3]{privault2004}}]\label{thm:levymarginals}
Let $\pi_r,\pi_v$ be probability measures on $\R^d$, and suppose
$h(r,i,v,m)$ is continuous in $(i,m)$ and strictly positive. Then there
exist measures $\eta_r^\ast(\dd i)$ and $\eta_v(\dd m)$ such that
\begin{equation}
\pi_r(\dd i)=\eta_r^\ast(\dd i)\!\int_{\R^d}\!h(r,i,v,m)\,\eta_v(\dd m),
\qquad
\pi_v(\dd m)=\eta_v(\dd m)\!\int_{\R^d}\!h(r,i,v,m)\,\eta_r^\ast(\dd i).
\end{equation}
\end{theorem}

Combining Theorems~\ref{thm:levy_bridge} and~\ref{thm:levymarginals}
produces, for any admissible pair of marginals, a Markovian bridge of the
L\'evy process: solve the coupled system for the boundary functions, then
feed them into the first theorem. Two features of this construction
deserve emphasis, and we state them as prose rather than as theorems
because they summarise the thrust of the paper rather than quote a single
result. First, the resulting process is the $h$-transform of the
reference L\'evy process from both time endpoints, so its law at time $t$
factorises as $\eta_t^\ast\eta_t$ --- the probabilistic image of writing a
quantum density as $\psi^\ast\psi$, which is the point of contact with
Euclidean quantum mechanics; the adjoint pair
$-\partial_t\eta_t^\ast=H^\dagger\eta_t^\ast$,
$\partial_t\eta_t=H\eta_t$ is the nonlocal counterpart of the
forward--backward heat system \eqref{eq:fk}--\eqref{eq:bk} of the
Brownian theory. Second, Privault and Zambrini further show that the
processes so constructed are, under suitable conditions on the Bernstein
kernel, \emph{exactly} the reciprocal Markov processes with jumps, and
they characterise them variationally through a stochastic control problem
whose value function solves a Hamilton--Jacobi--Bellman equation ---
the jump analogue of the second-order structure of
Sect.~\ref{subsec:local-second}, not a relative-entropy minimisation.\\

We now give several types of bridges as examples. \\

{  
(i) \emph{Brownian}: for a Brownian motion with covariance matrix
$\Sigma$ per unit time, conditioned on $X_s=x$ and $X_u=y$ with
$s<t<u$, the bridge is Gaussian with
\[
\mathbb E[X_t\mid X_s=x,X_u=y]
=x+\frac{t-s}{u-s}(y-x),\qquad
\operatorname{Cov}(X_t\mid X_s,X_u)
=\frac{(t-s)(u-t)}{u-s}\,\Sigma .
\]
The previously displayed formula is the special case $s=0$.\\

(ii) \emph{Poisson}: let $(N_t)$ be a homogeneous Poisson process of
rate $\lambda$, conditioned on $N_s=n_s$ and $N_u=n_u$ with
$n_s\le n_u$. Its bridge is a time-inhomogeneous pure-birth Markov
process with predictable jump intensity
\[
q_t(n,n+1)=\frac{n_u-n}{u-t},\qquad s<t<u,
\]
when the current count is $N_{t-}=n$. This follows directly from the
Doob $h$-transform since
$h(t,n)=\mathbb P(N_u=n_u\mid N_t=n)$; the original rate $\lambda$
cancels in the ratio $\lambda h(t,n+1)/h(t,n)$. The bridge is Markov
(and hence reciprocal), but it is not reversible in the detailed-balance
sense; its time reversal is instead a time-inhomogeneous pure-death
Markov bridge.\\
}\par

(iii) \emph{Stable}: bridges of symmetric $\alpha$-stable processes
inherit self-similarity, with transition densities obtained by
conditioning the stable kernel on the endpoints. \\

\textbf{The quantum motivation.}
{  
Here ``quantum motivation'' refers specifically to the Euclidean
semigroup/Feynman--Kac representation of a Schr\"odinger-type operator;
it does not invoke Nelson stochastic mechanics or a Bohm-potential
reconstruction.
}
For the Euclidean Schr\"odinger equation in momentum representation,
$\partial_t\eta=-(\Psi(p)+V)\eta$ with pseudo-differential kinetic symbol
$\Psi$, there is a Feynman--Kac representation
\begin{equation}
\eta(t,p)=\mathbb E\Bigl[\eta(T,\xi_T)\,
e^{\int_t^T c(s,\xi_s)\dd s}\,\Big|\,\xi_t=p\Bigr]
\label{eq:levyFK}
\end{equation}
over a L\'evy process $\xi$ with exponent $\Psi$
\citep{garbaczewski1995, privault2004}: the bridges of
Theorem~\ref{thm:levy_bridge} provide probabilistic propagators for
nonlocal Schr\"odinger operators, generalising the diffusion picture of
Sect.~\ref{subsec:local-extensions} to nonlocal kinetic energies.

\subsection{The Partial Integro-Differential Equation Viewpoint}
\label{subsec:levy-pide}

Consider the controlled jump-diffusion
\begin{equation}
\dd X_t=b(t,X_t,u_t)\dd t+\sigma(t,X_t)\dd W_t
+\int_{\R^d}\gamma(t,X_{t^-},z)\,\tilde N(\dd t,\dd z),
\label{eq:jumpdiff}
\end{equation}
{  
with $\tilde N$ the compensated Poisson random measure. To avoid
overloading the symbol $V$, throughout this subsection we denote HJB
value functions by $\Phi$ and reserve $V$ for physical or killing
potentials. Thus the value function $\Phi$ solves the backward
integro-differential HJB equation
$\partial_t\Phi+\sup_u L^u\Phi=0$ with the nonlocal generator
}
\begin{equation}
L^u\Phi=b\!\cdot\!\nabla\Phi
+\tfrac12\mathrm{Tr}(\sigma\sigma^{\!\top}\nabla^2\Phi)
+\int_{\R^d\setminus\{0\}}\!\!
\bigl[\Phi(x+\gamma)-\Phi(x)
-\gamma\!\cdot\!\nabla\Phi\,\mathbf 1_{|z|\le1}\bigr]\nu(\dd z),
\label{eq:nonlocalgen}
\end{equation}
and a verification theorem holds under standard smoothness hypotheses
{  \citep{Oksendal2019}}.
{  
Dually, the density solves a forward Fokker--Planck PIDE. For
stable-type jump kernels the nonlocal part has fractional order
$\alpha\in(0,2)$, while for a general jump diffusion its structure is
determined by the L\'evy integral in \eqref{eq:nonlocalgen}. This
forward--backward pairing is the nonlocal analogue of
\eqref{eq:fphjb}; the word ``second order'' below is therefore reserved
for geometric Hessians, connections, and accelerations, not for the
spatial order of this PIDE.
}\\

\textbf{Bridges for jump-diffusions.}
Zlotchevski and Chen \citep{zlotchevski2024} study the Schr\"odinger
bridge problem --- the minimisation of the relative entropy
$\mathrm{KL}(\mathbf P\,\|\,\mathbf R)$ over path measures with prescribed
endpoint marginals --- when the reference $\mathbf R$ is the law of a jump
diffusion. Writing $(\mathfrak f,\mathfrak g)$ for the solution of the
associated Schr\"odinger system and
$\mathfrak h(t,x)=\int\mathfrak g(y)\,P_{t,T}(x,\dd y)$, their central
structural result is that the bridge is an $h$-transform of a reference
path measure,
\begin{equation}
\hat{\mathbf P}=\frac{\mathfrak h(T,X_T)}{\mathfrak h(0,X_0)}\,\mathbf P ,
\label{eq:zchtransform}
\end{equation}
and that, when $\mathfrak h$ is harmonic, $\hat{\mathbf P}$ solves the
martingale problem for a jump-diffusion generator $L^{\mathfrak h}$
obtained from $L$ by a shift of the drift and a rescaling of the L\'evy
measure by $\mathfrak h(t,x+\gamma)/\mathfrak h(t,x)$. Thus the bridge
remains a jump diffusion of the same type as the reference; the
$h$-transform structure of the Brownian theory
(Theorem~\ref{thm:fg}) survives intact in the nonlocal world. When the
reference admits a sufficiently regular transition density, the pair
$(\varphi,\hat\varphi)$ representing the bridge solves a forward--backward
system of PIDEs generalising the local Schr\"odinger system, and the
marginal density of $\hat{\mathbf P}$ is the product $\varphi\hat\varphi$
\citep{zlotchevski2024}. What one loses relative to the Brownian case is
only the Hopf--Cole \emph{gain}: the transform still linearises the
problem, but the resulting equations are nonlocal. \\


{  
The preceding paragraph concerns a genuine path-space Schr\"odinger
bridge. We now separate one genuine extension from a neighboring
control model so that the scope of the literature remains explicit.
}\\

\textbf{Killing and unbalanced marginals.}
When the marginals carry unequal mass, no conservative process can
connect them. Chen, Georgiou and Pavon \citep{chen2021unbalanced}
resolved this by allowing a killing rate $V(t,X_t)$ --- a jump to an
absorbing graveyard state --- and showed that the most likely evolution
of diffusing-and-vanishing particles is again a Schr\"odinger bridge
after embedding into a conservative system on the augmented state
space. Since killing \emph{is} a jump, unbalanced transport is
naturally a chapter of the nonlocal theory, with applications from
single-cell biology to population dynamics.

\textbf{A neighboring nonlocal-in-time control model.}
Not every fractional transport/control model should be presented as a Schr\"odinger bridge. A verified neighboring example is the time-fractional mean-field control problem of Zheng, Yang, Li and Wang \citep{zheng2023timefractional}. Its first-order optimality system couples a time-fractional transport equation with a backward time-fractional Hamilton--Jacobi equation and is designed for subdiffusive advective transport.

\subsection{The Geometric Viewpoint I: First-Order}
\label{subsec:levy-geom}

The geometric viewpoint asks for the analogue, under a jump reference,
of Otto's construction. The answer --- a nonlocal Benamou--Brenier
formula --- was developed by Erbar \citep{erbar2014} for the
gradient-flow structure and by Slep\v cev and Warren
\citep{slepcev2023nonlocal} in the general metric setting; we follow the latter.\\

\textbf{Nonlocal vector calculus.}
Fix a symmetric jump kernel $J(x,y)=J(y,x)\ge0$ and an interpolation
function $\Theta:[0,\infty)^2\to[0,\infty)$, symmetric, jointly concave,
positively $1$-homogeneous. The nonlocal gradient of a function is
$\ngrad\phi(x,y)=\phi(y)-\phi(x)$; the nonlocal divergence of an
antisymmetric flux $V(x,y)=-V(y,x)$ is
$\ndivg V(x)=\int V(x,y)J(x,y)\dd y$, and a nonlocal integration by
parts holds. A curve of measures satisfies the nonlocal continuity
equation
\begin{equation}
\partial_t\rho_t+\ndivg V_t=0 .
\label{eq:nlcontinuity}
\end{equation}

\begin{definition}[Nonlocal Wasserstein distance
{\citep{slepcev2023nonlocal, erbar2014}}]\label{def:nlW}
\begin{equation}
\begin{aligned}
\WW_{J,\Theta}(\mu_0,\mu_1)^2
=\inf\Bigl\{\frac12\int_0^1\!\!\iint
&\frac{|V_t(x,y)|^2}{\Theta\bigl(\rho_t(x),\rho_t(y)\bigr)}
\,J(x,y)\dd x\dd y\dd t\;:\\[-1mm]
&\ \partial_t\rho_t+\ndivg V_t=0 ,\ \rho_0=\mu_0,\ \rho_1=\mu_1\Bigr\}.
\end{aligned}
\label{eq:nlBB}
\end{equation}
\end{definition}

{  
In \eqref{eq:nlBB} the double integral is over ordered pairs $(x,y)$;
the factor $1/2$ compensates for this double counting. On a finite graph,
\eqref{eq:discBB} uses the sum $\sum_{x<y}$ over unordered edges and
therefore carries no additional factor $1/2$. Moreover, if
$m(x,y):=J(x,y)V(x,y)$ denotes the mass flux, then the nonlocal action
can equivalently be written with integrand
$m(x,y)^2/[J(x,y)\Theta(\rho(x),\rho(y))]$.

For the logarithmic mean $\Theta=\Lam$, the jump semigroup with kernel
$J$, which is \ the fractional heat semigroup for
$J(x,y)=|x-y|^{-d-2s}$, is the gradient flow of the entropy in this
metric \citep{erbar2014}, extending
Theorem~\ref{thm:masterGF} from graphs to $\R^d$; for
translation-invariant kernels the induced distance is comparable to a
fractional Sobolev norm. The framework simultaneously contains the
discrete metric of Definition~\ref{def:maasmetric} (take a graph as base
space) and, as we shall see, the classical $W_2$ as a limit. Notice,
however, the discrete lesson at work: only $\Theta=\Lam$ yields the
gradient-flow identification, while the metric theory below is
developed for general $\Theta$.\\

\textbf{The topology of $\WW_{J,\Theta}$.}
The topology induced by $\WW_{J,\Theta}$ depends delicately on the
weight kernel near the origin and on the boundary value $\Theta(1,0)$ of
the interpolation, through the cost of expelling mass from a Dirac. The
key estimate is the following (in the notation of
Definition~\ref{def:nlW}, $J$ is the weight kernel and $\Theta$ the
interpolation, both assumed isotropic and to satisfy the standing
structural assumptions).

\begin{proposition}[{\cite[Prop.~1.1]{slepcev2023nonlocal}}]
\label{prop:expel}
Let $\nu$ be any compactly supported probability measure singular to
$\delta_0$. Depending on $J$ and $\Theta$:
\begin{enumerate}
\item[\rm(i)] if $\Theta(1,0)=0$ and $\int_{B(0,1)}J(|y|)\dd y<\infty$,
then $\WW_{J,\Theta}(\delta_0,\nu)=\infty$;
\item[\rm(ii)] if $\Theta(1,0)>0$ and $\int_{B(0,1)}J(|y|)\dd y<\infty$,
then $\infty>\WW_{J,\Theta}(\delta_0,\nu)\ge\bigl(2\int_{\R^d}J(|y|)\dd
y\bigr)^{-1/2}$;
\item[\rm(iii)] if $J$ has an algebraic singularity at the origin, i.e.\
$J(|y|)\ge c\,|y|^{-d-s}$ for some $s>0$ and small $|y|$, then
$\WW_{J,\Theta}(\delta_0,m_{B(0,\delta)})\le C\delta^{s/2}$, so that
$\inf_{\nu\perp\delta_0}\WW_{J,\Theta}(\delta_0,\nu)=0$.
\end{enumerate}
\end{proposition}

Slep\v cev and Warren emphasise that Proposition~\ref{prop:expel} is
\emph{not} a trichotomy: the intermediate case in which
$\int_{B(0,1)}J(|y|)\dd y=\infty$ but $J$ grows no faster than $|y|^{-d}$
near the origin is left open. The two integrable regimes (i)--(ii) are
sharpened into a global topological statement.

\begin{theorem}[{  Topology of $\WW_{J,\Theta}$;
\citealp[Thm.~1.2]{slepcev2023nonlocal}}]\label{thm:swtopology}
Suppose $J$ and $\Theta$ satisfy the standing assumptions.
If $\Theta(1,0)>0$ and $\int_{B(0,1)}J(|y|)\dd y<\infty$, then on a
compact domain there is a constant $C$ with
\[
\tfrac1C\,\mathrm{TV}(\mu,\nu)^{1/2}\le\WW_{J,\Theta}(\mu,\nu)
\le C\,\mathrm{TV}(\mu,\nu)^{1/2}
\qquad\text{for all }\mu,\nu ,
\]
so $\WW_{J,\Theta}$ metrizes the \emph{strong} topology. Conversely, if
$\int_{B(0,1)}J(|y|)\dd y=\infty$ while $J$ has finite second moment,
then one has only the lower bound
$W_1(\mu,\nu)\le C\,\WW_{J,\Theta}(\mu,\nu)$; and if moreover
$J(|y|)\ge c\,|y|^{-d-s}$ for some $s>0$ near the origin, then on a
compact domain $\WW_{J,\Theta}$ metrizes the \emph{weak} topology.
\end{theorem}

{  
Proof sketch.
The proof combines two comparison mechanisms. The lower estimates are
obtained by testing the nonlocal continuity equation and estimating the
resulting flux by the action; this yields control of $W_1$ when the
kernel has finite second moment and, in the integrable regime with
$\Theta(1,0)>0$, quantitative control of total variation. For the upper
estimate in the integrable regime, one decomposes the signed difference
of the two measures and transports the positive and negative parts by
spreading and recombining mass. The finite cost of moving mass away from
a point, quantified by Proposition~\ref{prop:expel}, gives
$\WW_{J,\Theta}\lesssim \mathrm{TV}^{1/2}$ on compact sets. When the
kernel has a non-integrable algebraic singularity, the corresponding
short-range spreading estimate makes weakly nearby measures cheap to
connect; together with the $W_1$ lower bound and compactness this yields
weak-topology metrization. These ingredients are organized in
Slep\v cev--Warren, Sections~3.1--3.3
\citep[Thm.~1.2]{slepcev2023nonlocal}.
}

Read together, these results show that the geometry of $\WW_{J,\Theta}$
is governed by two independent switches --- whether the kernel $J$ is
integrable at the origin, and whether the interpolation $\Theta$
vanishes on the boundary. When $J$ is integrable and $\Theta(1,0)=0$, as
for the logarithmic mean, the space is highly disconnected: expelling a
Dirac costs infinite action, so a delta mass cannot be moved at finite
cost. When $J$ is integrable but $\Theta(1,0)>0$, as for the arithmetic
mean, the distance is bi-Lipschitz to $\sqrt{\mathrm{TV}}$ and metrizes
the strong topology. When $J$ has a non-integrable algebraic singularity
$J(|y|)\gtrsim|y|^{-d-s}$, as for the $s$-stable jump kernel driving
fractional diffusion, it metrizes the weak topology, exactly as the
classical Wasserstein distance does. The disconnectedness of the first
regime has no analogue on a finite graph
(Sect.~\ref{subsec:discrete-geom}): it is a purely infinite-dimensional
phenomenon, and the first sign that second-order analysis will be
delicate.\\

\textbf{The Localization Limit.}
The precise bridge back to the local world rescales the kernel:
$J_\eps(x-y)=\eps^{-d}J\bigl((x-y)/\eps\bigr)$ concentrates jumps at
scale $\eps$.

\begin{theorem}[{  Localization to $W_2$;
\citealp[Thm.~1.3]{slepcev2023nonlocal}}]
\label{thm:localization}
Assume $J$ has finite second moment $M_2(J)$, and let $\mu,\nu$ be
supported in a fixed domain. There is a constant $C$ such that for all
$\eps\in(0,1]$
\[
\sqrt{\tfrac{M_2(J)}{2d}}\;\eps\,\WW_{J_\eps,\Theta}(\mu,\nu)
\;\le\;\bigl(1+\sqrt\eps\bigr)^2 W_2(\mu,\nu)+C\sqrt\eps ,
\]
\[
W_2(\mu,\nu)^2\;\le\;\tfrac{M_2(J)}{2d}\,\eps^2\,
\WW_{J_\eps,\Theta}(\mu,\nu)^2+C\sqrt\eps .
\]
In particular, on measures supported in a compact domain,
$\sqrt{M_2(J)/2d}\;\eps\,\WW_{J_\eps,\Theta}\to W_2$ as $\eps\to0$ in the
sense of Gromov--Hausdorff. The upper bound is proved by convolving the
optimal $W_2$ interpolation into an exact solution of the nonlocal
continuity equation; the lower bound follows from a dual formulation via
nonlocal Hamilton--Jacobi \emph{subsolutions},
\begin{equation}
\partial_t\phi(t,x)+\frac12\int
\bigl(\phi(t,y)-\phi(t,x)\bigr)^2 J(x,y)\dd y\;\le\;0 ,
\label{eq:nlHJ}
\end{equation}
the nonlocal replacement of the eikonal constraint
$\partial_t\phi+\tfrac12|\nabla\phi|^2\le0$ of the local duality.
\end{theorem}

{  
Proof sketch.
For the upper bound, start from an admissible Benamou--Brenier pair for
$W_2$. Slep\v cev and Warren construct from it an \emph{exact} solution
of the nonlocal continuity equation by convolution with a kernel
determined by $J_\varepsilon$; the resulting nonlocal action is then
compared with the local kinetic action. After normalization by the
second moment $M_2(J)$ this gives the first estimate (their
Corollary~4.3). The reverse estimate uses the dual side: suitably
smoothed nonlocal Hamilton--Jacobi subsolutions are compared with local
Hamilton--Jacobi subsolutions, and the resulting duality inequality
controls $W_2^2$ by the rescaled nonlocal action up to the
$O(\sqrt{\varepsilon})$ error (their Corollary~5.12). Combining the two
bounds yields the stated Gromov--Hausdorff convergence; see
\citep[Thm.~1.3]{slepcev2023nonlocal}.
}

\textbf{Neighbouring geometries.}
Four constructions extend the same variational template in different
directions, and we record them for orientation. \emph{Covariance-modulated
transport} \citep{burger2025} weights the kinetic action by the inverse
covariance of the current measure,
$\mathcal E=\tfrac12\int(v-\bar v)^{\!\top}\Sigma_\rho^{-1}(v-\bar v)
\dd\rho$, yielding the gradient-flow geometry of mean-field ensemble
Kalman methods, Gaussianity-preserving and weakly equivalent to $W_2$.
\emph{Configuration spaces}: on the space of point configurations,
Dello~Schiavo, Herry and Suzuki \citep{delloschiavo2023} built a
nonlocal continuity equation from the difference operator of the
Poisson space and proved a Benamou--Brenier formula together with Ricci
bounds for Poisson measures --- a rare instance where curvature
statements survive in an infinite-dimensional nonlocal setting, thanks
to the exact structure of the Poisson measure. \emph{Coarse extrinsic
curvature}: Li \citep{li2025coarse} used $W_1$-costs of small balls to
define a coarse mean curvature of submanifolds converging to the
classical one, bridging optimal transport and extrinsic geometry.
\emph{Graph limits}: for kernels supported on graphs the framework
recovers discrete transport distances, closing the triangle with
Sect.~\ref{sec:discrete}.

\subsection{The Geometric Viewpoint II: What Exists, What Breaks, and Open
Problems}
\label{subsec:levy-second}

On
the local side, the second-order theory rests on the Bochner formula;
on graphs, on finite dimensionality. The continuous nonlocal setting
has neither, and each of the four pillars of
Sects.~\ref{subsec:local-second}
and~\ref{subsec:discrete-second} fails for its own reason.\\

{  
\textbf{The diffusion-style Bochner route does not transfer verbatim.}
For a jump generator
\[
Lf(x)=\int\bigl(f(y)-f(x)\bigr)J(x,y)\,\dd y
\]
the carr\'e du champ $\Gamma$ and the iterated form $\Gamma_2$ are
well-defined, but the diffusion chain rule behind \eqref{eq:hessent} is
absent. Accordingly, one should not expect in general a decomposition
of $\Gamma_2$ into ``Hessian squared plus a Ricci tensor'' of the
classical diffusion type. For the fractional Laplacian on
$\mathbb R^d$, an explicit double-integral formula gives
$\Gamma_2\ge0$, hence $CD(0,\infty)$; what is known to fail is
$CD(\kappa,N)$ for every $\kappa\in\mathbb R$ and every finite
$N>0$ \citep{spener2020fractional}. Thus the relevant obstruction here
is not the nonexistence of $\Gamma_2$, but the loss of the
finite-dimensional diffusion-style Bochner/chain-rule mechanism.

\textbf{A general second-order transport calculus is not yet available.}
Locally, the elliptic equation
$-\nabla\!\cdot(\rho\nabla\phi)=\dot\rho$ gives a robust tangent
representation, and on a finite graph the corresponding inversion is
finite-dimensional. In the continuous nonlocal metric, one instead has
to invert a density-dependent operator of the form
$\phi\mapsto\ndivg(\Theta\,\ngrad\phi)$. The references used in this
chapter do not provide a general regularity and inversion theory strong
enough to construct a Levi--Civita connection, parallel transport, or
Jacobi fields for the full class of kernels considered here. This is a
statement about the present scope of the theory, not a theorem of
nonexistence.

\textbf{Topology can obstruct a uniform smooth-geometric interpretation.}
Proposition~\ref{prop:expel} and Theorem~\ref{thm:swtopology} show that
different kernels can produce disconnected, strong, or weak topologies.
Consequently, one cannot impose a single smooth Riemannian geodesic
calculus uniformly across all regimes. Even in the weak-topology
regime, the dual characterization used in this chapter is formulated
through Hamilton--Jacobi subsolutions rather than through a smooth
eikonal equation that could simply be differentiated twice.

\textbf{Nonquadratic dissipation is a structural warning, not an
obstruction theorem.}
Large-deviation structures for jump processes naturally lead to
generalized gradient systems with nonquadratic, including cosh-type,
dissipation potentials \citep{Peletier2022}. This observation motivates
the study of Finsler-type or more general dissipative geometries in
addition to quadratic nonlocal Wasserstein metrics. It does \emph{not},
by itself, prove that the quadratic metric $\WW_{J,\Theta}$ is the
``wrong'' geometry, nor does it rule out a useful second-order calculus
for special kernels.

The following questions are therefore best read as
\emph{author-formulated research directions} distilled from the gaps
between the local, discrete, and continuous nonlocal theories; they are
not quoted theorems from the cited literature.

\begin{openproblem}\label{op:connection}
Construct, under explicit analytic hypotheses on $J$ and $\Theta$, a
second-order calculus (connection, parallel transport, and Jacobi
equation) on $(\Prob(\mathbb R^d),\WW_{J,\Theta})$, or identify precise
conditions that obstruct such a construction.
\end{openproblem}

\begin{openproblem}\label{op:gamma2}
Identify a curvature or information criterion adapted to the nonlocal
transport metric---for example, an appropriate entropy-convexity or
generalized-dissipation condition---that yields useful displacement
convexity and sharpened functional inequalities in
Sect.~\ref{subsec:levy-fi}.
\end{openproblem}

\begin{openproblem}\label{op:newton}
Determine whether a Newton-type second-order equation can be derived
for a nontrivial class of L\'evy-driven Schr\"odinger bridges, and, if
so, identify the nonlocal information functional that plays the role of
the Fisher-information force in \eqref{eq:newton}.
\end{openproblem}

\begin{openproblem}\label{op:smallnoise}
Clarify the small-noise limit of L\'evy-driven Schr\"odinger bridges
and its relation, if any, to the nonlocal transport metrics
$\WW_{J,\Theta}$. In contrast with the Brownian case, this relation
should be established from path-space large-deviation and dissipation
structures rather than assumed a priori.
\end{openproblem}

These questions mark the boundary between results established in the
references and the research programme suggested by the comparison
developed in this chapter.
}

\subsection{Functional Inequalities
in Nonlocal Wasserstein Space}
\label{subsec:levy-fi}

The pattern of Sects.~\ref{subsec:local-first-complete}
and~\ref{subsec:discrete-fi} continues, with partial results. Erbar's
identification of jump semigroups as entropy gradient flows
\citep{erbar2014} yields energy-dissipation identities --- the
first-order half of the LSI story --- and, for stationary point
processes, Huesmann and Stange \citep{huesmann2025} proved that the
Ornstein--Uhlenbeck semigroup is the gradient flow of the specific
relative entropy for a nonlocal transport metric and derived HWI- and
Talagrand-type inequalities in that setting, together with the
existence of solutions of the nonlocal continuity equation between
arbitrary stationary point processes. What is conspicuously missing is
the third vertex of the local triangle: inequalities whose \emph{proof}
requires convexity of the entropy along geodesics --- displacement
convexity, HWI in its curvature form, sharpened Talagrand constants ---
because those proofs consume precisely the second-order input that, as
the preceding subsection explains, the continuous nonlocal setting does not
possess.

\section{Information Geometry and Generative Modelling}
\label{sec:information-geometry-generative}

Transport geometry measures the cost of moving probability mass, whereas
information geometry describes statistical distinguishability through
metrics, divergences, and dual connections. The Schr\"odinger bridge links
these viewpoints: it selects a probability evolution by minimising
relative entropy under endpoint constraints. We introduce the essential
information-geometric structures, explain the bridge as an entropy
projection, and connect this interpretation to transport-based
optimisation and generative modelling.

\subsection{Fisher Geometry and Divergences}
\label{sec:ig-foundations}

Consider a smooth, identifiable family of positive densities
$p_\theta(x)$, $\theta\in\Xi\subset\mathbb R^n$, with common support.
Assume differentiation under the integral is valid, all expectations
below are finite, and the Fisher matrix is positive definite.
Writing $\ell_\theta=\log p_\theta$,
the score has zero mean, and the Fisher metric is
\begin{equation}
 g_{ij}(\theta)
 =\mathbb E_\theta[\partial_i\ell_\theta\,\partial_j\ell_\theta]
 =-\mathbb E_\theta[\partial_i\partial_j\ell_\theta].
 \label{eq:ig-fisher}
\end{equation}
It measures the leading change in relative entropy:
\begin{equation}
 D_{\mathrm{KL}}(p_\theta\Vert p_{\theta+h})
 =\frac12\sum_{i,j}g_{ij}(\theta)h^ih^j
   +o(\lVert h\rVert^2).
 \label{eq:ig-kl-local}
\end{equation}
Thus the Fisher--Rao distance captures local statistical separation;
it does not encode the spatial cost of transporting mass
\cite{amari2016}.

For a parameter displacement $h$, its squared Fisher length is the
variance of the directional score $\sum_i h^i\partial_i\ell_\theta$.
Directions that barely change the likelihood therefore have small
Fisher length. Equivalently,
$g_{ij}=4\int\partial_i\sqrt{p_\theta}\,\partial_j\sqrt{p_\theta}\,dx$:
the metric is the pullback of the $L^2$ metric under the square-root
embedding $p\mapsto2\sqrt p$.

Information geometry supplements the metric with dual affine
connections. Let
$T_{ijk}=\mathbb E_\theta[
\partial_i\ell_\theta\partial_j\ell_\theta\partial_k\ell_\theta]$
be the Amari--Chentsov tensor. The $\alpha$-connections satisfy
\begin{equation}
 \Gamma^{(\alpha)}_{ij,k}
 =\Gamma^{(0)}_{ij,k}-\frac{\alpha}{2}T_{ijk},
 \label{eq:ig-alpha-connection}
\end{equation}
where the indices are lowered using $g$ and $\nabla^{(0)}$ is its
Levi--Civita connection. The pair $\nabla^{(\alpha)}$,
$\nabla^{(-\alpha)}$ is dual with respect to $g$; $\alpha=1$ and
$\alpha=-1$ give the exponential and mixture connections, respectively.
Their geodesics generally differ from Fisher--Rao geodesics. In a regular
minimal exponential family, the natural parameter $\theta$ and expectation
parameter $\eta=\nabla\psi(\theta)$ are dual affine coordinates, with
$g=\nabla^2\psi$ for the log-partition function $\psi$. This is the basic
example of a dually flat statistical manifold.
Dual flatness concerns the exponential and mixture connections; it
does not assert that the Fisher--Rao Riemannian curvature vanishes.

These structures are also generated by divergences. With the convention
\begin{equation}
 D^{(\alpha)}(p\Vert q)
 =\frac{4}{1-\alpha^2}
  \left(1-\int p^{(1-\alpha)/2}q^{(1+\alpha)/2}\,dx\right),
 \qquad \alpha\ne\pm1,
 \label{eq:alphadiv}
\end{equation}
the endpoint limits are
$D^{(-1)}(p\Vert q)=D_{\mathrm{KL}}(p\Vert q)$ and
$D^{(1)}(p\Vert q)=D_{\mathrm{KL}}(q\Vert p)$.
Here $D_{\mathrm{KL}}(p\Vert q)=\int p\log(p/q)\,dx$.
Under suitable integrability assumptions, the same local quadratic
expansion as in \eqref{eq:ig-kl-local} holds for every $\alpha$.
Different divergences can therefore share the Fisher metric while
assigning different costs to distributions farther apart
\cite{amari2016}.
Their third derivatives determine the dual connections, so changing
$\alpha$ modifies the associated affine geometry while leaving the
infinitesimal metric unchanged. This separation of metric and connection
is useful when comparing estimation objectives that have the same local
sensitivity but different global behaviour.

\subsection{Schr\"odinger Bridges as Entropy Projections}
\label{sec:ig-projection}

Let $R_{01}$ be the endpoint joint law of a reference process and
$\Pi(\mu_0,\mu_1)$ the couplings of the prescribed endpoint laws.
The static Schr\"odinger problem is
\begin{equation}
 \pi^*\in\underset{\pi\in\Pi(\mu_0,\mu_1)}{\operatorname{argmin}}
 D_{\mathrm{KL}}(\pi\Vert R_{01}),
 \qquad
 d\pi^*(x,y)=f(x)g(y)\,dR_{01}(x,y).
 \label{eq:ig-static-bridge}
\end{equation}
The factorisation holds under the existence and regularity assumptions
for the Schr\"odinger system discussed earlier \cite{leonard2014}.
The marginal constraints are affine in the coupling, so this minimiser
is an information, or $I$-, projection of $R_{01}$ onto their intersection.

Its geometric meaning is expressed by a Pythagorean identity.
Assume the factors are positive and their logarithms are integrable
against the prescribed marginals. For every feasible coupling for which
the relative entropies below are finite,
\begin{equation}
 D_{\mathrm{KL}}(\pi\Vert R_{01})
 =D_{\mathrm{KL}}(\pi\Vert\pi^*)
  +D_{\mathrm{KL}}(\pi^*\Vert R_{01}).
 \label{eq:ig-pythagorean}
\end{equation}
Indeed, $\log(d\pi^*/dR_{01})=\log f(x)+\log g(y)$ has the same
integral under every coupling with these marginals. Subtracting the
three entropies proves the identity, which also gives uniqueness of the
minimiser.

The path-space formulation is recovered by retaining the reference
conditional bridges. If $R^{xy}$ denotes the reference law conditioned
on its endpoints, then
$P^*=\int R^{xy}\,\pi^*(dx\,dy)$.
The entropy chain rule shows that changing these conditional laws adds
a nonnegative conditional entropy \cite{leonard2014}. Thus the static
projection selects the endpoint coupling, while the reference process
supplies the conditional dynamics between endpoints.

Iterative proportional fitting alternately projects onto the
two marginal constraint sets. For a strictly positive discrete reference
matrix, these projections are precisely the row and column rescalings
of the Sinkhorn algorithm; its geometric formulation is discussed in
\cite{Modin2024}.

This projection interpretation does not make the bridge's marginal
curve a Fisher--Rao geodesic. Nor is an entropic interpolation generally
a $W_2$-geodesic. For Brownian reference dynamics, displacement
interpolation is recovered in the appropriate small-noise limit
\cite{leonard2014}. This global transport limit and the local Fisher
expansion of relative entropy are different statements.

\subsection{Transport Information Geometry}
\label{sec:ig-transport}

Transport information geometry combines density-space transport with
statistical optimisation. For a regular parametrised family in
$\mathcal P_2(\mathbb R^d)$ with finite Wasserstein tangent norms, pulling
back the $W_2$ metric produces a parameter metric $G_W(\theta)$.
When this matrix is invertible, the associated natural-gradient flow
is $\dot\theta=-G_W(\theta)^{-1}\nabla_\theta\mathcal F(p_\theta)$
\cite{LiMontufar2018}. Using the Fisher matrix instead gives
the Fisher natural gradient. The metric choice therefore changes the
meaning of steepest descent, even when the objective is unchanged.

More explicitly, transport potentials represent parameter variations by
\begin{equation}
 \begin{aligned}
  -\nabla\!\cdot(p_\theta\nabla\phi_i)&=\partial_i p_\theta,\\
  (G_W)_{ij}(\theta)&=\int
    \nabla\phi_i\!\cdot\nabla\phi_j\,p_\theta\,dx.
 \end{aligned}
 \label{eq:ig-wasserstein-pullback}
\end{equation}
These equations require suitable decay or no-flux boundary conditions
and solvability of the weighted Poisson problems.
They encode how each parameter variation redistributes probability mass.
The same density family can therefore carry different metrics according
to whether perturbations are compared through likelihood or through
transport. Parametrisation turns these geometric choices into matrices
that can be used in optimisation.

Wang and Li extend this viewpoint to accelerated information gradient
flows \cite{wangLi2022}. For the Wasserstein metric, a smooth positive
density and potential velocity $v_t=\nabla\Phi_t$ satisfy
\begin{equation}
 \begin{aligned}
  \partial_t\rho_t+\nabla\!\cdot(\rho_t v_t)&=0,\\
  \partial_t v_t+(v_t\!\cdot\nabla)v_t+\gamma_t v_t
   &=-\nabla\frac{\delta\mathcal F}{\delta\rho}(\rho_t),
 \end{aligned}
 \label{eq:ig-accelerated-flow}
\end{equation}
with damping $\gamma_t\ge0$. Here
$\partial_t v_t+(v_t\!\cdot\nabla)v_t$ is the Wasserstein covariant
acceleration, rather than a curvature tensor. Convergence statements
require the convexity and regularity hypotheses of the accelerated-flow
theory. The same notion of acceleration appears in the Schr\"odinger
bridge Newton equation derived earlier, where the driving term is the
Wasserstein gradient of the Fisher information functional
\cite{conforti2019}.

Li also develops transport-Hessian metrics and transport Bregman
divergences \cite{li2021Hessian,li2021bregman}. A Wasserstein Hessian
measures variation of a functional's gradient along transport directions;
using it as a metric requires suitable positivity and nondegeneracy.
These constructions relate information functionals to transport
dynamics without identifying the Fisher and Wasserstein metrics.
For L\'evy models, the same distinction separates parameter
distinguishability in a specified statistical observation model from
nonlocal transport of probability densities.

\subsection{Connections to Generative Modelling}
\label{sec:ig-generative}

Generative models also describe curves of probability laws, but differ
in how the evolution is selected. Score-based diffusion models prescribe
a forward noising process and learn the score
$s_t=\nabla\log\rho_t$ used in its reverse-time dynamics
\citep{SongEtAl2021}. Score estimation alone does not impose a separately
prescribed two-marginal entropy-minimisation problem.

The Schr\"odinger bridge adds that selection principle: among path laws
with the required endpoint marginals, it minimises relative entropy
with respect to a reference process. Diffusion Schr\"odinger bridge
algorithms approximate iterative proportional fitting through
alternating forward and backward regressions
\citep{DeBortoliEtAl2021}.
They connect score learning to the entropy-projection structure above.
In the construction of De~Bortoli et al.~\citep{DeBortoliEtAl2021},
the first iteration recovers the score-based procedure associated
with the reference noising process, while subsequent iterations
reduce endpoint discrepancies.

Flow matching learns a velocity for a chosen probability path satisfying
the continuity equation, typically through regression against tractable
conditional velocities \citep{LipmanEtAl2023}. The continuity equation
is the density evolution law for the learnt deterministic transport.
Diffusion paths and displacement paths can both supply conditional
training constructions, depending on the chosen coupling. The equation
ensures consistency of the ideal density and velocity, but velocity
regression alone does not select a Wasserstein geodesic or an entropic
interpolation. Such an identification requires the corresponding
optimal-transport or entropy-minimisation principle.


Transport geometry also guides computational methods on probability
distributions.

\noindent\textbf{Nonlocal Hamiltonian transport.}
Zheng, Gao and Duan~\cite{ZhengGaoDuan2026Nonlocal} combine
logarithmic-mean mobility, symmetric L\'evy-type edges and a
Kullback--Leibler energy in a deterministic, damped Hamiltonian flow
on a sparse graph.
Near a positive target, the linearised modes satisfy
\[
  \ddot a_k+\gamma\dot a_k+\lambda_k a_k=0,
  \qquad \gamma_*=2\sqrt{\lambda_1},
\]
where $\lambda_1$ is the smallest positive weighted-graph eigenvalue
and $\gamma_*$ optimises the slowest-mode asymptotic decay rate.
For the reported graphs, stronger nonlocality enlarges the spectral
gap, while the relative benefit of inertia depends on that gap
(Figure~\ref{fig:book-nonlocal-spectral}). Convergence is analysed
locally; low-dimensional experiments report improved mode balance.

\begin{figure}[H]
  \centering
  \includegraphics[width=\linewidth]{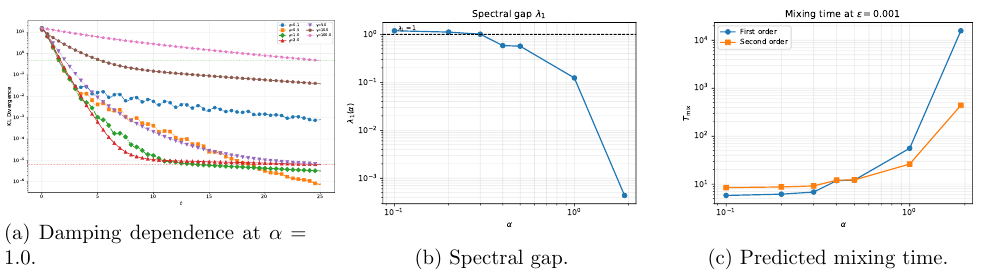}
  \caption{Nonlocal transport: (a) numerical KL histories of the
  nonlinear flow for different damping coefficients at $\alpha=1$;
  (b) computed sparse-graph
  spectral gaps; (c) first- and second-order settling-time
  \emph{predictions from the linearised modal equations}, at tolerance
  $10^{-3}$. Reproduced from \cite{ZhengGaoDuan2026Nonlocal}.}
  \label{fig:book-nonlocal-spectral}
\end{figure}

\noindent\textbf{Wasserstein trust-region optimisation.}
Wan, Gao and Duan~\cite{WanGaoDuan2026WTR} use second variations of
smooth potential--interaction energies to construct quadratic models.
Truncated conjugate gradients compute a trial displacement satisfying
\[
  \rho_k^+=(\mathrm{Id}+s_k)_\#\rho_k,
  \qquad
  W_2(\rho_k,\rho_k^+)
  \leq\|s_k\|_{L^2(\rho_k)}
  \leq\Delta_k.
\]
The ratio of actual to predicted energy decrease determines acceptance
and updates the radius. For $C^2$ lower-bounded external and even
interaction potentials with globally bounded Hessians, the prescribed
trust-region scheme reaches gradient tolerance $\varepsilon$ in
$O(\varepsilon^{-2})$ total trials, including rejections.
This first-order stationarity guarantee applies to empirical measures.
For equal-weight quadrature (Figure~\ref{fig:book-wtr-speedup}), at
RMSE $10^{-4}$ over $75$ Gaussian test functions, reported median
speedups are about $1.8$ over L-BFGS and $1.3$ over a Krylov
trust-region solver.

\begin{figure}[H]
  \centering
  \includegraphics[width=0.54\linewidth]{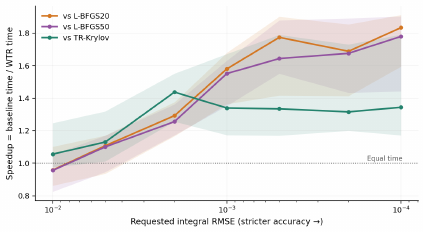}
  \caption{Median baseline time divided by WTR time at prescribed
  integration RMSE, for two-moons and ring targets with
  $N=64,128,256$ nodes. Bands show interquartile ranges; values
  above one favour WTR. Reproduced from
  \cite{WanGaoDuan2026WTR}.}
  \label{fig:book-wtr-speedup}
\end{figure}

\section*{Problems}
\addcontentsline{toc}{section}{Problems}


\begin{problem}[$\alpha$-divergences unify the KL pair]
\label{prob:alphadiv}
For the convention \eqref{eq:alphadiv}, prove that
\[
\lim_{\alpha\to-1}D^{(\alpha)}(p\|q)=\KL(p\|q),\qquad
\lim_{\alpha\to1}D^{(\alpha)}(p\|q)=\KL(q\|p).
\]
If $p$ is the target density and $q=q_\theta$ is the fitted model,
explain why the $\alpha\to-1$ limit corresponds to minimising the
forward KL (typically mass/mode-covering), whereas the $\alpha\to1$
limit corresponds to the reverse KL (typically mode-seeking). \\
\emph{Hint:} Set $u=(1-\alpha)/2$ and expand at $u=0$ and $u=1$.
\end{problem}

\begin{problem}[The quantum--stochastic duality]\label{prob:quantum}
Apply the Wick rotation $t\mapsto it$ to the Euclidean pair
\eqref{eq:euclideanpair} and show that it becomes the forward/backward
quantum Schr\"odinger equations, with
$\rho=\varphi\hat\varphi$ turning into Born's rule $|\psi|^2$. What
does the duality assert about most probable evolutions of classical
clouds versus quantum wave mechanics? \\
\emph{Hint:} Track how diffusion coefficient acquires a factor of
$i$ under rotation, turning the real, dissipative forward-backward
system \eqref{eq:euclideanpair} into a unitary, conservative pair;
compare the continuity equation for $|\psi|^2$ with the
Fokker-Planck equation for $\rho$.
\end{problem}

\begin{problem}[Bending the metric]\label{prob:bending}
Explain, at the level of the metric tensor, how the ``speed limit'' for
transporting mass between two separated clusters changes when the
local Otto geometry of Sect.~\ref{subsec:local-otto} is replaced by the
nonlocal geometry of Definition~\ref{def:nlW} with an
$\alpha$-stable kernel.  \\
\emph{Hint:} In the local metric, transport cost grows with the square
of the distance moved, forcing mass to pass continuously through every
intermediate point; the nonlocal kernel instead allows a jump cost that
stays finite over long range, lowering the effective speed limit
between well-separated clusters. Let the stability index of the kernel
approach the diffusive limit and check that
Theorem~\ref{thm:localization} recovers the local metric.
\end{problem}

\begin{problem}[The chain rule forces the logarithmic mean]
\label{prob:chainrule}
Prove identity \eqref{eq:chainrule} and deduce
Theorem~\ref{thm:masterGF}. Then show the rigidity statement: if the
mobility $w_{xy}=\theta(\rho_x,\rho_y)q_{xy}$ with a continuous
symmetric mean $\theta$ makes the master equation \eqref{eq:master} the
gradient flow of $H(\cdot\,|\,\pi)$, then $\theta=\Lam$. \\
\emph{Hint:} Differentiate $H(\rho_t\,|\,\pi)$ along
\eqref{eq:master} and match the result term-by-term against the
gradient-flow structure; the chain rule for $\log$ forces $\theta$ to
satisfy a functional equation whose only continuous symmetric solution
is the logarithmic mean $\Lam$.
\end{problem}


\begin{problem}[Fisher metric from $\alpha$-divergences]
\label{prob:fisherfromalpha}
Starting from the family $D^{(\alpha)}(p\|q)$ of
Sect.~\ref{sec:ig-foundations} , show that for any fixed
$\alpha$, the second-order Taylor expansion of
$D^{(\alpha)}(p_\theta\|p_{\theta+d\theta})$ in $d\theta$ recovers the
Fisher information metric, independent of $\alpha$, while the
third-order term recovers the $\alpha$-connection. Explain why this
shows the Fisher metric is the unique Riemannian object common to the
entire $\alpha$-family. \\
\emph{Hint:} Expand $D^{(\alpha)}(p_\theta\|p_{\theta+d\theta})$ in
powers of $d\theta$; the first-order term vanishes, the second-order
term is $\alpha$-independent by construction of \eqref{eq:alphadiv},
and the $\alpha$-dependence first appears at third order, matching the
definition of the dual $\alpha$-connections.
\end{problem}

\bibliographystyle{plain}
\bibliography{ref}

\backmatter


\end{document}